\documentclass{article}
\usepackage{amsmath}
\usepackage{latexsym}
\usepackage{amssymb}
\newtheorem{thm}{Theorem}[section]
\newtheorem{la}[thm]{Lemma}
\newtheorem{Defn}[thm]{Definition}
\newtheorem{Remark}[thm]{Remark}
\newtheorem{Conj}[thm]{Conjecture}
\newtheorem{prop}[thm]{Proposition}
\newtheorem{cor}[thm]{Corollary}
\newtheorem{Example}[thm]{Example}
\newtheorem{Number}[thm]{\!\!}
\newenvironment{defn}{\begin{Defn}\rm}{\end{Defn}}
\newenvironment{example}{\begin{Example}\rm}{\end{Example}}
\newenvironment{rem}{\begin{Remark}\rm}{\end{Remark}}

\newenvironment{numba}{\begin{Number}\rm}{\end{Number}}
\newenvironment{proof}{{\noindent\bf Proof.}}%
                  {\nopagebreak\hspace*{\fill}$\Box$\medskip\par}

\newcommand{\cW}{{\mathcal W}}
\newcommand{\cV}{{\mathcal V}}
\newcommand{\cO}{{\mathcal O}}
\newcommand{\cK}{{\mathcal K}}
\newcommand{\cE}{{\mathcal E}}
\newcommand{\cL}{{\mathcal L}}
\newcommand{\cC}{{\mathcal C}}
\newcommand{\cA}{{\mathcal A}}
\newcommand{\cF}{{\mathcal F}}
\newcommand{\cB}{{\mathcal B}}
\newcommand{\cS}{{\mathcal S}}
\newcommand{\cU}{{\mathcal U}}
\newcommand{\sbull}{{\scriptstyle\bullet}}
\newcommand{\ve}{\varepsilon}
\newcommand{\R}{{\mathbb R}}
\newcommand{\N}{{\mathbb N}}
\newcommand{\Z}{{\mathbb Z}}
\newcommand{\Sph}{{\mathbb S}}
\newcommand{\mto}{\mapsto}
\newcommand{\sub}{\subseteq}
\newcommand{\wt}{\widetilde}

\newcommand{\wb}{\overline}
\DeclareMathOperator{\id}{id}
\DeclareMathOperator{\Spann}{span}

\DeclareMathOperator{\Pol}{Pol}
\DeclareMathOperator{\Sym}{Sym}
\DeclareMathOperator{\Supp}{supp}
\DeclareMathOperator{\pr}{pr}
\DeclareMathOperator{\Ext}{Ext}

\DeclareMathOperator{\ex}{ex}

\DeclareMathOperator{\ev}{ev}
\DeclareMathOperator{\Sem}{Sem}
\DeclareMathOperator{\im}{im}
\DeclareMathOperator{\reg}{reg}
\DeclareMathOperator{\conv}{conv}

\DeclareMathOperator{\fb}{fb}
\newcommand{\tensor}{\hspace*{.5mm} \wt{\otimes}_\ve\hspace*{.2mm}}
\newcommand{\dl}{{\displaystyle \lim_{\longrightarrow}}}
\newcommand{\pl}{{\displaystyle \lim_{\longleftarrow}}}
\begin{document}
\begin{center}
{\bf\Large Smoothing operators for vector-valued functions\\[1.5mm]
and extension operators}\\[4mm]
{\bf Helge Gl\"{o}ckner}\vspace{2mm}
\end{center}
\begin{abstract}
Let $\Omega\sub\R^d$ be open
and $\ell\in \N_0\cup\{\infty\}$.
Given a locally convex topological vector space~$F$,
endow $C^\ell(\Omega,F)$ with the compact-open $C^\ell$-topology.
For $\ell<\infty$, we describe a sequence $(S_n)_{n\in\N}$
of continuous linear operators $S_n\colon C^\ell(\Omega,F)\to C^\infty(\Omega,F)$
such that
$S_n(\gamma)\to\gamma$ in $C^\ell(\Omega,F)$ as $n\to\infty$ for each
$\gamma\in C^\ell(\Omega,F)$,
and
$S_n(\gamma)\in F\otimes C^\infty_c(\Omega,\R)$ for
each $n\in\N$.
We\vspace{.3mm} show that $C^\ell(R,F)=F\tensor C^\ell(R,\R)$
for each $\ell\in\N_0\cup\{\infty\}$, complete locally convex space~$F$
and convex subset
$R\sub\R^d$ with dense interior. Moreover, we study the
existence of continuous linear right inverses
for restriction maps $C^\ell(\R^d,F)\to C^\ell(R,F)$,
$\gamma\mto\gamma|_R$ for subsets
$R\sub\R^d$ with dense interior.
As an application, we construct continuous linear
right inverses for restriction operators
between spaces of sections in vector bundles
in many situations, and smooth local right inverses 
for restriction operators between manifolds of mappings.
We also obtain smoothing results for sections in fibre bundles.
\end{abstract}
\section{Introduction and statement of the results}
Let $d\in \N$, $\Omega\sub\R^d$ be an open subset,
and $F$ be a locally convex, Hausdorff, real topological vector space
(which need not be complete).
Given $\ell\in\N_0\cup\{\infty\}$,
we let $C^\ell(\Omega,F)$ be the vector
space of all continuous functions $\gamma\colon \Omega\to F$
such that the partial dervatives $\partial^\alpha\gamma\colon \Omega\to F$
exist for all multi-indices $\alpha=(\alpha_1,\ldots,\alpha_d)\in(\N_0)^d$
with $|\alpha|:=\alpha_1+\cdots+\alpha_d\leq  \ell $,
and are continuous functions.\\[2.3mm]
A subset $R\sub X$ of a topological space~$X$
is said to be a \emph{regular subset} if
its interior $R^0$ (relative~$X$) is dense in~$R$.
Given a regular subset $R\sub\R^d$,
we let $C^\ell(R,F)$ be the vector space of all
continuous functions $\gamma\colon R\to F$ whose restriction
$\gamma|_{R^0}$ is in $C^\ell(R^0,F)$, and such that
$\partial^\alpha(\gamma|_{R^0})\colon R^0\to F$
extends to a (necessarily unique) continuous function
$\partial^\alpha\gamma \colon R\to F$ for all $\alpha\in\N_0^d$ such that $|\alpha|\leq \ell$.
We give $C^\ell(R,F)$ the compact-open $C^\ell$-topology, viz.\
the initial topology with respect to the maps $C^\ell(R,F)\to C(R,F)$, $\gamma\mto\partial^\alpha\gamma$
for $|\alpha|\leq\ell$,
where the space $C(R,F)$
of continuous functions carries the compact-open topology.\footnote{Which coincides with the topology of uniform convergence on
compact subsets of~$R$.}\\[2.3mm]
Standard arguments show that the linear map
\[
\theta\colon F\otimes C^\ell(R,\R)\to C^\ell(R,F)
\]
determined by $\theta(v\otimes \gamma)=\gamma v\colon R\to F$, $x\mto \gamma(x)v$
is injective, and we write $F\otimes C^\ell(R,\R)$ for its image.
If $\gamma\in C^\ell(R,F)$, then $\gamma\in F\otimes C^\ell(R,\R)$
if and only if the $\R$-linear span of $\gamma(R)$ in~$F$
has finite dimension.\footnote{The necessity of this condition is clear.
If the span $E$ has finite dimension, let $b_1,\ldots, b_n$ be a basis for~$E$
and $\lambda_1,\ldots,\lambda_n$ be continuous linear functionals
$F\to\R$ such that $v=\sum_{j=1}^n\lambda_j(v)b_j$ for all $v\in E$.
Then $\gamma=\sum_{j=1}^n b_j\otimes (\lambda_j \circ\gamma)$.}\\[2.3mm]
For an open set $\Omega\sub\R^d$,
it is well known that $F\otimes C^\infty(\Omega,\R)$
is dense in $C^\ell(\Omega,F)$ (see, e.g.,
\cite[Proposition~44.2]{Tre}).
We obtain a finer result concerning smoothing operators.
Given a compact set $K\sub\Omega$,
let $C^\ell_K(\Omega,F)\sub C^\ell(\Omega,F)$ be the vector subspace
of all $\gamma\in C^\ell(\Omega,F)$ with support $\Supp(\gamma)\sub K$.
Recall that a sequence $(K_n)_{n\in\N}$
of compact subsets of a locally compact topological space~$X$
is called a \emph{compact exhaustion} of~$X$ if $X=\bigcup_{n\in\N}K_n$
holds and $K_n$ is contained in the interior~$K_{n+1}^0$
for each $n\in\N$.\vspace{.5mm}
\begin{thm}\label{thmsmoo}
Let $\Omega\sub \R^d$ be an open set,
$F$ be a
locally convex space,
$K_1\sub K_2\sub\cdots$ be a compact exhaustion of~$\Omega$,
and $\ell\in \N_0$.
There exist continuous linear operators
\[
S_n\colon C^\ell(\Omega,F)\to C^\infty(\Omega,F)
\]
for $n\in\N$ with the following properties:
\begin{itemize}
\item[\rm(a)]
$S_n(\gamma)\to\gamma$ in $C^\ell(\Omega, F)$ as $n\to\infty$,
for all $\gamma\in C^\ell(\Omega,F)$;
\item[\rm(b)]
$S_n(\gamma)\in F\otimes C^\infty_{K_{n+1}}(\Omega,\R)$
for all $n\in\N$ and $\gamma\in C^\ell(\Omega,F)$; and
\item[\rm(c)]
$S_m(\gamma)\in F\otimes C^\infty_{K_{n+1}}(\Omega,\R)$
for all $n\in\N$, $m\geq n$, and
$\gamma\in C^\ell_{K_n}(\Omega,F)$.\vspace{1.5mm}
\end{itemize}
\end{thm}
The convergence in~(a) is uniform for $\gamma$ in compact subsets
of $C^\ell(\Omega,F)$ (see Remark~\ref{oncpset}).
Compact convergence of smoothing operators is essential
for the proof of our main smoothing result in fibre bundles
(Theorem~\ref{thm-steen}).\\[2.3mm]
We mention that the smoothing operators are not constructed with the help
of convolutions (as in other approaches, like \cite[Proposition~44.2]{Tre}).
Rather, we use $\ell$th order Taylor polynomials
around suitable points and mix them using a
partition of unity. Note that~$F$ need not satisfy any completeness
properties.\\[2.3mm]
Our notation concerning topological tensor
products of locally convex spaces is as in~\cite{Tre}.
Recall that a locally convex space~$F$ is called \emph{complete}
if each Cauchy net in~$F$ is convergent; it is called
\emph{sequentially complete}
if each Cauchy sequence converges.
Using Theorem~\ref{thmsmoo}, we show:
\begin{prop}\label{etensor}
Let $F$ be a complete locally convex space, $\ell\in \N_0\cup\{\infty\}$,
$d\in\N$ and $R\sub \R^d$ be a convex subset with dense interior.
Then
\[
C^\ell(R,F)\;\, \cong \; \, F\hspace*{.3mm}\tensor \hspace*{.2mm} C^\ell(R,\R)
\]
as a locally convex space.
\end{prop}
The proposition and its proof are analogous
to the well-known case of open domains (see, e.g., \cite[Theorem~44.1]{Tre}).
Using the preceding proposition,
we see that continuous linear extension operators
exist for vector-valued functions whenever
they exist for scalar-valued functions
(as long as the range is sequentially complete).
\begin{thm}\label{scalarseq}
Let $\ell\in\N_0\cup\{\infty\}$ and $R\sub\R^d$ be a closed, convex subset with dense interior.
Assume that the restriction map
\[
C^\ell(\R^d,\R)\to C^\ell(R,\R)
\]
has a continuous linear right inverse.
Then the restriction map
\[
C^\ell(\R^d,F)\to C^\ell(R, F)
\]
has a continuous linear right inverse,
for each sequentially complete locally\linebreak
convex space~$F$.
\end{thm}
Given $d\in\N$, let $\|\cdot\|_2$ be the euclidean norm on $\R^d$
and $B_r(x):=\{y\in\R^d \!:$\linebreak
$\|y-x\|_2<r\}$
be the open ball around $x\in\R^d$ of radius $r>0$.
The following\linebreak
definition from~\cite{RaS} (which is formulated there
for closed subsets of metric spaces)\footnote{The authors require
in (b) that $\|y-x\|<\rho\ve^r$ implies $y\in R$ and
$\|z-y\|<\ve$,
but the condition $\|z-y\|<\ve$ is automatic for
suitable choices of $\ve_0$ and $\rho$,
whence we chose to omit it in the formulation.}
combines conditions from~\cite{Bie} and~\cite{Fre}.
\begin{defn}\label{cuspco}
A closed, regular subset $R\sub\R^d$ is said to satisfy the
\emph{cusp condition} if (a) and (b) are satisfied:
\begin{itemize}
\item[(a)]
$R$ has \emph{no narrow fjords},
viz., for each $x\in R$ there exists $n\in\N$,
a compact neighbourhood $K$ of~$x$ in~$R$ and
$C>0$ such that all $y,z\in K$ can be joined by a rectifyable curve~$\gamma$
lying inside~$R^0$ except perhaps for finitely many points,
and the path length $L(\gamma)$ satisfies $L(\gamma)\leq C\sqrt[n]{\|z-y\|_2}$.
\item[(b)]
$R$ has \emph{at worst polynomial outward cusps},
viz., for all compact subsets $K\sub\R^d$, there exist
$\ve_0,\rho>0$ and $r\geq 1$ such that for all $z\in K\cap\partial R$
and $0< \ve\leq\ve_0$, there is an $x\in R$ with $\|x-z\|_2<\ve$
such that $B_{\rho\ve^r}(x)\sub R$.
\end{itemize}
\end{defn}
\begin{rem}\label{RS-ops}
Condition~(a) ensures that $C^\infty(R,\R)$ can be identified
with the locally convex space $\cE(R)$ of smooth Whitney jets on~$R$
(see \cite{Bie}). Due to condition~(b),
a result of Frerick~\cite{Fre}
provides a continuous right inverse for the natural
map $C^\infty(\R^d,\R)\to \cE(R)$. Combining these facts,
one concludes (see \cite[Corollary~3.5]{RaS}):\\[2.3mm]
\emph{If $R\sub\R^d$ is a closed, regular subset which satisfies
the cusp condition, then the restriction operator
\[
C^\infty(\R^d,\R)\to C^\infty(R,\R)
\]
has a continuous linear right inverse.}
\end{rem}
Every convex, closed subset $R\sub \R^d$ with dense
interior satisfies the cusp condition (as recalled in Lemma~\ref{convthencusp}).
Applying Theorem~\ref{scalarseq} to the extension operators by
Roberts and Schmeding described in Remark~\ref{RS-ops},
we obtain:
\begin{cor}\label{convex-ext}
Let $F$ be a sequentially complete locally convex space,
$d\in\N$ and $R\sub\R^d$ be a closed, convex subset with dense
interior. Then the restriction map
\[
C^\infty(\R^d,F)\to C^\infty(R,F)
\]
has a continuous linear right inverse.
\end{cor}
We also have results concerning extensions of functions on infinite-dimensional
domains.
It is possible to define $C^\ell$-functions
$R\to F$ if $R$ is a regular subset of a locally convex space~$E$
(see Definition~\ref{defck}, which follows \cite{Woc}), and to endow $C^\ell(R,F)$
with a compact-open $C^\ell$-topology
(see Definition~\ref{defcktop}).
Extension operators for vector-valued functions to arbitrary
locally convex spaces
were studied in
the recent work~\cite{Han} by Hanusch,
generalizing classical results by Seeley~\cite{See}.
Given $\ell\in \N_0\cup\{\infty\}$,
numbers $-\infty\leq a<\tau<b<\infty$,
locally convex spaces $E$ and $F$ and a regular subset $R\sub E$,
Hanusch defines (rewritten in our notation)\footnote{The open set $V:=R^0$
is dense in~$\cV:=R$, whence a function space $\cC^\ell_\cV((a,b)\times V,F)$
can be defined as in~\cite{Han}. Given $\gamma\in C^\ell(]a,b]\times R,F)$,
we have $f:=\gamma|_V\in C^\ell_\cV((a,b)\times V,F)$
as the differentials $d^{(j)}\gamma\colon \,]a,b]\times R\times (\R\times E)^j$
(as in Proposition~\ref{eqCk})
are continuous extensions (denoted $\Ext(f,j)$ by Hanusch) of the $d^{(j)}f$,
for all $j\in\N_0$ with $j\leq \ell$. If $f\in\cC^\ell_\cV((a,b)\times V,F)$,
then $\gamma:=\Ext(f,0)\in C^\ell(]a,b]\times R,F)$
as $\gamma|_{]a,b[\,\times R^0}=f$ is $C^\ell$
and the $\Ext(f,j)$ are continuous extensions for the $d^{(j)}(\gamma|_{]a,b[\,\times R^0})=d^{(j)}f$.
The restriction map
\[
C^\ell(]a,b]\times R,F)\to \cC^\ell_\cV((a,b)\times V,F),\;\,
\gamma\mto\gamma|_{]a,b[\,\times V}
\]
therefore is an isomorphism of vector spaces with inverse $f\mto\Ext(f,0)$.}
a linear map
\begin{equation}\label{thehanusch}
\cE_\tau\colon C^\ell(]a,b]\times R,F)\to C^\ell(]a,\infty[\,\times R,F)
\end{equation}
which is a right inverse for the restriction
map
\[
C^\ell(]a,\infty[\,\times R,F)\to C^\ell(]a,b]\times R,F).
\]
He obtains estimates for the extended functions
(see \cite[Theorem~1, part~2)]{Han}),
but did not prove continuity of $\cE_\tau$
with respect to the compact-open $C^\ell$-topology
(he broaches continuity questions in \cite[Remark~4]{Han}
but considers them out of the scope of his article).
Using Hanusch's estimates, we shall observe:
\begin{prop}\label{Hancts}
For $\ell\in\N_0\cup\{\infty\}$, Hanusch's extension operators
\[
\cE_\tau\colon C^\ell(]a,b]\times R,F)\to C^\ell(]a,\infty[\,\times R,F)
\]
are continuous if we use the compact-open $C^\ell$-topology
on domain and~range.
\end{prop}
Given $n\in\N$ and $-\infty\leq a_j<\tau_j<b_j<\infty$ for $j\in\{1,\ldots,n\}$,
Hanusch also obtained extension operators
\begin{equation}\label{hanumult}
C^\ell(]a_1,b_1]\times \cdots\times\,]a_n,b_n]\times R,F)
\to C^\ell(]a_1,\infty[\,\times\cdots\times \,]a_n,\infty[\,\times R,F).
\end{equation}
As these are, essentially, compositions of extension operators
of the form~(\ref{thehanusch}), they are continuous
(see Remark~\ref{hanumultrem} for details).
We deduce (see Section~\ref{extcorner}):
\begin{cor}\label{extcorner}
For each $\ell\in\N_0\cup\{\infty\}$,
$d\in\N$, $m\in \{0,\ldots,d\}$
and locally convex space~$F$,
the restriction map
\[
C^\ell(\R^d,F)\to C^\ell([0,\infty[^m\times\R^{d-m},F)
\]
has a continuous linear right inverse.
Moreover, the restriction map
\[
C^\ell(\R^d,F)\to C^\ell([0,1]^d,F)
\]
has a continuous linear right inverse.\vspace{1mm}
\end{cor}
We observe that extension operators for smooth functions on a domain $S$
automatically induce extension operators for smooth functions
on $R\times S$.
\begin{prop}\label{autopara}
Let $d\in \N$, $S\sub \R^d$ be a closed, regular subset,
$E$ and $F$ be locally convex spaces and $R\sub E$ be a
regular subset. If the restriction map
\[
C^\infty(\R^d,F)\to C^\infty(S,F)
\]
admits a continuous linear right inverse $C^\infty(S,F)\to C^\infty(\R^d,F)$,
then also the restriction map
\[
C^\infty(R\times \R^d,F)\to C^\infty(R\times S,F)
\]
admits a continuous linear right inverse.
\end{prop}
An analogous result for $C^\ell$-functions (Proposition~\ref{Cellcompatible})
is more technical and requires additional hypotheses.\\[2.3mm]
While the previous results just discussed were devoted
to functions on open (or more general) subsets of locally
convex spaces, in the second part of the article we turn to
questions concerning vector-valued functions on manifolds,
mappings to Lie groups modelled
on locally convex spaces, and sections in vector bundles
or fibre bundles.\\[2.3mm]
Beyond $C^\ell$-manifolds
modelled on locally convex spaces,
we consider\linebreak
$C^\ell$-manifolds \emph{with rough boundary}
modelled on locally convex, regular subsets
of locally convex spaces,
with chart changes taking interior
to interior (see \cite{GaN}).
The familiar $C^\ell$-manifolds with corners
(as in \cite{Cer},
\cite{Dou}, \cite{Mic})
are a prominent special case.\footnote{Compare
\cite{BAN}
for infinite-dimensional analogues
in the context of continuously Fr\'{e}chet differentiable
mappings between Banach spaces.}
More generally, some results are valid
for \emph{rough $C^\ell$-manifolds}
modelled on arbitrary regular subsets of locally convex spaces
(again with chart changes taking interior
to interior), as in~\cite{Rou},
and are formulated accordingly
(see \ref{defn-rough}
for more information on the preceding concepts).
Our main interest, however,
are ordinary manifolds and their most natural
generalizations, like manifolds
with corners or with rough boundary,
on which many central aspects of mathematical
analysis work very well (cf.\ \cite{GaN}).\\[2.3mm]
Note that every regular subset $R\sub M$
of a $C^\ell$-manifold~$M$ inherits
a rough $C^\ell$-manifold
structure (and likewise if~$M$ is a rough $C^\ell$-manifold).\footnote{These are the `full-dimensional
submanifolds' encountered later,
see Remark~\ref{char-full}.}\\[2.3mm]
Combining our results concerning
extension operators for vector-valued functions with
localization techniques
as in~\cite{RaS}, we obtain results concerning continuous
linear extension
operators between spaces of sections in vector bundles,
and results concerning locally defined smooth extension operators
between manifolds of mappings of the form $C^\ell(M,N)$.
In Section~\ref{sec-outlook}, we describe these findings,
with proofs in Section~\ref{proofs-outlook}.
We record two sample results, with terminology as in
Sections~\ref{sec-outlook} and~\ref{proofs-outlook}.
Notably,
the reader is referred to Definitions~\ref{defn-split-sub}
and \ref{full-submfd} for the relevant concepts of submanifolds.
The first theorem combines Proposition~\ref{aprop1}
and Theorem~\ref{thm-locops}.
For case~(a),
see already~\cite[\S4]{RaS}.
\begin{thm}\label{spec-loco}
Let $\ell\in \N_0\cup\{\infty\}$,
$F$ be a locally convex space and $M$ be a paracompact,
locally compact rough $C^\ell$-manifold.
Let $L\sub M$ be a closed subset.
Assume that at least one of the following conditions
is satisfied.
\begin{itemize}
\item[\rm(a)]
$\ell=\infty$, $M$ is a $\sigma$-compact
Riemannian manifold without boundary, $L\sub M$ a regular subset
satisfying the cusp condition with respect to the
metric arising from a Riemannian metric
$($as in {\rm\cite[Definition~3.1]{RaS})}
and $F$ has finite dimension;
\item[\rm(b)]
$\ell=\infty$,
$L\sub M$ is a full-dimensional submanifold with rough boundary
and $F$ sequentially complete;
\item[\rm(c)]
$L\sub M$ is a full-dimensional submanifold with corners;
or
\item[\rm(d)]
$L\sub M$ is a split rough submanifold.
\end{itemize}
Then the following holds
for each $C^\ell$-vector bundle
$E\to M$ with typical fibre~$F$:
\begin{itemize}
\item[\rm(i)]
The restriction map
\[
\Gamma_{C^\ell}(E)\to \Gamma_{C^\ell}(E|_L)
\]
between spaces of $C^\ell$-sections
admits a continuous linear right inverse;
\item[\rm(ii)]
The restriction map
\[
\Gamma_{C^\ell_c}(E)\to \Gamma_{C^\ell_c}(E|_L)
\]
admits a continuous linear right inverse
$($for the spaces of compactly supported $C^\ell$-sections in the vector bundles$)$.
\end{itemize}
\end{thm}
\begin{numba}
Let $\ell\in\N_0\cup\{\infty\}$.
If $M$ is a paracompact, locally compact, rough $C^\ell$-manifold
(or a paracompact, locally compact topological space, if $\ell=0$),
and $N$ a smooth manifold modelled on a locally convex space~$F$
such that $N$ admits a local addition (see, e.g., \cite{AGS}
for this concept),
then $C^\ell(M,N)$ admits a smooth manifold structure
independent of the local addition (see \cite{Rou}).
In fact, in the case of compact~$M$,
the manifold structure on $C^\ell(M,N)$
can be constructed as in~\cite{AGS},
noting that the vital exponential laws for manifolds
with rough boundary in
\cite{AaS} remain valid for rough manifolds
(the proof given in \cite{GaN} can be adapted
with minor modifications, as explained in~\cite{Rou}).
For general~$M$, let $(K_j)_{j\in J}$
be a locally finite family of compact, regular subsets of~$M$
whose interiors~$K_j^0$ cover~$M$;
then the image $S:=\rho(C^\ell(M,N))$ of the injective map
\[
\rho\colon C^\ell(M,N)\to{\prod_{j\in J}}^{\fb}C^\ell(K_j,N),\quad
\gamma\mto(\gamma|_{K_j})_{j\in J}
\]
is a submanifold of the fine box product
of the smooth manifolds $C^\ell(K_j,N)$,
and we endow $C^\ell(M,N)$ with the smooth manifold structure
turning $\rho|^S$ into a $C^\infty$-diffeomorphism
(see \cite{Rou} for details).
\end{numba}
\begin{numba}
We here use the following concept: If $(M_j)_{j\in J}$
is any family of smooth manifolds modelled
in locally convex spaces, we endow the
cartesian product $P:=\prod_{j\in J}M_j$ with the
so-called \emph{fine box topology} $\cO_{\fb}$,
the final topology with respect to the mappings
\begin{equation}\label{para-box}
\psi_\phi \colon \left(\bigoplus_{j\in J}E_j\right)\cap\prod_{j\in J}U_j\to P,
\;\, (x_j)_{j\in J}\to(\phi_j^{-1}(x_j))_{j\in J},
\end{equation}
for $\phi:=(\phi_j)_{j\in J}$ ranging through the
families of charts $\phi_j\colon U_j\to V_j\sub E_j$
of~$M_j$ such that $0\in V_j$;
here $E_\phi:=\bigoplus_{j\in J}E_j$
is endowed with the locally convex direct sum topology,
and the left-hand side $V_\phi$ of (\ref{para-box}),
which is an open subset of~$E_\phi$,
is endowed with the topology induced by~$E_\phi$.
Then $U_\phi:=\psi_\phi(V_\phi)$ is open in~$P$,
$\psi_\phi$ is a homeomorphism onto its image
and the maps $(\psi_\phi|^{U_\phi})^{-1}\colon U_\phi\to V_\phi\sub E_\phi$
are smoothly compatible (as a consequence of
\cite[Proposition~7.1]{Mea})
and hence form an atlas for a $C^\infty$-manifold
structure on~$P$. Following~\cite{Rou},
we write $P^{\fb}$ for~$P$, endowed with the topology~$\cO_{\fb}$
and the smooth manifold structure just described,
and call $P^{\fb}$ the \emph{fine box product}.\footnote{An
analogous construction works for families $(M_j)_{j\in J}$
of $C^\ell$-manifolds for finite $\ell$,
with a possible loss of one order
of differentiability if $J$ is uncountable (see \cite{Rou}).}
\end{numba}
As to the smooth manifold structure on $C^\ell(M,N)$,
see \cite{AGS} for the special case
where $M$ is a compact $C^\infty$-manifold with rough boundary
(or already \cite{RaS} if $\ell=\infty$ and~$N$ is a Fr\'{e}chet manifold);
\cite{Mic} for the case where $M$ is a paracompact, finite-dimensional
smooth manifold with corners and $\dim(N)<\infty$
(for compact manifolds~$M$
and finite-dimensional~$N$,
see also \cite{Eel} and \cite{Wit};
for discussions in the convenient setting,
see \cite{KaM}).\\[2.3mm]
The concept of a submersion is recalled
in Definition~\ref{def-subm}.
Theorem~\ref{thm-mapmfd} subsumes the following result.
\begin{thm}\label{spec-mapma}
Let $\ell\in \N_0\cup\{\infty\}$
and $M$ be a paracompact, locally compact, rough $C^\ell$-manifold.
Let $L\sub M$ be a closed subset and $N$ be a smooth manifold modelled
on a locally convex space~$F$. If $N$ admits a smooth local addition,
then the restriction map
\[
C^\ell(M,N)\to C^\ell(L,N),\quad \gamma\mto\gamma|_L
\]
is a smooth submersion in each of the following cases:
\begin{itemize}
\item[\rm(a)]
$\ell=\infty$, $M$ is a Riemannian manifold without boundary,
$L\sub M$ is a regular subset satisfying the cusp condition,
and $\dim(F)<\infty$ $($cf.\ {\rm\cite{RaS}}$)$;
\item[\rm(b)]
$\ell=\infty$,
$L$ is a full-dimensional submanifold with rough boundary of~$M$
and $F$ is sequentially complete;
\item[\rm(c)]
$L$ is a full-dimensional submanifold with corners of~$M$; or
\item[\rm(d)]
$L\sub M$ is a split rough submanifold. 
\end{itemize}
\end{thm}
We mention that the conclusions of Theorems~\ref{spec-loco}
and \ref{spec-mapma} remain valid if $\ell=0$,
$M$ is any metrizable, locally compact topological
space and $L\sub M$ any closed subset
(see Proposition~\ref{use-dugu}).
This is a consequence of the
Dugundji
Extension Theorem~\cite{Dug},
which establishes a linear
extension operator
\[
\cE\colon C(Y,F)\to C(X,F)
\]
for each locally convex space~$F$,
metric space~$X$ and closed subset $Y\sub X$.
As the image of $\cE(f)$ is contained
in the convex hull of $f(Y)$ for each $f\in C(Y,F)$,
it is clear that $\cE$ restricts to a continuous linear map
\begin{equation}\label{dudubc}
BC(Y,F)\to BC(X,F)
\end{equation}
between spaces of bounded continuous functions,
endowed with the topology of uniform convergence
(as already observed in \cite{Dug}
for the special case $F=\R$).\\[2.3mm]
We obtain a variant of Dugundji's extension
operators by a simple construction, which can dispense
with the complicated geometry pervading the work~\cite{Dug}.
In the simplified approach, we can establish continuity with respect to
the compact-open topology under
mild hypotheses:
\begin{prop}\label{own-dug}
Let $(X,d)$ be a metric space and $Y\sub X$
be a closed subset. If $Y$ is locally compact
or $Y$ is complete in the metric
induced by~$d$,
then the restriction map
\[
C(X,F)\to C(Y,F)
\]
admits a continuous linear right inverse,
for each locally convex space~$F$.
\end{prop}
The conclusion of Proposition~\ref{own-dug}
remains valid if $X$ is any paracompact topological
space and $Y\sub X$ a closed, locally compact, metrizable
subset (see Corollary~\ref{cor-triv-bun}),
or if $X$ is a completely regular topological space
and $Y\sub X$ a compact, metrizable subset (see Corollary~\ref{unif-cor}).\\[2.3mm]
Extension operators also play a role in the approach
to manifolds of mappings pursued in~\cite{Con}.\\[2.3mm]
Using Theorem~\ref{thmsmoo},
we show that the test function group
$C^\infty_c(M,G)$ is dense in the Lie group
$C_0(M,G)$ of continuous functions vanishing at infinity
for each finite-dimensional smooth manifold~$M$
and Lie group~$G$ modelled on a locally convex space,
and (for paracompact~$M$)
also in the Lie group $C_c^\ell(M,G)$ for all
$\ell\in\N_0$
(see Propositions~\ref{ctscaseprop}
and~\ref{densetf}, which include results
for locally compact rough manifolds and locally compact
manifolds with rough boundary).
Moreover, $C^\infty_c(\R^d,G)$
is dense in the Lie group $\cS(\R^d,G)$
of rapidly decreasing $G$-valued smooth mappings
on~$\R^d$ for all $d\in\N$ (see Remark~\ref{denserapid}).\footnote{More generally, we study density of $C^\infty_c(\Omega,G)$
in Walter's weighted mapping groups on an open subset $\Omega\sub\R^d$
(as in \cite{Wal}),
see Proposition~\ref{densewt}.}
This information is used in~\cite{HDL}
to calculate the homotopy groups of the Lie groups
$\cS(\R^d,G)$; the result is that
\[
\pi_k(\cS(\R^d,G))\cong \pi_{k+d}(G)
\]
for all $k\in\N_0$ (as conjectured
and formulated as an open problem~in~\cite{BCR}).
For the density of $C^\infty_c(M,G)$ in $C_c(M,G)$
when~$M$ is a $\sigma$-compact finite-dimensional smooth
manifold without boundary, cf.\ already \cite[Theorem~A.3.3]{Cen}
or \cite[Lemma~A.5]{Cur},
where also the inclusion map $C^\infty_c(M,G)\to C_0(M,G)$
is considered and shown to be a weak homotopy equivalence
(see \cite[Theorem~A.10]{Cur}).\\[2.3mm]
We also obtain information on function spaces on locally
compact rough $C^\ell$-manifolds.
For example, we show
that $C^\ell(M,F)\cong F\tensor C^\ell(M,\R)$
for each $\ell\in\N_0\cup\{\infty\}$,
locally compact $C^\ell$-manifold~$M$
with rough boundary and complete locally convex space~$F$
(see Proposition~\ref{tensorcases}(b)).
Using extension operators as a tool, we deduce
(see Proposition~\ref{nucity}):
\begin{prop}
If $M$ is a locally compact
$C^\infty$-manifold with rough boundary
and $F$ a nuclear locally convex space,
then also $C^\infty(M,F)$ is nuclear.
In particular, $C^\infty(M,\R)$ is nuclear.
If~$M$ is $\sigma$-compact and $F$ a nuclear
Fr\'{e}chet space,
then $C^\infty(M,F)$ is a nuclear Fr\'{e}chet
space.
\end{prop}
As a tool, we construct smoothing operators
for vector-valued
functions also on manifolds with corners, notably on $[0,1]^d$
(Lemma~\ref{smoocube}).\\[2.3mm]
Recall that a Hausdorff topological space~$X$
is called a \emph{$k$-space}
if a subset $A\sub X$ is closed if and only if $A\cap K$ is closed
for each compact subset $K\sub X$.
If $X$ is a $k$-space, then a function
$f\colon X\to Y$ to a topological space~$Y$
is continuous if and only if $f|_K$ is continuous for each compact
subset $K\sub X$.
It is well known and easy to check that
every locally compact Hausdorff space is
a $k$-space, and every metrizable
topological space.\\[2.3mm]
A Hausdorff topological space~$X$
is called a \emph{$k_\R$-space}
if for every real-valued function $f\colon X\to\R$,
continuity of~$f$ is equivalent to continuity of
$f|_K$ for each compact subset $K\sub X$
(the same then holds for functions from~$X$
to completely regular topological spaces).
Every $k$-space
is a $k_\R$-space, but not conversely.
For example, the product topology makes
$\prod_{j\in J}X_j$ a $k_\R$-space
for each family $(X_j)_{j\in J}$
of locally compact Hausdorff spaces (\cite{Nob},
also \cite{GaM}); but
$\R^J$ fails to be a $k$-space
if the cardinality of $J$ is at least
$2^{\aleph_0}$ (see, e.g., \cite[Remark~A.5.16]{GaN}).\\[2.3mm]
For later use, we prove the following
result (see Corollary~\ref{eval-spec}),
which slightly generalizes previous
findings in \cite{HYP}:
\begin{prop}
Let $E$, $F$, and $X$ be locally convex spaces,
$\ell\in \N_0\cup\{\infty\}$,
and $R\sub X$ be a regular subset.
Let $\gamma\colon R\to\cL(E,F)$
be a $C^\ell$-map to the space of continuous
linear operators, endowed with the compact-open topology.
Let $\eta\colon R\to E$ be a $C^\ell$-map.
If $R\times X$ is a $k_\R$-space, then
\[
R\to F,\quad x\mto\gamma(x)(\eta(x))
\]
is a $C^\ell$-map.
\end{prop}
An analogous result is obtained
if $\cL(E,F)$ is replaced with a space
of $C^\ell$-functions (Proposition~\ref{eval-chain}),
and also for the composition map between spaces
of linear operators instead of the evaluation map
(Corollary~\ref{compo-spec}).
Regarding the composition map
on spaces of differentiable functions,
we obtain the following (as a special case
of Proposition~\ref{strange-compo}):
\begin{prop}
Let $E$, $F$, $X$, and $Z$ be locally convex spaces,
$A\sub Z$ and $R\sub X$ be regular subsets,
and $S\sub E$
be a locally convex, regular subset.
Assume that $k,\ell\in \N_0\cup\{\infty\}$.
Let $\gamma\colon A\to C^{\ell+k}(S,F)$
and $\eta\colon A\to C^\ell(R,E)$
be $C^k$-maps such that $\eta(z)(R)\sub S$
for all $z\in A$.
If $R\times X$ and $A\times Z$
are $k_\R$-spaces, then
\[
\zeta\colon A\to C^\ell(R,F),\quad z\mto \gamma(z)\circ\eta(z)
\]
is a $C^k$-map.
\end{prop}
We mention that each rough manifold~$M$ has a formal
boundary $\partial^\circ M\sub M$;
its complement $M^\circ$ is called the formal
interior of~$M$
(see \ref{defn-rough}).
While $C^\ell$-maps
from a rough $C^\ell$-manifold~$M$
to a $C^\ell$-manifold~$N$ (or a $C^\ell$-manifold $N$ with rough boundary)
can be defined
in a straightforward fashion,
$C^\ell$-maps
$f\colon M\to N$ between general rough $C^\ell$-manifolds
$M$ and $N$ can only be defined if $f(M^\circ)\sub N^\circ$;
to keep the concepts apart, we call the latter maps
\emph{restricted $C^\ell$-maps} (or $RC^\ell$-maps, for short).\\[2.3mm]
In the penultimate section, we establish smoothing results for sections
in fibre bundles. 
Our main theorem in this context
(Theorem~\ref{thm-steen})
generalizes a result in~\cite{Wo2}
devoted to the case $(\ell,r)=(0,\infty)$
(which assumes that $M$ is a connected
$C^\infty$-manifold with corners,
and gives less detailed information
concerning properties of the homotopies).
Wockel's result, in turn, generalizes
a classical fact by Steenrod (\S6.7
in \cite{Ste}). In the case $r=\infty$,
our result reads as follows.
\begin{thm}
Let $\ell\in\N_0$.
If $\ell=0$, let $M$ be a $\sigma$-compact,
locally compact rough $C^\infty$-manifold;
if $\ell>0$, let $M$ be a $\sigma$-compact,
locally compact $C^\infty$-manifold with corners.
Let $\pi\colon N\to M$
be a locally trivial
smooth fibre bundle over~$M$
whose fibres are smooth manifolds
modelled on locally convex spaces.
Let $\sigma\in \Gamma_{C^\ell}(M\leftarrow N)$,
$\Omega\sub \Gamma_{C^\ell}(M\leftarrow N)$
be a neighbourhood of~$\sigma$ in the Whitney $C^\ell$-topology,
$U\sub M$ be open and $A\sub M$ be a closed
subset such that $\sigma$ is smooth
on an open neighbourhood of $A\setminus U$ in~$M$.
Let $V_{\reg}$ be the largest open subset
of~$M$ on which~$\sigma$ is smooth.
Then there exists a section $\tau\in \Omega$
and a homotopy
$H\colon [0,1]\times M\to N$ from $\sigma=H(0,\cdot)$
to $\tau=H(1,\cdot)$ such that $H_t:=H(t,\cdot)\in\Omega$
for all $t\in [0,1]$ and the following holds:
\begin{itemize}
\item[\rm(a)]
$\sigma|_{M\setminus U}=H_t|_{M\setminus U}$
for all $t\in [0,1]$;
\item[\rm(b)]
$H_t$ is smooth on $V_{\reg}$, for all $t\in [0,1]$;
\item[\rm(c)]
There exists an open neighbourhood $W$ of~$A$ in~$M$
such that $H_t|_W$ is smooth for all $t\in\,]0,1]$.
\end{itemize}
Moreover, one can achieve that $H$ is $RC^{0,\ell}$,
the restriction $H|_{]0,1]\times M}$ is $RC^{\infty,\ell}$,
and
the restriction of~$H$ to a map
$]0,1]\times (V_{\reg}\cup W)\to N$ is smooth.
\end{thm}
Here, we used the following concepts.
Consider a map $f\colon R\times S \to F$,
where $R$ and $S$ are regular subsets
of locally convex spaces $E_1$ and $E_2$,
respectively, and $F$ is a locally
convex space. Given $k,\ell\in\N_0\cup\{\infty\}$,
the map $f$ is called $C^{k,\ell}$
if $f$ is continuous, for all $i,j\in\N_0$ with
$i\leq k$ and $j\leq \ell$,
we can form iterated directional derivatives of $f$
at $(x,y)\in R^0\times S^0$
in the second variable
in directions $w_1,\ldots, w_j\in E_2$,
followed by iterated directional derivatives
in the first variable in directions $v_1,\ldots, v_i\in E_1$,
and the $F$-valued function of $(x,y,v_1,\ldots,v_i,w_1,\ldots, w_j)\in
R^0\times S^0\times E_1^i\times E_2^j$
so obtained admits a continuous extension
\[
d^{\,(i,j)}f\colon R\times S\times E_1^i\times E_2^j\to F
\]
(see \ref{defCkl} for details).
Testing in charts, one obtains a concept of
$C^{k,\ell}$-maps $f\colon M_1\times M_2\to M$
if $M_1$ is a $C^k$-manifold with rough boundary,
$M_2$ a $C^\ell$-manifold
with rough boundary, and $M$ a $C^{k+\ell}$-manifold
with rough boundary (see \cite{AaS}
and \cite{GaN}),
and a corresponding concept of
$RC^{k,\ell}$-maps
if manifolds with rough boundary are replaced
with rough manifolds (see \cite{Rou}).\\[2.3mm]
The final section compiles
applications of smoothing techniques in algebraic topology
(many of which are
known, or known in special cases; compare, e.g., \cite{Ste}
and \cite{MaW}). For smoothing results
in the (inequivalent) convenient setting of analysis,
see~\cite{KaM} and \cite{Kih}.\\[3mm]
{\bf Acknowledgement.}
Thanks are due to Johanna Jakob (Paderborn)
for drawing attention to an error
in the former proof of Theorem~1.3
in the case $\ell=\infty$.
The proof has now been corrected by the author.
\section{Preliminaries of infinite-dimensional calculus}\label{sec-prels}
In this section, we compile necessary notation and prerequisites
concerning calculus in locally convex topological vector spaces.\\[2.3mm]
A differential calculus for mappings on regular subsets of locally convex spaces
was sketched in~\cite{Woc} and fully worked out in~\cite{Rou},
where also so-called \emph{rough $C^k$-manifolds}
modelled on locally convex spaces are considered,
which locally look like a regular subset of
a modelling locally convex space and whose chart changes (transition maps)
are $C^k$-maps taking interior into interior.
We recall necessary prerequisites concerning calulus in this section
and refer to \cite{Rou} for proofs. The theory closely
parallels calculus on \emph{locally convex} regular subsets of locally convex spaces
as in~\cite{GaN}, where also corresponding
``manifolds with rough boundary'' have been studied
(which locally look like a locally convex, regular set).
For mappings on open sets, the approach to infinite-dimensional
calculus goes back to A. Bastiani~\cite{Bas},
and is also known as Keller's $C^k_c$-theory
(see \cite{Res}, \cite{GaN}, \cite{Ham}, \cite{Mic}, and~\cite{Mil}
for expositions; cf.\ also~\cite{BGN}).
\begin{numba} (General conventions).
We write $\N:=\{1,2,\ldots\}$ and $\N_0:=\N\cup\{0\}$.
Locally convex, Hausdorff topological vector spaces over~$\R$
will simply be called \emph{locally convex spaces}.
If $E$ is a locally convex space, $q\colon E\to[0,\infty[$
a continuous seminorm, $x\in E$ and $r>0$, we write
\[
B^q_r(x):=\{y\in E\colon q(y-x)<r\}\quad\mbox{and}\quad
\wb{B}^q_r(x):=\{y\in E\colon q(y-x)\leq r\}
\]
for the open ball and closed ball around~$x$ of radius~$r$
with respect to~$q$, respectively.
If $(E,\|\cdot\|)$ is a normed space, we simply write $B^E_r(x)$ (or $B_r(x)$)
for the open ball and $\wb{B}^E_r(x)$ (or $\wb{B}_r(x)$)
for the closed ball with respect to~$\|\cdot\|$, if no confusion is possible.
All vector spaces considered are vector spaces over~$\R$
and we shall use $E\otimes F$ as a shorthand for the vector space $E\otimes_\R F$,
for vector spaces $E$ and~$F$.
If $E$ and $F$ are locally convex spaces, we write $\cL(E,F)$ for the vector space
of continuous linear mappings from~$E$ to~$F$.
We write $E':=\cL(E,\R)$ for the dual space.
As usual, continuous mappings are also called $C^0$.
If $X$ is a topological space and $S\sub X$ a subset,
then a subset $U\sub X$
is called a \emph{neighbourhood of $S$ in $X$}
if~$S$ is contained in the interior~$U^0$ of~$U$ in~$X$.
Following~\cite{GaN},
we say that a subset $M$ of a locally convex space~$E$
is \emph{locally convex} if each $x\in M$ has a neighbourhood in~$M$ which is a convex set.
Then every neighbourhood of~$x$ in~$M$ contains
a convex neighbourhood (see~\cite{GaN}).
If $E$ is a vector space and $M\sub E$ a subset,
we write $\conv(M)$ for its convex hull.
If $X$ is a topological
space, we define partitions of unity
on~$X$ as usual;
they are always assumed locally finite
(in contrast to more general conventions in~\cite{Eng}).
If $X$ and $Y$ are topological spaces,
we write $C(X,Y)$ for the set of all continuous functions
from $X$ to~$Y$, including the case that $Y$ is a locally convex
space.\footnote{In contrast to conventions in other
parts of mathematical analysis, where
$C(X,Y)$ refers to the space
of bounded continuous functions.}
In the latter case, we write $BC(X,Y)$
for the vector space of bounded continuous
functions from~$X$ to~$Y$.
\end{numba}
\begin{defn}\label{praeck}
Let $E$ and $F$ be locally convex spaces,
and $U\sub E$ be an open subset.
A continuous map $f\colon U\to F$
is called $C^1$ if it is continuous, the directional derivative
\[
df(x,y):=(D_yf)(x):=\frac{d}{dt}\Big|_{t=0}f(x+ty)
\]
exists for all $x\in U$ and $y\in E$, and $df\colon U\times E\to F$ is continuous.
\end{defn}
\begin{defn}\label{defck}
Let $E$ and $F$ be locally convex spaces
and $R\sub E$ be a regular subset.
A mapping $f\colon R\to F$ is called $C^1$
if $f$ is continuous, $f|_{R^0}\colon R^0\to F$
is~$C^1$ and $d(f|_{R^0})\colon R^0\times E\to F$
has a (necessarly unique) continuous extension
$df\colon R\times E\to F$.
Recursively,
for $\ell\in\N$ with $\ell\geq 2$, we say that $f$ is $C^\ell$
if $f$ is $C^1$ and $df\colon R\times E\to F$
is $C^{\ell-1}$. If $f$ is $C^\ell$ for all $\ell\in\N_0$,
then $f$ is called $C^\infty$ or \emph{smooth}.
\end{defn}
The following fact is proved in Appendix~\ref{appA} (using~\cite{Rou}),
where also more references are given.
\begin{prop}\label{eqCk}
Let $E$ and $F$ be locally convex spaces,
$R\sub E$ be a regular subset, $\ell\in \N_0\cup\{\infty\}$,
and $f\colon R\to F$ be a mapping.
Then the following conditions are equivalent:
\begin{itemize}
\item[\rm(a)]
$f$ is $C^\ell$.
\item[\rm(b)]
$f$ is continuous and has the following property:
For all $j\in\N$ with $j\leq\ell$, the iterated directional derivative
\[
d^{\,(j)}(f|_{R^0})(x,y_1,\ldots,y_j):=(D_{y_j}\cdots D_{y_1}f)(x)
\]
exists for all $x\in R^0$ and all $y_1,\ldots, y_j\in E$,
and the function $d^{\,(j)}(f|_{R^0})\!:$\linebreak
$R^0\times E^j\to F$
so obtained has a $($necessarily unique$)$ continuous extension
\[
d^{\,(j)}f\colon R\times E^j\to F.
\]
\end{itemize}
If $E=\R^n$ for some $n\in\N$, then also {\rm (c)} is equivalent to {\rm(a)}
and {\rm(b)}:
\begin{itemize}
\item[\rm(c)]
$f$ is continuous, the partial differentials
\[
\partial^\alpha (f|_{R^0})(x):=\left(\frac{\partial^{\alpha_1}}{\partial x_1^{\alpha_1}}\cdots
\frac{\partial^{\alpha_n}}{\partial x_n^{\alpha_n}}f\right)(x)
\]
exist for all multi-indices $\alpha=(\alpha_1,\ldots,\alpha_n)
\in \N_0^n$ with $1\leq |\alpha|\leq \ell$,
and the functions $\partial^\alpha (f|_{R^0})\colon R^0\to F$ so obtained
have $($necessarily unique$)$ continuous extensions
\[
\partial^\alpha f\colon R\to F.
\]
\end{itemize}
\end{prop}
Thus $d^{\,(1)}f=df$. In the situation of~(b),
we also write $d^{\,(0)}f:=f$.
In the situation of~(c), as usual $\partial^\alpha f :=f$ if $\alpha=0$.\\[2.3mm]
By (b), a continuous map $f\colon R\to F$ is $C^\ell$ if and only if
$f|_{R^0}\colon R^0\to F$ is $C^\ell$ and $d^{\,(j)}(f|_{R^0})$
admits a continuous extension $d^{\,(j)}f\colon R\times E^j\to F$
for each $j\in \N$ such that $j\leq \ell$.
\begin{numba}\label{sy-mulin}
If $f\colon E\supseteq R\to F$ (as in Definition~\ref{defck})
is $C^\ell$, then the continuous map
\[
d^{(j)}f(x,\cdot)\colon E^j\to F
\]
is symmetric and $j$-linear
for each $j\in \N$ such that $j\leq\ell$
(this is true for $x\in R^0$ by
\cite[Proposition~1.3.17]{GaN}
and passes to $x$ in the closure $R=\wb{R^0}$
by continuity).
\end{numba}
\begin{numba}\label{difflin}
In particular, the map $f'(x):=df(x,\cdot)\colon E\to F$ is continuous and linear
for each $x\in R$ if $f\colon E\supseteq R\to F$ is~$C^1$.
\end{numba}
\begin{numba}\label{diffmultilin}
If $E$ and $F$ are locally convex spaces
and $\alpha\colon E\to F$ is a continuous linear map, then
$\alpha$ is $C^\infty$ and $\alpha'(x)=\alpha$ for all $x\in E$
(see, e.g., \cite[Example~1.3.5]{GaN}).
\end{numba}
The Chain Rule holds in the following form~\cite{Rou}
(also \cite{Woc} if $g(R^0)\sub S^0$;
for both $R$ and $S$ locally convex, also~\cite[Proposition~1.4.10]{GaN};
for $R$ and $S$ open, also \cite{BGN}, \cite{Res},
\cite{Ham}, \cite{Mic}):
\begin{prop}\label{unchained}
Let $E$, $F$ and $Y$ be locally convex spaces,
$R\sub E$ and $S\sub F$ be regular subsets,
$\ell\in \N\cup\{\infty\}$
and $f\colon S\to Y$ as well as $g\colon R\to F$
be $C^\ell$-functions such that $g(R)\sub S$.
If $S$ is locally convex or $g(R^0)\sub S^0$,
then $f\circ g\colon R\to Y$ is~$C^\ell$ and
\[
d(f\circ g)(x,y)=df(g(x),dg(x,y))\quad\mbox{for all $\,(x,y)\in R\times E$.}
\]
Thus $(f\circ g)'(x)=f'(g(x))\circ g'(x)$.
\end{prop}
Recall that a subset $A$ of a topological space $X$ is called
\emph{sequentially closed} if it has the following property:
If $x\in X$ and there exists a sequence  $(a_n)_{n\in\N}$ in $A$
such that $a_n\to x$ in~$X$ as $n\to\infty$, then $x\in A$.
Every closed subset is sequentially closed.\\[2.3mm]
The following simple fact is frequently useful.
See \cite{Rou} for the proof
(or \cite[Lemma~1.4.16]{GaN} if $R$ is locally convex;
or \cite[Lemma~10.1]{BGN} if $R$ is open).
\begin{la}\label{intoclo}
Let $E$ and $F$ be locally convex spaces,
$F_0\sub F$ be a closed vector subspace,
$R\sub E$ be a regular subset, and $\ell\in\N_0\cup\{\infty\}$.
Let $f\colon R\to F$ be a mapping with image $f(R)\sub F_0$.
Then $f\colon R\to F$ is $C^\ell$ if and only if its co-restriction
$f|^{F_0}\colon R\to F_0$, $x\mto f(x)$ is $C^\ell$.
In this case,
$d^{\,(j)}f(x,y_1,\ldots,y_j)=d^{\,(j)}(f|^{F_0})(x,y_1,\ldots,y_j)\in F_0$
for all $j\in \N$ with $j\leq\ell$ and all $(x,y_1,\ldots,y_j)\in R\times E^j$.
If $E$ is metrizable or $R$ is locally convex, then
the same conclusions hold if $F_0$ is sequentially closed
in~$F$.
\end{la}
\begin{numba}\label{theco}
Recall that the compact-open topology on the set $C(X,Y)$
of continuous functions between Hausdorff topological spaces
is the topology with subbasis
\[
\lfloor K,U\rfloor:=\{\gamma\in C(X,Y)\colon \gamma(K)\sub U\},
\]
for $K$ ranging through the set $\cK(X)$
of compact subsets of~$X$ and
$U$ in the set of open subsets of~$Y$.
\end{numba}
\begin{numba}\label{propsco}
If $F$ is a locally convex space,
then the compact-open topology makes $C(X,F)$ a locally convex space;
moreover, it coincides with the topology of compact convergence
and a basis of open $0$-neighbourhoods is given by the sets
$\lfloor K, U\rfloor$,
for $K$ ranging through a cofinal\footnote{For each $K\in\cK(X)$,
there exists $L\in\cK$ such that $K\sub L$.}
subset $\cK$ of $\cK(X)$
and $U$ through a basis $\cB$ of open $0$-neighbourhoods in~$F$.\\[2.3mm]
The seminorms $\|\cdot\|_{K,q}\colon C(X,F)\to[0,\infty[$ given by
\begin{equation}\label{vorschrsemi}
\|\gamma\|_{K,q}\; :=\; \sup_{x\in K}\, q(\gamma(x))
\end{equation}
for $\gamma\in C(X,F)$
define the compact-open topology on~$C(X,F)$,
for $K$ in a cofinal subset $\cK\sub\cK(X)$ and $q$ in a directed\footnote{Thus $\Gamma\not=\emptyset$
and we assume that for all $q_1,q_2\in\Gamma$, there exists $q\in\Gamma$ such that
$q_1(x)\leq q(x)$ and $q_2(x)\leq q(x)$ for all $x\in F$.}
set $\Gamma$ of continuous seminorms on~$F$
defining its topology (as $B^{\|\cdot\|_{K,q}}_r(0)=\lfloor K, B^q_r(0)\rfloor$
for all $K\in\cK$, $q\in\Gamma$,~$r>0$).
\end{numba}
We define the compact-open $C^\ell$-topology
on a function space $C^\ell(R,F)$ as in the case
of a locally convex set (treated in \cite{GaN}).\footnote{As functions $f$
with domain $C^\ell(R,F)$ are a frequent topic in infinite-dimensional
calculus, we now denote the elements of $C^\ell(R,F)$
by $\gamma$ (rather than $f$).}
\begin{defn}\label{defcktop}
If $E$ and $F$ are locally convex spaces,
$R\sub E$ is a regular subset and $\ell\in \N_0\cup\{\infty\}$,
we define the \emph{compact-open $C^\ell$-topology}
on $C^\ell(R,F)$ as the initial topology with respect to the
linear mappings
\[
d^{\,(j)}\colon C^\ell(R,F)\to C(R\times E^j,F),\quad \gamma \mto
d^{\,(j)}\gamma
\]
for $j\in\N_0$ such that $j\leq \ell$,
where $C(R\times E^j,F)$ is endowed with the compact-open topology.
\end{defn}
\begin{rem}\label{remcktop}
(a) The compact-open $C^\ell$-topology makes $C^\ell(R,F)$ a locally convex space
and the point evaluations $\ve_x\colon C^\ell(R,F)\to F$, $\gamma\mto\gamma(x)$
are continuous linear maps for all $x\in R$. By definition, the compact-open $C^\ell$-topology
turns the injective linear map
\begin{equation}\label{inprodcts}
(d^{\,(j)})_{\N_0\ni j\leq \ell}\colon C^\ell(R,F)\to\prod_{\N_0\ni j\leq \ell}C(R\times E^j,F),
\;\, \gamma\mto (d^{\,(j)}\gamma)_{\N_0\ni j\leq\ell}
\end{equation}
into a topological embedding (a homeomorphism onto the image).\medskip

\noindent
(b) Moreover, it is easy to see that
\[
C^\infty(R,F)\;=\; \pl\; C^n(R,F)\vspace{-1.3mm}
\]
as a locally convex space, using the inclusion maps $C^\infty(R,F)\to C^n(R,F)$
for $n\in\N_0$ as the limit maps,
where the inclusion maps $C^m(R,F)\to C^n(R,F)$
for integers $0\leq n\leq m$ are the bonding maps.\footnote{Let
$P:=\pl\, C^n(R,F)\sub\prod_{n\in\N_0}C^n(R,F)$ be the standard projective limit.
One readily verifies that the linear map $\phi\colon C^\infty(R,F)\to P$, $\gamma\mto(\gamma)_{n\in\N_0}$
is bijective, continuous, and that $\phi^{-1}$ is continuous.}
\end{rem}
We now compile various known properties
of the compact-open $C^\ell$-topology. A proof for part~(c) is given in
Appendix~\ref{appA};
proofs for parts (a), (b), (d), (e), (f), and (g)
can be found in~\cite{Rou}
(for locally convex regular sets, see already~\cite{GaN} for all assertions).
\begin{la}\label{sammelsu}
Let $E$, $F$ and $Y$ be locally convex spaces and $R\sub E$
be a regular subset. Then the following holds:
\begin{itemize}
\item[\rm(a)]
The map $(d^{\,(j)})_{\N_0\ni j\leq \ell}$ as in {\rm(\ref{inprodcts})}
is a linear topological embedding with closed image.
\item[\rm(b)]
If $F$ is vector subspace of~$Y$ and carries the induced topology,
then the compact-open $C^\ell$-topology on $C^\ell(R,F)\sub C^\ell(R,Y)$
is the topology induced by the compact-open $C^\ell$-topology on $C^\ell(R,Y)$.
\item[\rm(c)]
If $\ell\geq 1$ holds, $\cK$ is cofinal in $\cK(R)$,
$\cL$ cofinal in $\cK(E)$ and $\Gamma$ a directed subset
of continuous seminorms on~$F$ defining its locally convex vector topology,
then the seminorms $\|\cdot\|_{C^j,K,L,q}\colon C^\ell(R,F)\to[0,\infty[$,
\[
\|\gamma\|_{C^j,K,L,q}:=
\max\{\|\gamma\|_{K,q},\|\gamma\|_{1,K,L,q},\ldots,\|\gamma\|_{j,K,L,q}\}
\]
for $K\in\cK$, $L\in \cL$ and $q\in\Gamma$
define the compact-open $C^\ell$-topology on $C^\ell(R,F)$,
where $\|\cdot\|_{K,q}\colon C^\ell(R,F)\to[0,\infty[$
is defined as in {\rm(\ref{vorschrsemi})}
and
\[
\|\gamma\|_{i,K,L,q}:=\|d^{\,(i)}\gamma\|_{K\times L^i,q}
\]
for $i\in \{1,\ldots, j\}$ and $\gamma\in C^\ell(R,F)$.
\item[\rm(d)]
If $E$ is metrizable and $F$ is complete $($resp., sequentially complete$)$,
then $C^\ell(R,F)$ is complete $($resp., sequentially complete$)$.
\item[\rm(e)]
If $S\sub Y$ is a regular subset,
$f\colon S\to E$ is a $C^\ell$-map with $f(S)\sub R$
and $R$ is locally convex or
$f(S^0)\sub R^0$,
then the map
\[
C^\ell(f,F)\colon C^\ell(R,F)\to C^\ell(S,F),\;\, \gamma\mto \gamma\circ f
\]
is continuous and linear. Notably, the restriction map
\[
C^\ell(R,F)\to C^\ell(S,F), \;\; \gamma\mto\gamma|_S
\]
is continuous and linear for each regular subset $S\sub R$.
\item[\rm(f)]
If $(U_i)_{i\in I}$ is a cover of~$R$ be relatively open subsets,
then the map
\[
C^\ell(R,F)\to\prod_{i\in I}C^\ell(U_i,F),\quad \gamma\mto(\gamma|_{U_i})_{i\in I}
\]
is a linear topological embedding with closed image.
\item[\rm(g)]
If $f\colon R\times E\to F$ is $C^\ell$,
then the mapping
\[
f_*\colon C^\ell(R,E)\to C^\ell(R,F),\;\;\gamma\mto f\circ (\id_R,\gamma)
\]
is continuous. In particular,
the multiplication operator
\[
m_h\colon C^\ell(R,E)\to C^\ell(R,E),\quad \gamma\mto h\cdot \gamma
\]
with $(h\cdot\gamma)(x)=h(x)\gamma(x)$ is continuous and linear
for each $h\in C^\ell(R,\R)$
$($as $m_h=f_*$ for the $C^\ell$-map $f$ given by $f(x,y):=h(x)y)$.
\end{itemize}
\end{la}
The following fact is proved in Appendix~\ref{appA}.
\begin{la}\label{toppartial}
If $F$ is a locally convex space, $d\in\N$, $\ell\in\N_0\cup\{\infty\}$
and $R\sub\R^d$ a regular subset,
then the compact-open $C^\ell$-topology on $C^\ell(R,F)$ is the initial topology
with respect to the mappings
\[
\partial^\alpha\colon C^\ell(R,F)\to C(R,F),\;\;\gamma\mto \partial^\alpha\gamma
\]
for $\alpha\in \N_0^d$ such that $|\alpha|\leq \ell$,
if $C(R,F)$ is endowed with the compact-open topology.
The seminorms
\[
\|\cdot\|_{C^j\!,K,q}^\partial\colon C^\ell(R,F)\to[0,\infty[,\;\;
\gamma\, \mto\, \max_{|\alpha|\leq j}\; \sup_{x\in K}\, q(\partial^\alpha\gamma(x))
\]
define the compact-open $C^\ell$-topology
for $j\in \N_0$ with $j\leq\ell$,
$K$ ranging in a cofinal set $\cK\sub \cK(R)$ of compact sets
and $q$ in a directed set of continuous seminorms on~$F$
defining its locally convex vector topology.
\end{la}
If $F=\R$, we can take $q:=|\cdot|$
and abbreviate $\|\cdot\|_{C^j,K}:=\|\cdot\|_{C^j\!,K,|\cdot|}^\partial$.
\begin{numba}
Given $d\in\N$, $\ell\in\N_0\cup\{\infty\}$
and an open subset $\Omega\sub\R^d$,
the support $\Supp(\gamma)$
of $\gamma\in C^\ell(\Omega,F)$
is defined as the closure of $\{x\in\Omega\colon\gamma(x)\not=0\}$
in~$\Omega$.
Given a compact subset $K\sub\Omega$, we endow
the closed vector subspace
\[
C^\ell_K(\Omega,F):=\{\gamma\in C^\ell(\Omega,F)\colon \Supp(\gamma)\sub K\}
\]
of $C^\ell(\Omega,F)$ with the induced topology.
As usual, we endow the space
\[
C^\ell_c(\Omega,F) = \bigcup_{K\in \cK(\Omega)}C^\ell_K(\Omega,F)
\; =\; \dl\, C^\ell_K(\Omega,F)\vspace{-1mm}
\]
of all compactly supported $F$-valued $C^\ell$-maps with the locally convex
direct limit topology.
\end{numba}
See \cite{Rou} for more
details on the following concept
(introduced in \cite{AaS} for functions on a product of two locally convex regular sets;
see also \cite{GaN}).
\begin{defn}\label{defCkl}
Let $E_1$, $E_2$ and $F$ be locally convex spaces
and $R_1\sub E_1$ as well as $R_2\sub E_2$ be regular subsets.
Given $k,\ell\in \N_0\cup\{\infty\}$,
we say that a continuous function $f\colon R_1\times R_2\to F$ is
$C^{k,\ell}$ if the iterated directional derivatives
\[
d^{\,(i,j)}f(x,y,v_1,\ldots,v_i,w_1,\ldots,w_j):=
(D_{(v_i,0)}\cdots D_{(v_1,0)}D_{(0,w_j)}\cdots D_{(0,w_1)}f)(x,y)
\]
exist for all $i,j\in\N_0$ with $i\leq k$ and $j\leq\ell$
and all $x\in R_1^0$, $y\in R_2^0$, $v_1,\ldots,v_i\in E_1$ and $w_1,\ldots,w_j\in E_2$,
and extend to $($necessarily unique$)$ continuous mappings
\[
d^{\,(i,j)}f\colon R_1\times R_2\times E_1^i\times E_2^j\to F.
\]
\end{defn}
We recall the concept of a submersion (see \cite{Sub},
or also \cite{Ham} in the case of Fr\'{e}chet manifolds).
\begin{defn}\label{def-subm}
Let $M$ and $N$ be smooth manifolds
modelled on locally convex spaces.
A smooth map $f\colon M\to N$ is called
a \emph{submersion} if it has the following property:
For each $x\in M$, there exists a chart $\phi\colon U_\phi\to V_\phi\sub E_\phi$
of~$M$ with $x\in U_\phi$, a chart $\psi\colon U_\psi\to V_\psi\sub E_\psi$
of~$N$ with $f(U_\phi)\sub U_\psi$ and a continuous linear
map $\pi\colon E_\phi\to E_\psi$ admitting a continuous linear right inverse
such that $\pi(V_\phi)\sub V_\psi$ and $\psi\circ f|_{U_\phi}=\pi\circ \phi$.
\end{defn}
Finally, let us recall notation and facts
concerning Taylor expansions.
See \cite{BGN} and \cite{GaN} for discussions
of the following concepts.
\begin{numba}
Let $E$ and $F$ be locally convex spaces and $j\in\N_0$.
A map $p\colon E\to F$ is called a \emph{homogeneous $F$-valued
polynomial of degree~$j$ on~$E$}
if there exists a continuous $j$-linear map $\beta\colon E^j\to F$
such that
\[
p(y)=\wb{\beta}(y):=\beta(\underbrace{y,\ldots,y}_{\text{$j$ times}})
\]
for all $y\in E$ (if $j=0$, take the value
of the constant map $\beta\colon E^0=\{0\}\to F$).
We may always assume that $\beta$ is symmetric since
\[
\beta_{\text{sym}}(y_1,\ldots,y_j):=\frac{1}{j!}\sum_{\pi\in S_j}\beta(y_{\pi(1)},\ldots,
y_{\pi(j)})
\]
is symmetric (where $S_j$ is the group of all permutations
of $\{1,\ldots, j\}$)
and $\wb{\beta}=\wb{\beta_{\text{sym}}}$.
If $\beta$ is symmetric, then $\beta$ can be recovered
from the homogeneous polynomial
$p=\wb{\beta}$ by means of the Polarization Formula~\cite[Theorem~A]{BaS}:
\begin{equation}\label{polform}
\beta(x_1,\ldots, x_j)\;=\; \frac{1}{j!}\sum_{\ve_1,\ldots, \ve_j=0}^1
(-1)^{j-(\ve_1+\cdots+\ve_j)}p(\ve_1x_1+\cdots+\ve_j x_j).
\end{equation}
\end{numba}
\begin{numba}
Given $\ell\in\N_0$,
a function $p\colon E\to F$ between locally convex spaces
is 
called a \emph{continuous polynomial} of degree $\leq\ell$
if there exist continuous homogeneous polynomials $p_j\colon E\to F$
of degree~$j$
for $j\in \{0,\ldots,\ell\}$ such that $p=\sum_{j=0}^\ell p_j$.
\end{numba}
\begin{numba}
If $E$ and $F$ are locally convex spaces,
$U\sub E$ is open, $\ell\in\N_0$, $\gamma \colon U\to F$ a $C^\ell$-map
and $x\in U$, then
\[
\delta^j_x(\gamma)(y)\; :=\;
\frac{d^j}{dt^j}\Big|_{t=0}\gamma(x+ty)
\;=\;
d^{\,(j)}\gamma(x,y,\ldots, y)
\]
defines a continuous homogeneous
polynomial $\delta^j_x(\gamma)\colon E\to F$ of degree~$j$
(the \emph{$j$th G\^{a}teaux differential of~$\gamma$ at~$x$})
for $j\in\{0,\ldots,\ell\}$.
The \emph{$\ell$th order
Taylor polynomial of~$\gamma$ at~$x$} is defined as
\begin{equation}\label{deftaypol}
P^\ell_x(\gamma)\colon E \to F\,,\quad
P^\ell_x(\gamma)(y):=\sum_{j=0}^\ell\frac{\delta^j_x(\gamma)(y)}{j!}.\vspace{-1mm}
\end{equation}
\end{numba}
\begin{numba}\label{poly-is-tensor}
If $d\in\N$, $\ell\in\N$
and $F$ is a locally convex space,
then every continuous polynomial $p\colon \R^d\to F$ of degree $\leq \ell$
is of the form $p(x)=\sum_{|\alpha|\leq\ell}x^\alpha a_\alpha$
with suitable $a_\alpha\in F$ (where $x^\alpha:=x_1^{\alpha_1}\cdots x_d^{\alpha_d}$),
as usual.\footnote{It suffices to show this if $p=\wb{\beta}$ is homogeneous of degree
$j\leq \ell$. Using the standard basis vectors $e_1,\ldots, e_d$ of~$\R^d$,
we have $p(x)=\sum_{1\leq i_1,\ldots,i_j\leq d}x_{i_1}\cdots x_{i_j}\beta(e_{i_1},\ldots, e_{i_j})$,
from which the assertion follows.}
Hence $p=\sum_{|\alpha|\leq \ell}(\pr_1^{\alpha_1}\cdots\pr_d^{\alpha_d})
a_\alpha$ $\in F\otimes C^\infty(\R^d,\R)$.
\end{numba}
\begin{numba}\label{tay}
We shall use Taylor's Theorem (see \cite[Theorem~1.6.14]{GaN}):\\[2.3mm]
\emph{Let $E$ and $F$ be locally convex spaces, $U\sub E$ be an open subset,
$\ell\in\N$ and $\gamma\colon U\to F$ be a $C^\ell$-map. If $x,y\in U$
and the line segment joining $x$ and $y$ is contained in~$U$, then
then the remainder term
\[
R_x(y):=\gamma(y)-P^\ell_x\gamma(y-x)
\]
can be written as the weak integral}
\[
R_x(y)=\frac{1}{(\ell-1)!}\int_0^1(1-t)^{\ell-1}(\delta^\ell_{x+t(y-x)}\gamma-\delta^\ell_x\gamma)(y-x)\, dt.
\]
\end{numba}
We mention that the seminorms described in Lemma~\ref{sammelsu}(c)
are useful to establish the link to Hanusch's work.
The approach via partial differentials and the seminorms
from Lemma~\ref{toppartial}
are used in the proof of Proposition~\ref{etensor} (and links our work to~\cite{Tre}).
In \ref{weiterprer}, we shall introduce yet another
family of seminorms defining the compact-open $C^\ell$-topology,
which is useful for the formulation
of Theorem~\ref{smoohoo} and its proof.
\section{Proof of Theorem~\ref{thmsmoo}}\label{sec-smoo}
We now construct smoothing operators
and obtain Theorem~\ref{thmsmoo}
as part of the following more technical lemma.
The seminorms in~(b) and~(c) are as in~\ref{weiterprer}.
\begin{la}\label{smoohoo}
Let $\Omega\sub \R^d$ be an open set,
$\ell\in \N_0$ and $F$ be a
locally convex topological vector space.
Then there exist continuous linear operators
\[
\wt{S}_n\colon C^\ell(\Omega,F)\to C^\infty(\Omega,F)
\]
for $n\in \N$ with the following properties:
\begin{itemize}
\item[\rm(a)]
$\wt{S}_n(\gamma)\to\gamma$ in $C^\ell(\Omega,F)$ as $n\to\infty$,
for each $\gamma\in C^\ell(\Omega,F)$.
\item[\rm(b)]
There exists $C \in [0,\infty[$
such that,
for each compact set $K\sub \Omega$,
compact neighbourhood $L$ of~$K$ in~$\Omega$,
continuous seminorm $q$ on~$F$,
$\gamma \in C^\ell(\Omega,F)$ and $n\in\N$ with
$K+[{-\frac{1}{n}},\frac{1}{n}]^d\sub L$, we have
\[
\|\wt{S}_n(\gamma)\|_{C^\ell\!,K,q}\;\leq \; C\, \|\gamma\|_{C^\ell\!,L,q}\, .
\]
\item[\rm(c)]
For each compact exhaustion $K_1\sub K_2\sub\cdots$ of $\Omega$,
there exists a sequence $(S_n)_{n\in\N}$
of continuous linear operators $S_n\colon C^\ell(\Omega,F)\to C^\infty(\Omega,F)$
satisfying the conditions of Theorem~{\rm\ref{thmsmoo}},
such that for some positive integers $m_1<m_2<\cdots$, we have
\begin{equation}\label{fnewprop}
S_n(\gamma)|_{K_n}=\wt{S}_{m_n}(\gamma)|_{K_n}
\end{equation}
for all $n\in\N$
and $\gamma\in C^\ell(\Omega,F)$;
moreover,
\begin{equation}\label{snewprop}
\|S_n(\gamma)\|_{C^\ell\!,K_n,q}\,\leq\, C\,
\|\gamma\|_{C^\ell\!,K_{n+1},q}
\end{equation}
for each continuous seminorm~$q$ on~$F$, with~$C$
as in~{\rm(b)}.
\end{itemize}
\end{la}
To obtain smooth approximations for $\gamma\in C^\ell(\Omega,F)$,
we shall use the $\ell$th order Taylor
polynomials $P^\ell_x(\gamma)\colon \R^d\to F$
of $\gamma$ around suitable points~$x$ (as in (\ref{deftaypol}))
and blend them using a smooth partition of unity. We begin with some
preparations.
\begin{numba}
If $(E,\|.\|)$ is a normed
space, $F$ a locally convex space, $q$
a continuous seminorm on~$F$,
$j\in \N_0$ and $\beta\colon E^j\to F$
a continuous $j$-linear map,
let
\[
\|\beta\|_q\; :=\;
\sup\{q(\beta(x_1,\ldots, x_j))\colon x_1,\ldots, x_j\in \wb{B}^E_1(0)\}.
\]
For the corresponding continuous homogenous polynomial $p:=\wb{\beta}$,
we let
\begin{equation}\label{hompolsemi}
\|p\|_q\;:=\;
\sup\{q(p(x))\colon x\in \wb{B}^E_1(0)\}.
\end{equation}
For symmetric $j$-linear maps $\beta$,
the Polarization Formula (\ref{polform})
entails that the seminorms $\beta\mto \|\wb{\beta}\|_q$
and $\beta\mto \|\beta\|_q$ are equivalent; we have
\begin{equation}\label{conspol}
\|p\|_q\; \leq \; \|\beta\|_q \;\leq\; \frac{(2j)^j}{j!}\,\|p\|_q\,.
\end{equation}
\end{numba}
\begin{numba}
Given a locally convex space $F$, a normed (or locally convex)
space~$E$
and $j\in \N_0$,
let $\Pol^j(E,F)$ be
the space of all continuous homogeneous polynomials
of degree~$j$ from $E$ to~$F$,
and $\Sym^j(E,F)$ be the space of continuous symmetric
$j$-linear maps from $E^j$ to~$F$.
We give both
spaces the compact-open topology.
\end{numba}
\begin{la}\label{co=coinft}
The inclusion map $\lambda\colon \Pol^j(E,F)\to C^\infty(E,F)$
is a topological embedding.
\end{la}
\begin{proof}
Consider the continuous
map $\Delta_j\colon E\to E^j$, $x\mto (x,\ldots,x)$.
Then
\[
\Theta_j\colon \Sym^j(E,F)\to \Pol^j(E,F)\,,\quad \beta\mto\beta\circ \Delta_j
\]
is an isomorphism of vector spaces
and continuous linear (by Lemma~\ref{sammelsu}(e)).
Combining Lemma~\ref{sammelsu}(e) with the Polarization
Formula, we see that also $\Theta_j^{-1}$ is continuous.
For all $i\in \N$ such that $i\leq j$,
we have
\[
d^{\,(i)}(\Theta_j(\beta))(x,y_1,\ldots, y_i)\;=\; \frac{j!}{(j-i)!}\,
\beta(x,\ldots, x,y_1,\ldots, y_i)
\]
(with $x$ in $j-i$ slots)
for $x,y_1,\ldots, y_i\in E$
(see \cite[Lemma~1.6.10]{GaN}).
Define $h \colon E^{i+1}\to E^j$, $(x,y_1,\ldots, y_i)\mto
(x,\ldots, x,y_1,\ldots, y_i)$.
By the above, $d^{\,(i)}\circ \Theta_j\colon
\Sym^j(E,F)\to C(E^{i+1},F)$ is a restriction of
the continuous map
\[
{\textstyle \frac{j!}{(j-i)!}}\, C(h,F)\colon C(E^j,F) \to C(E^{i+1},F)
\]
(see Lemma~\ref{sammelsu}(e))
and hence continuous.
Thus $\lambda$ is continuous.
The inclusion map $\Lambda\colon C^\infty(E,F)\to C(E,F)$
is continuous and $\Lambda\circ\lambda$ a topological
embedding. Thus $\lambda^{-1}=(\Lambda\circ \lambda)^{-1}\circ\Lambda$
is continuous and
$\lambda$ is a topological embedding.\vspace{-3mm}
\end{proof}
\begin{numba}\label{weiterprer}
In this section, we endow $\R^d$ with the maximum norm,
$\|(x_1,\ldots, x_d)\|_\infty:=\max\{|x_1|,\ldots, |x_d|\}$
and write $\wb{B}_r(x)\sub\R^d$ for corresponding closed balls.
Let $\Omega\sub \R^d$ be open, $F$ be a locally convex space, $K\sub \Omega$
be a compact subset and
$\gamma\colon \Omega\to F$
a $C^\ell$-map, where $\ell\in \N_0\cup\{\infty\}$.
For $k\in\N_0$ such that $k\leq\ell$,
we define
\[
\|\gamma\|_{C^k\!,K,q}\; :=\; \max_{j=0,\ldots, k}\, \sup_{x\in K} \|\delta_x^j\gamma\|_q
\]
if $q\colon F\to [0,\infty[$ is a continuous seminorm.
If $ h\colon \R^d\to\R$ is a $C^\ell$-map with compact support,
we let
\[
\|h\|_{C^k}\; :=\; \max_{j=0,\ldots,k}\,\sup_{x\in\R^d}\|\delta_x^j h\|_{|\cdot|}.
\]
\end{numba}
The following observation shall be proved in Appendix~\ref{appA}.
\begin{la}\label{S3new}
If $K$ ranges through a cofinal subset $\cK\sub\cK(\Omega)$,
$k$ through the set of non-negative integers $\leq\ell$
and $q$ ranges through a directed set $\Gamma$ of seminorms on~$F$ defining the locally convex topology
of~$F$, then the $\|\cdot\|_{C^k\!,K,q}$
form a directed set of continuous seminorms on $C^\ell(\Omega,F)$
which define the compact-open $C^\ell$-topology.
\end{la}
{\bf Proof of Lemma~\ref{smoohoo}.}
To construct a well-behaved
``periodic'' partition
of unity on $\R^d$,
let $\xi\colon \R^d\to\R$ be a non-negative
$C^\infty$-map
with support $\Supp(\xi)\sub \; ]{-1},1[^d$,
such that $\xi(x)>0$ for each $x\in [{-\frac{1}{2}},\frac{1}{2}]^d$.
Given $z\in \Z^d$, define
a non-negative
smooth function
$h_z\colon \R^d\to \R$ via
\[
h_z(x)\; :=\; \frac{\xi(x-z)}{\sum_{w\in \Z^d}\xi(x-w)}\,.
\]
Then the supports
$\Supp(h_z)\sub z+\; ]{-1},1[^d$
form a locally finite cover of~$\R^d$, and
$\sum_{z\in \Z^d}h_z=1$ pointwise.
Furthermore,
\[
h_z(x)\;=\; h_0(x-z)\quad\mbox{for all $z\in \Z^d$ and $x\in \R^d$.}
\]
By construction,
each point $x\in \R^d$ has a neighbourhood~$W$
on which at most~$2^d$ of the functions~$h_z$
are non-zero.\footnote{If $\Supp(\xi)\sub [{-(1-\ve)},1-\ve]^d$ with $\ve\in \, ]0,1[$,
we can take $W:=B_\ve(x)$.}
Given $n\in \N$, let $M_n$ be the set of all
$z\in \Z^d$ such that $\frac{z}{n}  +[{-\frac{1}{n}},\frac{1}{n}]^d\sub \Omega$.
For each $z\in M_n$, we define
\[
h_{n,z}\colon \Omega\to \R\,,\quad h_{n,z}(x)\, :=\, h_z(n x);
\]
then $\Supp(h_{n,z})\sub \frac{z}{n}+\,]{-\frac{1}{n}},\frac{1}{n}[^d$.
We now associate to $\gamma\in C^\ell(\Omega, F)$
a smooth function $\wt{S}_n(\gamma)\colon \Omega\to F$ via
\begin{equation}\label{glory}
\wt{S}_n(\gamma)(x)\; :=\;
\sum_{z\in M_n}h_{n,z}(x)\cdot P_{\frac{z}{n}}^\ell(\gamma)(x-{\textstyle \frac{z}{n}}).
\end{equation}
It is clear that $\wt{S}_n\colon C^\ell(\Omega,F)\to C^\infty(\Omega,F)$ is linear.
To see that~$\wt{S}_n$ is continuous,
it suffices to show that
for each relatively compact, open subset
$U\sub \Omega$,
the map
\[
g_U\colon C^\ell(\Omega,F)\to C^\infty(U,F)\, , \quad g_U(\gamma):=\wt{S}_n(\gamma)|_U
\]
is continuous (see Lemma~\ref{sammelsu}(f)).
But $\Phi_n(U) := \{z\!\in\! M_n\colon \Supp(h_{n,z})\cap U\not=\emptyset\}$
is finite and
\[
g_U(\gamma)=
\sum_{z\in \Phi_n(U)}h_{n,z} \cdot P_{\frac{z}{n}}^\ell(\gamma)(\cdot
-{\textstyle \frac{z}{n}}).
\]
In view of Lemma~\ref{sammelsu}(e) (which applies to translations in the domain)
and the fact that
the multiplication
operator $C^\infty(U,F)\to C^\infty(U,F)$, $\eta\mto h_{n,z}\cdot\eta$
is continuous (Lemma~\ref{sammelsu}(g)),
$g_U$ will be continuous if the map
\[
\zeta \colon C^\ell(\Omega,F)
\to \Pol^k(\R^d,F) \sub C^\infty(\R^d,F)\,,\quad
\gamma\mto \delta^k_{\frac{z}{n}}(\gamma)
\]
is continuous
for each fixed $m\in \Phi_n(U)$
and each $k\in \N_0$ such that $k\leq \ell$.
By Lemma~\ref{co=coinft},
we need only show that $\zeta$ is continuous
as a map to $C(\R^d,F)$
with the compact-open topology.
Consider the continuous map
\[
\alpha\colon \R^d\to \Omega\times (\R^d)^k\,,\quad
\alpha(y):=({\textstyle \frac{z}{n}},y,\ldots, y)\, .
\]
Then $\zeta= C(\alpha, F)\circ d^{\,(k)}$
is continuous, as
$d^{\,(k)}\colon C^\ell(\Omega,F)\to C(\Omega\times (\R^d)^k, F)$,
$\gamma\mto d^{\,(k)}(\gamma)$
is continuous by definition of the compact-open $C^\ell$-topology
and $C(\alpha, E)\colon C(\Omega\times (\R^d)^k,E)\to C(\R^d,E)$
is continuous (see Lemma~\ref{sammelsu}(e)).\\[2.3mm]
(a) Let
$K\sub \Omega$ be compact,
$L$ be a compact neighbourhood
of~$K$ in~$\Omega$,
and $\gamma\in C^\ell(\Omega, F)$.
There is $n_0\in \N$ such that $K+\wb{B}_{\frac{1}{n_0}}(0)\sub L$,\vspace{-.3mm}
where $\wb{B}_\frac{1}{n_0}(0)$ is as in~{\bf\ref{weiterprer}}.
If $x\in K$, let
$M_n(x):= \{z\in M_n\colon x\in \frac{z}{n}+[-\frac{1}{n},\frac{1}{n}]^d\}$.
Then\footnote{$A:=\bigcup_{z\in M_n\setminus M_n(x)}
(\frac{z}{n}+[{-\frac{1}{n}},\frac{1}{n}]^d)$ is closed in~$\Omega$ by local
finiteness. Hence $\Omega\setminus A$ is an open neighbourhood of~$x$ in~$\Omega$
and $\wt{S}_n(\gamma)(y)=\sum_{z\in M_n(x)}h_{n,z}(y)\big(P^\ell_{\frac{z}{n}}\gamma\big)(y-\frac{z}{n})$
for $y\in \Omega\setminus A$.}
\begin{eqnarray}
\hspace*{-5mm}\lefteqn{\delta^k_x(\gamma-\wt{S}_n(\gamma))(y)
\, =\,
\delta^k_x\left(
\sum_{z\in M_n(x)}h_{n,z}(\cdot) \big(\gamma(\cdot)-\big(P^\ell_{\frac{z}{n}}\gamma\big)(\cdot-{\textstyle\frac{z}{n}})
\big)
\right)\!(y)}\quad \notag \\
& = &
\sum_{z\in M_n(x)}\sum_{j=0}^k
\left(
\begin{array}{c}
k\\
j
\end{array}
\right)
\delta^{k-j}_x(h_{n,z})(y)\cdot
\delta^j_x\big(\gamma-
\big(P^\ell_{\frac{z}{n}}\gamma\big)(\cdot -{\textstyle\frac{z}{n}})\big)(y) \label{usex3}
\end{eqnarray}
for $k\in \{0,\ldots, \ell\}$, $x\in K$ and $y\in \R^d$.
Consider the auxiliary function
\[
\phi\colon \Omega \times \Omega\to F\, ,\quad
\phi(x,z):=\gamma(x)-P^\ell_z(\gamma)(x-z)\, .
\]
Abbreviate $\phi^z:=\phi(\cdot,z)$.
By definition of $P^\ell_z(\gamma)$,
we have $\delta^j_0(P^\ell_z(\gamma))=\delta^j_z(\gamma)$
for $j\in \{0,\ldots, \ell\}$ and hence
$\delta^j_z (\phi^z)=0$
for $j\in \{0,\ldots, \ell\}$. The map
\[
\psi_j\colon \Omega\times \Omega\to \Pol^j(\R^d,F)\,,\quad
(x,z)\mto \delta^j_x(\phi^z)=(\delta^j_{(x,z)}\phi)(\cdot,0)
\]
is $C^{\ell-j}$.
In fact: The mapping $d^{\,(j)}\phi\colon (\Omega\times \Omega)\times
(\R^d\times \R^d)^j\to F$ is $C^{\ell-j}$
by \cite[Remark~1.3.13]{GaN}, whence also
\[
b\colon (\Omega\times \Omega)\times \R^d\to F,\;\;
b(x,z,y):=d^{\,(j)}\phi((x,z),(y,0),\ldots, (y,0))
\]
is $C^{\ell-j}$ and thus $C^{\ell-j,0}$.
Consequently,
\[
b^\vee\colon \Omega\times\Omega\to C(\R^d,F)\,,\quad
b^\vee(x,z)(y):=b(x,z,y)
\]
is $C^{\ell-j}$ (see \cite[Theorem~A]{AaS}).
Thus $\psi_j$ is $C^{\ell-j}$,
as the corestriction of
$b^\vee$ to the closed vector subspace
$\Pol^j(\R^d,F)$ of $C(\R^d,F)$ (see Lemma~\ref{intoclo}).\\[2.5mm]
Set $\psi_j^z:=\psi_j(\cdot, z)\colon \Omega\to \Pol^j(\R^d,F)$.
For $w\in\R^d$, let $\ve_w\colon \Pol^j(\R^d,F)\to F$, $p\mto p(w)$
be the evaluation at~$w$, which is continuous and linear.
For $i\in \{0,\ldots, \ell-j\}$, $x\in\Omega$ and $v,w\in\R^d$, we then have
\begin{eqnarray}
\delta^i_x(\psi_j^z)(v)(w) &=&
\ve_w(\delta^i_x(\psi^z_j)(v)) \;=\;
\ve_w\Big(\frac{d^i}{dt^i}\Big|_{t=0}\psi^z_j(x+tv)\Big)\notag \\[.5mm]
&=& \frac{d^i}{dt^i}\Big|_{t=0}\ve_w(\psi^z_j(x+tv))\;=\;
\frac{d^i}{dt^i}\Big|_{t=0}(\delta^j_{x+tv}\phi^z)(w)\notag\\[.3mm]
&= &  \frac{d^i}{dt^i}\Big|_{t=0}d^{\,(j)}\phi^z(x+tv,\underbrace{w,\ldots,w}_{j}\,)\notag\\[-1mm]
&=& d^{\,(i+j)}\phi^z(x,\underbrace{w,\ldots, w}_{j},
\underbrace{v,\ldots, v}_{i}).\label{dousta}\vspace{-.7mm}
\end{eqnarray}
Since $(P^\ell_z\gamma)(\cdot-z)$
is a polynomial of degree $\leq \ell$,
the $\ell$th G\^{a}teaux differential $\delta^\ell_y((P^\ell_z\gamma)(\cdot-z))$
is independent of $y\in\R^d$.
Thus
\begin{equation}\label{newq}
\psi_\ell(x,z)\;=\; \delta^\ell_x(\phi^z)\;=\; \delta^\ell_x\gamma-\delta^\ell_z\gamma
\end{equation}
for all $(x,z)\in\Omega\times\Omega$, as
\[
\delta^\ell_x(P^\ell_z\gamma)(\cdot-z))
\,=\, \delta_z^\ell((P^\ell_z\gamma)(\cdot-z))\,=\,
\delta^\ell_z\gamma.
\]
Since $d^{\,(i+j)}(\phi^z)(z,\cdot)$
can be recovered from
$\delta^{i+j}_z(\phi^z)=0$ via the Polarization Formula,
we have $d^{\,(i+j)}\phi^z(z,\cdot)=0$
and hence
\begin{equation}\label{allzero}
\delta^i_z(\psi_j^z) = 0\quad\mbox{for all $\, j\in \{0,\ldots,\ell\}\,$
and $\,i \in\{0,\ldots,\ell-j\}$.}
\end{equation}
For $j\in \{0,1,\ldots, \ell-1\}$,
using Taylor's Theorem (see \ref{tay}) and (\ref{allzero})
we obtain
\begin{eqnarray}
\psi_j(x,z) &=& \psi_j^z(x)\,=\psi_j^z(z+(x-z))\,=\,
\psi_j^z(z+(x-z))-
\sum_{i=0}^{\ell-j} \frac{\delta^i_z (\psi_j^z)(x-z)}{i!}\notag\\
&=&
\frac{1}{(\ell-j-1)!}\int_0^1\! (1-t)^{\ell-j-1}
\big(\delta^{\ell-j}_{z+t(x-z)}\psi_j^z-\delta^{\ell-j}_z\psi^z_j)(x-z)\, dt\notag \\
&=&
\frac{1}{(\ell-j-1)!}\int_0^1\! (1-t)^{\ell-j-1}
(\delta^{\ell-j}_{z+t(x-z)}\psi_j^z)(x-z)\, dt \label{goodtay}
\end{eqnarray}
for all $x\in K$ and $z\in \wb{B}_{1/n_0}(x)\sub L$.
Since $\zeta:= z+t(x-z)\in L$, we infer
that
$\|\psi_j(x,z)\|_q  \leq
\frac{2\,(2\ell)^\ell \|\gamma\|_{C^\ell,L,q}\, (\|x-z\|_\infty)^{\ell-j}}{\ell!(\ell-j-1)!}$
and thus
\begin{equation}\label{reuest}
\hspace*{-2mm}\|\psi_j(x,z)\|_q  \; \leq\;
\frac{2\,(2\ell)^\ell}{\ell!} \|\gamma\|_{C^\ell,L,q}\, (\|x-z\|_\infty)^{\ell-j}
\end{equation}
for all $j\in\{0,\ldots, \ell-1\}$, $x\in K$, $z\in \wb{B}_{1/n_0}(x)$
and continuous seminorm~$q$ on~$F$,
using (\ref{conspol}), (\ref{dousta})
and (\ref{newq})
to estimate seminorms as in (\ref{hompolsemi}) via
\begin{eqnarray}
\|(\delta^{\ell-j}_\zeta\psi^z_j)(x-z)\|_q
&= & \|y\mto d^{\,(\ell)}\phi^z(\zeta ,\underbrace{x-z,\ldots ,
x-z}_{\ell-j},\underbrace{y,\ldots, y}_{j})\|_q\notag \\[-.5mm]
&\leq & \|d^{\,(\ell)}\phi^z(\zeta,\sbull)\|_q\cdot \|x-z\|^{\ell-j}_\infty\notag \\
&\leq &
\frac{(2\ell)^\ell}{\ell!} \|\delta^\ell_\zeta(\phi^z)\|_q \cdot \|x-z\|^{\ell-j}_\infty\notag \\
&\leq & \frac{(2\ell)^\ell}{\ell!}
\|\delta^\ell_\zeta\gamma-\delta^\ell_z\gamma\|_q \cdot \|x-z\|^{\ell-j}_\infty\label{fststr}
\end{eqnarray}
We mention that (\ref{reuest}) also holds for $x$, $z$, $q$ as before and $j=\ell$,
as
\[
\|\psi_\ell(x,z)\|_q\,=\, \|\delta^\ell_x\gamma-\delta^\ell_z\gamma\|_q\,\leq
\,\|\delta^\ell_x\gamma\|_q+\|\delta^\ell_z\gamma\|_q\,\leq\,
2\|\gamma\|_{C^\ell,L,q}
\]
for all $x,z\in L$.\\[2.3mm]
If $\ve>0$, the uniform continuity
of $L\to \Pol^\ell(\R^d,F)$, $x\mto \delta^\ell_x\gamma$
implies that there is $n_1\geq n_0$ such that
\begin{equation}\label{nowforall}
\|\delta^\ell_z\gamma-\delta^\ell_x\gamma \|_q\;\leq\; \frac{\ell!}{(2\ell)^\ell}\ve\quad
\mbox{for all $x,z\in L$ such that $\|x-z\|_\infty\leq \frac{1}{n_1}$,}
\end{equation}
whence
\begin{equation}\label{forellcase}
\|\delta^\ell_z\gamma-\delta^\ell_x\gamma\|_q\leq \ve
\end{equation}
in particular.
Combining this with
(\ref{goodtay}) and (\ref{fststr}), we obtain
$\|\psi_j(x,z)\|_q \leq \frac{\ve\, \|x-z\|_\infty^{\ell-j}}{(\ell-j-1)!}$
and thus
\begin{equation}\label{usex2}
\|\psi_j(x,z)\|_q \; \leq\;  \ve\, \|x-z\|_\infty^{\ell-j}
\end{equation}
for all $j\in \{0,\ldots, \ell-1\}$,
$x\in K$ and $z\in\R^d$
with $\|x-z\|_\infty\leq\frac{1}{n_1}$.
By~(\ref{newq}) and~(\ref{forellcase}),
we have $\|\psi_\ell(x,z)\|_q\leq \ve$ for all $x,z\in
L$ with $\|x-z\|_\infty\leq\frac{1}{n_1}$,
whence (\ref{usex2}) also holds for $j=\ell$.\\[2.3mm]
Since $\|h_{n,z}\|_{C^{k-j},K,|\cdot|}\leq n^{k-j}\|h_0\|_{C^{k-j}}$
and $M_n(x)$ has at most $2^d$ elements,
(\ref{usex3}) and (\ref{usex2}) yield
\begin{eqnarray*}
\|\delta^k_x(\gamma-\wt{S}_n(\gamma))\|_q
& \!\!\leq\!\! &
\sum_{z\in M_n(x)}\sum_{j=0}^k
\left(
\begin{array}{c}
k\\
j
\end{array}
\right) n^{k-j}\|h_0\|_{C^{k-j}}
\cdot
\underbrace{\|\psi_j(x,{\textstyle\frac{z}{n}})\|_q}_{\leq\ve (\frac{1}{n})^{\ell-j}}\\[-2mm]
& \!\! \leq \!\! &
\ell!\, 2^d\, \|h_0\|_{C^\ell}
\sum_{j=0}^k
n^{k-j}({\textstyle\frac{1}{n}})^{\ell-j} \ve
\; \leq \; \ell!\, (\ell+1)\, 2^d\, \|h_0\|_{C^\ell}\, \ve\\[-7mm]
\end{eqnarray*}
for all $k\in\{0,\ldots,\ell\}$, $n\geq n_1$ and $x\in K$.
Thus
\begin{equation}\label{betterso}
\|\gamma-\wt{S}_n(\gamma)\|_{C^\ell\!,K,q} \; \leq\;
\ell! \, (\ell+1) \, 2^d \, \|h_0\|_{C^\ell}\, \ve\quad\mbox{for all $\,n\geq n_1$,}
\end{equation}
which is arbitrarily small for small~$\ve$.\\[2.3mm]
(b) Let $n_0$ be as in the proof of~(a)
and~$q$ be a continuous seminorm on~$F$.
Repeating the estimates leading to~(\ref{betterso})
with (\ref{reuest}) instead of (\ref{usex2}), we obtain
\begin{equation}\label{elsewhere}
\|\gamma-\wt{S}_n(\gamma)\|_{C^\ell\!,K,q} \; \leq\;
(\ell+1) \, 2^{d+1} \, (2\ell)^\ell \, \|h_0\|_{C^\ell}\,
\|\gamma\|_{C^\ell\!,L,q}\quad \mbox{for all $n\geq n_0$.}
\end{equation}
Thus $\|\wt{S}_n(\gamma)\|_{C^\ell\!,K,q}\leq\|\gamma\|_{C^\ell\!,K,q}+
\|\wt{S}_n(\gamma)-\gamma\|_{C^\ell\!,K,q}\leq C\|\gamma\|_{C^\ell\!,L,q}$
with
$C:=1+(\ell+1)2^{d+1}(2\ell)^\ell\|h_0\|_{C^\ell}$.\\[2.3mm]
(c) We find $m_1<m_2<\cdots$ such that
\[
K_j+{\textstyle[{-\frac{2}{m_j}},\frac{2}{m_j}]^d}\,\sub\, K_{j+1}^0
\]
for all $j\in \N$. For each $j\in\N$,
\[
\Phi_j:=\{z\in M_{m_j}\colon \Supp(h_{m_j,z})\cap K_j\not=\emptyset\}
\]
is a finite set; we define $S_j(\gamma)\in C^\infty(\Omega,F)$ via
\[
S_j(\gamma)(x)\; :=\;
\sum_{z\in \Phi_j}h_{m_j,z}(x)\cdot P_{\frac{z}{m_j}}^\ell(\gamma)(x-{\textstyle \frac{z}{m_j}})\vspace{-.5mm}
\]
for $\gamma\in C^\ell(\Omega,F)$ and $x\in\Omega$.
It is clear that $S_j\colon C^\ell(\Omega,F)\to C^\infty(\Omega,F)$ is linear.
As $S_j(\gamma)$ is a finite sum of summands
which are continuous functions of~$\gamma$ (as
already shown), the function $S_j\colon C^\ell(\Omega,F)\to C^\infty(\Omega,F)$
is continuous.
We have
\[
\Supp(S_j(\gamma))\sub\bigcup_{z\in \Phi_j}\Supp(h_{m_j,z})
\sub\bigcup_{z\in\Phi_j}B_{\frac{1}{m_j}}(z)\sub K_j
+{\textstyle [{-\frac{2}{m_j}},\frac{2}{m_j}]^d}
\sub K_{j+1}^0
\]
and thus $S_j(\gamma)\in C^\infty_{K_{j+1}}(\Omega,F)$.
As $P_{\frac{z}{m_j}}^\ell(\gamma)(\cdot-{\textstyle \frac{z}{m_j}})\colon\R^d\to F$
is a polynomial, it is an element of
$F\otimes C^\infty(\R^d,\R)$. The restriction of the polynomial
to the domain~$\Omega$ then is an element of $F\otimes C^\infty(\Omega,\R)$
and this is unchanged if we multiply with $h_{m_j,z}$.
Thus $S_j(\gamma)$ is a sum of
summands which are elements of the tensor product, and thus
also $S_j(\gamma)\in F\otimes C^\infty(\Omega,\R)$.
Thus $S_j(\gamma)(\Omega)$ spans a finite-dimensional
vector subspace of~$F$.
Since $S_j(\gamma)\in C^\infty_{K_{j+1}}(\Omega,F)$,
this implies that
\[
S_j(\gamma)\in F\otimes C^\infty_{K_{j+1}}(\Omega,\R).
\]
The definition of $\Phi_j$ entails that
\begin{equation}\label{willu}
S_j(\gamma)|_{K_j}=\wt{S}_{m_j}(\gamma)|_{K_j}
\end{equation}
for all $j\in\N$ and $\gamma\in C^\ell(\Omega,F)$.
In fact, $S_j(\gamma)$ and $\wt{S}_{m_j}(\gamma)$
coincide on the open complement of the closed set
\[
\bigcup_{z\in M_{m_j}\setminus \Phi_j}\Supp(h_{m_j,z}),
\]
whence their $k$th G\^{a}teaux derivatives
coincide on~$K_j$ for all $k\in \N_0$.
If~$q$ is a continuous
seminorm on~$F$ and $j\in\N$, we can use $n_0:=m_j$
in the proof of~(a) with $K:=K_j$,
$L:=K_{j+1}$.
Thus
\[
\|S_j(\gamma)\|_{C^\ell\!,K,q}=
\|\wt{S}_{m_j}(\gamma)\|_{C^\ell\!,K,q}\leq C\,\|\gamma\|_{C^\ell\!,L,q}
\]
for all $\gamma\in C^\ell(\Omega,F)$,
using (b) and (\ref{willu}).
Pick~$n_1\geq n_0=m_j$ as in the proof of~(a).
There is $j_0\geq j$ such that
$m_i\geq n_1$ for all $i\geq j_0$. 
By (\ref{betterso}) and (\ref{willu}), we then have
\begin{equation}\label{newso}
\|\gamma-S_i(\gamma)\|_{C^\ell\!,K,q}\,=\, \|\gamma-\wt{S}_{m_i}(\gamma)\|_{C^\ell\,,K,q}
\, \leq\,
(\ell+1) \, 2^d \, \|h_0\|_{C^\ell}\, \ve
\end{equation}
for all $i\geq j_0$, showing that $S_i(\gamma)\to\gamma$ in $C^\ell(\Omega,F)$
as $i\to\infty$.\\[2.3mm]
To see that also condition~(c) of Theorem~\ref{thmsmoo}
is satisfied, let $i\in\N$ and $\gamma\in C^\ell_{K_i}(\Omega,F)$.
For integers $j\geq i$, define
$\Phi_{i,j}:=\{z\in\Phi_j\colon \frac{z}{m_j}\in K_i\}$.
Since $P^\ell_{\frac{z}{m_j}}(\gamma)=0$ if $\frac{z}{m}\not\in K_i$,
we have
\[
S_j(\gamma)=\sum_{z\in {\textstyle \Phi_{i,j}}h_{m_j,z}\cdot
P^\ell_{\frac{z}{m_j}}(\gamma)(\cdot-\frac{z}{m_j})}
\]
and thus $\Supp(S_j(\gamma))\sub K_{i+1}^0$, using that
\[
{\textstyle
\Supp(h_{m_j,z})\,\sub \,\frac{z}{m_j}+\,]{-\frac{1}{m_j}},\frac{1}{m_j}[^d\,
\, \sub\,  K_i+\,]{-\frac{1}{m_j}},\frac{1}{m_j}[^d
\, \sub\,  K_i+\,]{-\frac{1}{m_i}},\frac{1}{m_i}[^d
\, \sub\, K_{i+1}^0}
\]
for all $z\in\Phi_{i,j}$. $\,\square$
\begin{rem}\label{oncpset}
We mention that $\wt{S}_n(\gamma)\to\gamma$ as $n\to\infty$
and $S_i(\gamma)\to\gamma$ as $i\to\infty$
in $C^\ell(\Omega,F)$
uniformly for $\gamma$ in compact sets.\\[2mm]
To see this, let $B\sub C^\ell(\Omega,F)$ be compact.
The map
\[
h\colon C^\ell(\Omega,F)\to C(\Omega\times\R^d,F),\quad
h(\gamma)(x,y):=\delta^\ell_x\gamma(y)
\]
is continuous by definition of the compact-open $C^\ell$-topology
and Lemma~\ref{sammelsu}(e).
As $\Omega\times\R^d$ is locally compact,
the evaluation map $\ev\colon C(\Omega\times\R^d,F)\times\Omega\times\R^d\to F$,
$(\eta,x,y)\mto\eta(x,y)$ is continuous.
Hence
\[
g\colon C^\ell(\Omega,F)\times\Omega\times\R^d\to F,\quad
(\gamma,x,y)\mto \ev(h(\gamma),x,y)=\delta^k_x\gamma(y)
\]
is continuous and hence also the map
\[
g^\vee\colon C^\ell(\Omega,F)\times \Omega\to\Pol^\ell(\R^d,F)\sub C(\R^d,F),\;\,
(\gamma,x)\mto \delta^\ell_x\gamma.
\]
For $K$, $L$, $\ve$, $q$, and $n_0$ as in the proof of Lemma~\ref{smoohoo}(a),
the restriction
\[
g^\vee|_{B\times L}\to\Pol^\ell(\R^d,F)
\]
is uniformly continuous with respect to the unique
compatible uniformity on the compact set $B\times L$.
As a consequence, there exists $n_1\geq n_0$ such that
(\ref{nowforall}) holds for all $\gamma\in B$.
Thus~(\ref{betterso})
holds for all $\gamma\in B$ and also~(\ref{newso}).
\end{rem}
\section{Proof of Proposition~\ref{etensor} and Theorem~\ref{scalarseq}}
We now prove our results concerning completed tensor products and passage
from extension operators for real-valued functions
to extension operators for vector-valued functions.
\begin{la}\label{densy}
Let $F$ be a locally convex space,
$\ell\in\N_0\cup\{\infty\}$, $d\in\N$ and
$R\sub\R^d$ be a convex subset with dense
interior.
Then $F\otimes C^\infty(R,\R)$ is dense
in $C^\ell(R,F)$.
If $\ell<\infty$
and $\gamma\in C^\ell(R,F)$,
then there exist elements $\gamma_{k,n}\in F\otimes C^\infty(R,\R)$
for $k,n\in \N$ such that
\begin{equation}\label{3limit}
\gamma=\lim_{k\to\infty}\lim_{n\to\infty}\gamma_{k,n}
\;\;\mbox{{\rm in} $\, C^\ell(R,F)$.}
\end{equation}
If $\ell<\infty$ and $R$ is closed in~$\R^d$,
then $F\otimes C^\infty_c(R,\R)$ is dense
in $C^\ell(R,F)$ and the $\gamma_{k,n}$ can be chosen
in $F\otimes C^\infty_c(R,\R)$.
\end{la}
{\bf Proof of Lemma~\ref{densy}.}
Let $x_0\in R^0$; after replacing $R$ with $R-x_0$,
we may assume that $0\in R^0$. For each $t\in \;]1,\infty[$,
the set $R$ is contained in the interior $tR^0$ of $tR$.
Let $(t_k)_{k\in\N}$ be a sequence in $]1,\infty[$
with $t_k\to 1$ as $k\to\infty$.
The~map
\[
f\colon [1,\infty[\,\times R \to F,\quad f(t,x):=\gamma(x/t)
\]
is $C^\ell$ and thus $C^{0,\ell}$, entailing that
\[
f^\vee\colon [1,\infty[\,\to C^\ell(R,F),\quad t\mto f(t,\cdot)
\]
is continuous (see \cite[Theorem~A]{AaS}). Hence
\begin{equation}\label{1in3}
\gamma=\lim_{k\to\infty}\gamma_k\;\,\mbox{in $\, C^\ell(R,F)$}
\end{equation}
with $\gamma_k:=f^\vee(t_k)\colon R\to F$, $x\mto \gamma(x/t_k)$. 
Define $\eta_k\in C^\ell(t_kR^0,F)$ via $\eta_k(x):=\gamma(x/t_k)$.\\[2.3mm]
\emph{The case $\ell<\infty$.}
By Theorem~\ref{thmsmoo},
there exist $\eta_{k,n}\in F\otimes C^\infty_c(t_kR^0,\R)$
with
\[
\eta_k=\lim_{n\to\infty}\eta_{k,n}\;\,\mbox{in $\,C^\ell(t_kR^0,F)$.}
\]
Then $\gamma_{k,n}:=\eta_{k,n}|_R\in F\otimes C^\infty(R,\R)$
and
\[
\gamma_k=\eta_k|_R=\lim_{n\to\infty}\gamma_{k,n}\;\,
\mbox{in $\,C^\ell(R,F)$,}
\]
as the restriction map $C^\ell(t_kR^0,F)\to C^\ell(R,F)$ is continuous.
Substituting this into (\ref{1in3}), we obtain~(\ref{3limit}).
If $R$ is closed, then $\gamma_{k,n}$ has compact support
and thus $\gamma_{k,n}\in F\otimes C^\infty_c(R,\R)$.\\[2.3mm]
\emph{The case $\ell=\infty$.}
Given $\gamma\in C^\infty(R,F)$,
let $K\sub R$ be a compact subset,
$q\colon F\to\R$ be a continuous seminorm,
$j\in\N_0$, and $\ve>0$.
By the preceding, applied with $j$ in place of~$\ell$,
we find $\gamma_{k,n}\in F\otimes C^\infty(R,\R)$
such that
\[
\gamma=\lim_{k\to\infty}\lim_{n\to\infty}\gamma_{k,n}
\;\;\mbox{{\rm in} $\, C^j(R,F)$,}
\]
and we may assume $\gamma_{k,n}\in F\otimes C^\infty_c(R,\R)$ if $R$
is closed.
Thus, we find $k$ with
\[
\Big\|\gamma-\lim_{n\to\infty}\gamma_{k,n}\Big\|_{C^j\!,K,q}<\ve.
\]
Fix $k$. For sufficiently large $n$, we then have
$\|\gamma-\gamma_{k,n}\|_{C^j\!,K,q}<\ve$.
$\,\square$\\[2.3mm]
We need more information
in the case $\ell=\infty$.
The following fact can be shown by standard
arguments.
\begin{la}\label{fairly-stand}
Let $F$ be a sequentially complete
locally convex space, $d\in\N$,
and $\theta\in C^\infty_c(\R^d,F)$.
Then there exist $\theta_{m,j}\in F\otimes C^\infty(\R^d,\R)$
for $m,j\in\N$
such that
\[
\theta=\lim_{m\to\infty}\lim_{j\to\infty}\theta_{m,j}\quad
\mbox{in $\,C^\infty(\R^d,F)$.}
\]
For each compact subset $L\sub \R^d$ with $\Supp(\theta)\sub L^0$,
the $\theta_{m,j}$ can be chosen in $F\otimes C^\infty_L(\R^d,\R)$.
\end{la}
\begin{proof}
Since $F$ is sequentially complete and $\theta$ has compact support,
the weak integral
\[
(g*\theta)(x):=\int_{\R^d}g(y)\theta(x-y)\,d\lambda_d(y)
\]
exists in~$F$ for each $x\in\R^d$ and each
$g\in C(\R^d,\R)$
(see, e.g., \cite[p.\,110]{BaG}). Moreover,
$g*\theta \in C^\infty(\R^d,F)$
and the linear map
\begin{equation}\label{lin-is-cts}
C(\R^d,\R)\to C^\infty(\R^d,F),\quad g\mto g* \theta
\end{equation}
is continuous (e.g., by the statement concerning
$\theta_K$ in \cite[Proposition 8.1]{BaG},
applied with $r:=0$ and $s:=t:=\infty$).
For $m\in \N$,
consider the Gaussian
\[
g_m\colon \R^d\to\R,\quad x\mto \left(\frac{m}{\pi}\right)^{d/2}e^{-m(\|x\|_2)^2}.
\]
For each $\delta>0$,
there exists $m_0\in\N$ such that $\int_{\|y\|_2>\delta}g_m(y)\,d\lambda_d(y)<\delta$
for all $m\geq m_0$.
Hence
\[
g_m*\theta\to\theta\quad\mbox{in $\,C(\R^d,F)$}
\]
as $m\to\infty$,
using the uniform continuity of the compactly supported continuous
function~$\theta$.
Likewise, using (3.4) in \cite[Proposition~3.2]{BaG}, we get
\[
\partial^\alpha(g_m*\theta)=g_m*\partial^\alpha\theta\to \partial^\alpha\theta\quad
\mbox{in $\,C(\R^d,F)$}
\]
for each $\alpha\in \N_0^d$, whence $\theta_m:=g_m*\theta\to \theta$ in $C^\infty(\R^d,F)$
(by Lemma~\ref{toppartial}).
For $j\in \N$, consider the
polynomial $g_{m,j}\colon \R^d\to\R$,
\[
g_{m,j}(x):=
\left(\frac{m}{\pi}\right)^{d/2}
\sum_{\nu=0}^j \frac{1}{\nu!}(-m(\|x\|_2)^2)^\nu\quad \mbox{for $\,x\in\R^d$.}
\]
Then $g_{m,j}\to g_m$ in $C(\R^d,\R)$
as $j\to\infty$. Using the continuity of
the map (\ref{lin-is-cts}), we deduce that
\[
g_{m,j}*\theta\to g_m*\theta\quad\mbox{in $\, C^\infty(\R^d,F)$}
\]
as $j\to\infty$.
It remains to observe that
$\theta_{m,j}:=g_{m,j}*\theta$ is a polynomial
for all $j$ and $m$,
since
\[
\theta_{m,j}(x)=
\int_{\R^d}g(x-y)\theta(y)\,d\lambda_d(y)
\]
for all $x\in\R^d$ (see, e.g., (3.3) in~\cite{BaG}).
Thus $\theta_{m,j}\in F\otimes C^\infty(\R^d,\R)$.
For the final assertion,
pick a smooth function $h\colon \R^d\to\R$
such that $\Supp(h)\sub L$ and $h(x)=1$ for
all $x\in \Supp(\gamma)$.
Then $h\theta_{m,j}\to h\theta_m$ for $j\to\infty$
in $C^\infty(\R^d,F)$,
and $h\theta_m\to h\theta=\theta$ as $m\to \infty$.
We can therefore replace $\theta_{m,j}$
with $h\theta_{m,j}$.
\end{proof}
The following conclusion
complements the case $\ell=\infty$ of Lemma~\ref{densy}.
\begin{la}\label{densy-infty}
Let $F$ be a sequentially complete locally convex space,
$d\in\N$,
$R\sub\R^d$ be a closed, convex subset with dense
interior, and
$\gamma\in C^\infty(R,F)$.
Then there exist elements $\gamma_{k,n,m,j}\in F\otimes C^\infty_c(R,\R)$
for $k,n,m,j\in \N$ such that
\begin{equation}\label{four-limit}
\gamma=\lim_{k\to\infty}\lim_{n\to\infty}
\lim_{m\to\infty}\lim_{j\to\infty}\gamma_{k,n,m,j}
\;\;\mbox{{\rm in} $\, C^\infty(R,F)$.}
\end{equation}
\end{la}
\begin{proof}
We may assume that $0\in R^0$.
Given $\gamma\in C^\infty(R,F)$,
let $\gamma_k\in C^\infty(R,F)$,
$t_k\in \,]1,\infty[$ and
and $\eta_k\in C^\infty(t_kR^0,F)$
be as in the proof of Lemma~\ref{densy} (applied with $\ell:=\infty$).
For each $k\in\N$,
let $(K_{k,n})_{n\in\N}$
be a compact exhaustion of $t_kR^0$
and $h_{k,n}\colon t_kR^0\to\R$
be a compactly supported
smooth function such that $h_{k,n}(x)=1$
for all $x\in K_{k,n}$.
Then $h_{k,n}\eta_k\to\eta_k$ in $C^\infty(t_kR^0,F)$
as $n\to\infty$ (cf.\ Lemma~\ref{toppartial}),
entailing that
\[
(h_{k,n}\eta_k)|_R\to \eta_k|_R=\gamma_k
\]
in $C^\infty(R,F)$. Since $h_{k,n}\eta_k$
has compact support in the open subset~$R^0$
of~$\R^d$, we can extend $h_{k,n}\eta_k$
via~$0$ to a smooth function
$\eta_{k,n}\colon \R^d\to F$. Then $\Supp(\eta_{k,n})$ is compact,
being a subset of $\Supp(h_{k,n})$.
By Lemma~\ref{fairly-stand},
there are $\theta_{k,n,m,j}\in F\otimes C^\infty_c(\R^d,\R)$
for $m,j\in\N$
such that
\[
\eta_{k,n}=\lim_{m\to\infty}\lim_{j\to\infty}\theta_{k,n,m,j}\quad\mbox{in $\,C^\infty(\R^d,F)$.}
\]
Then $\gamma_{k,n,m,j}:=\theta_{k,n,m,j}|_R\in F\otimes C^\infty_c(R,\R)$
and
\[
\gamma=\lim_{k\to\infty}\lim_{n\to\infty}
\eta_{k,n}|_R=\lim_{k\to\infty}\lim_{n\to\infty}\lim_{m\to\infty}\lim_{j\to\infty}
\gamma_{k,n,m,j}
\]
in $C^\infty(R,F)$.
\end{proof}
Lemma~\ref{densy} and \ref{densy-infty} are useful for us as they have
consequences for functions to sequentially complete
locally convex spaces.
Recall that a map $f\colon X\to Y$ between topological spaces is said to be
\emph{sequentially continuous} if $f(x_n)\to f(x)$ as $n\to \infty$
for each $x\in X$ and sequence $(x_n)_{n\in\N}$ in~$X$
such that $x_n\to x$. If $f$ is continuous, then $f$ is sequentially continuous.
\begin{la}\label{use-seqco}
Let $E$ and $F$ be locally convex spaces, $d\in\N$, $\ell\in \N_0\cup\{\infty\}$
and $R\sub\R^d$ be a convex subset with dense interior.
If $\ell=\infty$, assume that $R$ is closed in~$\R^d$
and~$F$ is sequentially complete.
Let $E_0\sub E$ be a vector subspace which is sequentially
complete in the induced topology.
If $f\colon C^\ell(R,F)\to E$ is a sequentially
continuous map such that
\[
f(F\otimes C^\infty(R,\R))\sub E_0,
\]
then $f(C^\ell(R,F))\sub E_0$.
\end{la}
\begin{proof}
If $v\in E$ and $(v_n)_{n\in\N}$ is a sequence in~$E_0$
with $v_n\to v$ as $n\to\infty$, then
\begin{equation}\label{seqinside}
v=\lim_{n\to\infty} v_n  \in E_0,
\end{equation}
as $E_0$ is sequentially complete.
The case $\ell<\infty$:
For each $\gamma\in C^\ell(R,F)$, Lemma~\ref{densy}
provides elements $\gamma_{k,n}\in F\otimes C^\infty(R,\R)$ for $k,n\in\N$
such that (\ref{3limit}) holds. As $f$ is assumed sequentially continuous,
we deduce that
\begin{equation}\label{half}
f(\gamma)=\lim_{k\to\infty}\lim_{n\to\infty}f(\gamma_{k,n}).
\end{equation}
Since $f(\gamma_{k,n})\in E_0$ for all $k,n\in\N$ by hypothesis,
using~(\ref{seqinside}) twice we deduce from~(\ref{half}) that $f(\gamma)\in E_0$.
If $\ell=\infty$,
we write $\gamma$ as a fourfold limit
as in~(\ref{four-limit})
and argue as before,
using~(\ref{seqinside}) four times now to see that
\[
f(\gamma)=\lim_{k\to\infty}\lim_{n\to\infty}\lim_{m\to\infty}\lim_{j\to\infty}
f(\gamma_{k,n,m,j})
\]
is in~$E_0$.
\end{proof}
We shall use a well-known fact:
\begin{la}\label{folklo}
Let $E$ and $F$ be complete locally convex spaces,
$E_0\sub E$ and $F_0\sub F$ dense vector subspaces and
$\lambda\colon E_0\to F_0$ be an isomorphism
of topological vector spaces.
Then the unique continuous linear map
$\Lambda\colon E\to F$ with $\Lambda|_{E_0}=\lambda$
is an isomorphism of topological vector spaces.
\end{la}
\begin{proof}
Let $\Theta\colon F\to E$ be the unique continuous
linear map which extends $\lambda^{-1}\colon F_0\to E_0$.
Then $\Theta\circ\Lambda|_{E_0}=\id_E|_{E_0}$, whence
$\Theta\circ\Lambda=\id_E$. Likewise, $\Lambda\circ\Theta=\id_F$.
Hence $\Theta=\Lambda^{-1}$.
\end{proof}
\begin{numba}\label{epstop}
Given locally convex spaces $E$ and $F$,
we write $E'_\tau$ for the space of continuous linear functionals
on~$E$, endowed with the Mackey topology
(the topology of uniform convergence on
absolutely convex weakly compact subsets of~$E$).
Write $\cL(E'_\tau,F)_\ve$ for $\cL(E_\tau',F)$,
endowed with the topology of uniform convergence
on equicontinuous subsets of~$E'$.
Recall that the $\ve$-topology on $E\otimes F$
is the initial topology with respect to the linear map
\[
\psi\colon E\otimes F\to \cL(E'_\tau,F)_\ve
\]
determined by $\psi(x,y)(\lambda):=\lambda(x)y$
for $x\in E$, $y\in F$ and $\lambda\in E'$
(cf.\ Definition 43.1 and \S42 in \cite{Tre}).
\end{numba}
\begin{numba}\label{semi-eps}
The locally convex topology on $\cL(E',F)_\ve$
is given by the seminorms
\[
\|\cdot\|_{S,q}\colon \cL(E',F)\to[0,\infty[,\quad
\alpha\, \mto \, \sup_{\lambda\in S}\, q(\alpha(\lambda)),\vspace{-1mm}
\]
for $S$ ranging through the set of equicontinuous subsets of~$E'$
and $q$ ranging through the set of all continuous seminorms on~$F$.
If $\Gamma$ is a directed set
of continuous seminorms defining the locally
convex topology on~$F$, then there exists $Q\in\Gamma$ and $r>0$ such that
$q\leq rQ$. As a consequence, $\|\cdot\|_{S,q}\leq \|\cdot\|_{S,rQ}=\|\cdot\|_{rS,Q}$.
Thus, it suffices to take $q\in \Gamma$.
Moreover, after increasing~$S$, we may assume that
\[
S=V^\circ:=\{\lambda\in E'\colon \lambda(V)\sub [{-1},1]\}
\]
is the polar of a $0$-neighbourhood $V\sub E$. After
replacing $V$ with a smaller $0$-neighbourhood (which increases the polar),
we may assume that $V=\wb{B}^p_1(0)$ for a continuous
seminorm $p$ on~$E$ (and, conversely, $\wb{B}^p_1(0)^\circ$
is equicontinuous for each continuous seminorm $p$ on~$E$).
Since
\[
\wb{B}^p_1(0)=(\wb{B}^p_1(0)^\circ)_\circ:=\{y\in E\colon (\forall \lambda\in \wb{B}^p_1(0)^\circ)\;
\lambda(y)\sub [{-1},1]\}
\]
by the Bipolar Theorem, we deduce that
\begin{equation}\label{semnormpolar}
p(y)\, =\, \sup\, \{|\lambda(y)|\colon \lambda\in \wb{B}^p_1(0)^\circ\}
\, =\, \sup_{\lambda\in S}\, |\lambda(y)|\vspace{-1mm}
\end{equation}
for all $y\in E$.
This well-known fact will be useful for us.
\end{numba}
{\bf Proof of Proposition~\ref{etensor}.}
Since $C^\ell(R,F)$ is complete (see Lemma~\ref{sammelsu}(d))
and $F\otimes C^\ell(R,\R)$
is a dense vector subspace of both $C^\ell(R,F)$
(see\linebreak
Theorem~\ref{densy}) and $F\tensor C^\ell(R,\R)$,
it suffices to
show that both $C^\ell(R,F)$ and $F\tensor C^\ell(R,\R)$
induce the same topology on $F\otimes C^\ell(R,\R)$;
the unique continuous linear map
\[
\wt{\Lambda}\colon C^\ell(R,F)\to F\tensor C^\ell(R,\R)
\]
determined by $\wt{\Lambda}(\gamma v)=v\otimes\gamma$
for $v\in F$ and $\gamma\in C^\ell(R,\R)$
then is an isomorphism of topological vector spaces (by Lemma~\ref{folklo}).
Let
\[
\Lambda\colon Y\to F\otimes C^\ell(R,\R)
\]
be the unique linear map on $Y:=F\otimes C^\ell(R,\R)\sub C^\ell(R,F)$
determined by $\Lambda(\gamma v)=v\otimes \gamma$. 
It suffices to show that the vector space
isomorphism $\Lambda$ is a homeomorphism
if we use the topology induced by $C^\ell(R,F)$ on~$Y$
and the topology induced by
$F\tensor C^\ell(R,\R)$ on the range of~$\Lambda$.
The latter holds if and only if the
linear map $h:=\psi\circ\Lambda$
is a topological embedding, for
\[
\psi\colon F\otimes C^\ell(R,\R)\to \cL(F'_\tau,C^\ell(R,\R))_\ve
\]
as in~\ref{epstop}. Note that $h(\gamma)(\lambda)=
(\psi\circ \Lambda)(\gamma)(\lambda)=\lambda\circ\gamma$
for all $\gamma\in Y$ and $\lambda\in F'$, as
$(\psi\circ\Lambda)(\gamma v)(\lambda)=\lambda(v)\gamma$ for all $v\in F$,
$\gamma\in C^\ell(R,\R)$, and $\lambda\in F'$.\\[2.3mm]
We claim that
\begin{equation}\label{keytensor}
\|h(\gamma)\|_{S,\|\cdot\|_{C^j,K}}=\|\gamma\|_{C^j,K,p}^\partial
\end{equation}
for each $j\in\N_0$ such that $j\leq \ell$, compact subset $K\sub R$,
and continuous seminorm $p$ on~$F$, with $S:=\wb{B}^p_1(0)^\circ\sub F'$.
If this is true, then the injective linear map~$h$ is an embedding.
In fact, for each continuous seminorm $Q$ on $\cL(F'_\tau,C^\ell(R,\R))_\ve$,
there exist a continuous seminorm~$p$ on~$F$, a compact subset
$K\sub R$ and $j\in\N_0$ with $j\leq\ell$ such that
\[
Q\, \leq \, \|\cdot\|_{S,\|\cdot\|_{C^j,K}}
\]
with $S:=\wb{B}^p_1(0)^\circ$, by \ref{semi-eps}.
Then $Q(h(\gamma))\leq \|h(\gamma)\|_{S,\|\cdot\|_{C^j,K}}=\|\gamma\|_{C^j,K,p}^\partial$
for each $\gamma\in Y$, by~(\ref{keytensor}),
showing that the linear map~$h$ is continuous.
If $P$ is a continuous seminorm on~$Y$, then there exist
$j\in \N_0$ with $j\leq \ell$, a compact subset $K\sub R$,
and a continuous seminorm~$p$ on~$F$ such that
\[
P(\gamma)\leq \|\gamma\|_{C^j,K,p}^\partial =\|h(\gamma)\|_{S,\|\cdot\|_{C^j,K}},
\]
where $S:=\wb{B}_1^p(0)^\circ$. The linear map
$(h|^{h(Y)})^{-1}$ is therefore continuous, and thus~$h$ is
a topological embedding.
But
\begin{eqnarray*}
\|h(\gamma)\|_{S,\|\cdot\|_{C^j,K}}&=&
\sup_{\lambda\in S}\, \|\lambda\circ\gamma\|_{C^j,K}\\[.3mm]
&=&\sup_{\lambda\in S}\,\max_{|\alpha|\leq j}\,\sup_{x\in K}\, \underbrace{|\partial^\alpha(\lambda\circ \gamma)(x)|}_{=|\lambda(\partial^\alpha\gamma(x))|}\\[.3mm]
&=&\max_{|\alpha|\leq j}\,\sup_{x\in K}\, \underbrace{\sup_{\lambda\in S}\, |\lambda(\partial^\alpha\gamma(x))|}_{=p(\partial^\alpha\gamma(x))}\;=\;
\|\gamma\|_{C^j,K,p}^\partial
\end{eqnarray*}
for each $\gamma\in Y$, using the Chain Rule and~(\ref{semnormpolar});
this proves the claim. $\,\square$\\[2.3mm]
{\bf Proof of Theorem~\ref{scalarseq}.}
Let $\cE\colon C^\ell(R,\R)\to C^\ell(\R^d,\R)$ be a continuous
linear map such that $\cE(\gamma)|_R=\gamma$ for each $\gamma\in C^\ell(R,\R)$.
Let~$\wt{F}$ be a completion of $F$ such that $F\sub\wt{F}$.
There is a unique continuous linear map
\[
\id_{\wt{F}}\tensor \cE\colon \, \wt{F}\tensor C^\ell(R,\R)\to \wt{F}\tensor C^\ell(\R^d,\wt{F})
\]
mapping elementary tensors $v\otimes \gamma$ to $v\otimes \cE(\gamma)$
(see \cite[Definition~43.5]{Tre}).
Using Proposition~\ref{etensor} twice
(with domains $R$ and $\R^d$, respectively),
we now interpret $\id_{\wt{F}}\tensor \cE$
as a map
\[
\cE_{\wt{F}}\colon C^\ell(R,\wt{F})\to C^\ell(\R^d,\wt{F}).
\]
The restriction $f:=\cE_{\wt{F}}|_{C^\ell(R,F)}$
is a continuous linear map $C^\ell(R,F)\to C^\ell(\R^d,\wt{F})$
such that $f(F\otimes C^\ell(R,\R))\sub F\otimes  \cE(C^\ell(R,\R)) \sub C^\ell(\R^d,F)$.
Moreover, $C^\ell(\R^d,\wt{F})$ induces the compact-open $C^\ell$-topology
on $C^\ell(\R^d,F)$ (see Lemma~\ref{sammelsu}(b)).
Since $C^\ell(\R^d,F)$ is sequentially complete by Lemma~\ref{sammelsu}(d),
Lemma~\ref{use-seqco} shows that $f(C^\ell(R,F))\sub C^\ell(\R^d,F)$.
Thus $f$ co-restricts to a continuous linear map
\[
\cE_F\colon C^\ell(R,F)\to C^\ell(\R^d,F).
\]
Let $\rho\colon C^\ell(\R^d,F)\to C^\ell(R,F)$, $\gamma\mto \gamma|_R$
be the restriction map, which is continuous and linear. Since
\[
\rho(\cE_F(\gamma v))=\rho(\cE(\gamma)v)=\cE(\gamma)|_R v=\gamma v
\]
for all $v\in F$ and $\gamma\in C^\ell(R,\R)$, passing to the linear span
we deduce that $\rho(\cE_F(\gamma))=\gamma$ for all $\gamma\in F\otimes C^\ell(R,\R)$.
The continuous linear maps $\rho\circ\cE_F$ and $\id_{C^\ell(R,F)}$
therefore coincide, as they coincide on a dense vector subspace of their domain.
$\,\square$\\[2.3mm]
We close with a fact concerning convex sets mentioned in the introduction.
\begin{la}\label{convthencusp}
Let $d\in \N$ and $R\sub\R^d$ be a closed, convex subset
with non-empty interior. Then $R$ satisfies the cusp
condition.
\end{la}
\begin{proof}
To see that $R$ has no narrow fjords,
let $x\in R$. Pick $r>0$. Then
\[
K:=\wb{B}_r(x)\cap R
\]
is a compact neighbourhood of~$x$ in~$R$
(using the euclidean ball).
Let $w\in R^0$.
If $y\not=z$ are points in~$K$,
we can choose $t\in \,]0,1]$ so small that $t\|w-y\|_2,t\|w-z\|_2\leq \|y-z\|_2/2$.
Then $y+s(w-y)\in R^0$ and $z+s(w-z)\in R^0$ for all $s\in\,]0,t]$
and the line segment joining $y+t(w-y)$  and
$z+t(w-z)$ (which has length
\[
\|y+t(w-y)-z-t(w-z)\|_2=(1-t)\|y-z\|_2)
\]
is contained in~$R^0$.
The polygonal path $\gamma$ with edges
$y$, $y+t(w-y)$, $z+t(w-z)$, $z$
therefore has length
\[
t\|w-y\|_2+(1-t)\|y-z\|_2+t\|w-z\|_2\leq 2\,\|y-z\|_2,
\]
using $1-t\leq 1$ and the choice of~$t$.
We can therefore take $C:=2$ and $n=1$
in the definition of no narrow fjords.\\[2mm]
To see that $R$ has at worst polynomal outward cusps,
let $K\sub\R^d$ be a compact set. Pick $w\in R^0$;
there is $\rho>0$ such that $B_\rho(w)\sub R^0$.
For each $z\in K\cap\partial R$ and $\ve\in\,]0,1]$,
setting $x:=z+\ve(w-z)$ we have
\begin{eqnarray*}
B_{\rho\ve}(x) &=&
B_{\rho\ve}(z+\ve(w-z))=
z+\ve(w-z)+B_{\rho\ve}(0)\\
&=& z+\ve(w-z)+\ve B_\rho(0)
=z+\ve (B_\rho(w)-z)\sub R^0\sub R.
\end{eqnarray*}
Thus, the condition in Definition~\ref{cuspco}(b)
is satisfied with $\ve_0:=1$ and~$r:=1$.
\end{proof}
\section{Proof of Proposition~\ref{Hancts} and Corollary~\ref{extcorner}}
We now prove the two results related to Hanusch's work.\\[2.3mm]
{\bf Proof of Proposition~\ref{Hancts}.}
Let the constants $C_i\in [1,\infty[$ for $i\in\N$ with $i\leq \ell$ be
as in part~2) of \cite[Theorem~1]{Han}.
Let $j\in\N_0$ such that $j\leq \ell$, $K\sub \,]a,\infty[\,\times R$
and $L\sub \R\times E$ be compact subsets, and~$q$ be a continuous seminorm on~$F$.
After increasing $K$ and $L$, we may assume that $K=[\alpha,\beta ]\times A$
with $a<\alpha <\tau< b<\beta$ and a compact subset $A\sub R$,
and $L=[{-r},r]\times B$ for some $r\geq 10$ and a compact subset $B\sub E$
such that $0\in B$. If $\ell\geq 1$, we may assume that $j\geq 1$.
By part~2) of \cite[Theorem~1]{Han}, we have
\[
q(\cE_\tau(\gamma)(t,x))\;\leq\; \sup\{q(\gamma(s,x))\colon s\in [\tau,b]\}\;
\leq\, \|\gamma\|_{[\alpha,\beta]\times A,q}
\]
for all $t\in \,]b,\beta]$ and $x\in A$.
Thus
\[
q(\cE_\tau(\gamma)(t,x))\;\leq\; \|\gamma\|_{[\alpha,\beta]\times A,q}
\]
holds for all $(t,x)\in [\alpha,\beta]\times A$,
as $\cE_\tau(\gamma)(t,x)=\gamma(t,x)$ in the remaining case
where $(t,x)\in [\alpha,b]\times A$
(making the inequality a triviality).
We therefore have
\begin{equation}\label{p3han}
\|\cE_\tau(\gamma)\|_{K,q}\leq \|\gamma\|_{K,q}
\end{equation}
for all $\gamma\in C^\ell(]a,b]\times R,F)$;
if $j=0$ (and thus $\ell=0$), this establishes continuity of~$\cE_\tau$.
If $j\geq 1$,
consider $i\in\{1,\ldots,j\}$ now.
Abbreviate $D:=\{(1,0)\}\cup(\{0\}\times B)\sub [{-r},r]\times B\sub\R\times E$.
Then
\begin{eqnarray*}
q(d^{\,(i)}(\cE_\tau \gamma )((t,x),(t_1,y_1),\ldots,(t_i,y_j)))
& \leq &
C_i r^i\|\gamma\|_{C^i,([\tau,b]\times \{x\}),D,q}\\
& \leq & C_ir^i\|\gamma\|_{C^i,[\alpha,\beta]\times A,[{-r},r]\times B,q}
\end{eqnarray*}
for all $t\in \,]b,\beta]$, $x\in A$, $y_1,\ldots, y_i\in B$ and
$t_1,\ldots, t_i\in [{-r},r]$,
using that $\max\{1,|t_1|,\ldots,|t_i|\}^i\leq r^i$.
Thus
\begin{equation}\label{p2han}
q(d^{\,(i)}(\cE_\tau \gamma)((t,x),(t_1,y_1),\ldots,(t_i,y_j)))
\;\leq\; C_ir^i\|\gamma\|_{C^i,K,L,q}
\end{equation}
holds for all $(t,x)\in [\alpha,\beta]\times A$ and
$(t_1,y_1),\ldots, (t_i,y_i)\in [{-r},r]\times B$,
as
\[
d^{\,(i)}(\cE_\tau\gamma)((t,x),(t_1,y_1),\ldots,(t_i,y_i))
=d^{\,(i)}\gamma((t,x),(t_1,y_1),\ldots, (t_i,y_i))
\]
in the remaining case where $t\in [\alpha,b]$
(making the inequality a triviality).\\[2.3mm]
Now
\[
C\;:=\; \max\big\{C_ir^i\colon i\in\{1,\ldots, j\}\big\}\,\geq \, 1.
\]
For all $\gamma\in C^\ell(]a,b]\times R,F)$,
we have $\|\cE_\tau(\gamma)\|_{K,q}\leq \|\gamma\|_{K,q}\leq
\|\gamma\|_{C^j,K,L,q}\leq C\,\|\gamma\|_{C^j,K,L,q}$,
by (\ref{p3han}).
Moreover, for $i\in\{1,\ldots, j\}$, we have
\[
\|\cE_\tau(\gamma)\|_{i,K,L,q}\; \leq \; C_ir^i\|\gamma\|_{C^i,K,L,q}
\;\leq\; C\,\|\gamma\|_{C^j,K,L,q},
\]
by (\ref{p2han}).
Thus
\[
\|\cE_\tau(\gamma)\|_{C^j,K,L,q}\;\leq\; C\,\|\gamma\|_{C^j,K,L,q}
\]
for all $\gamma\in C^\ell(]a,b]\times R,F)$.
As a consequence, $\cE_\tau$ is continuous. $\,\square$.\\[2.3mm]
The following observations prepare the proof of
Corollary~\ref{extcorner}.
\begin{la}\label{towa1}
Let $E_1$ and $E_2$ be finite-dimensional vector spaces,
$E$ and $F$ be locally convex spaces,
$R\sub E$ be a regular subset,
$\ell\in\N_0\cup\{\infty\}$, $f\colon U_1\to U_2$ be a $C^\ell$-diffeomorphism
between open subsets $U_1\sub E_1$ and $U_2\sub E_2$,
and $S_1$ be a closed, regular subset of~$U_1$. Let $S_2:=f(S_1)$.
If the restriction map
\[
C^\ell(U_1\times R,F)\to C^\ell(S_1\times R,F)
\]
has a continuous linear right inverse $\cE\colon C^\ell(S_1\times R,F)
\to C^\ell(U_1\times R,F)$, then
\[
\cF:=C^\ell((f^{-1})\times \id_R,F)\circ\cE\circ C^\ell(f|_{S_1}^{S_2}\times\id_R,F)\colon
C^\ell(S_2\times R,F)\to C^\ell(U_2\times R,F)
\]
is a continuous linear map and a right inverse for the restriction mapping
$C^\ell(U_2\times R,F)\to C^\ell(S_2\times R,F)$.
\end{la}
\begin{proof}
The composition~$\cF$ is continuous and linear as a composition of continuous
linear maps (using Lemma~\ref{sammelsu}(e)).
For $\gamma\in C^\ell(S_2\times R,F)$
we have
\[
\cF(\gamma)(x,y)=\cE(\gamma\circ (f|_{S_1}\times \id_R))(f^{-1}(x),y)
=(\gamma\circ (f|_{S_1}\times\id_R))(f^{-1}(x),y)=\gamma(x,y)
\]
for all $x\in S_2$ and $y\in R$ (using that $f^{-1}(x)\in S_1$).
Thus $\cF(\gamma)|_{S_2\times R}=\gamma$.
\end{proof}
\begin{la}\label{towa2}
Let $E_1$, $E_2$, and $F$ be locally convex spaces,
$d\in\N$, $\ell\in\N_0\cup\{\infty\}$,
$U\sub \R^d$ be an open subset
and $R_1\sub E_1$, $R_2\sub E_2$, and $S\sub U$ be
regular subsets such that $S$ is closed in~$U$.
If the restriction map
\[
C^\ell(U\times R_1\times R_2,F)\to C^\ell(S\times R_1\times R_2,F)
\]
has a continuous linear right inverse
\[
\cE\colon C^\ell(S\times R_1\times R_2,F)
\to C^\ell(U\times R_1\times R_2,F),
\]
then
also the restriction map
\[
C^\ell(R_1\times U \times R_2,F)\to C^\ell(R_1\times S\times R_2,F)
\]
has a continuous linear right inverse $C^\ell(R_1\times S\times R_2,F)
\to C^\ell(R_1\times U\times R_2,F)$.
\end{la}
\begin{proof}
The map $f\colon U\times R_1\times R_2\to R_1\times U\times R_2$,
$(x,y,z)\mto (y,x,z)$ is a $C^\ell$-diffeomorphism between
regular subsets of $\R^d\times E_1\times E_2$ and
$E_1\times \R^d\times E_2$,
respectively. Then
\[
C^\ell(f^{-1},F)\circ \cE\circ C^\ell(f|_{S\times R_1\times R_2},F)\colon
C^\ell(R_1\times S\times R_2,F)\to C^\ell(R_1\times U\times R_2,F)
\]
is the desired right inverse.
\end{proof}
\begin{rem}\label{hanumultrem}
It is now clear that Hanusch's extension operators
as in (\ref{hanumult}) are continuous.
If fact, given $\tau_j\in ]a_j,b_j[$
we have Hanusch's extension operator
\[
C^\ell(]a_j,b_j]\times R_j,F)\to C^\ell(]a,j,\infty[\,\times R_j,F)
\]
for the regular subset $R_j:=\,]a_1,\infty[\,\times\!\cdots\! \times\,]a_{j-1},\infty[\,\times
]a_{j+1},b_j]\times\!\cdots\! \times \,]a_n,b_n]\times R$\linebreak
$\sub
\R\times (\R^{n-1}\times E)$. Using Lemma~\ref{towa2},
we have a corresponding extension operator
\begin{eqnarray*}
\lefteqn{\cE_j\colon
C^\ell(]a_1,\infty[\,\times\cdots\times\,]a_{j-1},\infty[\,
\times \,]a_j,b_j]\times \cdots\times \,]a_n,b_n]\times R,F)}\qquad\\
& & \to
C^\ell(]a_1,\infty[\,\times\cdots\times\,]a_j,\infty[\,
\times \,]a_{j+1},b_{j+1}]\times \cdots\times \,]a_n,b_n]\times R,F).
\end{eqnarray*}
Then $\cE_n\circ \cdots\circ \cE_1$ is continuous linear and is the extension operator
in~(\ref{hanumult}).
\end{rem}
{\bf Proof of Corollary~\ref{extcorner}.}
Consider an extension operator
\[
C^\ell(]{-\infty},0]^m\times\R^{d-m},F)\to C^\ell(\R^d,F)
\]
by Hanusch (analogous to (\ref{hanumult})),
which is continuous linear by Remark~\ref{hanumultrem}.\linebreak
Using Lemma~\ref{towa1}, we have a corresponding continuous linear extension\linebreak
operator
$C^\ell([0,\infty[^m\times\R^{d-m},F)\to C^\ell(\R^d,F)$, as asserted.\\[2.3mm]
To prove the final assertion,
we claim that there exist continuous linear extension operators
\begin{equation}\label{intermedia}
\cE_j\colon C^\ell(\R^{j-1}\times [0,1]^{1+d-j},F)\to C^\ell(R^j\times [0,1]^{d-j},F)
\end{equation}
for all $j\in \{1,\ldots,d\}$. If this is true, then
\[
\cE_d\circ\cdots\circ\cE_1\colon C^\ell([0,1]^d,F)\to C^\ell(\R^d,F)
\]
is a continuous linear extension operator, as desired.
In view of Lemma~\ref{towa2}, operators as in (\ref{intermedia})
will exist if we can show that the restriction map
\[
C^\ell(\R\times R,F)\to C^\ell([0,1]\times R,F)
\]
has a continuous linear right inverse for each regular subset~$R$
of a locally convex space~$E$.
The restriction map
\[
\rho_1\colon C^\ell(\R\times R,F)\to C^\ell(]0,\infty[\,\times R,F)
\]
is continuous and linear, and so are the restriction maps
$\rho_2\colon C^\ell(\R\times R,F)\to C^\ell(]-\infty,1[\,\times R,F)$,
$r_1\colon C^\ell([0,1]\times R,F)\to C^\ell(]0,1]\times R,F)$
and $r_2\colon C^\ell([0,1]\times R,F)\to C^\ell([0,1[\,\times R,F)$.
Lemma~\ref{sammelsu}(f) entails that the linear map
\[
(\rho_1,\rho_2)\colon C^\ell(\R\times R,F)\to C^\ell(]0,\infty[\,\times R,F)
\times C^\ell(]{-\infty},1[\,\times R,F)
\]
is a topological embedding. Now Hanusch's work provides a continuous
linear extension operator
\[
\cF_1\colon C^\ell(]0,1]\times R,F)\to C^\ell(]0,\infty[\,\times R,F).
\]
Using Lemma~\ref{towa1}, we obtain a corresponding continuous linear
extension operator
\[
\cF_2\colon C^\ell([0,1[\,\times R,F)\to C^\ell(]{-\infty},1[\,\times R,F).
\]
For all $\gamma\in C^\ell([0,1]\times R,F)$,
$x\in \,]0,1[$ and $y\in R$, we have
\[
\cF_1(\gamma|_{]0,1]\times R})(x,y)=\gamma(x,y)=\cF_2(\gamma|_{[0,1[\,\times R})(x,y)
\]
and thus
\[
(\cF_1\circ r_1)(\gamma)|_{]0,1[\,\times R}=(\cF_2\circ r_2)|_{]0,1[\,\times R}.
\]
Hence
\[
\cF(\gamma)(x,y)\;:=\;
\left\{
\begin{array}{cl}
(\cF_1\circ r_1)(\gamma) &\mbox{if $\,(x,y)\in \,]0,\infty[\,\times R$;}\\
(\cF_2\circ r_2)(\gamma) &\mbox{if $\,(x,y)\in \,]{-\infty}, 1[\,\times R$}
\end{array}
\right.
\]
yields a well-defined $C^\ell$-map $\cF(\gamma)\in C^\ell(\R\times R,F)$
and clearly $\cF(\gamma)$ is linear in~$\gamma$.
As
\[
\rho_j\circ \cF=\cF_j\circ r_j
\]
is continuous for all $j\in\{1,2\}$, we deduce that $\cF$ is continuous.
By contruction, $\cF(\gamma)(x,y)=\gamma(x,y)$ for all $(x,y)\in [0,1]\times R$.
In fact, if $x\in \,]0,1]$, then $\cF(\gamma)(x,y)=\cF_1(r_1(\gamma))(x,y)=r_1(\gamma)(x,y)=\gamma(x,y)$.
The case $x\in [0,1[$ is analogous. $\,\square$
\section{Proof of Proposition~\ref{autopara} and a {\boldmath$C^\ell$}-analogue}
In this section, we prove Proposition~\ref{autopara}
and an analogue for $C^\ell$-maps.\\[2.3mm]
{\bf Proof of Proposition~\ref{autopara}.}
Let $\cE\colon C^\infty(S,F)\to C^\infty(\R^d,F)$ be the continuous
linear right inverse to the restriction map.
Since $S$ is locally compact, the map
\[
\Phi\colon C^\infty(R\times S,F)\to C^\infty(R,C^\infty(S,F)),\quad
\gamma\mto\gamma^\vee
\]
is an isomorphism of topological vector spaces, where $\gamma^\vee(x)(y):=\gamma(x)(y)$
(see \cite{Rou};
if $R$ and $S$ are locally convex sets,
see already~\cite{AaS}).
Likewise, the map
\[
\Psi\colon C^\infty(R\times\R^d,F)\to C^\infty(R,C^\infty(\R^d,F)),\quad
\gamma\mto\gamma^\vee
\]
is an isomorphism of topological vector spaces
and $\Psi^{-1}(\eta)=:\eta^\wedge$ is given by $\eta^\wedge(x,y):=\eta(x)(y)$
for $\eta\in C^\infty(R\times \R^d,F)$.
Now
\[
C^\infty(R,\cE)\colon C^\infty(R,C^\infty(S,F))\to C^\infty(R,C^\infty(\R^d,F)),\quad
\gamma\mto \cE\circ\gamma
\]
is a smooth (and hence continuous) linear map (cf.\ Lemma~\ref{sammelsu}(g)).
As a consequence, also
\[
\cF:=\Psi^{-1}\circ C^\infty(R,\cE)\circ\Phi\colon C^\infty(R\times S,F)\to
C^\infty(R\times\R^d,F)
\]
is continuous linear. It remains to observe that
\[
\cF(\gamma)(x,y)=(\cE\circ \gamma^\vee)(x)(y)=\cE(\gamma^\vee(x))(y)=\cE(\gamma(x,\cdot))(y)
=\gamma(x,y)
\]
for all $x\in R$ and $y\in S$, whence $\cF$ is a right inverse for the restriction map
$C^\infty(R\times \R^d,F)\to C^\infty(R\times S,F)$.$\,\square$
\begin{prop}\label{Cellcompatible}
Let $d\in \N$, $S\sub \R^d$ be a closed, regular subset,
$F$ be a locally convex space, and $\ell\in\N_0$.
Assume that there exists
an $(i+1)$-tupel $(\cE_0,\ldots,\cE_\ell)$ of right inverses
\[
\cE_i\colon C^i(S,F)\to C^i(\R^d,F)
\]
for the restriction map $C^i(\R^d,F)\to C^i(S,F)$
for $i\in\{0,\ldots,\ell\}$ which are compatible in the sense
that
\[
\cE_i(\gamma)=\cE_j(\gamma)\quad\mbox{for all $\, 0\leq i\leq j\leq \ell\,$
and $\,\gamma\in C^j(S,F)$.}
\]
Then also the restriction map
\[
C^\ell(R\times \R^d,F)\to C^\ell(R\times S,F)
\]
admits a continuous linear right inverse,
for each locally convex space~$E$ and regular subset $R\sub E$.
\end{prop}
\begin{proof}
For all $i,j\in \N_0$ such that $i+j\leq \ell$,
the following maps are isomorphisms of topological vector spaces:
\[
\Phi_{ij}\colon C^{i,j}(R\times S,F)\to C^i(R,C^j(S,F)),\quad \gamma\mto\gamma^\vee;
\]
\[
\Psi_{ij}\colon C^{i,j}(R\times \R^d,F)\to C^i(R,C^j(\R^d,F)),\quad \gamma\mto\gamma^\vee
\]
(see \cite{Rou};
if $R$ and $S$ are locally convex sets,
see already~\cite{AaS}).
Now
\[
C^i(R,\cE_j)\colon C^i(R,C^j(S,F))\to C^i(R,C^j(\R^d,F)),\quad
\gamma\mto \cE_j\circ\gamma
\]
is a smooth (and hence continuous) linear map (cf.\ Lemma~\ref{sammelsu}(g)).
As a consequence, also
\[
\cE_{ij}:=\Psi^{-1}_{ij}\circ C^i(R,\cE_j)\circ\Phi_{ij}\colon C^{i,j}(R\times S,F)\to
C^{i,j}(R\times\R^d,F)
\]
is continuous linear. If $\gamma\in C^\ell(R\times S,F)$, then
\[
\cE_{ij}(\gamma)(x,y)=\cE_j(\gamma(x,\cdot))(y)=\cE_0(\gamma(x,\cdot))(y)
=\cE_{00}(\gamma)(x)(y)
\]
is independent of $i$ and $j$ for $(x,y)\in R\times\R^d$.
As the inclusion map $C^\ell(R\times S,F)\to C^{i,j}(R\times S,F)$
is continuous linear for all $i,j$ as before and the map
\[
C^\ell(R\times\R^d,F)\to\prod_{i+j\leq\ell}C^{i,j}(R\times\R^d,F),\quad \gamma\mto(\gamma)_{i+j\leq \ell}
\]
is linear and a topological embedding (see \cite{Rou}),
we deduce that the map
\[
\cF\colon C^\ell(R\times S,F)\to C^\ell(R\times\R^d,F),\quad \gamma\mto \cE_{00}(\gamma)
\]
makes sense, is linear, and continuous. It remains to observe that
\[
\cF(\gamma)(x,y)=\cE_{00}(\gamma)(x,y)=
\cE_0(\gamma(x,\cdot))(y)
=\gamma(x,y)
\]
for all $x\in R$ and $y\in S$, whence $\cF$ is a right inverse for the restriction map
$C^\ell(R\times \R^d,F)\to C^\ell(R\times S,F)$.
\end{proof}
\section{Applications to spaces of sections and\\
manifolds of mappings}\label{sec-outlook}
We now describe some applications of our results,
combined with techniques and concepts from \cite{RaS}
and \cite{Rou}. The applications deal with continuous
linear extension operators between spaces of sections in vector bundles,
and locally defined smooth extension operators between
manifolds of mappings of the form $C^\ell(M,N)$.
Such manifolds of mappings are well known
in various special cases (as recalled below);
the construction of~\cite{Rou} in the context of \emph{rough manifolds}
is particularly well adapted to the extension questions.
\begin{numba}\label{defn-rough}
Given a set $\cE$ of locally convex spaces and $\ell\in \N_0\cup\{\infty\}$,
a \emph{rough $C^\ell$-manifold} modelled on~$\cE$
is a Hausdorff topological space~$M$, together with a maximal
set $\cA$ of homeomorphisms $\phi \colon U_\phi\to V_\phi$
(``charts'')
from open subsets $U_\phi\sub M$ onto regular subsets $V_\phi\sub E_\phi$
for some $E_\phi\in\cE$,
such that the domains $U_\phi$ cover~$M$
and the transition maps $\phi\circ\psi^{-1}$ take the
interior $\psi(U_\phi\cap U_\psi)^0$ relative~$E_\psi$
inside the interior $\phi(U_\phi\cap U_\psi)^0$ relative $E_\phi$
and are~$C^\ell$ (see \cite{Rou} for details).
If $\cE=\{E\}$ is a singleton,
then~$M$ is called a rough $C^\ell$-manifold modelled on~$E$;
such manifolds are also called \emph{pure}
manifolds. Note that $E_\phi\cong E_\psi$
if $U_\phi\cap U_\psi\not=\emptyset$;
therefore every rough $C^\ell$-manifold
admits a partition into open submanifolds
which can be considered as pure $C^\ell$-manifolds.
The rough $C^\ell$-manifolds we consider
are a slight generalization of
the \emph{$C^\ell$-manifolds with rough boundary} discussed
in~\cite{GaN},
where $V_\phi$ is assumed regular \emph{and locally convex.}
Every rough $C^\ell$-manfold~$M$ has a formal interior $M^\circ$
(the set of $x\in M$ such that $\phi(x)\in V_\phi^0$ for each chart
$\phi\colon U_\phi\to V_\phi$ with $x\in U_\phi$)
and a formal boundary $\partial^\circ M:=M\setminus M^\circ$.
\end{numba}
If a rough $C^\ell$-manifold~$M$
is locally compact, then the locally convex space~$E_\phi$
is locally compact and thus finite-dimensional
for each chart $\phi\colon U_\phi\to V_\phi\sub E_\phi$
of~$M$ such that $U_\phi\not=\emptyset$.
We can therefore model~$M$ on a set of finite-dimensional
vector spaces.
\begin{numba}
Let $M$ be a rough $C^\ell$-manifold.
For each $C^\ell$-manifold~$L$ (without boundary)
modelled
on locally convex spaces, or each $C^\ell$-manifold~$L$ with
rough boundary in the sense of~\cite{GaN} (e.g.,
a $C^\ell$-manifold with corners), one can define $C^\ell$-maps
$M\to L$ as continuous maps which are~$C^\ell$ in local charts
(see \cite{Rou}). If, more generally, $L$ is any rough $C^\ell$-manifold,
we define $RC^\ell$-mappings
$f\colon M\to N$ (``restricted $C^\ell$-maps'')
as continuous mappings from $M$ to~$N$
such that $f(M^\circ)\sub L^\circ$ and~$f$ is $C^\ell$ in local
charts. If $L$ is a $C^\ell$-manifold with rough boundary,
then $RC^\ell(M,L)\sub C^\ell(M,L)$
and (for $M\not=\emptyset$) equality holds if and only if~$\partial^\circ L$ is empty.\footnote{Pick
$y\in\partial^\circ L$
and let $f\colon M\to L$ be the constant map $x\mto y$.
Then $f$ is~$C^\ell$ but not $RC^\ell$.}
Thus $RC^\ell(M,L)=C^\ell(M,L)$ if $L$ is an ordinary $C^\ell$-manifold
(without boundary) modelled on locally convex spaces.
\end{numba}
We mention another (elementary)
source of extension operators.
\begin{defn}\label{defn-split-sub}
For $\ell\in \N_0\cup\{\infty\}$,
we call a subset $L$ of a rough $C^\ell$-manifold~$M$
a \emph{split $RC^\ell$-submanifold} (or also:
a \emph{split rough submanifold}) if, for each $x\in L$,
there exists a chart $\phi\colon U_\phi\to V_\phi\sub E_\phi$
of $M$ around~$x$ such that $\phi(x)=0$
and which is \emph{$RC^\ell$-adapted to~$L$}
in the sense that
\begin{itemize}
\item[(a)]
$V_\phi=W_1+ W_2$ with $W_1:=V_\phi\cap E_1$
and $W_2:=V_\phi\cap E_2$ for
some vector subspaces $E_1$ and $E_2$ of the modelling
space~$E_\phi$ such that $E_\phi=E_1\oplus E_2$ (internally)
as a topological vector space;
\item[(b)]
$W_1$ and $W_2$ are regular subsets of~$E_1$ and $E_2$, respectively;
\item[(c)]
$\phi(U_\phi\cap L)=W_1$; and
\item[(d)]
$0\in W_2^ 0$ relative $E_2$,
\end{itemize}
whence $U_\phi\cap L$ is open in~$L$ in the induced topology,
and the map
\[
\phi_L\colon U_\phi\cap L\to W_1,\;\, y\mto\phi(y)
\]
is a homeomorphism.
Moreover, we require
\begin{itemize}
\item[(e)]
For all charts $\phi\colon U_\phi\to V_\phi\sub E_\phi$
and $\psi\colon U_\psi\to V_\psi\sub E_\psi$
which are $RC^\ell$-adapted to~$L$,
we have for all $y\in L\cap U_\phi\cap U_\psi$
\[
\phi_L(y)\in\phi_L(L\cap U_\phi)^0\;\,\Leftrightarrow
\;\, \psi_L(L\cap U_\psi)^0
\]
for the interiors relative $\Spann \phi_L(L\cap U_\phi)$
($=E_1$ in~(a)) and $\Spann\psi_L(L\cap U_\psi)$,
respectively. 
\end{itemize}
If $M$ is a $C^\ell$-manifold
with rough boundary,
we say that a chart $\phi$
is $C^\ell$-adapted to~$L$
if (a), (b), and (c) hold
and~$V_\phi$ is convex.
If, moreover, the following condition (e)$'$ holds,
we call~$L$ a \emph{split submanifold} of~$M$.
\begin{itemize}
\item[(e)$'$]
For all charts $\phi\colon U_\phi\to V_\phi\sub E_\phi$
and $\psi\colon U_\psi\to V_\psi\sub E_\psi$
which are $C^\ell$-adapted to~$L$,
we have for all $y\in L\cap U_\phi\cap U_\psi$
\[
\phi_L(y)\in\phi_L(L\cap U_\phi)^0\;\,\Leftrightarrow
\;\, \psi_L(L\cap U_\psi)^0
\]
for the interiors relative $\Spann \phi_L(L\cap U_\phi)$
and $\Spann\psi_L(L\cap U_\psi)$,
respectively. 
\end{itemize}
\end{defn}
\begin{rem}
If $M$ is a rough $C^\ell$-manifold
modelled on locally convex spaces for
some $\ell\in\N_0\cup\{\infty\}$
and $L\sub M$ a rough $C^\ell$-submanifold,
then the maximal $C^\ell$-atlas
containing the charts $\phi_L$
for all charts $\phi$ of~$M$ which are
$RC^\ell$-adapted to~$L$
turns $L$ into a rough $C^\ell$-manifold
such that $L$ carries the topology induced by~$M$
and the inclusion map $L\to M$ is an
$RC^\ell$-map.\footnote{Condition~(d) ensures that
the inclusion map $L\to M$ takes $L^\circ$ into $M^\circ$.}\\[2.3mm]
Likewise,
if $M$
is a $C^\ell$-manifold with rough boundary modelled
on locally convex spaces
and $L\sub M$
a $C^\ell$-submanifold,
then the maximal $C^\ell$-atlas
of charts with
locally convex ranges containing the charts $\phi_L$
for all charts $\phi$ of~$M$ which are
$C^\ell$-adapted to~$L$
turns $L$ into a
$C^\ell$-manifold with rough boundary
such that $L$ carries the topology induced by~$M$
and the inclusion map $L\to M$ is
a $C^\ell$-map.
\end{rem}
\begin{rem}
If $E_1$ in condition~(a) of Definition~\ref{defn-split-sub}
is always finite-dimensional
(e.g., if each modelling space of~$M$
has finite dimension),
then condition~(e) (resp., (e)$'$)
is automatically satisfied, by Invariance of Domain.\\[2.3mm]
Moreover, condition~(e)$'$ is automatic if $\ell>0$,
by \cite[Lemma~3.5.6]{GaN}.
\end{rem}
\begin{numba}\label{projextensions}
In the situation of Definition~\ref{defn-split-sub},
the projection $\pr_1\colon V_\phi\to W_1$, $w_1+w_2\mto w_1$
is~$RC^\ell$.
Hence
\[
r\colon U_\phi\to U_\phi\cap L,\;\; y\mto \phi^{-1}(\pr_1(\phi(y)))
\]
is an $RC^\ell$-map such that $r(y)=y$ for each $y\in U_\phi\cap L$
(an $RC^\ell$-retraction from $U_\phi$ onto $U_\phi\cap L$).
As a consequence, the map
\[
C^\ell(r,F)\colon C^\ell(U_\phi\cap L,F)\to C^\ell(U_\phi, F),\;\;
\gamma\mto \gamma\circ r
\]
is a continuous linear right inverse for the restriction map
\[
C^\ell(U_\phi,F)\to C^\ell(U_\phi\cap L, F), \quad
\gamma\mto\gamma|_{U_\phi\cap L},
\]
for each locally convex space~$F$.
\end{numba}
Further concepts of submanifolds will be useful,
which are analogous to Definitions~3.1 and 6.1 in \cite{RaS}.
\begin{defn}\label{full-submfd}
Let $\ell\in \N_0\cup\{\infty\}$,
$M$ be
a rough $C^\ell$-manifold
modelled on locally convex spaces
and $L\sub M$ be a subset.
\begin{itemize}
\item[\rm(a)]
We say that $L$ is a \emph{full-dimensional rough submanifold} of~$M$
if, for each $x\in L$,
there exists a chart $\phi\colon U_\phi\to V_\phi\sub E_\phi$
of~$M$ around~$x$ such that $\phi(U_\phi\cap L)$
is a regular subset of~$V_\phi$.
\item[\rm(b)]
We say that $L$ is a \emph{full-dimensional
submanifold with rough boundary} of~$M$ if,
for each $x\in L$,
there exists a chart $\phi\colon U_\phi\to V_\phi\sub E_\phi$
of~$M$ around~$x$ such that $\phi(U_\phi\cap L)$
is a regular subset of~$V_\phi$ and locally convex.
\item[\rm(c)]
If $M$ has finite-dimensional modelling spaces,
we say that
$L\sub M$ is a \emph{full-dimensional submanifold with corners}
if, for each $x\in L$,
there exists a chart $\phi\colon U_\phi\to V_\phi\sub \R^d$
of~$M$ around~$x$ for some $d\in\N_0$
such that $\phi(x)=0$
and $\phi(U_\phi\cap L)$
is a relatively open subset of $[0,\infty[^m\times \R^{d-m}$
for some $m\in \{0,\ldots, d\}$.
\end{itemize}
\end{defn}
Note that every full-dimensional submanifold with corners
is a full-dimensional submanifold with rough boundary;
every full-dimensional submanifold with rough boundary
is a full-dimensional rough submanifold.
\begin{rem}\label{char-full}
It is easy to verify that the
following conditions are equivalent
for a subset $L$ of a rough $C^\ell$-manifold~$M$
modelled on locally convex spaces:
\begin{itemize}
\item[(a)]
$L$ is a full-dimensional rough submanifold of~$M$.
\item[(b)]
$L$ is a regular subset of~$M$.
\item[(c)]
For every chart
$\phi\colon U_\phi\to V_\phi\sub E_\phi$
of~$M$, the image $\phi(U_\phi\cap L)$
is a regular subset of~$V_\phi$.
\end{itemize}
In this case, the functions $\phi|_{U_\phi\cap L}\colon U_\phi\cap L\to
\phi(U_\phi\cap L)$ form an atlas of charts for~$L$
which (after passing to a maximal $C^\ell$-atlas containing it)
turns $L$ into a rough $C^\ell$-manifold with formal boundary
\begin{equation}\label{submfd-bdy}
\partial^\circ L\, =\,  \partial L\cup (L\cap \partial^\circ M),
\end{equation}
where $\partial L$ is the boundary of~$L$ as a subset of the topological
space~$M$.
The inclusion map $L\to M$ is an $RC^\ell$-map.
Likewise, every full-dimensional
submanifold with rough boundary can be turned into
a $C^\ell$-manifold
with rough boundary,
and every full-dimensional submanifold with corners
can be turned into a $C^\ell$-manifold with corners.
\end{rem}
\begin{defn}\label{locml}
Let $\ell\in \N_0\cup\{\infty\}$
and $M$ be
a rough $C^\ell$-manifold.
Let $L\sub M$ be a closed subset, endowed with a rough $C^\ell$-manifold structure
which is compatible with the topology induced by~$M$ on~$L$
and makes the inclusion map $j\colon L \to M$ an $RC^\ell$-map\footnote{If $M$ is
a $C^\ell$-manifold with rough boundary, it suffices to assume
that $j$ is a $C^\ell$-map.}
(for example, $L$ can be any closed, regular subset of~$M$).
Given a locally convex space~$F$,
we say that $M$ \emph{locally admits extension
operators for $F$-valued $C^\ell$-functions around~$L$}
if $M$ has the following property:
For each $x\in L$,
each $x$-neighbourhood in~$L$
contains an open $x$-neighbourhood
$P_x\sub L$ for which there exists a continuous linear
operator
\begin{equation}\label{localop}
\cE_x\colon C^\ell(P_x,F)\to C^\ell(Q_x,F)
\end{equation}
for some open $x$-neighbourhood
$Q_x\sub M$
such that $Q_x\cap L\subseteq P_x$
and
\[
\cE(\gamma)|_{Q_x\cap L}=\gamma|_{Q_x\cap L}
\]
for all $\gamma\in C^\ell(P_x,F)$.
We shall refer to $\cE_x$ as \emph{local extension operator}
around~$x$.
\end{defn}
\begin{rem}\label{locallocal}
If $\cE_x$ is a local extension operator around $x\in L$
as in~(\ref{localop}) and
$Q\sub M$ is an open $x$-neighbourhood,
then we can replace $Q_x$ with $Q_x\cap Q$
as the restriction map $C^\ell(Q_x,F)\to C^\ell(Q_x\cap Q,F)$
is continuous linear (see Lemma~\ref{sammelsu}(c)).
We can therefore choose $\cE_x$ such that $Q_x\sub Q$
for a given $x$-neighbourhood $Q\sub M$.
\end{rem}
\begin{defn}
Let $\ell\in\N_0\cup\{\infty\}$
and~$M$ be a rough $C^\ell$-manifold modelled on locally
convex spaces.
A partition of unity $(h_j)_{j\in J}$
on the topological space~$M$
is called a \emph{$C^\ell$-partition of unity}
on~$M$ if
each $h_j$ is a $C^\ell$-function.
The rough $C^\ell$-manifold~$M$
is said to be \emph{$C^\ell$-paracompact}
if for each open cover $(U_j)_{j\in J}$ of~$M$,
there exists a $C^\ell$-partition of unity
$(h_j)_{j\in J}$
on~$M$ which is subordinate to
$(U_j)_{j\in J}$ in the sense
that $h_j^{-1}(]0,1])\sub U_j$ for each $j\in J$.
Then $(h_j)_{j\in J}$ can be chosen such that
$\Supp(h_j)\sub U_j$
for each $j\in J$,
by standard arguments.\footnote{Since~$M$ is paracompact, $M$ is normal
and hence a regular topological space,
whence each $x\in M$ has an open neighbourhood $W_x$ whose
closure $\wb{W_x}$ in~$M$ is contained in $U_{j(x)}$
for some $j(x)\in J$. Let $(V_i)_{i\in I}$
be a locally finite open cover of~$M$ such that,
for each $i\in I$, we have $V_i\sub W_{x(i)}$
for some $x(i)\in M$ and thus $\wb{V_i}\sub U_{j(x(i))}$.
Then also $(\wb{V})_{i\in I}$
is locally finite. If we set $I(j):=\{i\in I\colon j(x(i))=j\}$,
then the sets $Q_j:=\bigcup_{i\in I(j)}V_i$
form an open cover of~$M$ and $\wb{Q_j}=\bigcup_{i\in I(j)}\wb{V_i}
\sub U_j$, using that $(\wb{V_i})_{i\in I(j)}$ is locally finite.
If we choose a $C^\ell$-partition
of unity $(h_j)_{j\in J}$ subordinate to $(Q_j)_{j\in J}$,
then $\Supp(h_j)\sub \wb{Q_j}\sub U_j$.}
\end{defn}
We shall see that every paracompact, locally compact
rough $C^\ell$-manifold is $C^\ell$-paracompact (Lemma~\ref{partu-ex}).\\[2.3mm]
For terminology concerning vector bundles,
we refer to Section~\ref{proofs-outlook}.
\begin{prop}\label{aprop1}
Let $\cF$ be a set of locally convex spaces, $\ell\in \N_0\cup\{\infty\}$
and $M$ be a
rough $C^\ell$-manifold.
Let $L\sub M$ be a closed subset,
endowed with a rough $C^\ell$-manifold structure
making the inclusion map $L\to M$ an $RC^\ell$-map.\footnote{If $M$ is a $C^\ell$-manifold with rough boundary, it suffices to assume
that the inclusion map $L\to M$ is a $C^\ell$-map.}
If $M$ locally admits extension
operators for $F$-valued $C^\ell$-functions around~$L$,
for each $F\in\cF$,
then the following holds
for each $C^\ell$-vector bundle
$E\to M$ all of whose fibres are isomorphic
to vector spaces $F\in\cF$:
\begin{itemize}
\item[\rm(a)]
If $M$ is $C^\ell$-paracompact,
then the restriction map
\[
\Gamma_{C^\ell}(E)\to \Gamma_{C^\ell}(E|_L)
\]
between spaces of $C^\ell$-sections
admits a continuous linear right inverse;
\item[\rm(b)]
If $M$ is locally compact
and paracompact, then the restriction map
\[
\Gamma_{C^\ell_c}(E)\to \Gamma_{C^\ell_c}(E|_L)
\]
admits a continuous linear right inverse
$($for the spaces of compactly supported $C^\ell$-sections in the vector bundles$)$.
\end{itemize}
\end{prop}
\begin{thm}\label{thm-locops}
Let $\ell\in \N_0\cup\{\infty\}$,
$F$ be a locally convex space and $M$ be a rough $C^\ell$-manifold.
Let $L\sub M$ be a closed subset.
Then $M$
locally admits extension
operators for $F$-valued $C^\ell$-functions around~$L$
in each of the following cases:
\begin{itemize}
\item[\rm(a)]
$\ell=\infty$ holds, $M$ is a $\sigma$-compact, finite-dimensional
Riemannian manifold without boundary, $L\sub M$ a regular subset
satisfying the cusp condition with respect to the
metric arising from Riemannian metric
$($as in {\rm\cite[Definition~3.1]{RaS})}
and $F$ has finite dimension $($as shown in {\rm\cite[\S4]{RaS});}
\item[\rm(b)]
$\ell=\infty$ holds, $M$ is locally compact, 
$L\sub M$ a full-dimensional submanifold with rough boundary
and $F$ sequentially complete;
\item[\rm(c)]
$M$ is locally compact and $L\sub M$ a full-dimensional submanifold with corners;
or
\item[\rm(d)]
$L\sub M$ is a split $RC^\ell$-submanifold.
\end{itemize}
\end{thm}
\begin{numba}\label{lonenum}
Let $\ell\in\N_0\cup\{\infty\}$.
If $\ell\geq 1$, let
$M$ be a paracompact, locally compact, rough $C^\ell$-manifold;
if $\ell=0$, let~$M$ be
a paracompact, locally compact topological space.
Let $N$ a smooth manifold (without boundary) modelled on a set $\cF$
of locally convex spaces
such that $N$ admits a local addition (see, e.g., \cite{AGS}
for this concept).
Then $C^\ell(M,N)$ admits a smooth manifold structure
independent of the local addition (\cite{Rou};
special cases and references to the literature
were recalled in the Introduction).
\end{numba}
Let $M$ be as in~\ref{lonenum}
and $L\sub M$ be a closed subset;
if $\ell\geq 1$,
assume that $L$ is endowed with a rough $C^\ell$-manifold structure
making the inclusion map $L\to M$ an $RC^\ell$-map.
Then the restriction map
\[
\rho\colon C^\ell(M,N)\to C^\ell(L,N),\quad \gamma\mto \gamma|_L
\]
is smooth, as it locally looks like the restriction map
\[
\rho_\gamma\colon \Gamma_{C^\ell_c}(\gamma^*(TN))\to\Gamma_{C^\ell_c}((\gamma|_L)^*(TN))
\]
between pullback bundles
in standard charts around~$\gamma\in C^\ell(M,N)$ and $\gamma|_L$
and this restriction map is continuous and linear (see Section~\ref{proofs-outlook}
for details).
Hence $\rho$ will be a smooth submersion (as in Definition~\ref{def-subm})
when each $\rho_\gamma$ admits a continuous linear right
inverse. Criteria for this were compiled in the preceding conditions (a)--(d),
and we shall deduce the following result:
\begin{thm}\label{thm-mapmfd}
Let $\ell\in \N_0\cup\{\infty\}$,
$M$ be a paracompact, locally compact, rough $C^\ell$-manifold
and $L\sub M$ be a closed subset. Let $N$ be a smooth manifold modelled
on a set $\cF$ of locally convex spaces. If $N$ admits a smooth local addition,
then the restriction map
\[
C^\ell(M,N)\to C^\ell(L,N),\quad \gamma\mto\gamma|_L
\]
is a smooth submersion in each of the following cases:
\begin{itemize}
\item[\rm(a)]
$\ell=\infty$, $M$ is a Riemannian manifold without boundary,
$L\sub M$ is a regular subset satisfying the cusp condition,
and $\dim(F)<\infty$ for each $F\in\cF$ $($cf.\ {\rm\cite{RaS}}$)$;
\item[\rm(b)]
$\ell=\infty$,
$L$ is a full-dimensional submanifold with rough boundary of~$M$
and each $F\in\cF$ is sequentially complete;
\item[\rm(c)]
$L$ is a full-dimensional submanifold with corners of~$M$; or
\item[\rm(d)]
$L\sub M$ is a split $RC^\ell$-submanifold. 
\end{itemize}
\end{thm}
Note that $\ell$ and~$\cF$ are arbitrary in~(c);
in~(d), all of~$M$, $\ell$, and the $F\in\cF$ are arbitrary.
A special case of~(d) (when $M$ and $L$ are compact
manifolds without boundary, $\ell<\infty$ and~$N$ is a Banach manifold)
was already considered in~\cite{Die}.\\[2.3mm]
As shown in \cite{Sub}, every smooth submersion admits smooth local sections.
The preceding theorem therefore subsumes
the existence of locally defined smooth extension operators
around each function in the image of the restriction map $C^\ell(M,N)\to C^\ell(L,N)$,
in all of the cases (a)--(d).
\begin{rem}
If $M$ is not only a rough $C^\ell$-manifold,
but a $C^\ell$-manifold with rough boundary,
then it suffices to assume that~$L$ is a split
submanifold in Theorem~\ref{thm-locops}(d) and thus
also in Theorem~\ref{spec-loco}(d); likewise in
Theorem~\ref{thm-mapmfd}(d).
Moreover, the local compactness of~$M$
is inessential for the conclusion~(i)
of Theorem~\ref{spec-loco}, assuming its
hypothesis~(d) (or its variant just described);
it suffices that~$M$ be $C^\ell$-paracompact.
The proofs in Section~\ref{proofs-outlook} apply
just as well in these situations.
\end{rem}
In the case $\ell=0$, the manifold structure on~$M$
is inessential for the preceding results.
Let us set up notation:
\begin{numba}\label{topsbundtop}
If $E\to X$ is a locally trivial
topological
vector bundle over a Hausdorff
topological space~$X$
whose fibres are locally convex spaces,
we endow the vector space
$\Gamma_{C^0}(E)$ of continuous sections
with the compact-open topology
and the closed vector subspace $\Gamma_{C^0_K}(E)$
of sections supported in a closed subset $K\sub X$
with the induced topology. If $X$ is locally compact,
we give $\Gamma_{C^0_c}(E)=\bigcup_{K\in\cK(X)}\Gamma_{C^0_K}(E)$
the locally convex direct limit topology.
\end{numba}
\begin{numba}\label{strangemfdmps}
If $X$ is a paracompact, locally compact topological
space and $N$ a $C^\infty$-manifold
modelled on locally convex spaces such that
$N$ has a local addition,
then $C(X,N)$ can be made a smooth manifold
in such a way that the modelling space at $\gamma\in C(X,N)$
is $\Gamma_{C^0_c}(\gamma^*(TN))$, as already mentioned in~\ref{lonenum}
(see \cite{Rou}).
\end{numba}
\begin{prop}\label{use-dugu}
Let $X$ be a Hausdorff topological
space and $Y\sub X$ be a closed subset.
Then the following holds:
\begin{itemize}
\item[\rm(a)]
If $X$ is paracompact and $Y$ is locally compact and
metrizable in the induced topology,
then
the restriction map
$\Gamma_{C^0}(E)\to\Gamma_{C^0}(E|_Y)$
has a continuous linear right inverse,
for each topological vector bundle
$E\to X$ over~$X$
whose fibres are locally convex spaces.
\item[\rm(b)]
If $X$ is metrizable
and there exists a metric~$d$ on~$X$
defining its topology for which
each $x\in Y$ has a neighbourhood $Z$ in~$Y$
such that $(Z,d|_{Z\times Z})$
is complete, then
the restriction map
$\Gamma_{C^0}(E)\to\Gamma_{C^0}(E|_Y)$
has a continuous linear right inverse,
for each topological vector bundle
$E\to X$ over~$X$
whose fibres are locally convex spaces.
\item[\rm(c)]
If $X$ is paracompact and locally compact
and $Y$ is metrizable,
then the restriction map
$\Gamma_{C^0_c}(E)\to\Gamma_{C^0_c}(E|_Y)$
has a continuous linear right inverse,
for each topological vector bundle~$E\to X$
over~$X$ whose fibres are locally convex spaces.
\item[\rm(d)]
If $X$ is paracompact and locally compact
and $Y$ is metrizable, then the restriction map
$C(X,N)\to C(Y,N)$ is a smooth submersion
between the manifolds of maps in~{\rm\ref{strangemfdmps}},
for each
smooth manifold~$N$ modelled
on locally spaces such that $N$ has a local
addition.
\end{itemize}
\end{prop}
\begin{rem}
Note that the hypothesis of Proposition~\ref{use-dugu}(b)
is satisfied if
$Y$ is locally compact in the induced topology\footnote{Take any compact
neighbourhood $Z$ of~$x$.} (duplicating a special case of~(a))
or $(Y,d|_{Y\times Y})$ is complete (take  $Z:=Y$).
\end{rem}
Applying Proposition~\ref{use-dugu}(a)
to a trivial vector bundle $X\times F\to X$,
we get:
\begin{cor}\label{cor-triv-bun}
If $X$ is a paracompact topological space, $Y\sub X$ a closed
subset which is locally compact and metrizable in the induced topology
and $F$ a locally convex space, then the restriction map
\[
C(X,F)\to C(Y,F)
\]
has a continuous linear right inverse
$($using the compact-open topology$)$.
\end{cor}
\section{Proof of Proposition~\ref{own-dug} and related results}
This section prepares the proof of Proposition~\ref{use-dugu}.
We begin with a review of Dugundji's result.
\begin{numba}\label{review-dug}
Let $(X,d)$ be a metric space, $Y\sub X$ be a closed subset
and $F$ be a locally convex space.
Dugundji~\cite{Dug} constructed a continuous extension
\[
\cE(\gamma)\colon X\to F
\]
for $\gamma\in C(Y,F)$ in such a way that the map
\[
\cE\colon C(Y,F)\to C(X,F)
\]
is linear (cf.\ \cite[p.\,359]{Dug})
and the image $\cE(\gamma)(X)$ is contained in the convex hull
of the image $\gamma(Y)$ (loc.\,cit.).
Since convex hulls of bounded sets
are bounded,
we deduce that $\cE(\gamma)$ is a bounded continuous
function if $\gamma$ is so and the restriction
\[
BC(Y,F)\to BC(X,F),\quad\gamma\mto\cE(\gamma)
\]
of $\cE$ to a map between
spaces of bounded continuous functions
is continuous with respect to the topology of
uniform convergence
(see already \cite[Theorem~5.1]{Dug}
if $F=\R$).
In fact,
\[
\|\cE(\gamma)\|_{\infty,q}=\|\gamma\|_{\infty,q}
\]
for each continuous seminorm~$q$ on~$F$
(if $\gamma(Y)\sub \wb{B}^{\,q}_r(0)$ for some $r>0$,
then also $\cE(\gamma)(X)\sub \wb{B}^{\,q}_r(0)$).
If $K:=Y\sub X$ is compact, we deduce that
\begin{equation}\label{notquite}
\cE\colon C(K,F)=BC(K,F)\to BC(X,F)\sub C(X,F)
\end{equation}
is continuous for the compact-open topologies
on $C(K,F)$ and $C(X,F)$.
\end{numba}
To establish Proposition~\ref{use-dugu}
in full generality, Dugundji's
results would not be sufficient.
To get~(a) in general,
They rely on Proposition~\ref{own-dug}
and its consequences.\\[2.3mm]
{\bf Proof of Proposition~\ref{own-dug}.}
The proof is based on a simplified variant of
Dugundji's construction.
\begin{numba}\label{numbaext}
Let $(X,d)$ be a metric space and $Y\sub X$
be a closed, non-empty subset.
Then the distance function
\[
d_Y\colon X\to[0,\infty[,\quad
x\mto\inf\{d(x,y)\colon y\in Y\}
\]
is Lipschitz continuous with Lipschitz constant~$1$,
and hence continuous. Hence
\[
W_n:=\{x\in X\colon d_Y(x)\in\,]2^{-n-1},2^{-n+1}[\}
\]
is an open subset of~$X$ for each $n\in\Z$ and
\[
X\setminus Y=\bigcup_{n\in\Z} W_n.
\]
Note that
\begin{equation}\label{neighbour}
(\forall k,\ell\in\Z)\;\,
W_k\cap W_\ell\not=\emptyset\;\Rightarrow\; |k-\ell|\leq 1.
\end{equation}
We let $\cU_n$ be a cover of~$W_n$
by open subsets $U\sub W_n$ of diameter
$\leq 2^{-n+1}$ and $\cU=\bigcup_{n\in\Z}\cU_n$.
Since $X\setminus Y$ is metrizable and hence
paracompact, there exists a partition of unity
$(h_j)_{j\in J}$ on $X\setminus Y$
such that, for each $j\in J$,
there exists $n(j)\in \Z$ and $U(j)\in\cU_{n(j)}$
such that
\[
\Supp(h_j)\sub U(j).
\]
After replacing $J$ with $\{j\in J\colon h_j\not=0\}$,
we may assume that $h_j\not=0$ for each $j\in J$.
We choose $x(j)\in\Supp(h_j)$.
Since $x(j)\in U(j)\sub W_{n(j)}$,
we have $d_Y(x(j))<2^{-n(j)+1}$
and find $y(j)\in Y$ such that $d(x(j),y(j))<2^{-n(j)+1}$.
Then
\begin{equation}\label{infosupp}
d(x,y(j))<d(x(j),y(j))+2^{-n(j)+1}<2^{-n(j)+2}
\mbox{ for all $x\in\Supp(h_j)$.}
\end{equation}
Given a locally convex space $F$ and $\gamma\in C(Y,F)$,
the assignment
\[
\cE(\gamma)(x):=
\left\{
\begin{array}{cl}
\sum_{j\in J}h_j(x)\gamma(y_j) &\mbox{ if $x\in X\setminus Y$;}\\
\gamma(x) & \mbox{ if $x\in Y$}
\end{array}
\right.
\]
defines a function $\cE(\gamma)\colon X\to F$.
\end{numba}
We now establish the following properties of
$\cE(\gamma)$ as constructed in~\ref{numbaext}
(thus establishing Proposition~\ref{own-dug}):
\begin{la}
In the situation of~{\rm\ref{numbaext}},
we have:
\begin{itemize}
\item[\rm(a)]
For each $\gamma\in C(Y,F)$,
the map $\cE(\gamma)\colon X\to F$ is continuous.
Moreover, $\cE(\gamma)|_Y=\gamma$
and
$\cE(\gamma)(X)$ is contained in the convex hull
of~$\gamma(Y)$.
\item[\rm(b)]
The map $\cE\colon C(Y,F)\to C(X,F)$
is linear.
\item[\rm(c)]
If $(Y,d|_{Y\times Y})$
is complete or $Y$ is locally compact, then
the linear map
$\cE\colon C(Y,F)\to C(X,F)$
is continuous.
\end{itemize}
\end{la}
\begin{proof}
(a) By construction,
$\cE(\gamma)|_Y=\gamma$ and $\cE(\gamma)(X)\sub\conv(\gamma(Y))$.
Each $x\in X\setminus Y$ has an open neighbourhood
$Q\sub X\setminus Y$ such that $J_Q:=\{j\in J\colon h_j|_Q\not=0\}$
is finite, whence $\cE(\gamma)|_Q=\sum_{j\in J_Q}h_j|_Q$
is continuous. Thus $f|_{X\setminus Y}$ is continuous.
To see that $\cE(\gamma)$ is also continuous at each $x\in Y$,
let $(x_n)_{n\in\N}$ be a sequence
in~$X$ such that $x_n\to x$. Let $q$ be a continuous seminorm
on~$F$ and $\ve>0$.
Since~$\gamma$ is continuous,
there exists $m\in \Z$ such that
\begin{equation}\label{easpart}
q(\gamma(y)-\gamma(x))\leq \ve
\mbox{ for all $y\in Y$ such that $d(y,x)<2^{-m}$.}
\end{equation}
There is $n_0\in\N$ such that
\begin{equation}\label{justcon}
d(x_n,x)<2^{-m-5}
\mbox{ for all $n\geq n_0$.}
\end{equation}
Let $n\geq n_0$. If $x_n\in Y$, then
$q(\cE(\gamma)(x_n)-\cE(\gamma)(x))=q(\gamma(x_n)-\gamma(x))\leq \ve$,
by~(\ref{easpart}).
If $x_n\in X\setminus Y$, then $x_n\in W_{k(n)}$
for some $k(n)\in\Z$;
thus
\begin{equation}\label{oneside}
2^{-k(n)-1}<d_Y(x_n)<2^{-k(n)+1}.
\end{equation}
For each $j\in J$,
the condition $x_n\in \Supp(h_j)$ implies that
$\Supp(h_j)\cap W_{k(n)}=\emptyset$ and hence
\[
W_{n(j)}\cap W_{k(n)}\not=\emptyset,
\]
whence $|n(j)-k(n)|\leq 1$ (see (\ref{neighbour}))
and thus
\[
d(x_n,y(j))<2^{-n(j)+2}\leq 2^{-k(n)+3}<2^4d_Y(x_n)
\leq 2^4d(x_n,x)<2^{-m-1},
\]
using (\ref{infosupp}),
(\ref{oneside}),
and (\ref{justcon}). As a consequence,
$d(x,y(j))\leq d(x,x_n)+d(x_n,y(j))<2^{-m}$ (recalling (\ref{justcon}))
and thus
$q(\gamma(y(j))-\gamma(x))\leq\ve$
if $x_n\in\Supp(h_j)$.
Then $h_j(x_n)q(\gamma(y(j))-\gamma(x))\leq h_j(x_n)\ve$
for all $j\in J$ and hence
\begin{eqnarray*}
q(\cE(\gamma)(x_n)-\cE(\gamma(x))
&=& q(\cE(\gamma)(x_n)-\gamma(x))
=q\left(\sum_{j\in J}h_j(x_n)(\gamma(y(j))-\gamma(x))\right)\\
&\leq&  \sum_{j\in J}h_j(x_n)q(\gamma(y(j))-\gamma(x))\leq\ve.
\end{eqnarray*}
Thus $q(\cE(\gamma)(x_n)-\cE(\gamma(x))\leq\ve$
for all $n\geq n_0$, and thus $\cE(\gamma)$ is continuous at~$x$.\\[2.3mm]
(b) The linearity is clear from the definition.\\[2.3mm]
(c) Let $K\sub X$ be a compact subset,
$J_K:=\{j\in J\colon \Supp(h_j)\cap K\not=\emptyset\}$
and
\begin{equation}\label{theP}
P:=(K\cap Y)\cup\{y(j)\colon j\in J_K\}.
\end{equation}
We claim that the closure $L:=\wb{P}$
of~$P$ in~$Y$ is compact.
If this is true, then we have
for each continuous seminorm~$q$
on~$F$ that
\[
\|\cE(\gamma)\|_{K,q}\leq \|\gamma\|_{L,q}
\]
for all $\gamma\in C(Y,F)$, and thus~$\cE$ is continuous.
In fact, $q(\cE(\gamma)(x))=q(\gamma(x))\leq \|\gamma\|_{L,q}$
for all $x\in K\cap Y$, as $K\cap Y\sub L$.
For all $x\in K\setminus Y$, we have
\[
q(\cE(\gamma)(x))
=q\left(\sum_{j\in J_K}h_j(x)\gamma(y(j))\right)
\leq\sum_{j\in J_K}h_j(x)\underbrace{q(\gamma(y(j)))}_{\leq\|\gamma\|_{L,q}}
\leq\|\gamma\|_{L,q}
\]
as well.\\[2.3mm]
We now prove the claim, starting with the case
that $(Y,d_{Y\times Y})$ is complete.
If we can show that~$P$ is precompact (totally bounded) in~$Y$,
then its closure~$L$ will be precompact and complete,
and thus~$L$ will be compact.\\[2.3mm]
For $x\in X$ and $\ve>0$, write
$B_\ve(x):=\{y\in X\colon d(x,y)<\ve\}$;
if $S\sub X$ is a subset, abbreviate
\[
B_\ve[S]:=\bigcup_{x\in S}B_\ve(x).
\]
To establish the precompactness of~$P$,
let $\ve>0$. Since~$K$ is compact and hence precompact,
we find a finite subset $\Phi\sub X$ such that
$K\sub B_{\ve/2}[\Phi]$,
whence
\begin{equation}\label{1inc}
K\cap Y\sub B_\ve[\Phi]
\end{equation}
in particular.
There exists $n_0\in\Z$ such that
\[
2^{-n_0+3}<\ve.
\]
Since
\[
H:=\{x\in K\colon d_Y(x)\geq 2^{-n_0-1}\}
\]
is a compact subset of $X\setminus Y$,
the set
\[
J_H:=\{j\in J\colon\Supp(h_j)\cap H\not=\emptyset\}
\]
is finite.
Thus
\[
\Psi:=\{y(j)\colon J_H\}
\]
is a finite subset of~$Y$.
Given $j\in J_K$, there exists $x\in \Supp(h_j)\cap K$.
If $n(j)\geq n_0$,
then
\[
d(y(j),x)<2^{-n(j)+2}\leq 2^{-n_0+2}<\ve/2
\]
by (\ref{infosupp}), and thus
\begin{equation}\label{2inc}
y(j)\in B_{\ve/2}[K]\sub B_\ve[\Phi]
\mbox{ if $n(j)\geq n_0$.}
\end{equation}
If $n(j)<n_0$,
then $\Supp(h_j)\sub W_{n(j)}$ entails
that $d_Y(x)>2^{-n(j)-1}>2^{-n_0-1}$,
whence $x\in H$ and thus $j\in J_H$.
Hence
\begin{equation}\label{3inc}
y(j)\in \Psi \mbox{ if $n(j)<n_0$.}
\end{equation}
By (\ref{1inc}), (\ref{2inc}), and (\ref{3inc}),
$P\sub B_\ve[\Phi\cup\Psi]=\bigcup_{z\in\Phi\cup\Psi}B_\ve(z)$.
Thus~$P$ is precompact.\\[2.3mm]
If $Y$ is locally compact,
then $K\cap Y$ has a relatively compact, open
neighbourhood~$Q$ in~$Y$.
There exists an open subset $V\sub X$
such that $Q=Y\cap V$.
After replacing $V$ with $V\cup (X\setminus Y)$,
we may assume that $K\sub V$. Since~$K$
is compact, there exists $\ve>0$ such that
$B_\ve[K]\sub V$. Then
\[
B_\ve[K]\cap Y\sub V\cap Y=Q
\]
is relatively compact in~$Y$ and
\begin{equation}\label{newi1}
K\cap Y\sub B_\ve[K]\cap Y.
\end{equation}
If we define~$n_0$, $H$, $J_H$, and the finite
subset $\Psi\sub Y$ as in the proof for complete~$Y$,
then (\ref{2inc}) and (\ref{3inc})
show that
\[
\{y(j)\colon j\in J_K\}\sub (B_\ve[K]\cap Y)\cup\Psi.
\]
Combining this with~(\ref{newi1}), we see that
\[
P\sub (B_\ve[K]\cap Y)\cup\Psi
\]
is a subset of the compact set $\wb{Q}\cup\Psi$
and thus relatively compact.
\end{proof}
\begin{rem}
If $F\not=\{0\}$ and the extension operator
$\cE\colon C(Y,F)\to C(X,F)$ given by~(\ref{numbaext})
is continuous, then the set $P$
defined in~(\ref{theP})
must be relatively compact in~$Y$
for each compact subset $K\sub X$,
as is easy to see.
\end{rem}
%
If $X$ is a set and $R\sub X\times X$ a subset,
we write
\begin{eqnarray*}
R^{-1} & := & \{(y,x)\in X\times X\colon (x,y)\in R\},\\
R\circ R &:= & \{(x,z)\in X\times X \colon (\exists y\in X)\;
(x,y)\in R\mbox{ and }(y,z)\in R\}
\end{eqnarray*}
and let $R^{\circ 3}$ be the set of all $(x,a)\in X\times X$
for which there exist $y,z\in X$ such that
$(x,y)$, $(y,z)$, $(z,a)\in R$.\\[2.3mm]
Recall that a uniform space is a set~$X$,
together with a filter $\cU$ of subsets
$U\sub X\times X$ which are supersets of the diagonal $\Delta_X$,
such that $U^{-1}\in\cU$
for each $U\in\cU$ and there exists $V\in\cU$
such that $V\circ V\sub U$
(see \cite[II.1.3]{Shu},
cf.\ \cite{Wei} and \cite{Eng}).
The elements $U\in\cU$ are called
\emph{entourages};
an entourage~$U$ is called \emph{symmetric}
if $U=U^{-1}$.
Every entourage~$U$ contains a symmetric
entourage, namely $U\cap U^{-1}$.
Every uniform space $(X,\cU)$ defines a topology on~$X$;
a basis of neighbourhoods of $x\in X$ is given by the sets
$U[x]:=\{y\in X\colon (x,y)\in U\}$
for $U\in\cU$. A Hausdorff topology
arises from a uniform structure
if and only if it is completely regular
(see \cite{Shu} or \cite{Eng}).
\begin{prop}\label{unif-case}
Let $(X,\cU)$ be a uniform space
and $Y\sub X$ be a closed subset.
Assume that the induced filter $\cU_Y:=\{U\cap (Y\times Y)\colon U\in\cU\}$
on $Y\times Y$ is generated by a countable filter basis.
Then the restriction map
\[
C(X,F)\to C(Y,F)
\]
admits a linear right inverse $\cE\colon C(Y,F)\to C(X,F)$
for each locally convex space~$F$,
such that $\cE(\gamma)(X)$ is contained in the convex hull
of $\gamma(Y)$ for each $\gamma\in C(Y,F)$.
If, moreover, $Y$ is locally compact
in the induced topology or the uniform space $(Y,\cU_Y)$
is complete, then $\cE$ as before can be chosen
as a continuous linear map with respect to the compact-open
topologies on domain and range.
\end{prop}
\begin{proof}
Let $\{W_n\colon n\in\N\}$ be a countable
filter basis for $\cU_Y$. For each $n\in\N$,
there exists $U_n\in\cU$ such that $W_n=U_n\cap (Y\times Y)$.
Let $V_0:=X\times X$.
Recursively, if $n\in\N$ and $V_0,\ldots, V_{n-1}\in\cU$ have been
have been determined, let $V_n\in\cU$
be a symmetric entourage such that $V_n \sub U_n$
and $V_n^{\circ 3}\sub V_{n-1}$.
Then there exists a pseudometric
\[
\rho\colon X\times X\to[0,\infty[
\]
on~$X$ such that for all $n\in\N$
\[
\{(x,y)\in X\times X\colon \rho(x,y)<2^{-n}\}
\sub V_n\sub \{(x,y)\in X\times X\colon \rho(x,y)\leq 2^{-n}\}
\]
(see \cite[Theorem~8.1.10]{Eng}).
By \cite[Proposition~8.1.9]{Eng},
$\rho$~is continuous.
Since
\[
B_n:=\{(x,y)\in Y\times Y\colon \rho(x,y)<2^{-n}\}
\sub V_n\cap(Y\times Y)\sub W_n
\]
and $\bigcap_{n\in\N} W_n=\Delta_Y$,
we deduce that $\rho|_{Y\times Y}$
is a metric on~$Y$ such that
$\{B_n\colon n\in\N\}$ is a filter basis for
$\cU|_Y$,
whence $\rho|_{Y\times Y}$ defines the topology
induced by~$X$ on~$Y$,
and Cauchy sequences in~$Y$ with respect to
$\cU_Y$ and those with respect to $\rho|_{Y\times Y}$
coincide. For $x,y\in X$, write $x\sim y$
if $\rho(x,y)=0$. Then $\sim$ is an equivalence
relation on~$X$. Let $X_\rho$
be the set of equivalence classes $[x]$
and $q\colon X\to X_\rho$, $x\mto [x]$ be the canonical map.
Then
\[
d\colon X_\rho\times X_\rho\to[0,\infty[,\;\,
([x],[y])\mto \rho(x,y)
\]
is a well-defined metric on $X_\rho$.
Note that $q|_Y$ is injective and $d([x],[y])=\rho(x,y)$
for all $x,y\in Y$, whence
\[
\phi\colon q|_Y\colon (Y,\rho|_{Y\times Y})\to (q(Y),d|_{q(Y)\times q(Y)})
\]
is an isometry (notably, a homeomorphism).
The map $q$ is continuous as $q^{-1}(\{[y]\in X_\rho\colon
d([x],[y])<\ve\})=\{y\in X\colon \rho(x,y)<\ve\}$
is open in~$X$ for all $x\in X$ and $\ve>0$.
If
\[
\cE\colon C(q(Y),F)\to C(X_\rho,F)
\]
is the extension operator defined in \ref{numbaext}, then
\[
\cF:=C(q,F)\circ \cE\circ C(\phi^{-1},F)
\colon C(Y,F)\to C(X,F),\quad \gamma\mto \cE(\gamma\circ\phi^{-1})\circ q
\]
is a linear right inverse for the restriction map
$C(X,F)\to C(Y,F)$.
Moreover,
\[
\cF(\gamma)(X)=\cE(\gamma\circ\phi^{-1})(X_\rho)
\sub\conv \gamma(\phi^{-1}(q(Y)))=\conv\gamma(Y)
\]
for all $\gamma\in C(Y,F)$.
If $(Y,\cU_Y)$ is
complete, then $(q(Y),d|_{q(Y)\times q(Y)})$
is complete, being isometric to $(Y,\rho|_{Y\times Y})$.
If $Y$ is locally compact, then also $q(Y)$ (being homeomorphic
to~$Y$). In both cases, $\cE$ is continuous and hence
also~$\cF$.
\end{proof}
\begin{cor}\label{unif-cor}
Let $X$ be a completely regular topological
space and $Y\sub X$ be a compact, metrizable subset.
For each locally convex space~$F$, the restriction map
\[
C(X,F)\to C(Y,F)
\]
then admits a continuous linear right inverse
$\cE\colon C(Y,F)\to C(X,F)$
such that
$\cE(\gamma)(X)$ is contained in the convex hull
of $\gamma(Y)$ for each $\gamma\in C(Y,F)$.
\end{cor}
\begin{proof}
Let~$d$ be a metric on~$Y$ defining its topology.
Since $X$ is completely regular,
its topology arises from a uniform structure $\cU$
on~$X$, and $\bigcap_{U\in\cU}=\Delta_X$ as~$X$
is Hausdorff. On the compact space~$Y$,
there is only one uniform structure defining its topology
and the latter is complete
(see Theorems~8.3.13 and 8.3.15 in \cite{Eng}).
Hence $(Y,\cU_Y)$ is complete and $\cU_Y$
has a countable filter basis (as it coincides
with the uniform structure given by~$d$).
So Proposition~\ref{unif-case} applies.
\end{proof}
\section{Proofs for Section~\ref{sec-outlook}}\label{proofs-outlook}
In this section, we prove Proposition~\ref{aprop1}, Theorem~\ref{thm-locops},
Theorem~\ref{thm-mapmfd}, and Proposition~\ref{use-dugu}.
We begin with preparations.
As in the case without boundary,
partitions of unity are available
(as we shall check in Appendix~\ref{appA}).
\begin{la}\label{partu-ex}
If $\ell\in\N_0\cup\{\infty\}$,
$M$ is a paracompact, locally compact rough $C^\ell$-manifold
and $(U_j)_{j\in I}$ an open cover of~$M$,
then there exists a $C^\ell$-partition of unity $(h_j)_{j\in J}$
on~$M$ such that $\Supp(h_j)\sub U_j$ for
all $j\in J$.
\end{la}
\begin{defn}\label{def-vbdl}
Let $M$ be a rough $C^\ell$-manifold
and $\cF$ be a set of locally convex spaces.
A rough $C^\ell$-manifold $E$, together with an $RC^\ell$-map
$\pi\colon E\to M$ and a vector space structure on $E_x:=\pi^{-1}(\{x\})$
for each $x\in M$,
is called a \emph{$C^\ell$-vector bundle over~$M$
with fibres in~$\cF$} if each $x\in M$ has an open neighbourhood
$U\sub M$ such that $E|_U:=\pi^{-1}(U)$ is trivializable,
i.e., there exists an $RC^\ell$-diffeomorphism
\[
\theta\colon E|_U\to U\times F
\]
for some $F\in\cF$
such that $\pr_2\circ\,\theta$
restricts to
vector space isomorphism $E_y\to F$
for all $y\in U$, and $\theta=(\pi|_{E|_U},\pr_2\circ \,\theta)$
(then $\theta$ is called a \emph{local trivalization} for~$E$).
If $\cF=\{F\}$ is a singleton in the preceding situation,
then~$E$ is called a $C^\ell$-vector bundle
\emph{with typical fibre~$F$}.
If $U=M$ can be chosen, then
$\theta\colon E\to M\times F$ is called
a \emph{global trivialization} for~$E$
and~$E$ is called $C^\ell$-trivializable.
\end{defn}
\begin{defn}\label{def-sections}
If $\pi\colon E\to M$ is a $C^\ell$-vector bundle
with fibres in~$\cF$, let $\Gamma_{C^\ell}(E)$
be the set of all $C^\ell$-sections of~$E$
(i.e., $RC^\ell$-maps $\sigma\colon M\to E$ with $\pi\circ\sigma=\id_M$).
Then $\Gamma_{C^\ell}(E)$ is a vector space under
pointwise operations and the compact-open $C^\ell$-topology\footnote{For rough $C^\ell$-manifolds
$M$ and $L$, the compact-open $C^\ell$-topology on
$RC^\ell(M,L)$ is defined as the initial topology with respect to the mappings
$RC^\ell(M,L)\to C(T^jM,T^jN)$, $\gamma\mto T^j\gamma$ to
spaces of continuous mappings between iterated tangent bundles,
endowed with the compact-open topology.}
induced by $RC^\ell(M,E)$ is a locally convex vector topology (see \cite{Rou}).
\end{defn}
\begin{rem}\label{basic-sectop}
In the situation of Definition~\ref{def-sections}, we have (see \cite{Rou}):\\[2.3mm]
(a) If $U\sub M$ is an open subset, then the restriction map
$\Gamma_{C^\ell}(E)\to\Gamma_{C^\ell}(E|_U)$ is continuous and linear.\\[2.3mm]
(b) If $(U_j)_{j\in J}$ is an open cover of~$M$,
then the linear map
\[
\Gamma_{C^\ell}(E)\to\prod_{j\in J}\Gamma_{C^\ell}(E|_{U_j}),\quad
\sigma\mto(\sigma|_{U_j})_{j\in J}
\]
is a topological embedding with closed image.\\[2.3mm]
(c) If $\theta\colon E|_U\to U\times F$ is a local trivialization,
then the map
\begin{equation}\label{representa}
\Theta \colon \Gamma_{C^\ell}(E|_U)\to RC^\ell(U,F),\quad \sigma\mto\pr_2 \circ\,
\theta\circ \sigma
\end{equation}
an isomorphism of topological vector spaces.
\end{rem}
\begin{defn}\label{defGammaK}
If $K\sub M$ is a closed subset in the situation of Definition~\ref{def-sections},
we endow the closed vector subspace
\[
\Gamma_{C^\ell_K}(E):=\{\sigma\in \Gamma_{C^\ell}(E)\colon (\forall x\in M\setminus K)\;\sigma(x)=0\}
\]
with the topology induced by $\Gamma_{C^\ell}(E)$.
\end{defn}
\begin{rem}\label{restr-cp}
(a) If $U\sub M$ is an open neighbourhood of~$K$ in~$M$,
then the restriction map
\[
\Gamma_{C^\ell_K}(E)\to \Gamma_{C^\ell_K}(E|_U)
\]
is an isomorphism of topological vector spaces
(apply Remark~\ref{basic-sectop}(b) to the open cover $\{U,M\setminus K\}$ of~$M$).
The inverse map takes $\sigma\in\Gamma_{C^\ell_K}(E|_U)$~to
\[
\wt{\sigma}\colon M\to E,\quad
x\mto\left\{
\begin{array}{cl}
\sigma(x) &\mbox{if $\,x\in U$;}\\
0\in E_x & \mbox{if $\, x\in M\setminus K$.}
\end{array}\right.
\]
(b) If $h\colon M\to \R$ is a $C^\ell$-function with support~$K$,
then the mutiplication operator
\[
\Gamma_{C^\ell}(E)\to\Gamma_{C^\ell_K}(E),\quad \sigma\mto h\cdot \sigma
\]
is linear and continuous
(as follows from parts (b) and (c) in Remark~\ref{basic-sectop}
and\linebreak
Lemma~\ref{sammelsu}(g)).
\end{rem}
\begin{defn}\label{cp-sect}
If $M$ is
locally compact,
we endow
\[
\Gamma_{C^\ell_c}(E):=\bigcup_{K\in\cK(M)}\Gamma_{C^\ell_K}(E)\vspace{-1mm}
\]
with the locally convex direct limit topology.
\end{defn}
\begin{rem}\label{basic-cpsec}
In the situation of Definition~\ref{cp-sect},
the inclusion mapping\linebreak
$\Gamma_{C^\ell_c}(E)\to\Gamma_{C^\ell}(E)$
is linear, and it is continuous (by the locally convex direct limit property) as
all its restrictions to inclusion maps $\Gamma_{C^\ell_K}(E)\to\Gamma_{C^\ell}(E)$
are continuous. Thus $\Gamma_{C^\ell_c}(E)$ is Hausdorff
and induces the given topology on each of its vector subspaces $\Gamma_{C^\ell_K}(E)$.
\end{rem}
Vector bundles can be pulled back along $C^\ell$-maps.
\begin{numba}\label{basic-pb}
Let $\ell\in\N_0\cup\{\infty\}$,
$M$ be a rough $C^\ell$-manifold and
$\pi\colon E\to M$ be a $C^\ell$-vector bundle
with fibres in a set~$\cF$ of locally convex spaces.
If $L$ is a rough $C^\ell$-manifold and $f\colon L\to M$
an $RC^\ell$-map (or a $C^\ell$-map if $M$ is a $C^\ell$-manifold
with rough boundary),
then
\[
f^*(E):=\bigcup_{x\in L}\{x\}\times E_{f(x)}\sub L\times E
\]
can be made a $C^\ell$-vector bundle over~$L$ with fibres in~$\cF$
in a natural way, with the projection $\pr_1\colon f^*(E)\to L$,
$(x,y)\mto x$ to the base~$L$; for each local trivialization
$\theta \colon E|_U\!\to U\!\times\! F$ of~$E$
and second component $\theta_2\colon E|_U\!\to F$,
the~map
\[
\theta^f\colon f^*(E)|_{f^{-1}(U)}\colon f^*(E)\cap (f^{-1}(U)\times E)
\to f^{-1}(U)\times F,\quad
(x,y)\mto (x,\theta_2(y))
\]
is a local trivialization for $f^*(E)$,
with inverse map $(x,z)\mto (x,\theta^{-1}(f(x),z))$
(see~\cite{Rou}).\\[2.3mm]
If $L$ is a subset of~$M$ here and $f\colon L\to M$ the inclusion map,
we also write $E|_L:=f^*(E)$. Then $\pr_2\colon E|_L\to E$, $(x,y)\mto y$
is an injective map and one may identify $E|_L$ with the image
of $\pr_2$, but we shall mainly do so when dealing explicitly with
an open subset $L$ of~$M$
(to retain the meaning of $E|_U$ in Definition~\ref{def-vbdl}).
\end{numba}
\begin{rem}\label{sectionops}
Let $\cF$ be a set of locally convex spaces
and assume that $M$ locally admits extension operators
for $F$-valued $C^\ell$-functions around~$L$ for all $F\in\cF$,
where $\ell$, $M$, and~$L$ are as in Definition~\ref{locml}.
Let $\pi\colon E\to M$ be a $C^\ell$-vector bundle
with fibres in~$\cF$.
If $x\in L$, $P\sub L$ is an $x$-neighbourhood
and $\theta\colon E|_U\to U\times F$
is a local trivialization for~$E$ with $x\in U$,
then there exists a local extension operator
\[
\cE_x\colon C^\ell(P_x,F)\to C^\ell(Q_x,F)
\]
for an open $x$-neighbourhood $P_x\sub L$ with $P_x\sub P\cap U$
and an open $x$-neighbourhood $Q_x\sub M$ with $Q_x\sub U$,
by Remark~\ref{locallocal}.
Let $E|_L$ be the pullback-bundle $j^*(E)$ for the inclusion map $j\colon L\to M$.
Then
\[
\theta_L:=\theta|_{(E|_L)|_{P_x}}\colon (E|_L)|_{P_x}\to P_x\times F
\]
is a local trivialization for~$E|_L$.
Let
\[
\Theta\colon \Gamma_{C^\ell}(E|_{Q_x})\to C^\ell(Q_x,F)\quad\mbox{and}\quad
\Theta_L\colon \Gamma_{C^\ell}((E|_L)|_{P_x})\to C^\ell(P_x,F)
\]
be the isomorphisms of topological vector spaces associated to
$\theta|_{E|_{Q_x}}$ and $\theta_L$, respectively
(analogous to~(\ref{representa})).
Then
\begin{equation}\label{thefx}
\cF_x\, :=\, \Theta^{-1}\circ \cE_x\circ \Theta_L\colon
\Gamma_{C^\ell}((E|_L)|_{P_x})\to\Gamma_{C^\ell}(E|_{Q_x})
\end{equation}
is a continuous linear map with $\cF_x(\sigma)|_{L\cap Q_x}=\sigma|_{L\cap Q_x}$
for all $\sigma\in \Gamma_{C^\ell}((E|_L)|_{P_x})$.
We shall also refer to $\cF_x$ as a \emph{local extension operator}
around~$x$.
\end{rem}
{\bf Proof of Proposition~\ref{aprop1}.}
For each $x\in L$,
there exists a local trivialization $\theta_x\colon E|_{U_x}\to U_x\times F$
of~$E$ such that $x\in U_x$;
if $M$ is locally compact,
we assume, moreover, that~$U_x$ is relatively compact in~$M$.
Using Remark~\ref{sectionops},
we find a local extension operator
\[
\cF_x\colon \Gamma_{C^\ell}((E|_L)|_{P_x})\to \Gamma_{C^\ell}(E|_{Q_x})
\]
for an open $x$-neighbourhood $P_x\sub U_x\cap L$
and an open $x$-neighbourhood $Q_x\sub U_x$ such that $Q_x\cap L\sub P_x$.
Then $(Q_x)_{x\in L}$, together with $M\setminus L$,
is an open cover of~$M$.
We let $(h_x)_{x\in L}$, together with $h$,
be a $C^\ell$-partition of unity for~$M$
such that $K_x:=\Supp(h_x)\sub U_x$ for all $x\in L$ and $\Supp(h)\sub M\setminus L$.
Set $J:=\{x\in L\colon h_x\not=0\}$.
Then $(h_x)_{x\in J}$, together with $h$, is a $C^\ell$-partition
of unity for~$M$. Notably, $(K_x)_{x\in J}$ is a locally finite family of closed
subsets of~$M$ (which are compact if~$M$ is locally compact).
Moreover, $(h_x|_L)_{x\in J}$ is a $C^\ell$-partition of unity for~$L$.\\[2.3mm]
(a) If $\sigma\in\Gamma_{C^\ell}(E|_L)$,
then the sum
\[
\cE(\sigma):=\sum_{x\in J} (h_j\cdot\cF_x(\sigma|_{P_x}))\hspace*{.2mm}\wt{\;}
\]
converges in $\Gamma_{C^\ell}(E)$.
In fact, by local finiteness of $(K_x)_{x\in J}$,
each $z\in M$ has an open neighbourhood~$U$ in~$M$ such that
\[
I:=\{x\in J\colon U\cap K_x\not=\emptyset\}
\]
is finite and
\[
\sum_{x\in J} (h_j\cdot\cF_x(\sigma|_{P_x}))\hspace*{.2mm}\wt{\;}\,|_U=\sum_{x\in I}
(h_j\cdot\cF_x(\sigma|_{P_x}))\hspace*{.2mm}\wt{\;}\,|_U
\]
a finite sum (as all other summands vanish),
from which convergence follows using Remark~\ref{basic-sectop}(b).
To see that the map
\[
\cE\colon \Gamma_{C^\ell}(E|_L)\to\Gamma_{C^\ell}(E)
\]
is continuous,
by Remark~\ref{basic-sectop}(b)
we only need to show that the previous restrictions
are continuous mappings of $\sigma$ to $\Gamma_{C^\ell}(E|_U)$.
It suffice to show continuity of the summand for each $x\in I$.
But this summand is the composition
\[
\Gamma_{C^\ell}((E|_L)|_{P_x})\to \Gamma_{C^\ell}(E|_U),
\;\;\sigma\mto (h_x\cdot \cF_x(\sigma))\hspace*{.2mm}\wt{\;}\, |_U,
\]
which is continuous by continuity of $\cF_x$, the multiplication operator
(analogous to Remark~\ref{restr-cp}(b)),
the extension map as in Remark~\ref{restr-cp}(a),
and a restriction map (see Remark~\ref{basic-sectop}(a)). Finally, observe that
$\cE$ is linear by construction. If $y\in L$ and $y\in K_x$ for some $x\in J$,
then $x\in Q_x\cap L$ and thus
\[
\cF_x(\sigma|_{P_x})(y)=\sigma(y).
\]
If $x\in J$ and $y\not\in K_x$, then $(h_x\cdot\cF_x(\sigma|_{P_x})\hspace*{.2mm}\wt{\;}(y)=0
=h_x(y)\sigma(y)$.
Hence
\[
\cE(\sigma)(y)=\sum_{x\in J}h_x(y)\sigma(y)=\sigma(y)
\]
and thus $\cE(\sigma)|_L=\sigma$.\\[2.3mm]
(b) If $\sigma\in \Gamma_{C^\ell_c}(E|_L)$, then
\[
\cF(\sigma):=\sum_{x\in J} (h_j\cdot\cF_x(\sigma|_{P_x}))\hspace*{.2mm}\wt{\;}
\]
is a finite sum an hence an element of $\Gamma_{C^\ell_c}(E)$.
By construction, the map
\[
\cF\colon \Gamma_{C^\ell_c}(E|_L)\to\Gamma_{C^\ell_c}(E)
\]
is linear. By the locally convex direct limit property, it will be continuous
if we can show that its restriction to a map
\[
\cF_K\colon \Gamma_{C^\ell_K}(E|_L)\to\Gamma_{C^\ell_c}(E)
\]
is continuous for each compact subset $K\sub L$. Now $I:=\{x\in J\colon K\cap K_x\not=\emptyset\}$
is a finite set, and thus $B:=\bigcup_{x\in I}K_x$ is a compact subset of~$M$.
The image of $\cF_K$ is contained in $\Gamma_{C^\ell_B}(E)$.
We therefore only need to show that $\cF_K$ is continuous as a map
\[
\cF_K\colon \Gamma_{C^\ell_K}(E|_L)\to\Gamma_{C^\ell_B}(E).
\]
But this map is a restriction of the continuous map~$\cE$ and hence continuous. $\,\square$\\[2.3mm]
{\bf Proof of Theorem~\ref{thm-locops}.}
As mentioned in the theorem, (a) can be found in~\cite{RaS}.
All of (b), (c), and (d) follow immediately
from Proposition~\ref{aprop1}.
In fact, $M$ locally admits extension operators for
$F$-valued $C^\ell$-functions around~$L$
in each case:
In case~(b), the local extension operators are provided
by Corollary~\ref{convex-ext}.
In~case~(c),
the local extension operators are provided by Corollary~\ref{extcorner}.
In case~(d),
the local extension operators are provided by~\ref{projextensions}.$\,\square$\\[2.3mm]
Before we prove Theorem~\ref{thm-mapmfd}, let us
recall an intrinsic description of
the smooth manifold structure on $C^\ell(M,N)$
(without recourse to the embedding into a fine box product
that was mentioned earlier).
\begin{numba}\label{def-map-mfd}
Let $\ell\in\N_0\cup\{\infty\}$
and $M$ be a paracompact, locally compact rough $C^\ell$-manifold.
Let $N$ be a $C^\infty$-manifold modelled on a locally convex space
such that~$N$ admits a local addition.
Then the smooth manifold structure
on $C^\ell(M,N)$ can be obtained as follows:\\[2.3mm]
Pick a local addition $\Sigma\colon U\to N$ for~$N$.
For $f\in C^\ell(K,N)$, we identify the space
\[
\Gamma_{C^\ell_c}(f^*(TN))
\]
of compactly supported $C^\ell$-sections of the pullback-bundle $f^*(TN)\to M$ with
the set
$\Gamma_f$ of all
$\tau\in C^\ell(M,TN)$
such that $\pi_{TN}\circ \tau=f$ and
\[
\Supp(\tau):=\wb{\{x\in M\colon \tau(x)\not=0_{f(x)}\}}\sub M
\]
is compact.
Transport the locally convex direct limit topology from $\Gamma_{C^\ell_c}(f^*(TN))$
to~$\Gamma_f$.
Then
\[
\cO_f:=\{\tau\in\Gamma_f\colon \tau(M)\sub \Omega\}
\]
is an open $0$-neighbourhood in~$\Gamma_f$.
We endow $C^\ell(M,N)$ with the final topology
with respect to the mappings
\[
\cO_f\to C^\ell(M,N),\quad \tau\mto\Sigma\circ \tau.
\]
Then the image~$\cO_f'$ of each of these maps becomes open in $C^\ell(M,N)$,
the map
\[
\phi_f\colon \cO_f\to\cO_f',\quad \tau\mto \Sigma\circ\tau
\]
becomes a homeomorphism, and $\{\phi_f^{-1}\colon f\in C^\ell(M,N)\}$
is an atlas of charts which defines a smooth manifold
structure on $C^\ell(M,N)$, modelled on the set
$\{\Gamma_f\colon f\in C^\ell(M,N)\}$ of locally convex spaces
(see \cite{Rou} for details). 
\end{numba}
{\bf Proof of Theorem~\ref{thm-mapmfd}.}
Consider the restriction map
\[
\rho\colon C^\ell(M,N)\to C^\ell(L,N),\quad \gamma\mto \gamma|_L.
\]
Let $\Sigma\colon U\to N$ be a local addition for~$N$.
Given $f\in C^\ell(M,N)$, let $\Gamma_f$, the open
set $\cO_f\sub \Gamma_f$, and the $C^\infty$-diffeomorphism
\[
\phi_f\colon \Omega_f\to\Omega_f'\sub C^\ell(M,N)
\]
be as in \ref{def-map-mfd}.
Define $\cO_{f|_l}\sub\Gamma_{f|_L}\sub C^\ell(L,N)$
and $\phi_{f|_L}\colon \cO_{f|_L}\to\cO_{f|_L}\sub C^\ell(L,N)$
analogously.
Then
\[
\gamma|_L\in \Omega_{f|_L}'
\]
for all $\gamma\in \Omega_f$ and
\[
(\phi_{f|_L})^{-1}\circ \rho\circ \phi_f
\]
equals the restriction map
\[
\Omega_f\to\Omega_{f|_L},
\]
which is a restriction of the restriction map
\[
\Gamma_f\to\Gamma_{f|_L},\quad\tau\mto \tau|_L
\]
which corresponds to the restriction map
\[
r\colon \Gamma_{C^\ell_c}(f^*(TN))\to\Gamma_{C^\ell_c}(f^*(TN)|_L)
=\Gamma_{C^\ell_c}((f|_L)^*(TN)),\;\;
\sigma\mto\sigma|_L.
\]
As the latter has a continuous linear right inverse by Theorem~\ref{thm-locops},
we deduce that~$\rho$ is a smooth submersion. $\,\square$
\begin{numba}
Let $M$ be a Hausdorff topological space
and $\cF$ be a set of locally convex spaces.
A Hausdorff topological space~$E$, together with a
continuous map
$\pi\colon E\to M$ and a vector space structure on $E_x:=\pi^{-1}(\{x\})$
for each $x\in M$,
is called a \emph{topological vector bundle} (or also:
\emph{$C^0$-vector bundle}) over~$M$
with fibres in~$\cF$ if each $x\in M$ has an open neighbourhood
$U\sub M$ such that $E|_U:=\pi^{-1}(U)$ is trivializable,
i.e., there exists a homeomorphism
\[
\theta\colon E|_U\to U\times F
\]
for some $F\in\cF$
such that $\pr_2\circ\,\theta$
restricts to
vector space isomorphism $E_y\to F$
for all $y\in U$, and $\theta=(\pi|_{E|_U},\pr_2\circ \,\theta)$.
If $\cF=\{F\}$ is a singleton in the preceding situation,
then~$E$ is called a topological vector bundle
\emph{with typical fibre~$F$}.
\end{numba}
\begin{rem}\label{fromCelltotop}
Replacing $C^\ell$ and $RC^\ell$ with $C^0$
and $M$ with~$X$,
Remark~\ref{basic-sectop}
remains valid for a topological
vector bundle
$\pi\colon E\to X$
over a Hausdorff topological space~$X$,
whose fibers
are locally connvex spaces,
by well-known facts concerning the compact-open
topology
(parts~(a), (b), and~(c)
follow from Remark~A.5.10, Lemma~A.5.11
and Lemma~A.5.3 in \cite{GaN}).
Likewise, Remark~\ref{restr-cp}
carries over (the part~(b)
follows from \cite[Lemma~A.5.2]{GaN}).
\end{rem}
\begin{numba}\label{locexttop}
Let $X$ be
a Hausdorff topological space and
$Y\sub X$ be a closed subset.
We say that $X$ \emph{locally admits extension
operators around~$Y$}
if,
for each $x\in Y$ and locally convex space~$F$,
each $x$-neighbourhood in~$Y$
contains an open $x$-neighbourhood
$P_x\sub Y$ for which there exists a continuous linear
operator
\begin{equation}\label{localoptop}
\cE_x\colon C(P_x,F)\to C(Q_x,F)
\end{equation}
for some open $x$-neighbourhood
$Q_x\sub X$
such that $Q_x\cap Y\subseteq P_x$
and
\[
\cE(\gamma)|_{Q_x\cap Y}=\gamma|_{Q_x\cap Y}
\]
for all $\gamma\in C(P_x,F)$
(using the compact-open topology
in~(\ref{localoptop})).
\end{numba}
\begin{prop}\label{aprop1top}
Let $X$ be a paracompact topological space
and $Y\sub X$ be a closed subset
such that $X$ locally admits extension operators
around~$Y$. If\linebreak
$E\to X$ is a vector bundle
whose fibres are locally convex spaces,
then we have:
\begin{itemize}
\item[\rm(a)]
The restriction map
\[
\Gamma_{C^0}(E)\to \Gamma_{C^0}(E|_Y)
\]
admits a continuous linear right inverse;
\item[\rm(b)]
If, moreover, $X$ is locally compact, then
the restriction map
\[
\Gamma_{C^0_c}(E)\to \Gamma_{C^0_c}(E|_Y)
\]
admits a continuous linear right inverse.
\end{itemize}
\end{prop}
\begin{proof}
In view of Remark~\ref{fromCelltotop},
we can repeat the proof of Proposition~\ref{aprop1}
with $X$ in place of~$M$,
$Y$ in place of $L$, and $\ell=0$.
\end{proof}
{\bf Proof of Proposition~\ref{use-dugu}.}
(a) Let $x\in Y$ and $P$ be an $x$-neighbourhood in~$Y$.
Since~$Y$ is locally compact, we find a compact $x$-neighbourhood
$P_x$ in~$Y$ such that $P_x\sub P$.
Let $Q_x$ be an open subset of~$X$ such that $Q_x\cap Y$
equals the interior $P_x^0$ of $P_x$ relative~$Y$.
Let $F$ be a locally convex space.
As $X$, being paracompact,
is normal and hence completely regular,
Corollary~\ref{unif-cor}
provides a continuous linear extension operator
\[
C(P_x,F)\to C(X,F).
\]
Composing the latter with the restriction map
$C(X,F)\to C(Q_x,F)$, we obtain a local extension
operator $C(P_x,F)\to C(Q_x,F)$ as in~\ref{locexttop}.
Thus Proposition~\ref{aprop1top}(a) applies.\\[2.3mm]
(b) Let $x\in Y$ and $P$ be an $x$-neighbourhood in~$Y$.
Since~$Y$ is metrizable, we find a closed 
$x$-neighbourhood~$A$ in~$Y$ such that $A\sub P$.
By hypothesis, there exists an $x$-neighbourhood
$Z$ in~$Y$ such that $(Z,d|_{Z\times Z})$
is complete. Then also $P_x:=Z\cap A$ is complete
in the metric induced by~$d$, being closed in~$Z$.
Moreover, $P_x\sub P$. Being complete, $P_x$ is closed
in~$X$.
Let $Q_x$ be an open subset of~$X$ such that $Q_x\cap Y$
equals the interior $P_x^0$ of $P_x$ relative~$Y$.
Let $F$ be a locally convex space.
Proposition~\ref{own-dug}
provides a continuous linear extension operator
\[
C(P_x,F)\to C(X,F).
\]
Composing the latter with the restriction map
$C(X,F)\to C(Q_x,F)$, we obtain a local extension
operator $C(P_x,F)\to C(Q_x,F)$ as in~\ref{locexttop}.
Thus Proposition~\ref{aprop1top}(a) applies.\\[2.3mm]
(c) If $X$ is, moreover, locally compact,
the proof of~(a) shows that
Proposition~\ref{aprop1top}(b) applies.\\[2.3mm]
(d) We can repeat the proof of Theorem~\ref{thm-mapmfd},
using part~(c) of the current proposition
in place of Theorem~\ref{thm-locops} and Proposition~\ref{aprop1}. $\,\square$
\begin{rem}
Let $F$ be a locally convex space,
$X$ be a topological space and $Y\sub X$
be a closed subset.
If a linear map
\[
\cE\colon C(Y,F)\to C(X,F)
\]
has the property
that
\begin{equation}\label{rangeprop}
\cE(\gamma)(X)\sub\conv \gamma(Y)
\mbox{ for each $\gamma\in C(Y,F)$,}
\end{equation}
then $\cE$ restricts to a continuous linear operator
\[
BC(Y,F)\to BC(X,F)
\]
with respect to the topology of uniform convergence
on domain and range (cf.\ \ref{review-dug}).
The property (\ref{rangeprop})
is satisfied by the extension operators
$C(Y,F)\to C(X,F)$ we have constructed,
in each of the following situations:
\begin{itemize}
\item[(a)]
Proposition~\ref{own-dug} (as stated in the conclusion).
\item[(b)]
Proposition~\ref{unif-case}
(as stated in the conclusion).
\item[(c)]
Corollary~\ref{cor-triv-bun},
if Proposition~\ref{aprop1top}
is applied in the proof of Proposition~\ref{use-dugu}
in appropriate form.
Namely, in the proof of Proposition~\ref{aprop1top}
(which re-uses the proof of Proposition~\ref{aprop1}),
let each $\theta_x$ be the global trivialization
$\pr_1\colon X\times F\to X$.
\end{itemize}
\end{rem}
\section{Density of test functions in mapping groups}\label{approxis}
In this section, we prove density of the set
of compactly supported
smooth functions
in various mapping groups.
\subsection*{Lie groups of mappings}
Let $G$ be a Lie group
modelled on a locally convex space~$F$,
with neutral element~$e$.
Let us recall various constructions
of Lie groups of $G$-valued mappings.
\begin{numba}
The group $C(K,G)$ of $G$-valued
continuous mappings on a compact topological space~$K$
can be turned in a Lie group modelled on $C(K,F)$,
endowed with the topology of uniform convergence,
as is well known;
more generally,
\[
C_K(X,G):=\{\gamma\in C(X,G)\colon \gamma|_{X\setminus K}=e\}
\]
is a Lie group modelled on $C_K(X,F)$ for each locally
compact topological space~$X$ and compact subset $K\sub X$
(see \cite{GCX}).
If $\phi\colon U\to V\sub F$
is a chart of $G$ with $e\in U$ and $\phi(e)=0$,
then $C_K(X,U):=\{\gamma\in C_K(X,G)\colon \gamma(K)\sub U\}$
is open in $C_K(X,G)$ and the map
\[
C_K(X,\phi)\colon C_K(X,U)\to C_K(X,V)\sub C_K(X,F),\quad
\gamma\mto \phi\circ\gamma
\]
is a chart for $C_K(X,G)$ around its neutral element,
the constant function~$e$ (see~\cite{Rou}).\footnote{In \cite{GCX},
we get this for small~$U$,
which would be good enough for the following proofs.}
\end{numba}
\begin{numba}
For $G$ and $F$ as before and locally compact space~$X$
which is not compact, let
$X^*:=X\cup\{\infty\}$ be the one-point compactification
of~$X$ and
$C_0(X,G)$ be the group of continuous functions $\gamma\colon X\to G$
which vanish at infinity in the sense that, for each
$e$-neighbourhood $U\sub G$, there is a compact set
$K\sub X$ with
\[
\gamma(X\setminus K)\sub U.
\]
Then
\[
\gamma^*\colon X^*\to G,\quad
\left\{\begin{array}{cl}
\gamma(x) &\mbox{ if $x\in X$;}\\
e &\mbox{ if $x=\infty$}
\end{array}
\right.
\]
is a continuous function and the map
\[
C_0(X,G)\to C(X^*,G)_*,\quad \gamma\mto \gamma^*
\]
is an isomorphism of groups (whose inverse takes
$\eta\in C(X^*,G)_*$ to $\eta|_X$),
where
\[
C(X^*,G)_*:=\{\eta\in C(X^*,G)\colon \eta(\infty)=e\}.
\]
If $\phi\colon U\to V\sub F$ is a chart for~$G$ with $e\in U$ and
$\phi(e)=0$, then
\[
C(X^*,\phi)(C(X^*,G)_*)=C(X^*,F)_*\cap C(X^*,V),
\]
whence $C(X^*,G)_*$ is a submanifold of $C(X^*,G)$,
modelled on the complemented vector subspace
$C(X^*,F)_*$ of $C(X^*,F)$,
and thus a Lie group.
As a consequence,
$C_0(X,G)$ is a Lie group modelled on $C_0(X,F)$;
for each chart~$\phi$ as before,
the subset $C_0(X,G)$ of $U$-valued functions
is open in $C_0(X,G)$~and
\[
C_0(X,\phi)\colon C_0(X,U)\to C_0(X,V)\sub C_0(X,F),\quad
\gamma\mto\phi\circ\gamma
\]
is a chart for $C_0(X,G)$.
\end{numba}
\begin{numba}
Let $\ell\in\N_0\cup\{\infty\}$.
If $\ell\geq 1$, let $M$ be a locally compact,
rough $C^\ell$-manifold;
if $\ell=0$, let $M$ be a locally compact
topological space and write $C^0$ for continuous functions.
If $\gamma\colon M\to G$ is a $C^\ell$-function,
we let
\[
\Supp(\gamma):=\overline{\{x\in M\colon \gamma(x)\not=e\}}\sub M
\]
be its support.
We let
\[
C^\ell_c(M,G)
\]
be the group of all $C^\ell$-functions $\gamma\colon M\to G$ with
compact support.
For each compact subset
$K\sub M$,
\[
C_K^\ell(M,G):=\{\gamma\in C^\ell(M,G)\colon \Supp(\gamma)\sub K\}
\]
is a Lie group modelled on $C_K^\ell(M,F)$,
endowed with the compact-open $C^\ell$-topology
(see \cite{GCX} for the $C^0$-case and ordinary
manifolds, \cite{Rou} for the generalization to rough manifolds).
If $\phi\colon U\to V\sub F$
is a chart of $G$ with $e\in U$ and $\phi(e)=0$,
then $C_K^\ell(M,U):=\{\gamma\in C_K^\ell(M,G)\colon \gamma(K)\sub U\}$
is open in $C_K^\ell(M,G)$ and the map
\[
C_K^\ell(M,\phi)\colon C_K^\ell(M,U)\to C_K^\ell(M,V)\sub C_K^\ell(M,F),\quad
\gamma\mto \phi\circ\gamma
\]
is a chart for $C_K^\ell(M,G)$.
If~$M$ is paracompact, then
$
C^\ell_c(M,G)
$
is a Lie group modelled on the locally convex direct limit
\[
C^\ell_c(M,F):=\dl\, C^\ell_K(M,F);\vspace{-1mm}
\]
for each $\phi$ as above, $C^\ell_c(M,U):=\{\gamma\in C^\ell_c(M,U)\colon \gamma(M)\sub U\}$
is an open identity neighbourhood in $C^\ell_c(M,G)$
and the map
\[
C^\ell_c(M,\phi)\colon C^\ell_c(M,U)\to C^\ell_c(M,V)\sub C^\ell_c(M,F),\quad
\gamma\mto\phi\circ\gamma
\]
is a $C^\infty$-diffeomorphism
(see \cite{Rou}; for $\sigma$-compact
locally compact spaces or ordinary manifolds, cf.\
already~\cite{GCX};
cf.\ also \cite{DIS}
for special cases).
Each of the groups $C^\ell_K(M,G)$ is a closed
subset of $C^\ell_c(M,F)$ and a submanifold,
as $C^\ell_c(M,\phi)$ takes $C^\ell_c(M,U)\cap C^\ell_K(M,G)=C^\ell_K(M,U)$
onto $C^\ell_c(M,V)\cap C^\ell_K(M,F)=C^\ell_K(M,V)$.
\end{numba}
\begin{numba}
Let $F$ be a locally convex space and $X$ be a locally compact space.
Note that $C_K(X,F)\sub C_0(M,F)$
for each compact set $K\sub X$
and that $C_0(X,F)$ induces the given topology
on $C_K(X,F)$.
In fact, the seminorms
$\|\cdot\|_{q,\infty}$ given~by
\[
\|\gamma\|_{q,\infty}:=\sup_{x\in M}\, q(\gamma(x))
\]
define the locally convex vector topology on $C_0(X,F)$,
for $q$ ranging thorugh the set of continuous seminorms
on~$F$. For $\gamma\in C_K(X,F)$, we have $\|\gamma\|_{q,\infty}=\|\gamma\|_{K.q}$,
from which the assertion follows.\\[2.3mm]
If $G$ is a Lie group modelled on~$F$,
then $C_K(X,G)\sub C_0(X,G)$ apparently
and this is a closed subset (the point evaluations
$C_0(X,G)\to G$, $\gamma\mto\gamma(x)$ being continuous)
and a smooth submanifold as $C_0(X,\phi)$ takes
\[
C_0(X,U)\cap C_K(X,G)=C_K(X,U)
\]
onto $C_0(X,V)\cap C_K(X,F)=C_K(X,V)$ and restricts
to the chart $C_K(X,\phi)$ of $C_K(X,G)$, for each chart
$\phi\colon U\to V$ of $G$ such that $e\in U$ and $\phi(e)=0$.
\end{numba}
\begin{numba}
Let $F$ be a locally convex space,
$d\in\N$, $\ell\in\N_0\cup\{\infty\}$,
$\Omega\sub\R^d$ be an open subset
and $\cW$ be a set of continuous functions
$f\colon \Omega\to\R$ which contains the constant function~$1$.
We write
\[
C^\ell_\cW(\Omega,F)
\]
for the vector space of all $C^\ell$-functions $\gamma\colon \Omega\to F$
such that
\[
\|\gamma\|_{q,f,k}
:=\max_{j\in\{0,\ldots,k\}} \,\sup_{x\in\Omega}\,|f(x)|
\,\|\delta^f_x\gamma\|_q\,<\,\infty
\]
for all $k\in\N_0$ with $k\leq \ell$, all $f\in\cW$
and all continuous seminorms~$q$ on~$F$,
using notation as in~(\ref{hompolsemi}
(with $E=\R^d$ and a fixed norm~$\|\cdot\|$ on~$E$),
cf.\ \cite[Definition~3.4.1]{Wal}.
If $Q\sub\Omega$ is an open subset,
we write $C^\ell_\cW(Q,F)$ as a shorthand
for $C^\ell_\cV(Q,F)$
with $\cV:=\{f|_Q\colon f\in \cW\}$.
Moreover, we abbreviate
\[
\|\gamma\|_{q,f,k}:=\|\gamma\|_{q,f|_Q,k}
\]
for $\gamma\in C^k(Q,F)$ and $q,f,k$ as above.
We let
\[
C^\ell_\cW(\Omega,F)^\sbull\sub C^\ell_\cW(\Omega,F)
\]
be the vector subspace of all $\gamma\in C^\ell_\cW(\Omega,F)$
with the following property:
For all $f\in\cW$, $\ve>0$ and each continuous
seminorm~$q$ on~$F$, there exists a compact subset $K\sub\Omega$
such that
\[
\|\gamma|_{\Omega\setminus K}\|_{q,f,k}\,<\,\ve.
\]
For each open $0$-neighbourhood $V\sub F$,
\[
C^\ell_\cW(\Omega,V)^\sbull
:=\{\gamma\in C^\ell_\cW(\Omega,F)^\sbull\colon \gamma(\Omega)\sub V\}
\]
is an open $0$-neighborhood in $C^\ell_\cS(\Omega,F)^\sbull$
(see \cite[Lemma~3.4.19]{Wal}).
If $G$ is a Lie group modelled on~$F$,
consider the set
\[
C^\ell_\cW(\Omega,G)^\sbull_{\ex}
\]
of all $C^\ell$-functions $\gamma\colon \Omega\to G$
for which there exists a chart $\phi\colon U\to V\sub F$ of~$G$
with $e\in U$ and $\phi(e)=0$, a compact set $K\sub \Omega$
and a function $h\in C^\infty_c(\Omega,\R)$
which is constant $1$ on a neighbourhood~$Q$ of~$K$ in~$\Omega$,
such that
\begin{equation}\label{cuto}
(1-h)\cdot (\phi\circ\gamma)|_{\Omega\setminus K}
\in C^\ell_\cW(\Omega\setminus K,F)^\sbull
\end{equation}
(see \cite[Definition~6.2.6]{Wal}).
The choices of~$\phi$ and $h$ do not play a role:
If $\gamma\in C^\ell_\cW(\Omega,G)$ and
$\phi\colon U\to V$ is any chart of~$G$ with $e\in U$ and $\phi(e)=0$,
then there exists a compact subset $K\sub\Omega$
such that $\gamma(\Omega\setminus K)\sub U$
and (\ref{cuto}) holds for all $h\in C^\infty_c(\Omega,\R)$
which are constant~$1$ on a neighbourhood of~$K$
(see \cite[Lemma~6.2.8(b)]{Wal}).
By \cite[Theorem 6.2.17]{Wal},
$C^\ell_\cW(\Omega,G)^\sbull_{\ex}$ can be given a
unique Lie group structure modelled on $C^\ell_\cW(\Omega,F)$
such that, for each chart $\phi\colon U\to V\sub F$ of~$G$
with $e\in U$ and $\phi(e)=0$,
there exists an open $0$-neighbourhood $V_0\sub V$
such that
$
\phi^{-1}\circ C^\ell_\cW(\Omega,V_0)^\sbull
$
is open in $C^\ell_\cW(\Omega,G)^\sbull_{\ex}$
and
\begin{equation}\label{thecha}
\phi^{-1}\circ C^\ell_\cW(\Omega,V_0)^\sbull\to C^\ell_\cW(\Omega,V_0)^\sbull,\quad
\gamma\mto\phi\circ \gamma
\end{equation}
is a chart for $C^\ell_\cW(\Omega,G)^\sbull_{\ex}$ around the neutral element.
\end{numba}
\begin{numba}\label{ext-weighted}
In the situation of~(\ref{cuto}),
the extension
\[
g\colon \Omega\to F,\quad
x\mto\left\{\begin{array}{cl}
(1-h(x))\phi(x) & \mbox{ if $x\in \Omega\setminus K$;}\\
0 &\mbox{ if $x\in Q$}
\end{array}\right.
\]
is $C^\ell$ and an element of $C^\ell_\cW(\Omega,F)^\sbull$.
In fact,
for $f\in\cW$, $k\in\N_0$ with $k\leq \ell$,
a continuous seminorm~$q$ on~$F$ and $\ve>0$,
there exists a compact set $L\sub \Omega\setminus K$~with
\[
\|f|_{(\Omega\setminus K)\setminus L}\|_{q,f,k}<\ve.
\]
But $K\cup L$ is compact in~$\Omega$
and $\|g|_{\Omega\setminus (K\cup L)}\|_{q,f,k}
=\|f|_{(\Omega\setminus K)\setminus L}\|_{q,f,k}$.
\end{numba}
\begin{numba}
We mention that $C^\ell_K(\Omega,F)\sub C^\ell_\cW(\Omega,F)^\sbull$
for each compact subset $K\sub\Omega$,
and that the inclusion map
\[
j\colon C^\ell_K(\Omega,F)\to C^\ell_\cW(\Omega,F)^\sbull
\]
is a topological embedding.
In fact, given $f\in\cW$, $k\in\N_0$ with
$k\leq\ell$ and a continuous seminorm~$q$ on~$F$,
we have
\[
\|\gamma\|_{q,f,k}\leq C\|\gamma\|_{C^k\!,K,q}\, <\,\infty.
\]
for each $\gamma\in C^\ell_K(\Omega,F)$
with $C:=\sup\{|f(x)|\colon x\in K\}$,
whence $j$ is continuous.
Choosing $f:=1\in\cW$, we get
\[
\|\gamma\|_{C^k\!,K,q}=\|\gamma\|_{q,f,k},
\]
entailing that~$j$ is an embedding.
\end{numba}
\begin{numba}
We mention that $C^\ell_K(\Omega,G)\sub C^\ell_\cW(\Omega,G)^\sbull_{\ex}$
for each compact subset $K\sub\Omega$,
and that the inclusion map
\[
j\colon C^\ell_K(\Omega,F)\to C^\ell_\cW(\Omega,F)^\sbull
\]
is a $C^\infty$-diffeomorphism onto a submanifold
of $C^\ell_\cW(\Omega,G)^\sbull_{\ex}$
which is a closed subset.
In fact, given $\gamma\in C^\ell_K(G)$
and a chart $\phi\colon U\to V\sub F$ of~$G$
with $e\in U$ and $\phi(e)=0$,
we have $\gamma(\Omega\setminus K)\sub\{e\}\sub U$
and
\[
(1-h)\cdot (\phi\circ \gamma|_{\Omega\setminus K})=0\in C^\ell_\cW(\Omega\setminus KI,F)
\]
for each $h\in C^\infty_c(\Omega,\R)$ which is $1$ on a neighbourhood
of~$K$. Thus $\gamma\in C^\ell_\cW(\Omega, G)^\sbull_{\ex}$.
Moreover, for $V_0\sub V$ as in (\ref{thecha})
and $U_0:=\phi^{-1}(V_0)$,
we have
\[
C^\ell_\cW(\Omega,V_0)^\sbull
\cap C^\ell_K(\Omega,F)=C^\ell_K(\Omega,V_0)
\]
and applying $\phi^{-1}$, we obtain the subset
$C^\ell_K(\Omega,U_0)$ of the domain~$D$ of the chart~(\ref{thecha})
which is contained in $D\cap C^\ell_K(\Omega,G)$
and equals the letter as $D$ consists of $U_0$-valued
functions and so $D\cap C^\ell(\Omega,G)\sub C^\ell_K(\Omega,U_0)$.
Thus $C^\ell_K(\Omega,G)$ is a submanifold of $C^\ell_\cW(\Omega,G)$
and as the chart~(\ref{thecha}) restricts to the chart
$C^\ell_K(\Omega,\phi|_{U_0}^{V_0})$ of $C^\ell_K(\Omega,G)$,
the submanifold structure copincides with the given Lie group
structure. Closedness of $C^\ell_K(\Omega,G)$ in $C^\ell_\cW(\Omega,G)^\sbull_{\ex}$
follows from the continuity
of the point evaluations $C^\ell_\cW(\Omega,G)^\sbull_{\ex}\to G$,
$\gamma\mto\gamma(x)$.
\end{numba}
\begin{numba}
In the preceding situation, with $\Omega:=\R^d$,
let us call
$\cW$ a set of \emph{BCR-weights}
if it
has the following properties:
\begin{itemize}
\item[(a)]
$f\geq 1$ for all $f\in \cW$;
\item[(b)]
For all $f_1,f_2\in\cW$, there exists $f\in \cW$
such that $f_1(x)\leq f(x)$ and $f_2(x)\leq f(x)$
for all $x\in\R^d$;
\item[(c)]
For each $f_1\in\cW$ there exists $f_2\in\cW$
such that
\[
(\forall\ve>0)(\exists n\in\N)\;\,
\|x\|\geq n\;\mbox{or}\;f_1(x)\geq n\;\Rightarrow\;
f_1(x)\leq\ve f_2(x).
\]
\end{itemize}
For a set $\cW$ of BCR-weights,
\cite[Definition~6.2.23]{Wal}
provides a group $\cS(\R^d,G,\cW)$
of $\cW$-rapidly decreasing smooth
$G$-valued functions on~$\R^d$,
which coincides with
the corresponding group considered in~\cite{BCR}
in the situation of the latter (i.e.,
when $G$ is a so-called LE-Lie group),
see \cite[Remark~6.2.29]{Wal}.
Moreover, $\cS(\R^d,G,\cW)$
is a Lie group
and
$\cS(\R^d,G,\cW)=C^\infty_\cW(\R^d,G)^\sbull_{\ex}$
actually, by~\cite[Lemma~6.2.28]{Wal}.
\end{numba}
We mention that weights $f\in\cW$
are not assumed continuous in \cite{Wal}
and may take the value~$\infty$;
likewise, BCR-weights are defined
in a more general setting in \cite[Definition~6.2.18]{Wal}
(including infinite values).
But the above special cases are
natural for our ends.
Notably, they subsume the Lie groups
\[
\cS(\R^d,G):=\cS(\R^d,G,\cW)
\]
of rapidly decreasing smooth $G$-valued functions
obtained by the choice $\cW:=\{x\mto (1+\|x\|_2^2)^n\colon n\in\N_0\}$
where $\|\cdot\|_2$ is the euclidean norm on~$\R^d$.
\subsection*{Density of {\boldmath$C^\infty_c(M,G)$} in {\boldmath$C_0(M,G)$}}
Let us start with the case of vector-valued functions.
We are mostly interested in the following results
in the case $r=\infty$.
\begin{la}\label{vectinc}
Let $F$ be a locally convex space.
\begin{itemize}
\item[\rm(a)]
If $X$ is a locally compact
space, then $F\otimes C_c(X,\R)$ is dense in $C_0(X,F)$.
\item[\rm(b)]
If $r\in \N_0\cup\{\infty\}$
and $M$ is a locally compact rough $C^r$-manifold,
then $F\otimes C^r_c(M,\R)$ is dense in $C_0(M,F)$.
\item[\rm(c)]
In the situation of~{\rm(b)},
$F\otimes C^r_c(M,\R)$ is also dense in $C(M,F)$,
endowed with the compact-open topology.
\item[\rm(d)]
If $K\sub M$ is compact and $L$ a compact neighbourhood of~$K$
in~$M$ in the situation of~{\rm(b)},
then
\[
C_K(M,F)\;\sub\;
\wb{F\otimes C^r_L(M,\R)}
\]
holds for the closure in $C(M,F)$.
\item[\rm(e)]
In~{\rm(a)}, $F\otimes C_c(X,\R)$
is dense in $C(M,F)$ in the compact-open topology.
\end{itemize}
\end{la}
\begin{proof}
(a)
Let $q$ be a continuous seminorm on~$F$, $\ve>0$
and $\gamma\in C_0(X,F)$.
There is a compact subset $K\sub X$ such that $q(\gamma(x))\leq\ve$
for all $x\in X\setminus K$.
We let $L$ be a compact subset
of~$X$ such that $K\sub L^0$.
Each $x\in K$ has an open neighbourhood $U_x\sub L^0$
such that
\[
q(\gamma(y)-\gamma(x))\leq\ve\mbox{ for all $y\in U_x$.}
\]
We let $(h_x)_{x\in B}$, together with $g\colon L\to\R$,
be a continuous partition of unity on~$L$
such that $\Supp(h_x)\sub U_x$ for each $x\in K$
and $\Supp(g)\sub L\setminus K$.
Then $h_x|_{L^0}\in C_c(L^0,\R)$
for each $x\in K$; extending by~$0$,
we get a continuous function $\wt{h}_x\in C_c(X,\R)$.
Now
\[
\Phi:=\{x\in K\colon \Supp(h_x)\cap K\not=\emptyset\}
\]
is a finite set, whence
\[
\eta(y):=\sum_{x\in\Phi} \wt{h}_x(y)\gamma(x)
\]
for $y\in X$ defines a function $\eta=\sum_{x\in\Phi}\gamma(x)\otimes \wt{h}_x
\in F\otimes C_c(X,\R)$.
For $y\in K$, we have $\sum_{x\in\Phi}\wt{h}_x(y)=1$, whence
$\gamma(y)=\sum_{x\in \Phi}\wt{h}_x(y)\gamma(y)$ and thus
\[
q(\gamma(y)-\eta(y))\leq\sum_{x\in\Phi}h_x(y)q(\gamma(y)-\gamma(x))
\leq\ve,
\]
since $h_x(y)\not=0$ implies $y\in U_x$.
If $y\in X\setminus K$, then
\[
q(\gamma(y)-\eta(y))\leq q(\gamma(y))+q(\eta(y))\leq\ve+\sum_{x\in \Phi}
\underbrace{\wt{h}_x(y)q(\gamma(x))}_{\leq 2\ve \wt{h}_x(y)}
\leq 3\ve
\]
since $\wt{h}_x(y)\not=0$ implies $y\in U_x$
and thus $q(\gamma(x))\leq q(\gamma(y))+q(\gamma(y)-\gamma(x))\leq2\ve$.
By the preceeding, $\|\gamma-\eta\|_{q,\infty}\leq 3\ve$.\\[2.3mm]
(b) We can copy the proof of (a) with $X:=M$,
with the following changes: After replacing $L$ with $\wb{L^0}$,
we may assume that $L$ is a regular subset of~$M$ and hence
a full-dimensional submanifold.
Moreover, we choose the $h_x$ and $g$ in $C^r(L,\R)$,
i.e., we choose a $C^r$-partition of unity.\\[2.3mm]
(c) follows from (b),
using that the linear inclusion map $C_0(M,F)\to C(M,F)$
is continuous (as $\|\gamma\|_{K,q}\leq\|\gamma\|_{q,\infty}$
for each compact set $K\sub M$, continuous seminorm~$q$ on~$F$
and $\gamma\in C_0(M,F)$)
and has dense image.
In fact, for $K$ and $q$ as before and $\gamma\in C(M,F)$,
we find a compact subset $L\sub M$ with $K\sub L^0$
and a continuous function $h\colon M\to[0,1]$
with $h|_K=1$ and $\Supp(h)\sub L^0$.
Then $h\gamma\in C_L(M,F)\sub C_0(M,F)$
and $\|h\gamma-\gamma\|_{K,q}=0$.\\[2.3mm]
(d) Let $\gamma\in C_K(M,F)$.
By~(c), there exists a net $(\gamma_a)_{a\in A}$
in $F\otimes C^r(M,\R)$ such that $\gamma_a\to\gamma$ in $C(M,F)$.
Let $h\colon M\to \R$ be a$C^r$-function
such that $h|_K=1$ and $\Supp(h)\sub L$.
Then $h\gamma_a\in F\otimes C^r_L(M,\R)$
for all $a\in A$
and $h\gamma_a\to h\gamma=\gamma$ in $C(M,F)$,
using that the multiplication operator
$C(M,F)\to C(M,F)$, $\eta\mto h\eta$ is continuous
(see, e.g., \cite[Lemma~A.5.24(f)]{GaN}).\\[2.3mm]
(e) Arguing as in~(c), we see that (e) follows from~(a).
\end{proof}
We now pass to the group case.
The following lemma can be re-used later when we discuss
density of $C^\infty_c(M,G)$ in $C^\ell_c(M,G)$.
\begin{la}\label{conviacov}
Let $G$ be a Lie group modelled on a locally convex space~$F$
and $\ell\in\N_0\cup\{\infty\}$.
If $r>0$, let
$M$ be a locally compact, rough $C^\ell$-manifold;
if $r=0$, let $M$ be a locally compact topological space.
Let $K\sub M$ be a compact subset,
$\gamma\in C^\ell_K(M,G)$
and $(\gamma_a)_{a\in A}$
be a net in $C^\ell_K(M,G)$.
If each $x\in M$ has an
open neighbourhood $S_x\sub M$ with compact closure~$K_x\sub M$
such that $\gamma_a|_{K_x}\to\gamma|_{K_x}$
in $C^\ell(K_x,G)$, then $\gamma_a\to\gamma$
in $C^\ell_K(M,G)$.
\end{la}
\begin{proof}
There is a finite subset $\Phi\sub K$
such that $K\sub\bigcup_{x\in\Phi}S_x$.
Then
\[
C^\ell(M,F)\to C^\ell(M\setminus K,F)\times \prod_{x\in\Phi}C^\ell(K_x,F),
\;
\gamma\mto (\gamma|_{M\setminus K},(\gamma|_{K_x})_{x\in\Phi})
\]
is a topological embedding (see \cite{Rou}),
entailing that
\[
\rho\colon C^\ell_K(M,F)\to\prod_{x\in\Phi}C^\ell(K_x,F),\;\,
\gamma\mto (\gamma|_{K_x})_{x\in \Phi}
\]
is a topological embedding.
Let $\phi\colon U\to V\sub F$ be a chart of~$G$
such that $e\in U$ and $\phi(e)=0$.
For each $x\in \Phi$, we find $a_x\in A$ such that
$(\gamma^{-1}\gamma_a)|_{K_x}\in C^\ell(K_x,U)$
for all $a\geq a_x$ in~$A$. Since $(A,\leq)$
is directed, there exists $b\in A$ such that $b\geq a_x$ for all $x\in\Phi$.
For all $a\in A$ with $a\geq b$, we now have
$(\gamma^{-1}\gamma_a)(K_x)\sub U$
for all $x\in\Phi$ and thus $\gamma^{-1}\gamma_a\in C^\ell_K(M,U)$.
Since
\begin{eqnarray*}
(\rho\circ C^\ell_K(M,\phi))(\gamma^{-1}\gamma_a)
& =& (C^\ell(K_x,\phi)((\gamma^{-1}\gamma_a)|_{K_x}))_{x\in\Phi}\\
&\to & (C^\ell(K_x,\phi)(e))_{x\in\Phi}=(0)_{x\in\Phi}=
(\rho\circ C^\ell_K(M,\phi))(e),
\end{eqnarray*}
we deduce that $\gamma^{-1}\gamma_a\to e$ in $C^\ell_K(M,G)$.
\end{proof}
\begin{prop}\label{ctscaseprop}
Let $G$ be a Lie group modelled
on a locally convex space~$F$,
and $r\in \N_0\cup\{\infty\}$.
If $r>0$, let $M$ be a locally compact, rough $C^r$-manifold.
If $r=0$, let $M$ be a locally compact topological space.
Then we have:
\begin{itemize}
\item[\rm(a)]
If $K$ and $L$ are compact subsets of $M$
such that $K\sub L^0$, then
\[
\wb{C^r_L(M,G)}\, \supseteq \, C_K(M,G)
\]
holds for the closure in $C_0(M,G)$,
which equals the closure in $C_L(M,G)$.
\item[\rm(b)]
$C^r_c(M,G)$ is dense in $C_0(M,G)$.
\end{itemize}
\end{prop}
\begin{proof}
As the point evaluations $C_0(M,G)\to G$, $\gamma\mto\gamma(x)$
are continuous, $C_L(M,G)$ is closed in
$C_0(M,G)$, whence the closures in~(a)
coincide.\\[2.3mm]
(a) Excluding a trivial case, we may assume that $r\geq 1$.
To enable re-using the proof, set $\ell:=0$.
Let $\gamma\in C_K^\ell(M,G)$
and $W$ be an open neighbourhood
of $\gamma$ in $C_L^\ell(M,G)$.
For each $x\in K$, there exists a chart
$\phi_x\colon U_x\to V_x$ of~$G$ such that $V_x$ is convex
and $\gamma(x)\in U_x$.
Let $Q_x$ be a relatively compact, open $x$-neighbourhood in~$L^0$
such that $\gamma(\wb{Q_x})\sub U_x$.
Let~$P_x$ be an
open $x$-neighbourhood
with compact closure $\overline{P_x}\sub Q_x$.
Let $h_x\colon Q_x\to\R$ be a $C^r$-function
with compact support $K_x:=\Supp(h_x)\sub Q_x$
such that $h_x(Q_x)\sub [0,1]$ and
$h_x|_{\overline{P_x}}=1$.
By compactness of~$K$,
we find a finite subset $\Phi\sub K$ such that $K\sub \bigcup_{x\in \Phi} P_x$.
Let $x_1,\ldots, x_m$ be the elements of~$\Phi$.
Let $\gamma_0:=\gamma$.
For $j\in\{1,\ldots,m\}$, we now construct
functions\footnote{We write $\lfloor B,U\rfloor=\{\eta\in C(M,G)\colon
\eta(B)\sub U\}$
if $B\sub M$ is compact and $U$
and open subset of~$G$, as in~\ref{theco}.}
\begin{equation}\label{inintersecvar}
\gamma_j\in W\cap\bigcap_{y\in \Phi}\lfloor \wb{Q_y},U_y\rfloor
\end{equation}
with the following properties:
\begin{itemize}
\item[(i)]
$\gamma_j|_{P_{x_j}}$ is $C^r$;
\item[(ii)]
If $U\sub M$ is open and $\gamma_{j-1}|_U$
is $C^r$, then also $\gamma_j|_U$ is~$C^r$.
\end{itemize}
Then $\gamma_m$ is~$C^r$ on $M\setminus K$
(like $\gamma_0=\gamma$)
and on $P_{x_1}\cup\cdots\cup P_{x_m}$,
and hence on all of~$M$. Thus $\gamma_m\in C^r_L(M,G)$.
Since $\gamma_m\in W$, we deduce that~$\gamma$ is in the closure
of $C^r_L(M,G)$.\\[2.3mm]
If $j\in\{1,\ldots,m\}$ and $\gamma_{j-1}$
has already been constructed, abbreviate $x:=x_j$.
Lemma~\ref{vectinc}(c)
provides a net $(\eta_{j,a})_{a\in A_j}$
in $F\otimes C^r(\wb{Q_x},\R)\sub C^r(Q_x,F)$~with
\[
\eta_{j,a}\to \phi_x\circ \gamma_{j-1}|_{\wb{Q_x}}
\]
in $C^\ell(\wb{Q_x},F)$.
Then
\[
\zeta_{j,a}\,:=\, (1-h_x)\cdot (\phi_x\circ\gamma_{j-1}|_{\wb{Q_x}})
+ h_x\cdot \eta_{j,a}
\]
is in $C^\ell_{K_x}(\wb{Q_x},F)$ and $C^r$ on each open subset
of~$\wb{Q_x}$ on which $\gamma_{j-1}$ is~$C^r$.
Moreover,
\[
\zeta_{j,a}\to \phi_x\circ \gamma_{j-1}|_{\wb{Q_x}}
\]
in $C^\ell(\wb{Q_x},F)$ as multiplication
operators are continuous linear.
We therefore find $a_j\in A_j$ such that
\[
\zeta_{j,a}(\wb{Q_x})\sub V_x\mbox{ for all $a\geq a_j$.}
\]
For $a\geq a_j$, we
define functions $\gamma_{j,a}\colon M\to G$ via
\[
\gamma_{j,a}(z):=\left\{
\begin{array}{cl}
\phi_x^{-1}(\zeta_{j,a}(z)) &\mbox{ if $z\in Q_x^0$;}\\
\gamma_{j-1}(z) &\mbox{ if $z\in M\setminus K_x$.}
\end{array}\right.
\]
Since $\zeta_{j,a}|_{Q_x^0\setminus K_x}=\phi_x\circ\gamma_{j-1}|_{Q_x^0\setminus K_x}$,
the function $\gamma_{j,a}$ is well defined and~$C^\ell$.
It is $C^r$ on each open set on which $\gamma_{j-1}$ is~$C^r$.
Moreover, $\Supp(\gamma_{j,a})\sub \Supp(\gamma_{j-1})\cup K_x
\sub L$.
Note that $\gamma_{j,a}\to\gamma_{j-1}$ in $C^\ell_L(M,G)$
as a consequence of Lemma~\ref{conviacov}
(which we apply with $S_y:=W_x$ for
each $y\in W_x$, and with open neighbourhoods $S_y\sub M\setminus K_x$
around points $y\in L\setminus W_x$). 
%
We therefore find a $b_j\geq a_j$ such that $\gamma_j:=\gamma_{j,b_j}$
satisfies~(\ref{inintersecvar}).
This completes the recursive construction and completes
the proof of~(a).\\[2.3mm]
(b) Let $\gamma\in C_0(M,G)$ and $W\sub C_0(M,G)$ be
a neighbourhood of~$\gamma$. Let $\phi\colon U\to V$
be a chart of~$G$ such that $e\in U$,
$\phi(e)=0$, and $V$ is convex.
There exists a compact subset $K\sub M$
such that $\gamma(M\setminus K)\sub U$
and a compact subset $L\sub M$ such that $K\sub L^0$.
Let $h\in C(M,[0,1])$ such that $h|_Q=1$ for an open neighbourhood
$Q$ of~$K$ in $L^0$, and $\Supp(h)\sub L^0$.
Then
\[
\zeta\colon M\to F,\quad x\mto
\left\{\begin{array}{cl}
(1-h(x))\phi(\gamma(x)) &\mbox{ if $x\in M\setminus K$;}\\
0 &\mbox{ if $x\in Q$}
\end{array}\right.
\]
is a continuous function; in fact, $\zeta\in C_0(M,V)$.
Thus $\gamma_1:=\phi^{-1}\circ\zeta\in C_0(M,G)$;
let $\gamma_2:=\gamma^{-1}\gamma$.
There are a $\gamma_1$-neighbourhood $W_1\sub C_0(M,G)$
and a $\gamma_2$-neighbourhood $W_2\sub C_0(M,G)$
such that $W_1 W_2\sub W$.
Lemma~\ref{vectinc}(b)
yields a net $(\zeta_a)_{a\in A}$ in $C^r_c(M,F)$
which converges to $\zeta$ in $C_0(M,F)$;
we may assume that $\zeta_a\in C_0(M,V)$ for all~$a$.
Then $\phi^{-1}\circ \zeta_a\to\phi^{-1}\circ\zeta=\gamma_1$
in $C_0(M,G)$. We therefore find $a$ such that
$\eta_1:=\phi\circ \zeta_a\in W_1$.
Then $\eta_1\in C^r_c(M,G)$.
By~(a), we find $\eta_2\in W_2\cap C^r_L(M,G)$.
Then $\eta_1\eta_2\in W_1W_2\sub W$
and $\eta_1\eta_2\in C^r_c(M,G)$.
\end{proof}
\subsection*{Density of {\boldmath$C^\infty_c(M,G)$} in {\boldmath$C^\ell_c(M,G)$}}
We now discuss density in Lie groups
of compactly supported $C^\ell$-functions.
\begin{prop}\label{densetf}
Let $r,\ell\in \N_0\cup\{\infty\}$ with $\ell\leq r$
and $G$ be a Lie group modelled on a locally convex space~$F$.
If $\ell=0$, let $M$ be a locally compact,
rough $C^r$-manifold;
if $\ell>0$, let
$M$ be a locally compact $C^r$-manifold with rough boundary.
Then the following holds:
\begin{itemize}
\item[\rm(a)]
If $K$ and $L$ are compact subsets of~$M$ such
that $K\sub L^0$,
then
\[
C^\ell_K(M,G)\, \sub \, \overline{C^r_L(M,G)}
\]
holds for the closure in $C^\ell_L(M,G)$.
\item[\rm(b)]
If $M$ is paracompact, then $C^r_c(M,G)$ is dense
in $C^\ell_c(M,G)$.
\end{itemize}
\end{prop}
\begin{proof}
(a) For the case $\ell=0$, see Proposition~\ref{ctscaseprop}(a).
Now assume that $\ell\geq 1$.
Let $\gamma\in C^\ell_K(M,G)$
and $W$ be an open neighbourhood
of $\gamma$ in $C^\ell_L(M,G)$.
We can repeat the proof
of Proposition~\ref{ctscaseprop}(a),
except that the sets~$Q_x$ have to be chosen more
carefully.
Having chosen~$U_x$ as in the cited proof,
let $\kappa_x\colon B_x\to D_x$ be a chart for~$M$
around~$x$ such that $\gamma(B_x)\sub U_x$
and~$D_x$ is a locally convex subset with dense interior
of~$\R^d$
for some $d\in\N_0$.
Now $\kappa_x(x)$ has a compact neighbourhood~$C_x$ in~$D_x$,
and~$C_x$ contains a convex $\kappa_x(x)$-neighbourhood
$E_x$. For some convex, open $\kappa_x(x)$-neighbourhood
$A_x$ in~$\R^d$, we have $A_x\cap D_x\sub E_x$.
Then $A_x\cap D_x=A_x\cap E_x$ is convex, open in~$D_x$,
and relatively compact in~$D_x$, whence $Q_x:=\kappa_x^{-1}(A_x\cap D_x)$
is an open $x$-neighbourhood with compact closure $\wb{Q_x}\sub B_x$.
Moreover, $Q_x$ is $C^r$-diffeomorphic to the convex subset $A_x\cap D_x\sub\R^d$
with dense interior.
We can now repeat the proof
of Proposition~\ref{ctscaseprop}(a),
with Lemma~\ref{densy} in place of Lemma~\ref{vectinc}(c).\\[2.3mm]
(b) is immediate from~(a).
\end{proof}
\subsection*{Density of {\boldmath$C^\infty_c(\Omega,G)$}
in {\boldmath$C^\ell_\cW(\Omega,G)^\sbull_{\ex}$}}
To obtain a density result, we need to impose
conditions on the set $\cW$ of weights.
\begin{numba}\label{situat}
In this section,
we endow $\R^d$ with some norm~$\|\cdot\|$,
and
let $\Omega \sub \R^d$ be an open subset.
We lat $\cW$ be a set of smooth functions
$f\colon \Omega \to \R$
such that
the constant function~$1$
belongs to~$\cW$ and the following conditions
are satisfied:
\begin{itemize}
\item[(a)]
$f(x)\geq 0$
for all $f\in \cW$
and $x\in \Omega$;
\item[(b)]
For each $x\in \Omega$,
there exists
$f\in \cW$ such that $f(x)>0$;
\item[(c)]
For all $N\in \N$,
$f_1,\ldots, f_N\in \cW$
and $k_1,\ldots, k_N\in \N_0$
with \mbox{$k_1,\ldots, k_N\leq r$,}
there exist $C>0$
and $f\in \cW$ such that
\[
\|\delta^{k_1}_xf_1\|_q \cdot\ldots\cdot \|\delta^{k_N}_xf_N\|_q \;\leq\;
C\, f(x)\quad\mbox{for all $x\in \Omega$,}
\]
with notation as in~(\ref{hompolsemi})
and using the norm $q\colon \R\to\R$, $y\mto|y|$.
\end{itemize}
\end{numba}
The conditions (a)-(c) imposed on $\cW$
imply a crucial property:
\begin{la}\label{isdnss}
If $\cW$ is a set of weights as in~{\rm\ref{situat}},
then
$C^\infty_c(\Omega,F)$ is dense
in $C^r_\cW(\Omega,F)^\sbull$
for each $r\in\N_0\cup\{\infty\}$
and locally convex space~$F$.
In fact, $F\otimes C^\infty_c(\Omega,\R)$
is dense.
\end{la}
\begin{proof}
If $F$ is finite-dimensional,
the assertion is immediate
from the scalar-valued
case treated in \cite[V.7\,a), p.\,224]{GWS}.
For the general case,
one first replaces $F$ with a completion $\wt{F}$
and reworks the proof of
\cite[V.7\,a), p.\,224]{GWS},
with minor modifications.\footnote{The completeness of $\wt{F}$
ensures that the relevant vector-valued (weak) integrals
exist. As one continuous
seminorm~$q$ on~$F$ suffices to describe a typical
neighbourhood of a given function in $C^r_\cW(\Omega,F)^\sbull$,
the proof goes through
if we replace the absolute value~$|\cdot|$~by~$q$.}
Then, in the last line of \cite[p.\,226]{GWS},
one replaces $(T_{m_1,m_2}f)(x_i^{(m_4)})\in \wt{F}$
by a nearby element in~$F$.
\end{proof}
\begin{prop}\label{densewt}
Given $d\in\N$ and $\ell\in\N_0\cup\{\infty\}$,
let $G$ be a Lie group modelled on a locally convex space~$F$;
let $\Omega$ be an open subset of $\R^d$ and $\cW\sub C^\infty(\Omega,\R)$
be a set of weights as in {\rm\ref{situat}.}
Then $C^\infty_c(\Omega,G)$ is dense
in the Lie group $C^\ell_\cW(\Omega,G)^\sbull_{\ex}$.
If $\Omega=\R^d$, $\ell=\infty$
and $\cW$ is, moreover,
a set of BCR-weights, then
$C^\infty_c(\R^d,G)$ is dense
in $\cS(\R^d,G,\cW)$.
Notably, $C^\infty_c(\R^d,G)$ is dense
in $\cS(\R^d,G)$.
\end{prop}
\begin{proof}
Let $\gamma\in C^\ell_\cW(\Omega,G)^\sbull_{\ex}$
and $W\sub C^\ell_\cW(\Omega,G)^\sbull_{\ex}$
be a $\gamma$-neighbourhood.
Let $\phi\colon U\to V$
be a chart of~$G$ such that $e\in U$
and $\phi(e)=0$.
Let $V_0\sub V$ be an open $0$-neighbourhood
such that $\phi^{-1}\circ C^\ell_{\cW}(\Omega,V_0)$
is an open identity neighbourhood in
$C^\ell_\cW(\Omega,G)^\sbull_{\ex}$
with $U_0:=\phi^{-1}(V_0)$;
after shrinking~$V_0$, we may assume that
$V_0$ is convex.
There exists a compact subset $K\sub \Omega$
such that $\gamma(\Omega\setminus K)\sub U_0$
and a $C^\infty_c$-function $h\colon \Omega\to\R$
with $h(\Omega)\sub[0,1]$
such that $h|_Q=1$ for some open neighbourhood~$Q$
of~$K$ in $\Omega$
and $(1-h)(\phi\circ\gamma)|_{\Omega\setminus K}\in C^\ell_\cW(\Omega\setminus K,F)$.
Then
\[
\zeta\colon \Omega\to F,\quad x\mto
\left\{\begin{array}{cl}
(1-h(x))\phi(\gamma(x)) &\mbox{ if $x\in \Omega\setminus K$;}\\
0 &\mbox{ if $x\in Q$}
\end{array}\right.
\]
is a $C^\ell$-function and in fact
$\zeta\in C^\ell_\cW(\Omega,V_0)^\sbull$,
by~\ref{ext-weighted}.
Thus $\gamma_1:=\phi^{-1}\circ\zeta\in C^\ell_\cW(\Omega,G)^\sbull_{\ex}$;
let $\gamma_2:=\gamma^{-1}_1\gamma$.
There are a $\gamma_1$-neighbourhood $W_1\sub C_\cW^\ell(\Omega,G)^\sbull_{\ex}$
and a $\gamma_2$-neighbourhood $W_2\sub C^\ell_\cW(\Omega,G)^\sbull_{\ex}$
such that $W_1 W_2\sub W$.
Lemma~\ref{isdnss}
yields a net $(\zeta_a)_{a\in A}$ in $C^\infty_c(\Omega,F)$
which converges to $\zeta$ in $C^\ell_\cW(\Omega,F)^\sbull$;
we may assume that $\zeta_a\in C^\infty_\cW(\Omega,V_0)$ for all~$a$.
Then $\phi^{-1}\circ \zeta_a\to\phi^{-1}\circ\zeta=\gamma_1$
in $C^\ell_\cW(\Omega,G)^\sbull_{\ex}$.
We therefore find $a$ such that
$\eta_1:=\phi\circ \zeta_a\in W_1$.
Then $\eta_1\in C^\infty_c(\Omega,G)$.
By Proposition~\ref{densetf}(a),
we find $\eta_2\in W_2\cap C^\infty_L(\Omega,G)$.
Then $\eta_1\eta_2\in W_1W_2\sub W$
and $\eta_1\eta_2\in C^\infty_c(M,G)$.
\end{proof}
\begin{rem}\label{denserapid}
Let $G$ be a Lie group modelled on a locally convex space~$F$.
If $\Omega=\R^d$ and $\cW\sub C^\infty(\R^d,\R)$
is a set of BCR-weights satisfying the conditions
(a)--(c) formulated in~\ref{situat},
then $C_\cW^\infty(\R^d,G)^\sbull_{\ex}=\cS(\R^d,G,\cW)$
as already mentioned, and thus $C^\infty_c(\R^d,G)$
is dense in $\cS(\R^d,G,\cW)$,
by Proposition~\ref{densewt}.
Note that $\cW:=\{x\mto (1+\|x\|_2^2)^n\colon n\in\N\}$
is a set of BCR-weights and satisfies the conditions~(a)--(c);
thus $C^\infty_c(\R^d,G)$ is dense in $\cS(\R^d,G)$.
\end{rem}
\section{More on density and the function spaces}\label{more-densy}
In this section,
we prove the following result concerning
tensor product realizations
for spaces of vector-valued functions
(which were previously known in special cases).
Moreover, we prove nuclearity for $C^\infty(M,F)$
under natural hypotheses (Proposition~\ref{nucity}).\\[2.3mm]
Recall that a Hausdorff topological space~$X$ is said to be a
\emph{$k_\R$-space} if for all
functions $f\colon X\to\R$,
continuity of~$f$ is equivalent to
continuity of~$f|_K$ for all compact subsets $K\sub X$.
Every locally compact space and every metrizable topological space
is a $k$-space\footnote{Subsets $A\sub X$ are closed
if and only if $A\cap K$ is closed for each compact set $K\sub X$.}
and hence a $k_\R$-space.
\begin{prop}\label{tensorcases}
Let $F$ be a complete locally convex space.
Then we have:
\begin{itemize}
\item[\rm(a)]
Using the compact-open topology,
\[
C(X,F)\cong F\tensor C(X,\R)
\]
for each topological space~$X$ which is a $k_\R$-space
and such that each compact subset $K\sub X$
has a closed, paracompact neighbourhood in~$X$;
for instance, $X$ may be any paracompact $k_\R$-space
$($e.g., any metrizable topological space$)$,
or any locally compact topological space.
\item[\rm(b)]
Using the compact-open $C^\ell$-topology,
\[
C^\ell(M,F)\cong F\tensor C^\ell(M,\R)
\]
holds
for each $\ell\in \N_0\cup\{\infty\}$ and
each locally compact $C^\ell$-manifold~$M$ with rough boundary.
\item[\rm(c)]
$
C^\ell_\cW(\Omega,F)^\sbull\cong F\tensor C^\ell_\cW(\Omega,\R)^\sbull
$
for each $\ell\in\N_0\cup\{\infty\}$,
each open subset $\Omega\sub\R^d$ and each set $\cW\sub C^\infty(\Omega,\R)$
of weights satisfying the conditions {\rm(a)--(c)}
in {\rm\ref{situat}}. Notably,
$\cS(\R^d,F)\cong F\tensor \cS(\R^d,\R)$.
\end{itemize}
\end{prop}
The proofs of the following three lemmas,
which are mostly variants
of results from earlier sections,
have been relegated to the appendix
(Appendix~\ref{appB}).
\begin{la}\label{in-parac}
Let $F$ be a locally convex space and
$X$ be a topological space such that
every compact subset of~$X$ has a closed, paracompact
neighbourhood in~$X$ $($e.g., $X$ might be any
paracompact topological space, or any locally compact
topological space$)$.
Then $F\otimes C(X,\R)$ is dense in $C(X,F)$,
endowed with the compact-open topology.
\end{la}
\begin{la}\label{roughcpsupp}
Let $F$ be a locally convex space,
$r,\ell\in\N_0\cup\{\infty\}$
with $\ell\leq r$, and~$M$ be a locally compact
$C^r$-manifold with rough boundary.
Then we have:
\begin{itemize}
\item[\rm(a)]
If $K\sub M$ is compact and $L$
a compact neighbourhood of~$K$ in~$M$, then
\[
C^\ell_K(M,F)\sub\overline{F\otimes C^r_L(M,\R)}
\]
holds for the closure in $C^r(M,F)$.
\item[\rm(b)]
$F\otimes C^r_c(M,\R)$ is dense in $C^\ell(M,F)$
in the compact-open $C^\ell$-topology.
\end{itemize}
\end{la}
If $F$ is a locally convex space and $q$ a continuous seminorm on~$E$,
then $\|x+q^{-1}(\{0\})\|_q:=q(x)$ defines
a norm on $F_q:=F/q^{-1}(\{0\})$.
We let $\wt{F}_q$ be a completion of $F_q$ with $F_q\sub \wt{F}_q$
and write $\|\cdot\|_q$ also for the norm on $\wt{F}_q$.
Let
\[
\alpha_q\colon F\to F_q,\quad x\mto x+q^{-1}(\{0\})
\]
be the canonical map and $\wt{\alpha}_q:=j_q\circ\alpha_q$,
where $j_q\colon F_q\to\wt{F}_q$ is the inclusion map.
For all continuous seminorms $p,q$ on~$F$ with
$p\leq q$ (in the sense that $p(x)\leq q(x)$
for all $x\in F$), the linear map
\[
\alpha_{p,q}\colon F_q\to F_p,\quad x+q^{-1}(\{0\})\to x+p^{-1}(\{0\})
\]
is continuous and has a unique continuous linear extension
$\wt{\alpha}_{p,q}\colon \wt{F}_q\to\wt{F}_p$.
Let $\Sem(F)$ be the set of continuous seminorms on~$F$.
We use the fact:
If $F$ is complete, then $(F,(\wt{\alpha}_q)_{q\in\Sem(F)})$
is a projective limit of the projective system
\[
((\wt{F}_q)_{q\in\Sem(F)}\colon (\wt{\alpha}_{p,q})_{p\leq q})
\]
in the category of locally convex spaces and continuous
linear mappings.
\begin{la}\label{weicompl}
Let $\ell\in\N_0\cup\{\infty\}$, $d\in\N$,
and $\Omega\sub\R^d$ be an open subset.
Let $F$ be a locally convex space.
\begin{itemize}
\item[\rm(a)]
Let $F$, together with continuous linear maps $\lambda_j\colon F\to F_j$,
be a projective limit of a projective system $\cS:=((F_j)_{j\in J},
(\lambda_{i,j})_{i\leq j})$ of locally convex spaces~$F_i$
and continuous linear maps $\lambda_{i,j}\colon F_j\to F_i$.
Then $C^\ell(\Omega,F)$, together
with the continuous linear maps
$
C^\ell(\Omega,\lambda_j)\colon C^\ell(\Omega,F)\to C^\ell(\Omega,F_j)$,
$\gamma\mto\lambda_j\circ\gamma$,
is a projective limit of the projective system
\begin{equation}\label{thefsys}
((C^\ell(\Omega,F_j))_{j\in J},(C^\ell(\Omega,\lambda_{i,j}))_{i\leq j})
\end{equation}
in the category of locally convex spaces and continuous linear mappings.
\item[\rm(b)]
Let $\cW\sub C(\Omega,\R)$ be a set of continuous weights
such that, for each $x\in\Omega$, there exists $f\in\cW$
with $f(x)\not=0$. If $F$ is complete, then
$C^\ell_\cW(\Omega,F)$, together with the continuous linear maps
$C^\ell_\cW(\Omega,\wt{\alpha}_q)^\sbull$ for $q\in\Sem(F)$,
is a projective limit of the projective system
\[
((C^\ell_\cW(\Omega,\wt{F}_q)^\sbull)_{q\in\Sem(F)},
(C^\ell_\cW(\Omega,\wt{\alpha}_{p,q})^\sbull)_{p\leq q})
\]
in the category of locally convex spaces and continuous linear mappings.
Moreover, $C^\ell_\cW(\Omega,F)$ is complete.
\end{itemize}
\end{la}
{\bf Proof of Proposition~\ref{tensorcases}.}
We can argue as in the proof of Proposition~\ref{etensor},
with the following modifications:\\[2.3mm]
(a) As $X$ is a $k_\R$-space,
$C(X,F)$ is complete
(see, e.g., \cite[Lemma~A.5.24(d)]{GaN}).
By Lemma~\ref{in-parac},
$F\otimes C(X,\R)$ is dense in $C(X,F)$.
Now use the seminorms $\|\cdot\|_{K,q}$
with continuous seminorms~$q$ on~$F$
and compact $K\sub X$.\\[2.3mm]
(b) $C^\ell(M,F)$ is complete (see \cite[Proposition~3.6.20]{GaN})
%
%
and $F\otimes C^\ell(M,\R)$ is dense in $C^\ell(M,F)$,
by Lemma~\ref{roughcpsupp}.
Now use the seminorms $\gamma\mto \|\gamma\circ\phi^{-1}\|_{C^k,K,q}$
with continuous seminorms~$q$ on~$F$, $k\in\N_0$ with $k\leq\ell$,
charts $\phi\colon U\to V$ for~$M$ and compact subsets $K\sub V$.\\[2.3mm]
(c) $C^\ell_\cW(\Omega,F)^\sbull$ is complete (see Lemma~\ref{weicompl}(b))
and $F\otimes C^\ell_\cW(\Omega,\R)^\sbull$ is dense in it by Lemma~\ref{isdnss}.
Now argue as above using the seminorms $\|\cdot\|_{q,f,k}$. $\,\square$\\[2.3mm]
We record a property of spaces of smooth scalar-valued
functions.
\begin{prop}\label{nucity}
Let $M$ be a locally compact $C^\infty$-manifold
with rough boundary $($e.g.,
any finite dimensional $C^\infty$-manifold with corners$)$.
Then we have:
\begin{itemize}
\item[\rm(a)]
$C^\infty(M,\R)$, endowed with the
compact-open $C^\infty$-topology,
is complete and a nuclear locally convex space.
Moreover, $C^\infty(M,F)$
is nuclear for each nuclear locally convex space~$F$.
\item[\rm(b)]
If~$M$ is, moreover, $\sigma$-compact,
then $C^\infty(M,\R)$ is a nuclear Fr\'{e}chet
space and so is $C^\infty(M,F)$
for each nuclear Fr\'{e}chet space~$F$.
\end{itemize}
\end{prop}
\begin{proof}
(a) For each $x\in M$,
there exists a chart $\phi_x\colon U_x\to V_x$
for~$M$ such that $x\in U_x$,
$\phi_x(x)=0$ and $V_x$
is a locally convex subset with dense interior
of a finite-dimensional vector space~$E_x$.
Let $K_x$ be a compact $0$-neighbourhood in~$V_x$;
there exists a closed, convex $0$-neighbourhood $B_x$ in~$V_x$
such that $B_x\sub K_x$.
Then~$B_x$ is compact and hence closed in~$E_x$.
As a consequence of Corollary~\ref{convex-ext},
$C^\infty(B_x,\R)$
is isomorphic to a complemented
vector subspace of $C^\infty(\R^d,\R)$.
Since $C^\infty(\R^d,\R)$
is a nuclear Fr\'{e}chet space
(cf.\ Example~3 in \cite[Chapter~III, \S8]{Sch}),
it follows that also $C^\infty(B_x,\R)$
is a nuclear Fr\'{e}chet space
(see \cite[Satz~5.1.1]{Pie})
and so is $C^\infty(A_x,\R)\cong C^\infty(B_x,\R)$,
considering $A_x:=\phi_x^{-1}(B_x)$
as a full-dimensional submanifold with rough boundary in~$M$.
Since the linear map
\[
C^\infty(M,\R)\to\prod_{x\in M} C^\infty(A_x,\R),
\quad\gamma\mto (\gamma|_{A_x})_{x\in M}
\]
is a topological embedding,
%
%
we deduce with \cite[Satz~5.2.3]{Pie}
that $C^\infty(M,\R)$ is nuclear.
If $F$ is a nuclear locally convex space,
then its completion $\wt{F}$ is nuclear
(see \cite[Satz~5.3.1]{Pie}).
As the inclusion map
$C^\infty(M,F)\to C^\infty(M,\wt{F})$
is linear and a topological embedding,
%
%
it suffices to show that $C^\infty(M,\wt{F})$
is nuclear (see \cite[Satz~5.1.1]{Pie}).
Since $C^\infty(M,\R)$ is nuclear (as just shown)
and $\wt{F}$ is nuclear,~also
\[
\wt{F}\,
\tensor \, C^\infty(M,\R)
\]
is nuclear (see \cite[Satz~5.4.1]{Pie}).
But
$C^\infty(M,\wt{F})\cong\wt{F}\,\tensor\, C^\infty(M,\R)$
by Proposition~\ref{tensorcases}(b), whence also $C^\infty(M,\wt{F})$
is nuclear.\\[2.3mm]
(b) If the locally compact smooth manifold~$M$
with rough boundary is $\sigma$-compact and $F$ is a Fr\'{e}chet
space, then also $C^\infty(M,F)$ is a Fr\'{e}chet space
(cf.\ \cite[Proposition~3.6.20]{GaN}). The assertion therefore
follows from~(a).
\end{proof}
\section{More constructions of smoothing operators}\label{moreops}
Smoothing operators can be constructed
in further situations. In this section, we record additional
results, relegating most proofs
to the appendix (Appendix~\ref{appC}).\\[2.3mm]
We start with a proposition devoted to regularizing operators,
which replace general continuous functions with compactly
supported functions whose image has finite-dimensional span.
It is a topological analogue of Theorem~\ref{thmsmoo}.
\begin{prop}\label{smoo-top}
Let $X$ be a $\sigma$-compact, locally compact
space which is metrizable, $K_1\sub K_2\sub\cdots$ be a compact exhaustion of~$X$
and $F$ be a locally convex space.
There is a sequence $(S_n)_{n\in\N}$
of continuous linear operators $S_n\colon C(X,F)\to C(X,F)$
with the following properties:
\begin{itemize}
\item[\rm(a)]
$S_n(\gamma)\to\gamma$ in $C(X, F)$ as $n\to\infty$,
uniformly for~$\gamma$ in compact subsets of~$C(X,F)$;
\item[\rm(b)]
$S_n(\gamma)\in F\otimes C_{K_{n+1}}(X,\R)$
for all $n\in\N$ and $\gamma\in C(X,F)$; and
\item[\rm(c)]
$S_m(\gamma)\in F\otimes C_{K_{n+1}}(X,\R)$
for all $n\in\N$, $m\geq n$, and
$\gamma\in C_{K_n}(X,F)$.
\end{itemize}
\end{prop}
Note that $S_n$ is also continuous as a map to $C_c(X,F)=\dl \,C_{K_n}(X,F)$,\vspace{-1mm}
as a consequence of~(b).\\[2.3mm]
In this connection, we mention:
\begin{prop}\label{are-metri}
For every $\ell\in \N_0\cup\{\infty\}$,
every paracompact, locally compact, rough $C^\ell$-manifold
is metrizable.
\end{prop}
Using extension operators as a tool,
we construct smoothing operators
on cubes.
\begin{la}\label{smoocube}
For $d\in\N$ and $\ell\in\N_0$,
there are continuous linear operators
$S_n\colon C^\ell([0,1]^d,F)\to C^\infty([0,1]^d,F)$
for $n\in\N$ with image in $F\otimes C^\infty([0,1]^d,\R)$
such that
\[
S_n(\gamma)\to\gamma
\]
in $C^\ell([0,1]^d,F)$ as $n\to\infty$,
uniformly for $\gamma$ in compact subsets
of $C^\ell([0,1]^d,F)$.
\end{la}
\begin{proof}
We let $\cE\colon C^\ell([0,1]^d,F)\to C^\infty(\R^d,F)$
be a continuous linear right inverse for the restriction
map $C^\ell(\R^d,F)\to C^\ell([0,1]^d,F)$,
as provided by Corollary~\ref{extcorner}.
We write
\[
H_n\colon C^\ell(\R^d,F)\to C^\infty(\R^d,F)
\]
for the smoothing operator $S_n$ in Theorem~\ref{thmsmoo},
applied with $\Omega:=\R^d$.
Finally, we let $\rho\colon C^\infty(\R^d,F)\to C^\infty([0,1]^d,F)$
be the restriction map, which is continuous and linear.
Then
\[
S_n:=\rho\circ H_n\circ \cE\colon C^\ell([0,1]^d,F)\to C^\infty([0,1],F)
\]
is continuous and linear. If $K\sub C^\ell([0,1]^d,F)$
is a compact set, then $\cE(K)$
is compact, whence $H_n(\eta)\to \eta$ uniformly
in $\eta\in\cE(K)$ as $n\to\infty$ (see Remark~\ref{oncpset}).
The continuous linear map $\rho$ being uniformly
continuous, we deduce that
\[
S_n(\gamma)=\rho(H_n(\cE(\gamma)))\to \rho(\cE(\gamma))=\gamma
\]
as $n\to\infty$, uniformly in $\gamma\in K$.
Moreover, $H_n(\cE(\gamma))=\sum_{j=1}^m \gamma_jv_j$
for some $m\in\N$, $v_1,\ldots, v_m\in F$ and suitable functions
$\gamma_1,\ldots,\gamma_m\in C^\infty(\R^d,\R)$.
Then
\[
S_n(\gamma)=\sum_{j=1}^m \gamma_j|_{[0,1]^d} v_j\in F\otimes
C^\infty([0,1]^d,\R),
\]
which completes the proof.
\end{proof}
\begin{prop}\label{onmfdsmoo}
Let $F$ be a locally convex space,
$r\in \N_0\cup\{\infty\}$, and $\ell\in\N_0$ with $\ell\leq r$.
If $\ell=0$, let~$M$ be a $\sigma$-compact, locally compact
rough~$C^r$-manifold;
if $\ell>0$, let $M$ be a $\sigma$-compact,
locally compact $C^r$-manifold with corners.
Let $K_1\sub K_2\sub\cdots$ be a compact exhaustion of~$M$.
Then there exists a sequence $(S_n)_{n\in\N}$
of continuous linear operators
\[
S_n\colon C^\ell(M,F)\to C^r_{K_{n+1}}(M,F)
\]
with the following properties:
\begin{itemize}
\item[\rm(a)]
$S_n(\gamma)\to \gamma$ in $C^\ell(M,F)$
as $n\to\infty$, uniformly for $\gamma$ in compact subsets
of~$C^\ell(M,F)$;
\item[\rm(b)]
$S_n(\gamma)\in F\otimes C^r_{K_{n+1}}(M,\R)$
for all $n\in\N$ and $\gamma\in C^\ell(M,F)$; and
\item[\rm(c)]
$S_m(\gamma)\in F\otimes C_{K_{n+1}}^r(M,\R)$
for all $n\in\N$, $m\geq n$, and
$\gamma\in C_{K_n}^\ell(M,F)$.
\end{itemize}
\end{prop}
\section{Properties of evaluation and composition}
If $X$, $Y$, and $Z$ are Hausdorff topological spaces,
it is well known that
the evaluation map
\[
C(X,Y)\times X\to Y,\quad
(\gamma,x)\mto \gamma(x)
\]
and the composition map
\[
C(Y,Z)\times C(X,Y)\to C(X,Z),\quad (\gamma,\eta)\mto\gamma\circ\eta
\]
is continuous if~$Y$
is locally compact (see Theorem~3.4.3
and Proposition~2.6.11 in~\cite{Eng}
for the first statement, \cite[Theorem~3.4.2]{Eng}
for the second).
In the case of locally convex spaces
and their subsets, differentiability
properties for evaluation of $C^\ell$-maps
and composition are well-known, again based on
local compactness of~$Y$
(see, e.g., \cite{Alz}, \cite{AaS},
\cite{Zoo}, and~\cite{GaN}).
As a tool for the next section,
we show that evaluation always is sequentially
continuous (irrespective of local compactness)
and has certain differentiability
properties which enable a limited version of the Chain
Rule. As the observations may be useful elsewhere,
we also record analogous findings
for composition maps.\\[2.3mm]
In this section,
spaces of continuous
functions are endowed with the compact-open topology.
Spaces of $C^\ell$-functions (or $RC^\ell$-functions)
are endowed
with the compact-open $C^\ell$-topology.
We first discuss evaluation maps,
starting with a topological setting.
\begin{prop}\label{evalu}
Let $X$ and $Y$ be Hausdorff topological spaces.
Then the evaluation map
\[
\ev\colon C(X,Y)\times X\to Y,\;\,
(\gamma,x)\mto \gamma(x)
\]
has the following properties:
\begin{itemize}
\item[\rm(a)]
For each compact subset $K\sub X$, the restriction of
$\ev$ to a mapping\linebreak
$C(X,Y)\times K\to Y$ is continuous.
\item[\rm(b)]
For each compact subset $K\sub C(X,Y)\times X$,
the restriction $\ev|_K\colon K\to Y$ is continuous.
Notably, $\ev(K)$ is a compact subset of~$Y$.
\item[\rm(c)]
$\ev$ is sequentially continuous.
\item[\rm(d)]
Let $Z$ be a Hausdorff topological space
and $f\colon Z\to C(X,Y)$ as well as $f\colon Z\to X$
be continuous mappings.
If $Z$ is a $k$-space or $Z$ is a $k_\R$-space
and~$Y$ completely regular, then
\[
\ev\circ (f,g)\colon Z\to Y,\quad
x\mto f(x)(g(x))
\]
is a continuous mapping.
\end{itemize}
\end{prop}
\begin{rem}\label{howto}
With a view towards~(d), recall
that every metrizable topological space~$Z$ is a $k$-space.
Every topological vector space~$Y$ (and every topological
group) is completely regular.
\end{rem}
{\bf Proof of Proposition~\ref{evalu}.}
(a) Let $K\sub X$ be compact.
The restriction map
$\rho_K\colon C(X,Y)\to C(K,Y)$, $\gamma\mto\gamma|_K$
is continuous
(see, e.g., \cite[Remark A.5.10]{GaN}).
Since~$K$ is locally compact, the evaluation map
$\ve\colon C(K,Y)\times K\to Y$, $(\zeta,x)\mto \zeta(x)$
is continuous (as just recalled). Hence
\[
\ev|_{C(X,Y)\times K}=\ve\circ (\rho_K\times \id_K)
\]
is continuous.\\[2.3mm]
(b) Let $\pr_2\colon C(X,Y)\times X\to X$ be the projection
$(\gamma,x)\mto x$. If $K$ is a compact subset of
$C(X,Y)\times X$,
then $L:=\pr_2(K)$ is compact in~$X$. By~(a),
\[
\ev|_K=(\ev|_{C(X,Y)\times L})|_K
\]
is continuous. As a consequence, $\ev(K)=\ev|_K(K)$ is compact.\\[2.3mm]
(c) If $(\gamma_n,x_n)_{n\in\N}$ is a convergent sequence
in $C(X,Y)\times X$ with limit $(\gamma,x)$,
then $K:=\{(\gamma_n,x_n)\colon n\in\N\}\cup\{(\gamma,x)\}$
is a compact subset of $C(X,Y)\times X$. Using~(b), we see that
\[
\ev(\gamma_n,x_n)=\ev|_K(\gamma_n,x_n)\to\ev|_K(\gamma,x)=\ev(\gamma,x)
\]
as $n\to\infty$.\\[2.3mm]
(d) For each compact subset $L\sub Z$,
the image $K:=(f,g)(L)$ is compact in $C(X,Y)\times X$.
Hence $\ev\circ (f,g)=\ev|_K\circ (f,g)|^K$ is continuous,
by~(b). As $Z$ is a $k$-space
(or $Z$ a $k_\R$-space and $Y$ complete regular),
the continuity of the restrictions $\ev\circ (f,g)|_L$
implies continuity of the map
$\ev\circ (f,g)\colon Z\to Y$. $\,\square$\\[2.3mm]
Using Proposition~\ref{evalu},
we get a version of the Chain Rule
for the
evaluation map on $C^\ell$-functions.
We shall use a consequence of the Mean
Value Theorem:
\begin{numba}\label{BGN}
If $E$ and $F$ are locally convex spaces,
$U\sub E$ is an open subset and $f\colon U\to F$ a $C^1$-map,
then $U^{[1]}:=\{(x,y,t)\in U\times E\times\R\colon
x+ty\in U\}$ is an open subset of $U\times E\times\R$
and the map
\[
f^{[1]}\colon U^{[1]}\to F,\;\,
(x,y,t)\mto\left\{
\begin{array}{cl}
\frac{f(x+ty)-f(x)}{t} &\mbox{ if $t\not=0$;}\\
df(x,y) & \mbox{ if $t=0$}
\end{array}\right.
\]
is continuous
(see \cite[Lemma~1.2.10]{GaN},
cf.\ \cite[Proposition~7.4]{BGN}).
The final conclusion remains valid
of $U\sub E$ is a locally convex, regular subset
(see \cite[Lemma~1.4.9]{GaN}).
\end{numba}
\begin{prop}\label{eval-chain}
Let $E$ and $F$ be locally convex spaces and
$V\sub E$ be a locally convex, regular subset.
Let $X$ be a locally convex space
and $U\sub X$ a regular subset.
Let $\ell\in \N_0 \cup\{\infty\}$
and $f\colon U\to C^\ell(V,F)$ as well as
$g\colon U\to V\sub E$ be $C^\ell$-maps.
If $U$ and $U\times X$ are $k_\R$-spaces, then
\[
h\colon U\to F,\quad x\mto f(x)(g(x))
\]
is a $C^\ell$-map. Moreover,
\begin{equation}\label{formderev}
dh(x,y)=df(x,y)(g(x))+df(g(x),dg(x,y))
\end{equation}
for all $(x,y)\in U\times X$, if $\ell\geq 1$.
\end{prop}
The same conclusion holds if~$V\sub E$ is any
regular subset, $f\colon U\to RC^\ell(V,F)$ is $C^\ell$
and $g\colon U\to V$ is an $RC^\ell$-map.
\begin{rem}\label{applyeval}
Note that $U$ and $U\times X$ are $k$-spaces (and hence
$k_\R$-spaces) in the situation of Proposition~\ref{eval-chain}
whenever~$X$ is metrizable.
\end{rem}
{\bf Proof of Proposition~\ref{eval-chain}.}
We prove Proposition~\ref{eval-chain} under the stronger
hypothesis that $U\times X^{2^j-1}$ is a $k_\R$-space
for all $j\in \N_0$ such that $j\leq\ell$.
The proof of the general case,
which is more technical,
can be found in Appendix~\ref{appcompo}.\\[2.3mm]
We may assume that $\ell\in\N_0$; the proof is by induction.
The case $\ell=0$ holds by Proposition~\ref{evalu}(d).
Now assume that $\ell\geq 1$ and assume that the assertion holds
for $\ell-1$ in place of~$\ell$.
Let
\[
\ve\colon C^{\ell-1}(V\times E,F)\times (V\times E) \to F
\]
be the evaluation map and $\pr_1\colon U\times X\to U$,
$(x,y)\mto x$ be the projection.
We claim that (\ref{formderev})
holds for all $(x,y)\in U^0\times X$.
If this is true, then $h$ is $C^1$ and
\begin{equation}\label{henceindu}
dh=\ev\circ (df,g\circ \pr_1)+\ve\circ (df,(g\circ \pr_1,dg)),
\end{equation}
as the right-hand side is a continuous function
by Proposition~\ref{evalu}(d) and extends
$d(h|_{U^0})$. In fact, the right-hand side
of (\ref{henceindu}) is $C^{\ell-1}$
by the inductive hypothesis. Since $h$ is $C^1$
and $dh$ is $C^{\ell-1}$, the map $h$ is~$C^\ell$.
By (\ref{henceindu}), the identity (\ref{formderev})
holds for all $(x,y)\in U\times X$.\\[2.3mm]
To establish the claim, let $(x,y)\in U^0\times X$.
Let $(t_n)_{n\in \N}$ be a sequence in $\R\setminus \{0\}$
such that $x+t_ny\in U^0$ for all $n\in\N$
and $t_n\to 0$ as $n\to\infty$.
Then
\begin{eqnarray*}
\lefteqn{\frac{f(x+t_ny)(g(x+t_ny))-f(x)(g(x))}{t_n}}\qquad\qquad\\
&=&
\frac{f(x+t_ny)-f(x)}{t_n}(g(x+t_ny))
+\frac{f(x)(g(x+t_ny))-f(x)(g(x))}{t_n}
\end{eqnarray*}
for all $n\in\N$. Note that the first summand tends to
$df(x,y)(g(x))$ as $n\to\infty$ by sequential continuity
of the evaluation map. The second summand
can be written as
\[
(f|_{U^0})^{[1]}\left(g(x),\frac{g(x+t_ny)-g(x)}{t_n},t_n\right)
\]
and hence converges to $(f|_{U^0})^{[1]}(g(x),dg(x,y),0)=
df(g(x),dg(x,y))$ as $n\to\infty$.
The claim is established.
$\,\square$\\[2.3mm]
Before we discuss composition maps,
it is useful to record facts and observations
concerning regularity properties of the compact-open
topology,
and compact-open $C^\ell$-toplogies.
\begin{numba}\label{co-reg}
Let $X$ and $Y$ be Hausdorff spaces.
If $Y$ is regular, then also
$C(X,Y)$ is regular when endowed with the compact-open topology
(see \cite[Theorem~3.4.13]{Eng}).
If $Y$ is completely regular,
then also $C(X,Y)$ is completely
regular (see \cite[Theorem~3.4.15]{Eng}).
\end{numba}
\begin{numba}\label{reg-bdle}
If $\pi\colon Y\to X$ is a locally trivial
fibre bundle over a regular
topological space~$X$ whose fibres
are regular topological spaces,
then $Y$ is a regular topological space.\\[2.3mm]
[Let $y\in Y$ and $V$ be an open neighbourhood
of~$y$ in~$Y$. Then $x:=\pi(y)$
has an open neighbourhood
$U$ in~$X$ for which there exists a homeomorphism
$\theta\colon \pi^{-1}(U)\to U\times F_x$
for a regular topological space~$F_x$,
such that $\pr_1\circ\theta=\pi|_{\pi^{-1}(U)}$.
Since $X$ is regular, also its subset~$U$
is regular and hence also $U\times F_x$,
entailing that $\pi^{-1}(U)$
is regular. We therefore find a neighbourhood
$B$ of $y$ in $\pi^{-1}(U)$ which is closed
in $\pi^{-1}(U)$ and such that $B\sub\pi^{-1}(U)\cap V$.
Since $X$ is regular, there exists a neighbourhood~$A$
on~$x$ in~$U$ which is a closed subset of~$X$.
Then $\pi^{-1}(A)$ is closed in~$Y$
and so is its closed subset $C:=B\cap \pi^{-1}(A)$.
Moreover, $C$ is a neighbourhood of~$y$ in~$Y$ and $C\sub V$.\,]
\end{numba}
Recall that a Hausdorff topological
space~$X$ is completely
regular if and only if, for each $x\in X$
and neighbourhood~$U$ of~$x$ in~$X$,
we find a continuous function $f\colon X\to\R$
such that $f(x)\not=0$ and $\Supp(f)\sub U$.
\begin{la}
Let $\ell\in\N_0\cup\infty$
and $M$ be a rough $C^\ell$-manifold
modelled on locally convex spaces.
Then the following holds:
\begin{itemize}
\item[\rm(a)]
The topological space underlying the manifold
$M$ is regular if and
only if~$M$ is completely regular.
\item[\rm(b)]
If $M$ is regular,
then the iterated tangent bundle
$T^k(M)$ is regular
for all $k\in\N_0$ such that $k\leq \ell$.
\item[\rm(c)]
If $N$ is a $C^\ell$-manifold with rough boundary
and~$M$ is regular, then $C^\ell(M,N)$
is regular when endowed with the compact-open
$C^\ell$-topology.
\item[\rm(d)]
If $N$ is a rough $C^\ell$-manifold
and~$M$ is regular, then $RC^\ell(M,N)$
is regular when endowed with the compact-open
$C^\ell$-topology.
\end{itemize}
\end{la}
\begin{proof}
(a) If $M$ is regular, let $x\in M$ and $U$ be a neighbourhood
of~$x$ in~$M$. let $\phi\colon U_\phi\to V_\phi$
be a chart of~$M$ such that $x\in U_\phi$.
Then $V_\phi$ is a subset of a locally convex space~$E$.
Since~$M$ is regular, there exists a closed subset $A$ of~$M$
such that $A\sub U\cap U_\phi$ and~$A$ is a neighbourhood of~$x$ in~$M$.
As~$E$ is completely regular, also $V_\phi$ is completely
regular. We therefore find a continuous function
$h\colon V_\phi\to \R$ such that $h(\phi(x))\not=0$
and $\Supp(h)\sub \phi(A)$.
Then
\[
f\colon M\to \R,\quad y\mto \left\{
\begin{array}{cl}
0 & \mbox{ if $y\in M\setminus A$;}\\
h(\phi(y)) & \mbox{ if $y\in U_\phi$}
\end{array}\right.
\]
is a continuous function such that $f(x)\not=0$
and $\Supp(f)\sub A\sub U$.\\[2.3mm]
(b) If $\ell\geq 1$, then $TM$ is regular
by \ref{reg-bdle}.
The assertion follows by induction.\\[2.3mm]
(c) By (b), $T^kN$ is regular
for all $k\in\N_0$ such that $k\leq\ell$.
Hence $T^kN$ is completely regular (by (a))
whence also $C(T^kM,T^kN)$ is completely regular,
as recalled in~\ref{co-reg}.
The topology on
$C^\ell(M,N)$ is initial with respect to the mapping
$T^k\colon C^\ell(M,N)\to C(T^kM,T^kN)$,
and the topology on $C(T^kM,T^kN)$
is initial with respect to the set of all
continuous maps $f\in C(T^kM,T^kN)\to\R$.
The topology on $C^\ell(M,N)$ is therefore
initial with respect to the continuous real-valued
mappings
$f\circ T^k\colon C^\ell(M,N)\to\R$.
The assertion follows.\\[2.3mm]
(d) We can repeat the proof of (c) with $RC^\ell(M,N)$
in place of $C^\ell(M,N)$.
\end{proof}
\begin{prop}\label{compocseq}
Let $X$, $Y$, and $Z$ be Hausdorff topological spaces.
Then the composition map
\[
c_{X,Y,Z}\colon C(Y,Z)\times C(X,Y)\to C(X,Z),
\;(\gamma,\eta)\mto \gamma\circ\eta
\]
has the following properties:
\begin{itemize}
\item[\rm(a)]
For each compact subset $K\sub C(X,Y)$,
the restriction of $c_{X,Y,Z}$ to a map
$C(Y,Z)\times K\to C(X,Z)$ is continuous.
\item[\rm(b)]
For each compact subset $K\sub C(Y,Z)\times C(X,Y)$,
the restriction $c_{X,Y,Z}|_K$ is continuous.
Notably, $c_{X,Y,Z}(K)$ is compact.
\item[\rm(c)]
$c_{X,Y,Z}$ is sequentially continuous.
\item[\rm(d)]
Let $A$ be a Hausdorff topological space
and $f\colon A\to C(Y,Z)$ as well as $g\colon A\to C(X,Y)$
be continuous mappings.
If $A$ is a $k$-space or $A$ is a $k_\R$-space
and~$Z$ completely regular, then
\[
c_{X,Y,Z}\circ (f,g)\colon A\to C(X,Z),\quad
x\mto f(x)\circ g(x)
\]
is a continuous mapping.
\end{itemize}
\end{prop}
\begin{proof}
(a) Given a compact subset $K\sub C(X,Y)$,
consider $h\colon C(Y,Z)\times K\to C(X,Z)$,
$(\gamma,\eta)\mto\gamma\circ\eta$.
The compact-open topology~$\cO$ on $C(X,Z)$ is initial
with respect to the restriction maps $\rho_L\colon C(X,Z)\to C(L,Z)$
for $K\in\cK(X)$; in fact, each $\rho_L$ is continuous
(see, e.g., \cite[Remark~A.5.10]{GaN})
and the pre-images $\rho_K^{-1}(C(L,U))=\lfloor L,U\rfloor$
generate~$\cO$ for $L\in \cK(X)$ and $U$
ranging through the open subsets of~$Z$.
We therefore only need to show that $\rho_L\circ h$
is continuous for all $L\in\cK(X)$.
Now
\[
B:=\{\gamma(x)\colon \gamma\in K,\;x\in L\}
\]
is a compact subset of~$Y$ by Proposition~\ref{evalu}(b)
and $\eta|_L\in C(L,B)$ for all $\eta\in K$.
The restriction maps
$r_B\colon C(Y,Z)\to C(B,Z)$ and $s_L\colon C(X,Y)\supseteq
K\to C(L,Y)$
are continuous and also the composition map $c_{L,B,Z}$
is continuous, by local compactness of~$B$.
Since
\[
\rho_L\circ h=c_{L,B,Z}\circ (r_B\times s_L),
\]
we see that $\rho_L\circ h$ (and hence~$h$) is continuous.\\[2.3mm]
(b) Let $\pr_2\colon C(Y,Z)\times C(X,Y)\to C(X,Y)$,
$(\gamma,\eta)\mto\eta$ be the projection.
Given $K$ as in~(b), also $L:=\pr_2(K)$
is compact. Using~(a), we see that
\[
c_{X,Y,Z}|_K=(c_{X,Y,Z}|_{C(Y,Z)\times L})|_K
\]
is continuous. Hence $c_{X,Y,Z}(K)=(c_{X,Y,Z}|_K)(K)$ is compact.\\[2.3mm]
(c)
$(\gamma_n,\eta_n)_{n\in\N}$ be a sequence in $C(Y,Z)\times
C(X,Y)$
which converges to
$(\gamma,\eta)\in C(Y,Z)\times C(X,Y)$.
Then $K:=\{(\gamma_n,\eta_n)\colon n\in\N\}\cup\{(\gamma,\eta)\}$
is a compact subset of $C(Y,Z)\times C(X,Y)$.
Then
\[
(\gamma_n\circ \eta_n)=c_{X,Y,Z}|_K(\gamma_n,\eta_n)
\to c_{X,Y,Z}|_K(\gamma,\eta)=\gamma\circ\eta
\]
as $n\to\infty$, using~(b).\\[2.3mm]
(d) Note that if $Z$ is completely regular,
then also $C(X,Z)$ is completely regular.
In fact, as the map
\[
\phi\colon Z\to \prod_{f\in C(Z,\R)}\R=:P
\]
taking $x\in Z$ to $(f(x))_{f\in C(Z,\R)}$
is a topological embedding, also
\[
C(X,\phi)\colon C(X,Z)\to C(X,P),\quad \gamma\mto \phi\circ \gamma
\]
is a topological embedding (see, e.g.,
\cite[Lemma~A.5.5]{GaN}).
Now $P$ is a locally convex space and hence also
$C(X,P)$ is a locally convex space,
entailing that $C(X,P)$ (and hence also
$C(X,Z)$) is completely regular.\\[2.3mm]
If $A$ is a $k$-space or $A$ is a $k_\R$-space and~$Z$
completely regular, then the map $c_{X,Y,Z}\circ (f,g)\colon A\to C(X,Z)$
will be continuous
if we can show that its restriction to~$L$
is continuous for each compact subset $L\sub A$. But $K:=(f,g)(L)$
is a compact subset of $C(Y,Z)\times C(X,Y)$.
Hence $c_{X,Y,Z}\circ (f,g)|_L=c_{X,Y,Z}|_K\circ (\gamma,\eta)|_L^K$
is continuous, by~(b).
\end{proof}
We mention a direct consequence.
\begin{prop}\label{compoclseq}
Let $\ell\in \N_0\cup\{\infty\}$
and $L$, $M$, and $N$ be $C^\ell$-manifolds
with rough boundary, which are modelled on locally
convex spaces. Then the composition map
\[
c^\ell_{L,M,N}\colon C^\ell(M,N)\times C^\ell(L,M)\to C^\ell(L,N),\quad
(\gamma,\eta)\mto\gamma\circ\eta
\]
has the following properties:
\begin{itemize}
\item[\rm(a)]
For each compact subset $K\sub C^\ell(L,M)$,
the restriction of $c^\ell_{L,M,N}$ to a map
$C^\ell(M,N)\times K\to C(L,N)$ is continuous.
\item[\rm(b)]
For each compact subset $K\sub C^\ell(M,N)\times C(L,M)$,
the restriction $c^\ell_{L,M,N}|_K$ is continuous.
\item[\rm(c)]
$c^\ell_{L,M,N}$ is sequentially continuous.
\item[\rm(d)]
Let $X$ be a Hausdorff topological space
and $f\colon X\to C^\ell(M,N)$ as well as $g\colon X\to C^\ell(L,M)$
be a continuous map.
If~$X$ is a $k$-space or~$X$ is a $k_\R$-space
and~$N$ is completely regular, then
\[
c^\ell_{L,M,N}\circ (f,g)\colon X\to C^\ell(L,N),\quad
x\mto f(x)\circ g(x)
\]
is a continuous mapping.
\end{itemize}
\end{prop}
\begin{proof}
Let $\gamma_n$ converge to $\gamma$ in $C^\ell(M,N)$
and $\eta_n$ converge to $\eta$ in $C^\ell(L,M)$.
The topology on $C^\ell(K,M)$ is initial with respect
to the maps $T^j\colon C^\ell(K,M)\to C(T^jK,T^jM)$
for $j\in\N_0$ such that $j\leq\ell$,
using the compact-open topology on spaces of continuous
functions. To see that $\gamma_n\circ\eta_n\to \gamma\circ\eta$
in $C^\ell(L,N)$, we therefore only need to show that
$T^j(\gamma_n\circ\eta_n)\to T^j(\gamma\circ\eta)$
in $C(T^jL,T^jN)$ as $n\to\infty$,
for all $j$ as before.
But
\begin{eqnarray*}
T^j(\gamma\circ\eta)&=& (T^j\gamma)\circ (T^j\eta)=c_{T^jL,T^jM,T^jN}(T^j\gamma_n,
T^j\eta_n)\\
&\to &
c_{T^jL,T^jM,T^jN}(T^j\gamma,T^j\eta)=
T^j \gamma\circ T^j\eta=T^j(\gamma\circ\eta),
\end{eqnarray*}
by Proposition~\ref{compocseq}.
\end{proof}
The same conclusion holds if $M$ and $N$ are as before
but $L$ merely
is a rough $C^\ell$-manifold modelled on locally convex
spaces. It also holds for the composition map
on spaces of $RC^\ell$-maps between rough $C^\ell$-manifolds
$L$, $M$, and~$N$.\\[2.3mm]
We shall not use the following more technical
result, the proof of which can
be looked up in Appendix~\ref{appcompo}.
\begin{prop}\label{strange-compo}
Let $E$, $F$, $X$, and $Z$ be locally convex spaces,
$A\sub Z$, $R\sub X$, and $S\sub E$
be regular subsets, and
$k,\ell\in \N_0\cup\{\infty\}$.
Let $\gamma\colon A\to C^{\ell+k}(S,F)$
and $\eta\colon A\to C^\ell(R,E)$
be $C^k$-maps such that $\eta(z)(R)\sub S$
for all $z\in A$;
if~$S$ is not locally convex, assume, moreover,
that $\eta(z)(R^0)\sub S^0$ for all $z\in A$.
If $\ell=0$,
assume that~$R$ is a $k_\R$-space;
if $\ell\geq 1$, assume that $R\times X$ is a $k_\R$-space.
If, moreover, $A\times Z$
is a $k_\R$-space, then
\[
\zeta\colon A\to C^\ell(R,F),\quad z\mto \gamma(z)\circ\eta(z)
\]
is a $C^k$-map.
\end{prop}
If $\ell=0$, the conclusion holds more generally
if $A$ and $Z$ are as before and
$R$ is any Hausdorff topological space
which is a $k_\R$-space.\\[2.3mm]
Recall that the space $\cL(X,F)$ of continuous
linear operators is a closed vector subspace
of $C^\infty(X,F)$ for all
locally convex spaces~$X$ and $F$; moreover,
the compact-open
topology on $\cL(X,F)$ coincides with the compact-open
$C^\infty$-topology (see Lemma~\ref{co=coinft}).
We therefore get the following
immediate corollary to Proposition~\ref{strange-compo}
and Proposition~\ref{compocseq}:
\begin{cor}\label{compo-spec}
Let $Z$, $X$, $E$, and $F$ be locally convex spaces,
$A\sub Z$ be a regular subset and $k\in \N_0\cup\{\infty\}$.
If $\gamma\colon A\to \cL(E,F)$ and $\eta\colon A\to\cL(X,E)$
are $C^k$-maps and $A\times Z$ is
a $k_\R$-space $($or $k=0$ and $A$ is a $k_\R$-space$)$,
then also the following map is $C^k$:
\[
A\to \cL(X,F),\quad z\mto\gamma(z)\circ\eta(z).
\]
Here $\cL(X,E)$, $\cL(E,F)$, and $\cL(X,F)$ carry
the compact-open topology.
\end{cor}
Likewise, Proposition~\ref{eval-chain} implies:
\begin{cor}\label{eval-spec}
Let $X$, $E$, and $F$ be locally convex spaces,
$R\sub X$ be a regular subset and $k\in \N_0\cup\{\infty\}$.
If $\gamma\colon R\to \cL(E,F)$ and $\eta\colon R\to E$
are $C^k$-maps and $R\times X$ is
a $k_\R$-space,
then also the following map is $C^k$:
\[
R\to F,\quad z\mto\gamma(z)(\eta(z)).
\]
\end{cor}
If $k=0$, the conclusion also holds
if $R$ is any Hausdorff topological space
which is $k_\R$, by Proposition~\ref{evalu}.
\begin{rem}
We mention that
special cases have been recorded before.
It is well known that the evaluation map
$\ve\colon \cL(E,F)\times E\to F$ is a hypocontinuous\footnote{We mean
hypocontinuity with respect to the set of compact
subsets of the second factor.}
bilinear map
for all locally convex spaces
$E$ and $F$, when the compact-open topology
is used on the space $\cL(E,F)$ of continuous linear
maps $E\to F$ (see \cite{HYP}, cf.\ \cite{Bou}).
Hence Corollary~\ref{eval-spec}
follows from \cite[Theorem~2.5]{HYP}
if $X^n$ is a $k$-space for each $n\in\N$
and $R\sub X$ is open
(the general setting of the cited theorem).
Moreover, the composition map
$\cL(E,F)\times \cL(X,E)\to \cL(X,F)$
is hypocontinuous if the compact-open topologies
are used
(see Proposition~9 in \cite[Ch.\,III, \S5, no.\,5]{Bou}).
Thus Corollary~\ref{compo-spec}
follows from \cite[Theorem~2.5]{HYP}
as well (for $X$ and $R$ as just explained).
The compact-open topology can be
replaced with the topology of bounded convergence
(see \cite[Corollary~2.6]{HYP}).\\[2.3mm]
For differential calculi based on $k$-spaces
and $k$-refinements of topologies
on direct products,
cf.\ also \cite{Sei} and \cite{EGF}.
\end{rem}
\section{Smoothing of sections in fibre bundles}
In this section, we consider fibre bundles
in the following generality:
\begin{defn}\label{deffib}
Let $r\in\N_0\cup\{\infty\}$ and
$M$ be a $\sigma$-compact, locally compact,
rough $C^r$-manifold.
A \emph{$C^r$-fibre bundle} over~$M$
is a pair $(N,\pi)$, where $N$ is a
rough $C^r$-manifold modelled
on locally convex spaces and $\pi\colon N\to M$
a surjective $RC^r$-map which is \emph{locally trivial}
in the following sense:
For each $x\in M$, there exists a $C^r$-manifold~$N_x$
(without boundary) modelled on locally convex spaces
and an $RC^r$-diffeomorphism
\[
\theta=(\theta_1,\theta_2)\colon N|_U\to U\times N_x
\]
(where $N|_U:=\pi^{-1}(U)$) such that $\theta_1(y)=\pi(y)$
for all $y\in U$ (a local trivialization around~$x$).
\end{defn}
Note that $\theta$ restricts to a bijection
$\pi^{-1}(\{x\})\to\{x\}\times N_x\cong N_x$;
when convenient, we may therefore assume that $N_x=\pi^{-1}(\{x\})$
and $\theta_2|_{N_x}$ is the identity map.
\begin{defn}\label{defsec}
Given a $C^r$-fibre bundle $\pi\colon N\to M$ over~$M$
and $\ell\in\N_0\cup\{\infty\}$ with $\ell\leq r$,
we let $\Gamma_{C^\ell}(M\leftarrow N)$
be the set of all $C^\ell$-sections~$\sigma$ of~$\pi$,
i.e., $RC^\ell$-functions $\sigma\colon M\to N$ such that
$\pi\circ \sigma=\id_M$.
\end{defn}
\begin{defn}
Let $M$ be a $\sigma$-compact, locally compact
rough $C^\ell$-manifold with $\ell\in\N_0$
and $N$ be a rough $C^\ell$-manifold modelled on locally convex spaces.
Let $RC^\ell(M,N)$ be the set of all restricted $C^\ell$-maps
from~$M$ to~$N$.
We define the \emph{Whitney $C^\ell$-topology} on $RC^\ell(M,N)$
as the initial topology with respect to the mapping
\begin{equation}\label{hereboxpro}
RC^\ell(M,N)\to{\prod_{n\in\N}}^b RC^\ell(M_n,N),\quad
\gamma\mto (\gamma|_{M_n})_{n\in\N},
\end{equation}
where $RC^\ell(M_n,N)$ is endowed
with the compact-open $C^\ell$-topology
and $(M_n)_{n\in\N}$ is a locally finite sequence
of (possibly empty) relatively compact,
regular subsets of~$M$ whose interiors $M_n^0$ cover~$M$.
The direct product in~(\ref{hereboxpro}) is endowed
with the box topology.
The Whitney $C^\ell$-topology is independent of the choice
of the sequence $(M_n)_{n\in\N}$
(see \cite{Rou} for details,
where also topologies on
$RC^\infty(M,N)$ are
discussed in analogy to classical concepts
as in \cite{Hir}, \cite{Mic}, \cite{Ill}, \cite{HaS}
and the references therein).
\end{defn}
\begin{defn}
For $r\in\N_0\cup\{\infty\}$,
$\ell\in\N_0$ with $\ell\leq r$
and a $C^r$-fibre bundle $\pi\colon N\to M$,
we give $\Gamma_{C^\ell}(M\leftarrow N)$
the topology induced by $RC^\ell(M,N)$,
endowed with the Whitney $C^\ell$-topology.
\end{defn}
Our goal is the next theorem that
generalizes a result in~\cite{Wo2}
devoted to the case $(\ell,r)=(0,\infty)$
(which assumes that $M$ is a connected
$C^\infty$-manifold with corners,
and gives less detailed information
concerning properties of the homotopies).
Wockel's result, in turn, generalizes
a classical fact by Steenrod (\S6.7
in \cite{Ste}).
\begin{thm}\label{thm-steen}
Let $r\in \N_0\cup\{\infty\}$ and $\ell\in\N_0$
with $\ell\leq r$.
If $\ell=0$, let $M$ be a $\sigma$-compact,
locally compact rough $C^r$-manifold;
if $\ell>0$, let $M$ be a $\sigma$-compact,
locally compact $C^r$-manifold with corners.
Let $\pi\colon N\to M$
be a $C^r$-fibre bundle over~$M$,
as in Definition~{\rm\ref{deffib}}.
Let $\sigma\in \Gamma_{C^\ell}(M\leftarrow N)$,
$\Omega\sub \Gamma_{C^\ell}(M\leftarrow N)$
be a neighbourhood of~$\sigma$ in the Whitney $C^\ell$-topology,
$U\sub M$ be open and $A\sub M$ be a closed
subset such that $\sigma$ is $RC^r$
on an open neighbourhood of $A\setminus U$ in~$M$.
Then there exists a section $\tau\in \Omega$
and a homotopy
$H\colon [0,1]\times M\to N$ from $\sigma=H(0,\cdot)$
to $\tau=H(1,\cdot)$ such that $H_t:=H(t,\cdot)\in\Omega$
for all $t\in [0,1]$ and the following holds:
\begin{itemize}
\item[\rm(a)]
$\sigma|_{M\setminus U}=H_t|_{M\setminus U}$
for all $t\in [0,1]$;
\item[\rm(b)]
For every open subset $V\sub M$ such that $\sigma|_V$
is $RC^r$, also $H_t|_V$ is $RC^r$ for all $t\in [0,1]$.
\item[\rm(c)]
There exists an open neighbourhood $W$ of~$A$ in~$M$
such that $H_t|_W$ is $RC^r$ for all $t\in\,]0,1]$.
\end{itemize}
Moreover, one can achieve that $H$ is $RC^{0,\ell}$,
the restriction $H|_{]0,1]\times M}$ is $RC^{r-\ell,\ell}$,
and that, for each $V$ as in~{\rm(b)},
the restriction of~$H$ to a map
$]0,1]\times (V\cup W)\to N$ is $RC^{0,r}$
$($resp., $RC^\infty$ if $r=\infty)$.
\end{thm}
Specializing to trivial fibre bundles,
we deduce:
\begin{cor}\label{smoothemps}
Let $r\in \N_0\cup\{\infty\}$ and $\ell\in\N_0$
with $\ell\leq r$.
If $\ell=0$, let $M$ be a $\sigma$-compact,
locally compact rough $C^r$-manifold;
if $\ell>0$, let $M$ be a $\sigma$-compact,
locally compact $C^r$-manifold with corners.
Let $N$
be a $C^r$-manifold modelled on locally convex spaces.
Let $\gamma\in C^\ell(M,N)$,
$\Omega\sub C^\ell(M,N)$
be a neighbourhood of~$\sigma$ in the Whitney $C^\ell$-topology,
$U\sub M$ be open and $A\sub M$ be a closed
subset such that~$\gamma$ is $C^r$
on an open neighbourhood of $A\setminus U$ in~$M$.
Then there exists a function $\eta\in \Omega$
and a continuous homotopy
$H\colon [0,1]\times M\to N$ from $\gamma=H(0,\cdot)$
to $\eta=H(1,\cdot)$ such that $H_t:=H(t,\cdot)\in\Omega$
for all $t\in [0,1]$ and the following holds:
\begin{itemize}
\item[\rm(a)]
$\gamma|_{M\setminus U}=H_t|_{M\setminus U}$
for all $t\in [0,1]$;
\item[\rm(b)]
For every open subset $V\sub M$ such that $\gamma|_V$
is $C^r$, also $H_t|_V$ is $C^r$ for all $t\in [0,1]$.
\item[\rm(c)]
There exists an open neighbourhood $W$ of~$A$ in~$M$
such that $H_t|_W$ is $C^r$ for all $t\in\,]0,1]$.
\end{itemize}
Moreover, one can achieve that $H$ is $C^{0,\ell}$,
the restriction $H|_{]0,1]\times M}$ is $C^{r-\ell,\ell}$,
and that, for all $V$ as in {\rm(b)},
the restriction of~$H$ to
$]0,1]\times (V\cup W)$ is $C^{0,r}$
$($resp., $C^\infty$ if $r=\infty)$. $\,\square$
\end{cor}
See Section~\ref{alg-top}
for typical applications of Corollary~\ref{smoothemps}
(or also its more limited precursors
in~\cite{Ste} and \cite{Wo2}).\\[2.3mm]
It will be useful for the proof of Theorem~\ref{thm-steen}
to interpolate the smoothing operators
obtained earlier, so that we get families
with a parameter $t\in\,]0,1]$.
\begin{la}\label{interpol}
Let $F$ be a locally convex space,
$d\in\N$, $\ell\in \N_0$
and $L\sub\R^d$ be a compact, regular subset.
If $\ell>0$, assume that $L=[0,1]^d$.
There is a family $(S_t)_{t\in \,]0,1]}$
of continuous linear operators
\[
S_t\colon C^\ell(L,F)\to C^\infty(L,F)
\]
with the following properties:
\begin{itemize}
\item[\rm(a)]
$S_t(\gamma)\in F\otimes C^\infty(L,\R)$
for all $t\in\,]0,1]$;
\item[\rm(b)]
The map $\,]0,1]\to \cL(C^\ell(L,F),C^\infty(L,F))$,
$t\mto S_t$ is smooth for each locally convex vector topology
on the space $\cL(C^\ell(L,F),C^\infty(L,F))$
of continuous linear operators;
\item[\rm(c)]
$S_t(\gamma)\to\gamma$ in $C^\ell(L,F)$ as $t\to 0$
uniformly for $\gamma$ in compact subsets of $C^\ell(L,F)$.
\end{itemize}
\end{la}
\begin{proof}
For $n\in\N$,
let us write $H_n$ for the operator $S_n$
in Lemma~\ref{smoocube}
(if $\ell>0$) and Proposition~\ref{onmfdsmoo} (if $\ell=0$,
with $K_n:=L$ for all $n$),
respectively.
There exists a monotonically increasing smooth function $\rho
\colon [0,1]\to \R$
such that
\[
\rho|_{[0,\ve]}=0\quad\mbox{and}\quad\rho|_{[1-\ve,1]}=1
\]
for some $\ve\in\,]0,\frac{1}{2}[$.
We choose
$1=t_1>t_2>\cdots$ with $t_n\to 0$
and define
\[
S_t\;:=\; H_{j+1}+\rho\left({\textstyle\frac{t-t_{j+1}}{t_j-t_{j+1}}}\right)
(H_j-H_{j+1})
\]
if $t\in \,]t_{j+1},t_j]$.
As $H_j\to\id$ uniformly on compact sets
and the $S_t$ are convex combinations
of $H_j$ and $H_{j+1}$, we get~(c).
The convex combinations depend smoothly
on $t\in\,]t_{j+1},t_j[$ and are constant on a neighbourhood
of $t_j$; thus (b) holds.
Since $H_j(\gamma)(\Omega)\sub F$
has finite-dimensional span for all~$j$,
the same holds if $H_j$ is replaced with a convex combination
of $H_j$ and $H_{j+1}$, like $S_t$ for $t\in[t_{j+1},t_j]$;
thus~(a) holds.
\end{proof}
\begin{rem}\label{cts-smoo}
If we define $S_0(\gamma):=\gamma$
for $\gamma\in C^\ell(\Omega,F)$,
then (b) and (c) in Lemma~\ref{interpol}
imply that the map
\[
[0,1]\to\cL(C^\ell(L,F),C^\ell(L,F)),\quad
t\mto S_t
\]
is continuous for the topology of compact convergence
on the space of continuous
linear operators.
\end{rem}
{\bf Proof of Theorem~\ref{thm-steen}.}
Let $V$ be the largest open subset of~$M$
such that $\sigma|_V$ is $RC^r$
(the union of all such open sets).
By hypothesis, $V$ is a neighbourhood of
$A\setminus U$ in~$M$.
If $A\sub V$, we can set $H(t,x):=\sigma(x)$
for all $(t,x)\in [0,1]\times M$.
Now assume that $A\setminus V$ is not empty.
There exists a locally finite family
$(Q_j)_{j\in J}$ of full-dimensional
compact submanifolds $Q_j$ of~$M$
and compact regular subsets $P_j\sub Q_j^0$
such that $(P_j^0)_{j\in J}$
is a cover of~$M$
and (i)--(iii) hold
for all~$j\in J$:
\begin{itemize}
\item[(i)]
$Q_j\sub M\setminus A$
or $Q_j\sub U$;
\item[(ii)]
If $\ell=0$, there exists an $RC^r$-diffeomorphism
$\kappa_j\colon Q_j\to L_j$
for a compact, non-empty, regular subset
$L_j\sub \R^{d_j}$ for some $d_j\in \N_0$.
If $\ell>0$, there exists an $RC^r$-diffeomorphism
$\kappa_j\colon Q_j\to L_j:=[0,1]^{d_j}\sub\R^{d_j}$ for
for some $d_j\in \N_0$ (where $[0,1]^0:=\{0\}\sub\R^0$);
\item[(iii)]
$Q_j$ is contained in an open subset $M_j\sub M$
such that there exists a local trivialization
\[
\theta_j=(\theta_{j,1},\theta_{j,2})
\colon N|_{M_j}\to M_j\times N_j
\]
for some $C^r$-manifold
$N_j$ modelled on locally convex spaces,
and $\theta_{j,2}(Q_j)\sub U_j$
holds for a chart $\phi_j\colon U_j\to V_j$
of~$N_j$ such that $V_j$ is an open, convex subset of
a locally convex space~$F_j$.
\end{itemize}
%
%
Let $J_0:=\{j\in J\colon Q_j\cap A\setminus V\not=\emptyset\}$;
then $Q_j\sub U$ for all $j\in J_0$.
We may assume that $J_0=\{n\in\N\colon n\leq n_0\}$
for some $n_0\in \N\cup\{\infty\}$.\\[2.3mm]
The topology on $\Gamma_{C^\ell}(M\leftarrow N)$
is initial with respect to the map
\[
\rho\colon \Gamma_{C^\ell}(M\leftarrow N)\to
{\prod_{j\in J}}^bRC^\ell(Q_j,N),\;
\tau\mto (\tau|_{Q_j})_{j\in J}
\]
and also with respect to the co-restriction
\[
\rho\colon \Gamma_{C^\ell}(M\leftarrow N)\to
{\prod_{j\in J}}^b\Gamma_{C^\ell}(Q_j\leftarrow N|_{Q_j})
\]
of the latter. The map
\[
f_j\colon \Gamma_{C^\ell}(Q_j\leftarrow N|_{Q_j})\to
C^\ell(Q_j,N_j),\;\, \tau\mto\theta_{j,2}\circ \tau
\]
is a homeomorphism, whence also
\[
f:=\prod_{j\in J}f_j\colon
{\prod_{j\in J}}^b\Gamma_{C^\ell}(Q_j\leftarrow N|_{Q_j})\to
{\prod_{j\in J}}^b C^\ell(Q_j,N_j),\;\,
(\tau_j)_{j\in J}\mto
(f_j(\tau_j))_{j\in J}
\]
is a homeomorphism. As a consequence,
\[
f\circ\rho\colon \Gamma_{C^\ell}(M\leftarrow N)\to {\prod_{j\in J}}^b
C^\ell(Q_j,N_j)
\]
is a topological embedding.
Note that $Y:=\prod_{j\in J}C^\ell(M_j,U_j)$ is an open
subset of $\prod_{j\in J} C^\ell(M_j,N_j)$ which contains $f(\rho(\sigma))$;
thus $\Omega':=(f\circ\rho)^{-1}(Y)$ is an open neighbourhood
of~$\sigma$ in $\Gamma_{C^\ell}(M\leftarrow N)$.
It consists of all $\tau\in\Gamma_{C^\ell}(M\leftarrow N)$
such that
\begin{equation}\label{halfcond}
\theta_{j,2}(\tau(Q_j))\sub U_j\quad\mbox{for all $\, j\in J$.}
\end{equation}
Now
\[
g_j\colon C^\ell(Q_j,U_j)\to C^\ell(Q_j,V_j)\sub C^\ell(Q_j,F_j),\;\,
\gamma\mto \phi_j\circ\gamma
\]
is a homeomorphism, whence also
\[
g:=\prod_{j\in J}g_j\colon
{\prod_{j\in J}}^bC^\ell(Q_j,U_j)\to {\prod_{j\in J}}^b C^\ell(Q_j,V_j)
\]
is a homeomorphism. Hence
\[
g\circ f\circ \rho|_{\Omega'}\colon \Omega'\to {\prod_{j\in J}}^b C^\ell(Q_j,V_j)
\]
is a topological embedding.
We therefore find open neighbourhoods
$W_j$ of $\phi_j\circ \theta_{j,2}\circ \sigma|_{Q_j}$
in $C^\ell(Q_j,F_j)$
such that $(g\circ f\circ\rho)(\Omega\cap \Omega')$
contains the set
\[
(g\circ f\circ \rho)(\Omega')\cap\prod_{j\in J}W_j.
\]
After shrinking the $\sigma$-neighbourhood
$\Omega$, we may assume that $\Omega$ is the set of all
$\tau\in \Omega'$ such that
\begin{equation}\label{sechalfcond}
\phi_j\circ\theta_{j,2}\circ \tau|_{Q_j} \in W_j
\mbox{ for all $j\in J$.}
\end{equation}
For each $j\in J_0$, choose operators $S_{j,t}\colon C^\ell(Q_j,F_j)\to
C^r(Q_j,F_j)$ for $t\in\,]0,1]$ as in
Lemma~\ref{interpol} (identifying $\gamma\in C^\ell(Q_j,F_j)$ with
$\gamma\circ\kappa_j^{-1}\in C^\ell(L_j,F_j)$
and $\eta\in C^\infty(L_j,F_j)\sub C^r(L_j,F_j)$ with
$\eta\circ \kappa_j\in C^r(Q_j,F_j)$).
Let $S_{j,0}\colon C^\ell(Q_j,F_j)\to C^\ell(Q_j,F_j)$
be the identity map.
Choose $\xi_j\in C^r(Q_j,\R)$
with $\Supp(\xi_j)\sub Q_j^0$
such that $\xi_j(x)=1$ for $x$ in a neighbourhood
of~$P_j$ in $Q_j^0$, and $\xi_j(Q_j)\sub[0,1]$.\\[2.3mm]
Define $H_0\colon [0,1]\times M\to N$, $(t,x)\mto\sigma(x)$.
For $j\in J_0$,
we now construct $C^{0,\ell}$-mappings
\[
H_j\colon [0,1]\times M\to N
\]
with $H_j|_{]0,1]\times M}$ a $C^{r-\ell,\ell}$-map,
such that $H_j|_{]0,1]\times P_1^0\cup\cdots\cup P_j^0\cup V}$
is $C^{0,r}$
(if $r<\infty$), resp.,
$C^\infty$ (if $r=\infty$)
and, moreover,
$H_j(t,\cdot)\in\Omega$ for all
$t\in[0,1]$ and
\[
H_j(t,x)=H_{j-1}(t,x) \mbox{ for all $x\in M\setminus \Supp(\xi_j)$.}
\]
Suppose that~$H_{j-1}$ has already been constructed.
Since $H_{j-1}|_{[0,1]\times P_j}$
is an R$C^{0,\ell}$-map with $H_j([0,1]\times Q_j)\sub \theta_j^{-1}(Q_j\times
U_j)$ and $\theta_{j,2}$ as well as $\phi_j$
are $RC^r$ (and thus $RC^\ell$),
also the composition
\[
a_j\colon [0,1]\times Q_j\to F_j,\quad (t,x)\mto
(\phi_j\circ \theta_{j,2}\circ H_{j-1})(t,x)
\]
is $C^{0,\ell}$, by the Chain Rule.
As a consequence,
\[
a_j^\vee\colon [0,1]\to C^\ell(Q_j,F_j),\quad
t\mto a_j(t,\cdot)
\]
is a continuous map.
Hence
\[
K_j:=\{\phi_j\circ \theta_{j,2}\circ H_{j-1}(t,\cdot)|_{Q_j}\colon t\in [0,1]\}
=a_j^\vee([0,1])
\]
is a compact subset of~$W_j$.
The linear map
\[
O_{j,t}\colon C^\ell(Q_j,F_j)\to C^\ell(Q_j,F_j),\;\,
\gamma\mto (1-\xi_j)\cdot\gamma+\xi_j\cdot S_{j,t}(\gamma)
\]
is continuous for each $t\in [0,1]$.
Since multiplication operators are continuous linear and
hence uniformly continuous, we see that
\[
O_{j,t}(\gamma)\to\gamma
\]
in $C^\ell(Q_j,F_j)$ as $t\to 0$, uniformly in $\gamma\in K_j$.
We therefore find $t_j\in \,]0,1]$ such that
$O_{j,t}(\gamma)\in C^\ell(Q_j,V_j)$
for all $t\in [0,t_j]$ and $\gamma\in K_j$.
For $t\in [0,1]$ and $x\in M$, we
define
\[
h_j(s,t,x):=\left\{
\begin{array}{cl}
H_{j-1}(t,x) & \mbox{ if $x\in M\setminus \Supp(\xi_j)$;}\\
\theta_j^{-1}(x,\phi_j^{-1}(O_{j,s\cdot t}(a_j^\vee(t))(x)))
& \mbox{ if $x\in Q_j^0$.}
\end{array}
\right.
\]
Note that the map
\[
[0,1]^2\to C^\ell(Q_j,F_j),\quad
(s,t)\mto O_{t\cdot t_j}(a_j^\vee(t))
\]
is continuous, by Corollary~\ref{eval-spec}
and Remark~\ref{cts-smoo}.
Hence
\[
[0,1]^2\to C^\ell(Q_j^0,F_j),\quad (s,t)\mto \phi_j\circ
\theta_{j,2}\circ h_j(s,t,\cdot)|_{Q_j^0}
\]
is a continuous map and hence also
\[
[0,1]^2\to C^\ell(Q_j^0,N),\quad t\mto h_j(s,t,\cdot)|_{Q_j^0}.
\]
We find relatively compact, open subsets
$Z_a$ of~$M\setminus \Supp(\xi_j)$ for some index set~$A$
which together with $Q_j^0$ form a locally finite cover of~$M$.
Since
\[
[0,1]^2\to C^\ell(Z_a,N),\quad (s,t)\mto h_j(s,t,\cdot)|_{Z_a}
\]
is a constant map for each $a\in A$,
the map
\[
[0,1]^2\to C^\ell(Q_j^0,N)\times{\prod_{a\in A}}^bC^\ell(Z_a,N),\;\,
(s,t)\mto (h_j(s,t,\cdot)|_{Q_j^0},(h_j(s,t,\cdot)|_{Z_a})_{a\in A})
\]
is continuous. Hence $h_j^\vee\colon [0,1]^2\to C^\ell(M,N)$,
$(s,t)\mto h_j(s,t,\cdot)$ is continuous with respect to
the Whitney $C^\ell$-topology,
by the description in~(\ref{hereboxpro}).
Since $h_j^\vee(0,t)=H_{j-1}(t,\cdot)\in \Omega$
for all $t\in [0,1]$,
we find $s_j\in \,]0,1]$
such that $h_j(s,t,\cdot)\in\Omega$
for all $s\in [0,s_j]$ and $t\in [0,1]$.
Let $H_j:=h_j(s_j,\cdot)\colon [0,1]\times M\to N$; thus
\[
H_j(t,x):=\left\{
\begin{array}{cl}
H_{j-1}(t,x) & \mbox{ if $x\in M\setminus \Supp(\xi_j)$;}\\
\theta_j^{-1}(x,\phi_j^{-1}(O_{j,s_j\cdot t}(a_j^\vee(t))(x)))
& \mbox{ if $x\in Q_j^0$.}
\end{array}
\right.
\]
To see that $H_j$ is a $C^{0,\ell}$-map,
recall that $H_j^\vee\colon [0,1]\to \Omega$,
$t\mto H_j(t,\cdot)=h_j(s_j,t,\cdot)$
is continuous. As a consequence, the map
\[
c_{i,j}\colon [0,1]\to C^\ell(Q_i,F_i),\;\,
t\mto \phi_i\circ\theta_{i,2}\circ H_j(t,\cdot)|_{Q_i}
\]
is continuous for all $i\in J$.
Since~$Q_i$ is compact and hence locally compact,
the Exponential Law shows that
\[
c_{i,j}^\wedge\colon [0,1]\times Q_i\to F_i,\quad
(t,x)\mto c_{i,j}(t)(x)
\]
is a $C^{0,\ell}$-map. By the Chain Rule,
also
\[
H_j|_{[0,1]\times Q_i}\colon [0,1]\times Q_i\to N,\;\,
(t,x)\mto \theta_i^{-1}(x,\phi_i^{-1}(c_{i,j}^\wedge(t,x)))
\]
is a $C^{0,\ell}$-map, entailing that $H_j$ is $C^{0,\ell}$.\\[2.3mm]
We know that $H_j|_{]0,1]\times (M\setminus \Supp(\xi_j))}=
H_{j-1}|_{]0,1]\times (M\setminus\Supp(\xi_j))}$ is $C^{r-\ell,\ell}$.
Since $H_{j-1}|_{]0,1]\times Q_j}$
is $C^{r-\ell,\ell}$, also the map
\[
b_j\colon \,]0,1]\times Q_j\to F_j, \;\,
(t,x)\mto (\phi_j\circ\theta_{j,2}\circ H_{j-1})(t,x)
\]
is $C^{r-j,\ell}$. Hence
\[
b_j^\vee\colon \,]0,1]\to C^\ell(Q_j,F_j),\;\,
t\mto b_j(t,\cdot)
\]
is $C^{r-\ell}$.
As the mapping
\[
]0,1]\to \cL(C^\ell(Q_j,F_j),
C^\infty(Q_j,F_j))\sub \cL(C^\ell(Q_j,F_j),
C^\ell(Q_j,F_j)), \;\, t\mto S_{j,s_j\cdot t}
\]
is $C^\infty$ and hence $C^{r-\ell}$
by Lemma~\ref{interpol}, we deduce with
Corollary~\ref{eval-spec} that the map
\[
]0,1]\to C^\ell(Q_j,F_j),\;\,
t\mto S_{j,s_j\cdot t}(b_j^\vee(t))
\]
and hence also the map $t\mto\xi_j\cdot S_{j,s_j\cdot t}(b_j^\vee(t))$
is $C^{r-\ell}$. Since $Q_j$ is locally compact,
using the Exponential Law we deduce that
the second summand in
\[
O_{j,s_j\cdot t}(b_j^\vee(t))(x)=
(1-\xi_j(x))b_j(t,x)+\xi_j(x)S_{j,s_j\cdot t}(b_j^\vee(t))(x)
\]
is $C^{r-\ell,\ell}$ in $(t,x)\in \,]0,1]\times Q_j$;
the first summand is $C^{r-\ell,\ell}$ since $b_j$ is so.
As a consequence,
$H_j(t,x)=\theta_j^{-1}(x,\phi_j^{-1}(O_{j,s_j\cdot t}(b_j^\vee(t))(x)))$
is $C^{r-\ell,\ell}$ in $(t,x)\in \,]0,1]\times Q_j^0$.\\[2.3mm]
Let $R_j:=V\cup P_1^0\cup\cdots\cup P_j^0$.
We now show that $H_j|_{]0,1]\times R_j}$
is $C^{0,r}$ (if $r>\infty$), resp., $C^\infty$
(if $r=\infty$). Setting $R_{j-1}:=V\cup P_1^0\cup\cdots\cup
P_{j-1}^0$, we know that $H_{j-1}|_{]0,1]\times R_{j-1}}$
is $C^{0,r}$ (if $r<\infty$), resp., $C^\infty$ (if $r=\infty$).
Hence
\[
H_j|_{]0,1]\times (R_{j-1}\setminus \Supp(\xi_j))}
=H_{j-1}|_{]0,1]\times (R_{j-1}\setminus \Supp(\xi_j))}
\]
is $C^{0,r}$ and $C^\infty$,
respectively. It remains to show that $H_j|_{]0,1]\times (R_j\cap Q_j^0)}$
is $C^{0,r}$ and $C^\infty$, respectively.
Like $H_{j-1}$, the map
\begin{equation}\label{fisuma}
(t,x)\mto
(1-\xi_j(x))(\phi_j\circ \theta_{j,2}\circ b_j(t,x)
\end{equation}
is $C^{0,r}$ (resp., $C^\infty$)
for $(t,x)\in\,]0,1]\times R_{j-1}\cap Q_j^0$
and also for $(t,x)\in\,]0,1]\times P_j^0$,
as the map vanishes for the latter arguments.
Hence (\ref{fisuma}) is a $C^{0,r}$-map
(resp., $C^\infty$-map) $\,]0,1]\times (R_j\cap Q_j^0)\to F_j$.
We have shown that the first summand in
\[
O_{j,s_j\cdot t}(b_j^\vee(t))(x)=
(1-\xi_j(x))b_j(t,x)+\xi_j(x)S_{j,s_j\cdot t}(b_j^\vee(t))(x)
\]
is $C^{0,r}$ (resp., $C^\infty$)
in $(t,x)\in \,]0,1]\times (R_j\cap Q_j^0)$.
We show that the second summand
is $C^{0,r}$ (resp., $C^\infty$)
even for $(t,x)\in \,]0,1]\times Q_j$;
as a consequence,
$H_j(t,x)=
\theta_j^{-1}(x,\phi_j^{-}(O_{j,s_j\cdot t}(b_j^\vee(t))(x)))$
will be $C^{0,r}$ (resp., $C^\infty$)
in $(t,x)\in\, ]0,1]\times (R_j\cap Q_j)$,
as required.
The map $\,]0,1]\to \cL(C^\ell(Q_j,F_j),
C^\infty(Q_j,F_j))\sub \cL(C^\ell(Q_j,F_j),
C^r(Q_j,F_j))$, $t\mto S_{j,s_j\cdot t}$
is $C^\infty$ by Lemma~\ref{interpol} and hence $C^0$ (resp., $C^\infty$).
Since $b_j^\vee$ is $C^{r-\ell}$ and thus
$C^0$ (if $r<\infty$) and $C^\infty$ (if $r=\infty$)
respectively,
we deduce with
Corollary~\ref{eval-spec} that the map
\[
]0,1]\to C^r(Q_j,F_j),\;\,
t\mto S_{j,s_j\cdot t}(b_j^\vee(t))
\]
and hence also the map $t\mto\xi_j\cdot S_{j,s_j\cdot t}(b_j^\vee(t))$
is $C^0$ and $C^\infty$, respectively.
Since $Q_j$ is locally compact,
using the Exponential Law we deduce that the map
$]0,1]\times Q_j\to F_j$, $(t,x)\mto
\xi_j(x)S_{j,s_j\cdot t}(b_j^\vee(t))(x)$
is $C^{0,r}$ and $C^\infty$, respectively.\\[2.3mm]
If $n_0$ is finite, we let $H:=H_{n_0}$.
If $n_0=\infty$, given $j\in J_0=\N$
we find $i_j>j$ such that
\[
Q_i\cap Q_j=\emptyset\mbox{ for all $i\geq i_j$.}
\]
Thus $H_i|_{[0,1]\times Q_j^0}=H_{i_j}|_{[0,1]\times Q_j^0}$
for all $i\geq i_j$, showing that
\[
H(t,x):=\lim_{i\to\infty}H_i(t,x)
\]
exists for all $(t,x)\in [0,1]\times Q_j^0$.
As $j$ was arbitrary, we obtain a function $H\colon [0,1]\times M\to N$.
By the preceding,
\[
H|_{[0,1]\times Q_j^0}=H_{i_j}|_{[0,1]\times Q_j^0}\quad\mbox{for all $j\in\N$.}
\]
Thus $H(t,\cdot)\in\Gamma_{C^\ell}(M\leftarrow N)$
for each $t\in [0,1]$. Moreover, we see that
$H$ (like each $H_{i_j}$) is $C^{0,\ell}$.
Also, if we set $W:=V\cup\bigcup_{j\in\N}P_j^0$,
then $H|_{]0,1]\times W}$ is $C^{0,r}$ (if $r<\infty$)
resp.\ $C^\infty$ (if $r=\infty$),
and~$H$ has all of the asserted properties.
$\,\square$
\section{Smoothing in algebraic topology}\label{alg-top}
In this section,
we compile some typical applications of smoothing
results in algebraic topology.
See also~\cite{Ste} and~\cite{MaW}.\\[2.3mm]
Smooth homotopies can be juxtaposed
to smooth homotopies if they are constant near
$t=0$ and $t=1$ (see \ref{collar}(d)), as is well known.
We shall use the following terminology
for this standard idea.
\begin{numba}\label{collar}
(a) If $X$ and $Y$ are topological
spaces, we say that a continuous map
$F\colon [0,1]\times X\to Y$ is a
\emph{homotopy with collar}
if there exists $\ve\in \,]0,\frac{1}{2}[$ such that
$F(t,x)=F(0,x)$ for all $t\in [0,\ve]$ and $x\in X$,
and $F(t,x)=F(1,x)$ for all $t\in [1-\ve,1]$
and $x$ in~$X$.\\[2.3mm]
(b) If there exists a homotopy $F$ from $\gamma\colon X\to Y$
to $\eta\colon X\to Y$, then there also exists a
homotopy with collar from $\gamma$ to $\eta$.
To this end, pick a smooth function $\tau\colon [0,1]\to\R$
with image in $[0,1]$ such that $\tau|_{[0,\frac{1}{3}]}=0$
and $\tau|_{[\frac{2}{3},1]}=1$.
Then
\[
F^c\colon [0,1]\times X\to Y,\quad (t,x)\mto F(\tau(t),x)
\]
has the desired properties. If $F$ is a homotopy relative
$A\sub X$, then also~$F^c$.\\[2.3mm]
(c) If $X$ and $Y$ are $C^\infty$-manifolds with rough boundary
and $F$ is
a homotopy with collar such that $F(0,\cdot)$
and $F(1,\cdot)$ are smooth, then $F$ is smooth on
$([0,\ve[\,\cup \,]1-\ve,1])\times X$ for some $\ve\in\, ]0,\frac{1}{2}[$.
Also note that if $F$ is smooth in the situation of~(b),
then also $F^c$ is smooth.\\[2.3mm]
(d) If $F, G\colon [0,1]\times X\to Y$
are smooth homotopies with collar such that $G(0,\cdot)=F(1,\cdot)$,
then the juxtaposed homotopy $F*G\colon [0,1]\times X\to Y$,
$(F*G)(t,x):=F(2t,x)$ for $t\in[0,\frac{1}{2}]$,
$(F*G)(t,x):=G(2t-1,x)$ for $t\in [\frac{1}{2},1]$
is smooth.\\[2.3mm]
(e) Let $X$ and $Y$ be $C^\infty$-manifolds with rough boundary
and $\gamma,\eta,\zeta\colon X\to Y$
be smooth maps. If there exists
a smooth homotopy $F$ from $\gamma$ to $\eta$ and a smooth homotopy
$G$ from $\eta$ to $\zeta$, then $F^c*G^c$ is a smooth homotopy from
$\gamma$ to $\eta$. If both $F$ and $G$ are homotopies
relative~$A$ for a subset $A\sub X$, then also $F^c*G^c$
is a homotopy relative~$A$.\\[2.3mm]
(f)
Let $(X,x_0)$ and $(Y,y_0)$ be pointed topological spaces.
We call a continuous map
$f\colon X\to Y$ 
a \emph{pointed map with collar}
if $f|_W=y_0$ for an $x_0$-neighbourhood $W\sub X$.\\[2.3mm]
(g) If $X$ and $Y$ are $C^\infty$-manifolds with rough boundary,
then a pointed map $(X,x_0)\to (Y,y_0)$
with collar is smooth on an open $x_0$-neighbourhood.\\[2.3mm]
(h) Let $(Y,y_0)$ be a pointed topological
space, $X$ be a locally compact $C^\infty$-manifold
with rough boundary and $x_0\in X$.
Then every pointed mapping\linebreak
$\gamma\colon (X,x_0)\to (Y,y_0)$
is homotopic relative~$\{x_0\}$
to a pointed map with collar.\\[2.3mm]
To see this, let
$\phi\colon P\to Q\sub\R^d$ be a chart for~$X$
with $x_0\in P$ such that $\phi(x_0)=0$.
After shrinking~$Q$, we may assume that~$Q$ is convex.
There exists a compact
$0$-neighbourhood $K\sub Q$
and a function $h\in C^\infty_c(Q,\R)$
with image in $[0,1]$
such that $h|_K=1$. Abbreviate $\chi:=1-h$
and $L:=\Supp(h)$.
Then $F\colon [0,1]\times X\to Y$,
\begin{equation}\label{nicehomot}
\;\,
(t,x)\mto
\left\{\begin{array}{cl}
\gamma(x) & \mbox{ if $\, x\in M\setminus \phi^{-1}(L)$;}\\
\gamma(\phi^{-1}((1-t)\phi(x)+t\chi(\phi(x))\cdot\phi(x)))&
\mbox{ if $x\in P$}
\end{array}\right.
\end{equation}
is a homotopy relative $\{x_0\}$
from $\gamma$ to $\gamma^c:=F(1,\cdot)$.
Note that $\gamma^c|_Z$ is constant (with value $x_0$)
on the $x_0$-neighbourhood
$Z:=\phi^{-1}(K)$.
\end{numba}
\begin{example}
Let
$N$ be a smooth manifold modelled on locally convex
spaces, $n_0\in N$,
and $k\in\N_0$. Let $e:=(0,\ldots, 0,1)\in \Sph_k\sub\R^{k+1}$.
Then we have:\\[2.3mm]
(a) \emph{Every homotopy class $[\gamma]\in \pi_k(N,n_0)
=[(\Sph_k,e),(N,n_0)]_*$ contains a smooth representative
$\eta\colon \Sph_k\to N$.}\\[2.3mm]
(b) \emph{If two smooth maps $\alpha\colon \Sph_k\to N$
and $\beta\colon \Sph_k\to N$ with
$\alpha(e)=\beta(e)=n_0$ are homotopic relative~$\{e\}$,
then there exists a smooth base-point preserving homotopy from
$\alpha$ to $\beta$.}\\[2.3mm]
[To prove~(a), let $\gamma\colon\Sph_k\to N$ be a continuous map
such that $\gamma(e)=n_0$.
By (h) and (g) in~\ref{collar},
we may assume that $\gamma$ is smooth on an open neighbourhood~$Y$ of $e$ in $\Sph_k$.
Applying Corollary~\ref{smoothemps}
to $\gamma$
with $r=\infty$, $\ell=0$, $A:=M:=\Sph_k$, $U:=\Sph_k\setminus\{e\}$
and $\Omega:=C(M,N)$,
we find a homotopy $H\colon [0,1]\times \Sph_k\to N$
from~$\gamma$ to a smooth map $\eta\colon\Sph_k\to N$
such that $H(t,e)=\gamma(e)=n_0$ for all $t\in[0,1]$.\\[2.3mm]
(b) Let $\phi\colon P\to Q$ be a chart of $\Sph_k$
such that $Q\sub \R^k$ is convex, $e\in P$, and $\phi(e)=0$.
Let $\chi$, $K$, and $L$ be as in \ref{collar}(h)
(with $X:=\Sph_k$, $Y:=N$).
Suppose that $\alpha$ and $\beta$ are smooth maps $\Sph_k\to N$
with $\alpha(e)=\beta(e)=n_0$
such that their exists a continuous homotopy~$F$ relative $\{e\}$
from~$\alpha$ to~$\beta$.
As the homotopy in (\ref{nicehomot}),
applied with $\alpha$ in place of~$\gamma$,
and the corresponding one for~$\beta$
are smooth, using~\ref{collar}(e)
it suffices to find a smooth homotopy relative~$\{e\}$
between
$\alpha^c$ and $\beta^c$.
Now
$F^c$ (as in \ref{collar}(b))
is a homotopy relative $\{e\}$
with collar from~$\alpha$ to~$\beta$.
Moreover, $F^c$
is smooth on the open neighbourhood
$Z:=([0,\frac{1}{3}[\,\cup\,]\frac{2}{3},1])\times\Sph_k$
of $\{0,1\}\times \Sph_k$ in $[0,1]\times\Sph_k$.
Now
\[
G\colon [0,1]\times \Sph_k\to N,\;
(t,x)\mto
\left\{\begin{array}{cl}
F^c(t,x) & \mbox{ if $\, x\in \Sph_k\setminus \phi^{-1}(L)$;}\\
F^c(t,\phi^{-1}(\chi(\phi(x))\cdot\phi(x)))&
\mbox{ if $x\in P$}
\end{array}\right.
\]
is a homotopy relative $\{e\}$ with colar
from $\alpha^c$ to $\beta^c$
such that $G|_{[0,1]\times \phi^{-1}(K)}=n_0$.
Thus $G$ is smooth on the open neighbourhood
$Z\cup ([0,1]\times \phi^{-1}(K^0))$
of the closed subset
$C:=(\{0,1\}\times \Sph_k)\cup ([0,1]\times \{e\})$
in $A:=M:=[0,1]\times\Sph_k$.
Applying Corollary~\ref{smoothemps}
with $G$ in place of~$\gamma$,
$\Omega:=C(M,N)$ and $U:=([0,1]\times\Sph_k)\setminus C$,
we get a smooth function $\eta\colon [0,1]\times\Sph_k\to N$
such that $\eta|_C=G|_C$, whence $\eta(0,\cdot)=\alpha^c$,
$\eta(1,\cdot)=\beta^c$, and $\eta(t,e)=n_0$ for all $t\in [0,1]$.\,]
\end{example}
\begin{defn}
Let $\ell\in\N_0\cup\{\infty\}$.
We say that a rough $C^\ell$-manifold~$N$
modelled on locally convex spaces is \emph{$RC^\ell$-contractible}
if there exists a homotopy
$F\colon [0,1]\times N\to N$
from $\id_N$ to the constant function
$c_{x_0}\colon N\to N$, $x\mto x_0$
for some $x_0\in N$, such that~$F$ is an $RC^\ell$-map.
If~$N$ is a $C^\ell$-manifold
with rough boundary modelled on locally
convex spaces
and there exists a homotopy
$F\colon [0,1]\times N\to N$
from $\id_N$ to the constant function
$c_{x_0}\colon N\to N$, $x\mto x_0$
for some $x_0\in N$, such that~$F$ is a $C^\ell$-map,
then~$N$ is called \emph{$C^\ell$-contractible}
(or also \emph{smoothly contractible}, if $\ell=\infty$).
\end{defn}
\begin{example}\label{algtopexs}
(a) \emph{If a $\sigma$-compact finite-dimensional
smooth manifold~$N$ is contractible,
then~$N$ is also smoothly contractible.}\\[2.3mm]
[Let $F\colon [0,1]\times N\to N$
be a homotopy from $\id_N$ to the constant function
$c_{x_0}\colon N\to N$, $x\mto x_0$
for some $x_0\in N$.
Then $F^c$ (as in \ref{collar}(b))
is a homotopy with collar from $\id_N$ to $c_{x_0}$
and smooth on $([0,\frac{1}{3}[\,\cup\, ]\frac{2}{3},1])\times M$.
Applying Corollary~\ref{smoothemps}
with $A:=M:=[0,1]\times N$,
$\gamma:=F^c$,
$U:=\,]\frac{1}{4},\frac{3}{4}[\,\times M$
and $\Omega:=C(M,N)$,
we obtain a smooth function $\eta\colon [0,1]\times N\to N$
which coincides with~$F^c$ on $([0,\frac{1}{4}]\cup[\frac{3}{4},1])\times N$
and thus is a homotopy from $\id_N$ to~$c_{x_0}$.\,]\\[2.3mm]
(b) \emph{If a singleton $\{x_0\}$ is a strong deformation retract of~$N$,
then there exists a smooth map $f\colon N\to N$
which is a
strong deformation retraction to~$\{x_0\}$.}\\[2.3mm]
[Let $\phi\colon P\to Q$, $\chi$, $K$, and $L$ be as in \ref{collar}(h),
with $X:=Y:=N$.
For $\gamma:=\id_N$, the homotopy~$F$
from $\id_X$ to $\id_X^c$
defined in (\ref{nicehomot}) is smooth and a homotopy
relative~$\{x_0\}$.
If $G$ is a homotopy relative~$\{x_0\}$
from $\id_N$ to~$c_{x_0}$,
then also $G^c$ is so. Moreover,
$G^c$ is smooth on the open neighbourhood
$Z:=([0,\frac{1}{3}[\,\cup\,]\frac{2}{3},1])\times N$
of $\{0,1\}\times N$ in $[0,1]\times N$.
Now
\[
J\colon [0,1]\times N\to N,\;
(t,x)\mto
\left\{\begin{array}{cl}
G^c(t,x) & \mbox{ if $\, x\in N \setminus \phi^{-1}(L)$;}\\
G^c(t,\phi^{-1}(\chi(\phi(x))\cdot\phi(x)))&
\mbox{ if $x\in P$}
\end{array}\right.
\]
is a homotopy relative $\{x_0\}$ with colar
from $\id^c$ to $c_{x_0}$
such that $J|_{[0,1]\times \phi^{-1}(K)}=x_0$.
Thus $J$ is smooth on the open neighbourhood
$Z\cup ([0,1]\times \phi^{-1}(K^0))$
of the closed subset
$C:=(\{0,1\}\times N)\cup ([0,1]\times \{x_0\})$
in $A:=M:=[0,1]\times N$.
Applying Corollary~\ref{smoothemps}
with~$J$ in place of~$\gamma$,
$\Omega:=C(M,N)$, and $U:=([0,1]\times N)\setminus C$,
we get a smooth function $\eta\colon [0,1]\times N\to N$
such that $\eta|_C=J|_C$, whence $\eta(0,\cdot)=\id_X^c$,
$\eta(1,\cdot)=c_{x_0}$, and  $\eta(t,e)=x_0$ for all $t\in [0,1]$.
Let $H:=\eta$.
Now $F^c*H^c$ is a smooth homotopy relative $\{x_0\}$
from $\id_N$ to $c_{x_0}$.\,]\\[2.3mm]
As before, let $N$ be a $\sigma$-compact, finite-dimensional smooth
manifold.
Let $S\sub N$ be a submanifold which is a closed
subset.\\[2.3mm]
(c)
\emph{If $S$ is a retract of~$N$,
then there exists a smooth retraction $\eta\colon N\to S$.}\\[2.3mm]
[We shall use the fact that~$S$ has a tubular
neighbourhood in~$N$ (see, e.g., \cite{Lan}).
Thus, there exists an open subset $T\sub N$
with $S\sub T$ and a $C^\infty$-diffeomorphism
$\psi\colon T\to B$ onto an open subset $B\sub E$
for some smooth vector bundle $\pi\colon E\to S$
such that $B$ contains $0_x\in E_x$
for all $x\in S$. After shrinking~$B$,
we may assume\footnote{Each $x\in S$
has an open neighbourhood $U_x\sub S$ such that
$V_x:=\theta^{-1}_x(U_x\times B_x)\sub B$
for some local trivialization $\theta_x$ of~$E$
around~$x$ and some balanced, open $0$-neighbourhood~$B_x$
in the typical fibre of the vector bundle at~$x$.
Replace~$B$ with the union of the $V_x$.}
that $tv\in B$ for all $v\in B$ and
$t\in [0,1]$, enabling us to define the smooth map
\[
[0,1]\times T\to T,\quad (t,x)\mto tx\cdot :=\psi^{-1}(t\psi(x)).
\]
Since $N$ is a normal topological
space, there exists a neighbourhood $C$ of~$S$ in~$T$
which is closed in~$M$. There exists $h\in C^\infty(N,\R)$
with $h(N)\sub[0,1]$ and $\Supp(h)\sub C$
such that such that $h|_O=1$ for $x$ for some open neighbourhood~$O$
of~$S$ in~$T$. Set $\chi:=1-h$.
If $r\colon N\to S$ is a continuous retraction, then
\[
\gamma\colon N\to S,\quad
x\mto\left\{\begin{array}{cl}
r(x) &\mbox{ if $x\in N\setminus \Supp(h)$;}\\
r(\chi(x)\cdot x) &\mbox{ if $x\in T$}
\end{array}\right.
\]
is a retraction and smooth on~$O$,
as $\gamma(x)=r(0.x)=\pi(\psi(x))$.
Applying Corollary~\ref{smoothemps}
with $A:=M:=N$, $S$ in place of~$N$,
$U:=N\setminus S$ and $\Omega:=C(N,S)$,
we get a smooth function $\eta\colon N\to S$
such that $\eta|_S=\gamma|_S$
and thus $\eta|_S=\id_S$.
Moreover, there is a homotopy $H\colon [0,1]\times N\to S$
from~$\gamma$ to~$\eta$
such that $H(t,\cdot)|_S=\gamma|_S=\id_S$
for all $t\in[0,1]$ and $H|_{]0,1]\times N}$
is smooth, which will be useful in~(d).\,]\\[2.3mm]
(d)
Let $j\colon S\to N$ be the inclusion map.\\[2.3mm]
\emph{If $S$ is a deformation retract of~$N$,
then there exists a smooth homotopy
$[0,1]\times N\to N$
from $\id_N$ to $j\circ \rho$
for a retraction $\rho\colon N\to S$.}\\[2.3mm]
[Let the tubular neighbourhood and corresponding notation be as
in the proof of~(b).
Then
\[
R\colon [0,1]\times N\to N,\;
x\mto\left\{\begin{array}{cl}
x &\mbox{ if $x\in N\setminus \Supp(h)$;}\\
\psi^{-1}((1-t)\psi(x)+t\chi(x)\psi(x)) &\mbox{ if $x\in T$}
\end{array}\right.
\]
is a smooth homotopy relative~$S$ from~$\id_M$ to $R(1,\cdot)$.
Let $F\colon [0,1]\times N\to N$ be a
homotopy from $\id_N$ to $j\circ r$ for
a retraction $r\colon N\to S$. Then
\[
G\colon N\to N,\quad
(t,x)\mto\left\{\begin{array}{cl}
F(t,x) &\mbox{ if $x\in N\setminus \Supp(h)$;}\\
F(t,\chi(x)\cdot x) &\mbox{ if $x\in T$}
\end{array}\right.
\]
is a homotopy from $R(1,\cdot)$ to $j\circ \gamma$ for
a retraction $\gamma\colon N\to S$
which is smooth on the tubular neighbourhood
$T$ of~$S$ in~$N$.
By the proof of~(c),
there is a homotopy
$H\colon [0,1]\times N\to S$ from $\gamma$ to a retraction
$\rho:=H(1,\cdot)$ such that $H(t,\cdot)|_S=\id_S$
for all $t\in[0,1]$
and $H|_{]0,1]\times N}$ is smooth.
Then
$(R^c*G^c)*(j\circ H)^c\colon [0,1]\times N\to N$
is a homotopy with collar from $\id_N$ to $j\circ H(1,\cdot)$
which is smooth on the open neighbourhood
$([0,\frac{1}{4}[\,\times N)\cup (]\frac{2}{3},1]\times N)$
of the closed subset $C:=\{0,1\}\times N$
of $A:=M:=[0,1]\times N$.
Applying Corollary~\ref{smoothemps}
with $U:=M\setminus C$,
$(R^c*G^c)*H^c$ in place of~$\gamma$ and $\Omega:=C(M,N)$,
we obtain a smooth function $\eta\colon [0,1]\times N\to N$
such that $\eta|_C=(R^c*G^c)*H^c|_C$ and thus $\eta(0,\cdot)=\id_N$
and $\eta(1,\cdot)=j\circ H(1,\cdot)$.\,]\\[2.3mm]
(e) \emph{If $S$ is a strong deformation retract of~$N$,
then there exists a smooth homotopy
$[0,1]\times N\to N$ relative~$S$
from $\id_N$ to $j\circ\rho$ for
a retraction $\rho\colon N\to S$.}\\[2.3mm]
[In fact, we can then choose $F$ as a homotopy relative~$S$
from $\id_N$ to $j\circ r$ in the proof of~(d).
Then also $G$ is a homotopy relative~$S$.
Since also~$R$ and~$H$ are homotopies relative~$S$, so are
$R^c$, $G^c$, $H^c$, and $(R^c*G^c)*H^c$. Moreover,
$(R^c*G^c)*H^c$ is smooth
on the open neighbourhood
$([0,\frac{1}{4}[\,\times N)\cup ([0,1]\times T)
\cup (]\frac{3}{4},1]\times N)$
of the closed subset $C:=(\{0,1\}\times N)\cup ([0,1]\times S)$
in $A:=M:=[0,1]\times N$.
Applying Corollary~\ref{smoothemps}
with $U:=M\setminus C$,
$(R^c*G^c)*H^c$ in place of~$\gamma$
and $\Omega:=C(M,N)$,
we obtain a smooth function $\eta\colon [0,1]\times N\to N$
such that $\eta|_C=(R^c*G^c)*H^c|_C$ and thus $\eta(0,\cdot)=\id_N$,
$\eta(1,\cdot)=j\circ \rho$, and $\eta(t,x)=x$ for all $t\in[0,1]$
and $x\in S$.\,]\\[2.3mm]
(f) Let us give an easy example:
If $E$ is a locally convex space
and $R\sub E$ a regular subset which is star-shaped,
then $R$ is $RC^\infty$-contractible
(in fact there exists a strong deformation
retraction from $R$ onto $\{x_0\}$ which is an $RC^\infty$-map).\\[2.3mm]
[We may assume that $x_0=0$. The map
\[
H\colon R\times [0,1]\to R,\;\,
(x,t)\mto tx
\]
is smooth as a map to~$E$ and takes
$(t,x)$ in the interior $]0,1[\,\times R^0$
of its domain to $tx\in tR^0\sub R^0$;
it therefore is an $RC^\infty$-map.
Moreover, $H(0,x)=0$ for all $x\in R$,
$H(1,x)=x$ and $H(t,0)=0$ for all $t\in[0,1]$.\,]
\end{example}
We close with related results concerning pull-backs
of vector bundles. In particular,
we shall see that all smooth vector bundles
over an $RC^\infty$-contractible, rough $C^\infty$-manifold
are trivial.
Our arguments vary those in~\cite[pp.\ 20--21]{Hat},
where finite-dimensional topological vector bundles
over paracompact topological spaces are considered.
We begin with two lemmas.
\begin{la}\label{hat1}
Let $\ell\in \N_0\cup\{\infty\}$,
$M$ be a rough $C^\ell$-manifold
modelled on locally convex spaces.
Let $a<b$ be real numbers
and $\pi\colon E\to ([a,b]\times M)$
be a $C^\ell$-vector bundle over $[a,b]\times M$
whose typical fibre is a locally convex space~$F$.
If there exist real numbers $\alpha<\beta$ in $\,]a,b[$
such that $E|_{[a,\beta[\,\times M}$
and $E|_{]\alpha,b]\times M}$ are trivializable
$C^\ell$-vector bundles, then~$E$
is $C^\ell$-trivializable.
\end{la}
\begin{proof}
Let $\theta_1\colon E|_{[a,\beta[\,\times M}\to
[\alpha,\beta[\,\times M\times F$
and $\theta_2\colon E|_{]\alpha,b]\times M}\to
\,]\alpha,b]\times M\times F$ be $C^\ell$-trivializations,
with second components $\theta_{1,2}$ and $\theta_{2,2}$,
respectively.
Then
\begin{equation}\label{trivcha}
]\alpha,\beta[\,\times M\times F\to \, ]\alpha,\beta[\,\times M\times F,
\;(t,x,y)\mto \theta_i(\theta_j^{-1}(t,x,y))
\end{equation}
is a $C^\ell$-diffeomorphism for $(i,j)\in\{(1,2),(2,1)\}$
and linear in the final argument.
We let $g_{i,j}\colon \,]\alpha,\beta[\,\times M\times F\to F$
be the third component of the diffeomorphism in~(\ref{trivcha});
thus
\[
\theta_1(\theta_2^{-1}(t,x,y))=(t,x,g_{1,2}(t,x,y)).
\]
Pick $r<s$ in $]\alpha,\beta[$.
There is a $C^\ell$-map $\tau\colon \,]\alpha,b]\to\R$
such that $\tau$ is monotonically increasing,
$\tau(t)=s$ for all $t\in [s,b]$
and $\tau(t)=t$ for all $t\in \,]\alpha,r]$.
Define $\theta\colon E\to [a,b]\times M\times F$
via
\[
\theta(v):=\left\{
\begin{array}{cl}
\theta_1(v) &\mbox{ if $t\in [a,r[$;}\\
(t,x,g_{1,2}(\tau(t),x,\theta_{2,2}(v)) &\mbox{ if $t\in \,]\alpha,b]$}
\end{array}\right.
\]
for $v\in E$,
with
$(t,x):=\pi(v)$.
Then $\theta$ is a $C^\ell$-diffeomorphism,
as we readily check that the inverse is the map taking
$(t,x,y)\in [a,b]\times M\times F$ to
$\theta_1^{-1}(t,x,y)$
and
$\theta_2^{-1}(t,x,g_{2,1}(\tau(t),x,y))$
if $t\in[a,r[$ and $t\in \,]\alpha,b]$,
respectively. Then $\theta$ is a global
$C^\ell$-trivialization for~$E$.
\end{proof}
\begin{la}\label{hat2}
Let $\ell\in \N_0\cup\{\infty\}$,
$M$ be a rough $C^\ell$-manifold
modelled on locally convex spaces,
$a<b$ be real numbers,
and $\pi\colon E\to [a,b]\times M$
be a $C^\ell$-vector bundle over $[a,b]\times M$
whose typical fibre is a locally convex space~$F$.
Then each $x\in M$ has an open neighbourhood
$U\sub M$ such that $E|_{[a,b]\times U}$
is $C^\ell$-trivializable.
\end{la}
\begin{proof}
For each $s\in [a,b]$,
there exists an open neighbourhood $V_s$ of $(s,x)$ in
$[a,b]\times M$ such that $E|_{V_s}$
is $C^\ell$-trivializable. We may assume
that $V_s=J_s\times U_s$
for open subsets $J_s\sub [a,b]$
and $U_s\sub M$.
Let $\delta>0$
be a Lebesgue number for the open cover $(J_s)_{s\in[a,b]}$
of the compact metric space $[a,b]$.
Pick $a=t_0<t_1<\cdots<t_n=b$ such that
$t_j-t_{j-1}<\delta/3$ for all $j\in\{1,\ldots,n\}$.
Set $t_{-1}:=t_0$ and $t_{n+1}:=t_n$.
For each $j\in\{1,\ldots,n\}$,
we find $s_j\in [a,b]$ such that
\[
[t_{j-2},t_{j+1}]\sub J_{s_j},
\]
as the interval on the left has length $<\delta$.
Then
\[
U:=U_{s_1}\cap\cdots\cap U_{s_n}
\]
is an open neighbourhood of~$x$ in~$M$.
A straightforward induction based
on Lemma~\ref{hat1} shows that
$E|_{[a,t_{j+1}]\times U}$ is $C^\ell$-trivializable
for all $j\in\{1,\ldots, n\}$.
Notably, $E|_{[a,b]\times U}$
is $C^\ell$-trivializable.
\end{proof}
If $M$ is a rough $C^\ell$-manifold
and $\pi_j\colon E_j\to M$
are $C^\ell$-vector bundle
for $j\in\{1,2\}$ with locally convex fibres,
we call a map $f\colon E_1\to E_2$
an \emph{isomorphism of $C^\ell$-vector bundles over $\id_M$}
if $f$ is a $C^\ell$-diffeomorphism,
$\pi_2\circ f=\pi_1$ and the restriction
of $f$ to a map $(E_1)_x\to (E_2)_x$
is linear for all $x\in M$.
\begin{prop}\label{eq-boundary}
Let $\ell\in\N_0\cup\{\infty\}$
and $M$ be a $C^\ell$-paracompact $C^\ell$-manifold
with rough boundary modelled on locally convex spaces.
Let $a<b$ be real numbers and $\pi\colon E\to [a,b]\times M$
be a $C^\ell$-vector bundle
whose fibres are locally convex spaces.
For $t\in [a,b]$, let $\lambda_t\colon M\to [a,b]\times M$
be the map $x\mto (t,x)$.
Then there exists a $C^\ell$-vector bundle
isomorphism $\lambda_a^*(E)\to\lambda_b^*(E)$
over $\id_M$. 
\end{prop}
\begin{proof}
Abbreviate $I:=[a,b]$.
By Lemma~\ref{hat2},
$M$ admits a cover $\cU$ by open sets $U\sub M$
such that $E|_{I\times U}$
is $C^\ell$-trivializable.
By $C^\ell$-paracompactness,
we find a $C^\ell$-partition of
unity $(h_j)_{j\in J}$ on~$M$
such that $S(j):=\Supp(h_j)\sub U(j)$ for some $U(j)\in\cU$,
for each $j\in J$.
Let $\cF$ be the set of finite subsets of~$J$.
For $\Phi\in\cF$, we define
\[
h_\Phi:=a+(b-a)\sum_{j\in\Phi}h_j\colon M\to I
\]
and consider the pullback bundle
$f_\Phi^*(E)$
over~$M$ determined by the $C^\ell$-map $f_\Phi\colon
M\to I \times M$, $x\mto (h_\Phi(x),x)$.
Thus $f_\emptyset=\lambda_a$.
For $\Phi\in\cF$, the interior
\[
W(\Phi):=\{x\in M\colon h_\Phi(x)=b\}^0
\]
is an open subset of~$M$ such that
the open $C^\ell$-vector subbundle $f_\Phi^*(E)|_{W(\Phi)}$
of $f_\Phi^*(E)$ coincides as a $C^\ell$-vector bundle
over $W(\Phi)$ with the open $C^\ell$-vector subbundle
$\lambda_b^*(E)|_{W(\Phi)}$ of $\lambda_b^*(E)$; thus
\[
f_\Phi^*(E)|_{W(\Phi)}=\lambda_b^*(E)|_{W(\Phi)}.
\]
Likewise,
\[
f_\Phi^*(E)|_{M\setminus S(\Phi)}=\lambda_a^*(E)|_{S(\Phi)}
\]
With $S(\Phi):=\bigcup_{j\in\Phi}\Supp(h_j)$.
We now construct a family $(\psi_\Phi)_{\Phi\in\cF}$
of $C^\ell$-vector bundle isomorphisms
\[
\alpha_\Phi\colon \lambda_a^*(E)\to f_\Phi^*(E)
\]
over $\id_M$
such that, for all
$\Phi,\Theta\in\cF$ with
$\Psi\sub\Theta$,
we have
\begin{equation}\label{loc-statio}
\alpha_\Phi(v)=\alpha_\Theta(v)\mbox{ for all $v\in \lambda_a^*(E)|_{W(\Phi)}$.}
\end{equation}
Once this is accomplished, we get a well-defined map
\[
\alpha\colon \lambda_a^*(E)\to \lambda_b^*(E)
\]
if we send $v\in \lambda_a^*(E)$ to
\[
\alpha(v):=\alpha_\Phi(v),
\]
independent of the choice of $\Phi\in \cF$
such that $v\in \lambda_a^*(E)|_{W(\Phi)}$.
By construction, $\alpha$ is an isomorphism
of $C^\ell$-vector bundles over~$\id_M$.\\[2.3mm]
To construct the isomorphisms $\alpha_\Phi$,
we fix a total order $\leq$ on~$J$
(e.g., a well-ordering).
For $j\in J$, we let
$\theta_j\colon E|_{I\times U(j)}\to I \times U(j)\times F(j)$
be a $C^\ell$-trivialization of $E|_{I\times U(j)}$,
with second component $\theta_{j,2}\colon E|_{I\times U(j)}\to F(j)$.
We let $\alpha_\emptyset$ be the identity map $\lambda_a^*(E)\to \lambda_a^*(E)$.
If $\Phi\in\cF$ has $n>1$ elements,
we write $\Phi=\{j_1,\ldots, j_n\}$
with $j_1<\cdots < j_n$ and set $\Psi:=\{j_1,\ldots, j_{n-1}\}$.
To define a $C^\ell$-vector bundle isomorphism
\begin{equation}\label{stepbyst}
\alpha_{\Phi,\Psi}\colon f_\Psi^*(E)\to f_\Phi^*(E)
\end{equation}
over $\id_M$, recall that
the local trivialization $\theta_{j_n}$
of~$E$ yields local trivializations
\[
\theta_{j_n}^\Phi\colon f_\Phi^*(E)|_{U(j_n)}\to U(j_n)\times F(j_n),\;
(x,y)\mto (x,\theta_{j_n,2}(y))
\]
and
\[
\theta_{j_n}^\Psi\colon f_\Psi^*(E)|_{U(j_n)}\to U(j_n)\times F(j_n),\;
(x,y)\mto (x,\theta_{j_n,2}(y))
\]
of the pullback bundles $f_\Phi^*(E)$
and $f_\Psi^*(E)$, respectively
(for $x\in U(j_n)$ and $y\in E_{f_\Phi(x)}$,
resp., $y\in E_{f_\Psi(x)}$).
Note that
\[
(\theta_{j_n}^\Phi)^{-1}(x,z)=(x,\theta_j^{-1}(h_\Phi(x),x,z))
\]
for $(x,z)\in U(j_n)\times F(j_n)$,
and an analogous formula holds with~$\Psi$ in place
of~$\Phi$. Since $h_\Phi(x)=h_\Psi(x)$
if $x\in U(j_n)\setminus S(j_n)$, we see that
\begin{equation}\label{hencewelldf}
(\theta_{j_n}^\Phi)^{-1}|_{(U(j_n)\setminus S(j_n))\times F(j_n)}
=
(\theta_{j_n}^\Psi)^{-1}|_{(U(j_n)\setminus S(j_n))\times F(j_n)}.
\end{equation}
We define a map $\alpha_{\Phi,\Psi}$ as in (\ref{stepbyst})
via $\alpha_{\Phi,\Psi}(v):=v$ if $v\in f_\Psi^*(E)|_{M\setminus S(j_n)}$
and
\[
\alpha_{\Phi,\Psi}(v):=((\theta_{j_n}^\Phi)^{-1}\circ\theta_{j_n}^\Psi)(v)
\]
for $v\in f_\Psi^*(E)|_{U(j_n)}$;
the map is well defined by~(\ref{hencewelldf}).
By construction, $\alpha_{\Phi,\Psi}$ is a $C^\ell$-vector
bundle isomorphism over~$\id_M$.
Moreover,
\begin{equation}\label{thusgood}
Y:=f_\Phi^*(E)|_{M\setminus S(j_n)}=f_\Psi^*(E)|_{M\setminus S(j_n)}
\mbox{ and } \alpha_{\Phi,\psi}|_Y=\id_Y.
\end{equation}
For $k\in\{0,1,\ldots, n\}$,
abbreviate
\[
\Phi(k):=\{j_1,\ldots, j_k\}.
\]
We obtain a $C^\ell$-vector bundle
isomorphism $\alpha_\Phi\colon \lambda_a^*(E)\to f_\Phi^*(E)$ over $\id_M$ via
\[
\alpha_\Phi:=\alpha_{\Phi(n),\Phi(n-1)}\circ \cdots\circ \alpha_{\Phi(1),\Phi(0)}.
\]
Now let $\Phi,\Theta\in\cF$ such that $\Phi\sub\Theta$.
Write $\Theta=\{i_1,\ldots, i_m\}$
with $i_1<\cdots< i_m$
and set $\Theta(k):=\{i_1,\ldots, i_k\}$
for $k\in\{0,\ldots, m\}$.
Note that for each $i_k\in \Theta\setminus \Phi$,
we have $S(i_k)\cap W(\Phi)=\emptyset$ and
thus $\alpha_{\Theta(k),\Theta(k-1)}|_{W(\Phi)}=\id$
(cf.\ (\ref{thusgood})). Hence
\begin{eqnarray*}
\alpha_\Theta|_{W(\Phi)} & = &
\alpha_{\Theta(m),\Theta(m-1)}|_{W(\Phi)}
\circ \cdots\circ \alpha_{\Theta(1),\Theta(0)}|_{W(\Phi)}\\
&=&
\alpha_{\Phi(n),\Phi(n-1)}|_{W(\Phi)}
\circ \cdots\circ \alpha_{\Phi(1),\Phi(0)}|_{W(\Phi)}
\, =\,  \alpha_\Phi|_{W(\Phi)},
\end{eqnarray*}
which completes the proof.
\end{proof}
\begin{rem}\label{dodge}
(a) If $t=a$ or $t=b$,
then $\lambda_t$ does not take $M^0$
inside the interior $]a,b[\,\times M^0$
of the range,
so that $\lambda_t$ is not an $RC^\ell$-map.
Thus, we cannot simply consider
rough $C^\ell$-manifolds
instead of $C^\ell$-manifolds with rough boundary
in
Proposition~\ref{eq-boundary}.
Yet, the author expects the proposition
to remain valid.\\[2.3mm]
(b) If $M$ is a $C^\ell$-paracompact
rough $C^\ell$-manifold
and $\pi\colon E\to I\times M$
is a $C^\ell$-vector bundle
for some interval $I\sub \R$ with $[a,b]\sub I^0$,
then we obtain the conclusion of
Proposition~\ref{eq-boundary}
for $E|_{[a,b]\times M}$.
In fact, we can pick $\alpha<a$ and $\beta>b$
with $[\alpha,\beta]\sub I$
and repeat the proof of Proposition~\ref{eq-boundary}
with $I:=[\alpha,\beta]$ in place of
$I:=[a,b]$.
\end{rem}
The following version of the Chain Rule
will help us to create a situation
as in Remark~\ref{dodge}(b),
and hence to dodge the problem
described in Remark~\ref{dodge}(a).
The lemma also allows $C^\infty$-manifolds
with rough boundary to be replaced with
rough $C^\infty$-manifolds
in parts~(c), (d), and~(e) of~\ref{collar}.
\begin{la}\label{technical-chain}
Let $X$, $E_1$, $E_2$, and $F$ be locally
convex spaces, $R\sub X$ and $S\sub E_2$
be regular subsets, $V\sub E_1$ be a locally convex,
regular subset, and $\ell\in \N_0\cup\{\infty\}$.
Let $g_1\colon R\to V$ be a $C^\ell$-map,
$g_2\colon R\to S$ be an $RC^\ell$-map
and $f\colon V\times S\to F$ be a $C^\ell$-map.
Then $h:=f\circ (g_1,g_2)\colon R\to F$ is~$C^\ell$;
if $\ell\geq 1$, then
\begin{equation}\label{chainhere}
dh(x,y)=df(g_1(x),g_2(x),dg_1(x,y),dg_2(x,y))
\mbox{ for all $(x,y)\in R\times X$.}
\end{equation}
\end{la}
\begin{proof}
We may assume that $\ell\in\N_0$
and proceed by induction. For $\ell=0$
the assertion holds as compositions
of continuous maps are continuous.
Let $\ell\in\N$ and assume the assertion holds
for $\ell-1$ in place of~$\ell$.
Abbreviate $g:=(g_1,g_2)\colon R\to V\times S\sub E_1\times E_2$.
For $x\in R^0$, $y\in X$, and $t\in\R\setminus\{0\}$
with $x+ty\in R^0$, we have
\begin{eqnarray*}
\frac{h(x+ty)-h(x)}{t}&=&
(f|_{V\times R^0})^{[1]}\left(g(x),\frac{g(x+ty)-g(x)}{t},t\right)\\
&\to & (f|_{V\times R^0})^{[1]}(g(x),dg(x,y),0)
=df(g(x),dg(x,y))
\end{eqnarray*}
as $t\to 0$, showing that
$h|_{R^0}$ is $C^1$ with $dh(x,y)=df(g(x),dg(x,y))$.
As $h$ is continuous and the right-hand side
of~(\ref{chainhere})
defines a continuous $F$-valued map on $R\times X$
which extends $d(h|_{R^0})$,
we deduce that~$h$ is~$C^1$ and (\ref{chainhere})
holds. Now
\[
\phi\colon (V\times E_1)\times (S\times E_2)\to F,\;\,
((y_1,z_1),(y_2,z_2))\mto df((y_1,y_2),(z_1,z_2))
\]
is $C^{\ell-1}$ and
\[
dh=\phi\circ (Tg_1,T g_2),
\]
where $Tg_1\colon R\times X\to V\times E_1$,
$(x,y)\mto (g_1(x),dg_1(x,y))$ is $C^{\ell-1}$
and $Tg_2\colon R\times X\to S\times E_2$,
$(x,y)\mto (g_2(x),dg_2(x,y))$ is $RC^{\ell-1}$.
Thus $dh$ is $C^{\ell-1}$,
by the inductive hypothesis,
and thus~$h$ is~$C^\ell$.
\end{proof}
We conclude that vector bundles
over a $C^\ell$-contractible
base are trivial, in larger generality
than previously known.
Recall that a
paracompact, finite-dimensional
$C^\ell$-manifold~$M$ (without boundary)
is $C^\ell$-contractible if and
only if it is contractible
as a topological space
(see Example~\ref{algtopexs}(a)
and its $C^\ell$-analogue).\footnote{If
$M$ is contractible,
then $M$ is connected.
Any paracompact, connected $C^\ell$-manifold
is $\sigma$-compact.}
\begin{cor}
Let $\ell\in \N_0\cup\{\infty\}$
and $M$ be a $C^\ell$-paracompact,
$RC^\ell$-contractible rough $C^\ell$-manifold
modelled on locally convex spaces
$($or a $C^\ell$-paracompact, $C^\ell$-contractible
$C^\ell$-manifold with rough boundary
modelled on locally convex spaces$)$.
Then every $C^\ell$-vector bundle
$E\to M$ with typical fibre
a locally convex space~$F$ is $C^\ell$-isomorphic over~$\id_M$
to the trivial bundle $M\times F$.
\end{cor}
\begin{proof}
Let $H\colon [0,1]\times M\to M$
be an $RC^\ell$-map (resp., a $C^\ell$-map)
which is a homotopy
from $\id_M$ to the constant map $c_{x_0}\colon M\to M$
for some $x_0\in M$.
Using Lemma~\ref{technical-chain} (resp. the Chain Rule),
we see that the corresponding homotopy with collar $H^c\colon [0,1]\times M\to M$
as in \ref{collar}(b) is~$C^\ell$
(as already observed in a special case in \ref{collar}(c)).
By Proposition~\ref{eq-boundary}
and Remark~\ref{dodge}(b)
(resp., by Proposition~\ref{eq-boundary}),
the $C^\ell$-vector bundles
$E=H(\frac{1}{4},\cdot)^*(E)$
and $H(\frac{3}{4},\cdot)^*(E)\cong M\times F$
are $C^\ell$-isomorphic over $\id_M$.
\end{proof}
Using analogous arguments,
it can be shown that all $C^\ell$-fibre
bundles over $C^\ell$-contractible bases which are $C^\ell$-manifolds
with rough boundary (or $RC^\ell$-contractible bases which are rough
$C^\ell$-manifolds)
are $C^\ell$-trivializable.
\appendix
\section{Details for Sections~\ref{sec-prels}, \ref{sec-smoo},
and \ref{proofs-outlook}}\label{appA}
We now prove Proposition~\ref{eqCk},
Lemma~\ref{sammelsu}(c),
Lemma~\ref{toppartial},
Lemma~\ref{S3new}, and Lemma~\ref{partu-ex}.\\[2.3mm]
{\bf Proof of Proposition~\ref{eqCk}.}
See Proposition~1.3.10 and Lemma~1.4.5 in \cite{GaN} for the equivalence of~(a) and (b)
if $R$ is locally convex (or~\cite[Lemma~1.14]{Res} if~$R$ is open);
the general case can be proved by the same arguments (see \cite{Rou}).\\[2.3mm]
Now assume that $E=\R^n$ for some $n\in\N$.
If (b) holds, then $\partial^\alpha(f|_{R^0})(x)$ exists for all $\alpha\in \N_0^n$
with $k:=|\alpha|\leq \ell$ and $x\in R^0$, and is given by
\[
\partial^\alpha(f|_{R^0})(x)\;=\;
d^{\,(k)}f(x,e_{i_1},\ldots,e_{i_k})
\]
with
\begin{equation}\label{theis}
(i_1,\ldots,i_k)=(\underbrace{n,\ldots, n}_{\alpha_n},\cdots,\underbrace{1,\ldots,1}_{\alpha_1}).
\end{equation}
A continuous extension $\partial^\alpha f\colon R\to F$ is obtained via
\begin{equation}\label{parviader}
\partial^\alpha(f)(x)\;=\;
d^{\,(k)}f(x,e_{i_1},\ldots,e_{i_k}),
\end{equation}
and thus (c) holds.
Now assume that (c) holds; we shall show that $f$ is $C^\ell$ and
\begin{equation}\label{derviapar}
d^{\,(k)}f(x,y_1,\ldots,y_k)\;=\;
\sum_{i_1,\ldots,i_k=1}^n y_{1,i_1}\cdots y_{k,i_k}\,\frac{\partial^k f}{\partial x_{i_k}\cdots
\partial x_{i_1}}(x)
\end{equation}
for all $k\in\N_0$ such that $k\leq\ell$, $x=(x_1,\ldots,x_n)\in R$
and $y_j=(y_{j,1},\ldots, y_{j,n})\in \R^n$ for $j\in\{1,\ldots, k\}$.
Let us assume first that~$R$ is open; we may assume
that $R=I_1\times\cdots\times I_n$ is a product of open intervals.
Fix $\alpha\in\N_0^n$ with $k:=|\alpha|\leq \ell$
Let $i_1,\ldots, i_k$ be as in (\ref{theis}).
As a consequence of~(c), the iterated
directional derivatives
\begin{eqnarray*}
\lefteqn{d^{\,(\alpha)}f(x,y_{1,1},\ldots,y_{1,\alpha_1},\ldots,
y_{n,1},\ldots, y_{n,\alpha_n})}\qquad\qquad\\
&:=
(D_{y_{1,\alpha_1}e_1}\cdots D_{y_{1,1}e_1}\cdots
D_{y_{n,\alpha_n}e_n}\cdots D_{y_{n,1}e_n}f)(x)
\end{eqnarray*}
exist for all $x\in R$ and $y=(y_{1,1},\ldots,y_{1,\alpha_1},\ldots,
y_{n,1},\ldots,y_{n,\alpha_n})\in\R^k$, and are given by
\[
\left(\prod_{j=1}^n\prod_{i=1}^{\alpha_j}y_{j,i}\right)\partial^\alpha f(x),
\]
which is a continuous functon of $(x,y)\in R\times\R^k$.
Hence $f$ is a $C^\alpha$-function in the sense of~\cite{Alz}.
Since $f$ is $C^\alpha$ for all $\alpha\in\N_0^n$ such that $|\alpha|\leq k$,
the map~$f$ is $C^k$, by \cite[Lemma~3.12]{Alz}
and a straightforward induction (see \cite{Ing}, also for generalizations).\\[2.3mm]
If $R\sub\R^n$ is any regular subset, then $f|_{R^0}$ is $C^\ell$
by the preceeding and
\[
d^{\,(k)}(f|_{R^0})(x,y_1,\ldots,y_k)\;=\;
\sum_{i_1,\ldots,i_k=1}^n y_{1,i_1}\cdots y_{k,i_k}\,\frac{\partial^k f|_{R^0}}{\partial x_{i_k}\cdots
\partial x_{i_1}}(x)
\]
for all $k\in\N_0$ such that $k\leq\ell$, $x=(x_1,\ldots,x_n)\in R^0$
and $y_j=(y_{j,1},\ldots, y_{j,n})\in \R^n$ for $j\in\{1,\ldots, k\}$,
using that $d^{\,(k)}(f|_{R^0})(x,\cdot)\colon \R^n\to F$ is $k$-linear.
Since (\ref{derviapar}) defines a continuous $F$-valued function
of $(x,y_1,\ldots,y_k)\in R\times(\R^d)^k$ which extends $d^{\,(k)}(f|_{R^0})$,
we see that~(b) is satisfied by~$f$. $\,\square$\\[2.3mm]
{\bf Proof of Lemma~\ref{sammelsu}(c).}
By definition as an initial topology, the compact-open $C^\ell$-topology on $C^\ell(R,F)$
turns the linear injective map
\[
C^\ell(R,F)\to\prod_{\N_0\ni j\leq\ell}C(R\times E^j,F),\quad\gamma\mto
(d^{\,(j)}\gamma)_{\N_0\ni j\leq\ell}
\]
into a topological embedding.
For $j\in\N_0$ with $j\leq\ell$,
the seminorms $\|\cdot\|_{K\times L^j,q}$
with $q\in\Gamma$, $K\in \cK$ and $L\in\cL$
define the locally convex topology on $C(R\times E^j,F)$
by~\ref{propsco}. The locally convex topology on $C^\ell(R,F)$
is therefore defined by the seminorms $\gamma\mto\|d^{\,(j)}\gamma\|_{K\times L^j,q}$,
and hence also by $\|\cdot\|_{K,q}$ and
the pointwise suprema $\|\cdot\|_{C^k,K,L,q}$
of finitely many such seminorms, as in~(c). $\,\square$\\[2.3mm]
{\bf Proof of Lemma~\ref{toppartial}.}
As a consequence of (\ref{parviader}),
we have
\[
\|\cdot\|_{C^j,K,q}^\partial \leq \|\cdot\|_{C^j,K,L,q}
\]
for all $j\in\N$ with $j\leq \ell$, compact subsets $K\sub R$ and
compact subsets $L\sub \R^n$ and each continuous seminorm~$q$ on~$F$,
if we use a compact subset $L\in\cL$
such that $\{e_1,\ldots, e_n\}\sub L$.
Conversely, given $j$, $K$, and $q$ as before
and $L\in\cL$,
there exists $r\in [1,\infty[$ such that $L\sub [{-r},r]^n$.
Using~(\ref{derviapar}), we see that
\[
\|d^{\,(j)}\gamma\|_{K\times L^j,q}\;\leq\; n^jr^j\|\gamma\|_{C^j,K,q}^\partial.
\]
Given $k\in\N$ with $k\leq\ell$, taking the maximum over $j\in \{1,\ldots,k\}$
and the estimate $\|\cdot\|_{K,q}\leq\|\cdot\|_{C^k,K,q}^\partial$,
we deduce that
\[
\|\cdot\|_{C^k,K,L,q}\leq n^kr^k\|\cdot\|_{C^k,K,q}^\partial.
\]
The locally convex vector topologies on $C^\ell(R,F)$ defined by
the families of seminorms in Lemma~\ref{sammelsu}(c)
and Lemma~\ref{toppartial} therefore concide.
As a consequence, the compact-open $C^\ell$-topology
on $C^\ell(R,F)$ is also initial with respect to the maps $\partial^\alpha$
(arguing as in the proof of Lemma~\ref{sammelsu}(c),
with $\partial^\alpha$ instead of the $d^{\,(j)}$). $\,\square$\\[2.3mm]
{\bf Proof of Lemma~\ref{S3new}.}
Abbreviate $B:=\wb{B}_1(0)\sub\R^d$.
If $K\in \cK$, $q\in\Gamma$, $k\in\N_0$ with $k\leq \ell$
and $j\in \{1,\ldots, k\}$,
then
\begin{eqnarray*}
\sup_{x\in K}\|\delta^j_x\gamma\|_q&=& \sup_{x\in K}\sup_{y\in B}q(\delta^j_x\gamma(y))
=\sup_{x\in K}\sup_{y\in B}q(d^{\,(j)}\gamma(x,y,\ldots,y))\\
&\leq& \|d^{\,(j)}\gamma\|_{K\times B^j,q}
\, \leq \, \|\gamma\|_{C^k,K,B,q}
\end{eqnarray*}
for all $\gamma\in C^\ell(\Omega,F)$,
whence
\[
\|\cdot\|_{C^k,K,q}\leq \|\cdot\|_{C^k,K,B,q}.
\]
Conversely, let $K\in\cK$, $L\sub\R^d$ is compact,
$q\in\Gamma$, and $k\in\N_0$ with $k\leq \ell$.
Then $L\sub [{-r},r]^d=\wb{B}_r(0)$ for some $r\in [1,\infty[$.
For all $i\leq j$ in $\{1,\ldots,k\}$, we have
\[
\underbrace{L+\cdots+L}_{i}\;\sub\; \wb{B}_{ir}(0)\;=\; riB\;\sub\; rj B
\]
for the $j$-fold sum.
Using the Polarization Formula (\ref{polform}),
we deduce that
\[
q(d^{\,(j)}\gamma(x,y_1,\ldots,y_j))\;\leq\;
\frac{(2rj)^j}{j!}\,\|\delta^j_x\gamma\|_q\;\leq\;\frac{(2rk)^k}{k!}\,\|\gamma\|_{C^k,K,q}
\]
for all $x\in K$ and $y_1,\ldots,y_k\in L$, whence
\[
\|\cdot\|_{C^k,K,L,q}\;\leq\; \frac{(2rk)^k}{k!}\,\|\cdot\|_{C^k,K,q}.
\]
The family of seminorms on $C^\ell(\Omega,F)$
described in Lemma~\ref{S3new})
therefore defines the same locally convex topology as the family
of seminorms given in Lemma~\ref{sammelsu}(c). $\,\square$\\[2.3mm]
Locally compact rough $C^\ell$-manfolds
are $C^\ell$-regular in the following sense.
\begin{la}\label{cutoff}
Let $\ell\in\N_0\cup\{\infty\}$,
$M$ be a locally compact, rough $C^\ell$-manifold, $x\in M$ and $U\sub M$
be an $x$-neighbourhood. Then there exists a $C^\ell$-function
$\xi\colon M\to\R$ with compact support $\Supp(\xi)\sub U$
such that $\xi(M)\sub [0,1]$ and $\xi(x)=1$.
\end{la}
\begin{proof}
There exists a chart $\phi\colon U_\phi\to V_\phi$ of~$M$
with $x\in U_\phi$ and $U_\phi\sub U$,
such that $V_\phi$ a regular subset of $\R^d$
for some $d\in\N_0$. As $M$ is locally compact,
we find a compact $x$-neighbourhood $B\sub U_\phi$;
then $\phi(B)$ is a compact $\phi(x)$-neighbourhood
in~$V_\phi$. As $V_\phi$ carries the topology induced by~$\R^d$,
we find an open $x$-neighbourhood $W\sub\R^d$ such that
\[
V_\phi\cap W\;\sub\; \phi(B).
\]
Let $h\colon \R^d\to\R$ be a smooth function with compact support
$\Supp(h)\sub W$
such that $h(\R^d)\sub[0,1]$ and $h(\phi(x))=1$.
Then
\[
\xi\colon M\to\R,\quad
y\mto\left\{
\begin{array}{cl}
h(\phi(y)) &\mbox{if $\,y\in U_\phi$;}\\
0 & \mbox{if $\,y\in M\setminus B$}
\end{array}\right.
\]
has the desired properties.
\end{proof}
\noindent
{\bf Proof of Lemma~\ref{partu-ex}.}
For each $x\in M$, there exists $j(x)\in J$ such that $x\in U_{j(x)}$.
By \cite[Theorem~5.1.27]{Eng},
there exists a family $(M_a)_{a\in A}$
of $\sigma$-compact, open subsets~$M_a$ of~$M$
such that $M=\bigcup_{a\in A}M_a$ and $M_a\cap M_b=\emptyset$
for all $a\not=b$ in~$M$.
For each $a\in A$, we let
\[
K_{a,1}\sub K_{a,2}\sub\cdots
\]
be a compact exhaustion\footnote{Each $K_{a,n}$ is a compact subset of~$M_a$
with $K_{a,n}\sub K_{a,n-1}^0$,
and $M_a=\bigcup_{n\in\N}K_{a,n}$.} of~$M_a$.
We define $K_{a,0}:=K_{a,-1}:=\emptyset$.
Let $a\in A$ and $n\in\N$.
For each $x\in K_{a,n}\setminus K_{a,n-1}^0$,
Lemma~\ref{cutoff} provides a $C^\ell$-function
$\xi_{a,n,x}\colon M\to\R$ with image in $[0,1]$ such that
\[
L_{a,n,x}\; :=\; \Supp\xi_{a,n,x}\;\sub\; (K_{a,n+1}^0\setminus K_{a,n-2})\cap U_{j(x)}
\]
and $\xi_{a,n,x}(x)=1$.
The sets $\{y\in M\colon \xi_{a,n,x}>0\}$
form an open cover of the compact set $K_{a,n}\setminus K_{a,n-1}^0$
for $x\in K_{a,n}\setminus K_{a,n-1}^0$.
We therefore find a finite subset $\Phi_{a,n}\sub K_{a,n}\setminus K_{a,n-1}^0$
such that
\[
\sum_{x\in \Phi_{a,n}}\xi_{a,n,x}(y)>0\quad\mbox{for all $y\in K_{a,n}\setminus
K_{a,n-1}^0$.}
\]
Let $I:=\{(a,n,x)\colon a\in A,\, n\in\N,\, x\in \Phi_{a,n}\}$.
Then the family $(L_{a,n,x})_{(a,n,x)\in I}$ of
compact subsets of~$M$ is locally finite. In fact,
if $z\in M$, then $z\in K_{a,n}\setminus K_{a,n-1}^0$
for some $a\in A$ and $n\in\N$.
Then $W:=K_{a,n+1}^0\setminus K_{a,n-2}$
is an open neighbourhood of~$z$ in~$M$ such that
\[
W\cap L_{b,m,y}\;\not\;=\emptyset
\]
for $(b,m,y)\in I$ implies $a=b$ and $|n-m|\leq 2$,
whence
$\{(b,m,y)\in I\colon L_{b,m,y}\cap W\not=\emptyset\}$ is a finite set.
Hence
\[
s\colon M\to \,]0,\infty[,\quad z\mto \sum_{(b,m,y)\in I}\xi_{(b,m,y)}(z)
\]
is a $C^\ell$-function and also
\[
g_{a,n,x}\colon M\to \R,\quad z\mto \frac{\xi_{a,n,x}(z)}{s(z)}
\]
is a $C^\ell$-function for all $(a,n,x)\in I$.
By construction, $(g_{a,n,x})_{(a,n,x)\in I}$
is a $C^\ell$-partition of unity on~$M$ such that $\Supp(g_{a,n,x})=L_{a,n,x}
\sub U_{j(x)}$ for all $(a,n,x)\in I$.
Given $j\in J$, let
\[
I_j\; :=\; \{(a,n,x)\in I\colon j(x)=j\}.
\]
Then
\[
h_j\colon M\to\R,\quad h_j(z):=\sum_{(a,n,x)\in I_j}g_{a,n,x}(z)
\]
is a $C^\ell$-function and $(h_j)_{j\in J}$ is a $C^\ell$-partiton of unity on~$M$
such that $\Supp(h_j)=\bigcup_{(a,n,x)\in I_j}L_{a,n,x}\sub U_j$
for all $j\in J$. $\,\square$
\section{Details for Section~\ref{more-densy}}\label{appB}
We now prove Lemmas~\ref{in-parac}, \ref{roughcpsupp},
and \ref{weicompl}.\\[2.3mm]
{\bf Proof of Lemma~\ref{in-parac}.}
Let $\gamma\in C(X,F)$, $K\sub M$ be a compact subset,
$q$ be a continuous seminorm on~$F$, and $\ve>0$.
By hypothesis, there exists a closed, paracompact subset
$L\sub M$ such that $K\sub L^0$.
Each $x\in K$ has an open neighbourhood $V_x\sub L^0$
such that
\[
(\forall y\in V_x)\quad q((\gamma(x)-\gamma(y))\,\leq\,\ve.
\]
We let $(h_x)_{x\in K}$,
together with $g\colon L\to[0,1]$,
be a partition of unity on~$L$
such that $\Supp(h_x)\sub V_x$
for each $x\in K$ and $\Supp(g)\sub L\setminus K$.
By local finiteness,
$A:=\bigcup_{x\in K}\Supp(h_x)$ is closed in~$L$
and hence in~$X$; moreover, $A\sub L^0$.
Thus
\[
\eta\colon X\to F,\quad
y\mto \left\{
\begin{array}{cl}
\sum_{x\in K}h_x(y)\gamma(x) &\mbox{ if $y\in L^0$;}\\
0 & \mbox{ if $y\in X\setminus A$}
\end{array}\right.
\]
is a continuous function. Since $h_x(y)\not=0$
implies $y\in V_x$, we get for $y\in K$
\begin{eqnarray*}
q(\gamma(y)-\eta(y))&=&q\left(\sum_{x\in K}h_x(y)(\gamma(y)-\gamma(x))\right)
\leq \sum_{x\in K}h_x(y)q(\gamma(y)-\gamma(x))\\
&\leq & \sum_{x\in K}h_x(y)\ve=\ve.
\end{eqnarray*}
Hence $\|\gamma-\eta\|_{K,q}\leq\ve$. $\,\square$\\[2.3mm]
{\bf Proof of Lemma~\ref{roughcpsupp}.}
(a) For each $x\in K$,
the linear map $\phi_x:=\id_F$
is a $C^r$-diffeomorphism and thus a chart for~$F$.
Given an open neighbourhood $W\sub C^\ell_L(M,F)$
of $\gamma$ in the compact-open $C^\ell$-topology,
we set $\gamma_0:=\gamma$ and find $\gamma_1,\ldots,\gamma_m$
as in the proof of Proposition~\ref{ctscaseprop},
which we vary as described in the proof of
Proposition~\ref{densetf}.
Then $\gamma_m \in W\cap C^r_L(M,F)$.
We show that the subset
\[
\gamma_i((M\setminus K)\cup P_{x_1}\cup\cdots\cup P_{x_i})
\]
of~$F$ has finite-dimensional span for all $j\in\{0,1,\ldots,m\}$.
If this is true, then $\gamma_m(M)$ has finite-dimensional
span and thus $\gamma_m\in F\otimes C^r_L(M,\R)$
(using that $M=(M\setminus K)\cup P_{x_1}\cup\cdots \cup P_{x_m}$).\\[2.3mm]
The induction starts as $\gamma_0=\gamma$ and
$\gamma(M\setminus K)\sub\{0\}$ has finite-dimensional span.
Let $j\in\{1,\ldots, m\}$ and assume that the assertion holds for $j-1$
in place of~$j$. Since $\gamma_j$ and $\gamma_{j-1}$
coincide on $M\setminus Q_{x_j}^0$,
the image of
\[
((M\setminus K)\cup P_{x_1}\cup\cdots\cup P_{x_{j-1}})\cap (M\setminus
Q_{x_j}^0)
\]
under $\gamma_j$
has finite-dimensional span.
If $y\in P_{x_j}$ or
\[
y\in ((M\setminus K)\cup P_{x_1}\cup
\cdots\cup P_{x_{j-1}})\cap Q_{x_j}^0=:S_j,
\]
then $\gamma_j(y)$ is a linear combination of
elements of $\eta_{j,b_j}
(Q_{x_j}^0)$
and 
$\gamma_{j-1}(S_j)$,
which are subsets of~$F$ with finite-dimensional span.\\[2.3mm]
(b) is immediate from~(a).$\,\square$\\[2.3mm]
{\bf Proof of Lemma~\ref{weicompl}.}
(a) As $\cS$ is a projective system,
$\lambda_{j,j}=\id_{E_j}$ for all $j\in J$
and $\lambda_{i,j}\circ \lambda_{j,k}=\lambda_{i,k}$ for $i\leq j\leq k$.
The limit maps $\lambda_j$ satisfy $\lambda_{i,j}\circ\lambda_j=\lambda_i$
for all $i\leq j$ in~$J$ and $(F,(\lambda_j)_{j\in J})$
has the familiar universal property.~If
\[
P:=\Big\{(y_j)_{j\in J}\in\prod_{j\in J}F_j\colon (\forall i\leq j)\;
\lambda_{i,j}(y_j)=y_i\Big\}\vspace{-1mm}
\]
is the standard projective limit of~$\cS$, then the map
\[
\lambda\colon F\to P,\quad y\mto (\lambda_j(y))_{j\in J}
\]
is an isomorphism of topological vector spaces.
For $i\in J$, let $\pr_i\colon P\to F_i$, $(y_j)_{j\in J}\mto y_i$
be the projection.
As $C^\ell(\Omega,\lambda_{j,j})$ is the identity map on $C^\ell(\Omega,F_j)$
and $C^\ell(\Omega,\lambda_{i,j})\circ C^\ell(\Omega,\lambda_{j,k})
=C^\ell(\Omega,\lambda_{i,k})$ if $i\leq j\leq k$,
indeed (\ref{thefsys})
is a projective system of locally convex spaces and continuous linear maps.
Let
\[
Q:=\Big\{(\gamma_j)_{j\in J}\in\prod_{j\in J}C^\ell(\Omega,F_j)\colon
(\forall i\leq j)\; \lambda_{i,j}\circ \gamma_j=\gamma_i\Big\}\vspace{-1mm}
\]
be the standard projective limit of~(\ref{thefsys}),
endowed with initial topology with respect to the linear maps
$\pi_i\colon P\to C^\ell(\Omega,F_i)$, $(\gamma_j)_{j\in J}\mto \gamma_i$
(the topology induced by $\prod_{j\in J}C^\ell(\Omega,F_j)$).
If $\gamma\in C^\ell(\Omega, F)$,
then $\lambda_j\circ \gamma\in C^\ell(\Omega,F_j)$
for all $j\in J$
and $\lambda_{i,j}\circ (\lambda_j\circ\gamma)=\lambda_i\circ\gamma$
if $i\leq j$.
Hence
\[
\phi(\gamma):=(\lambda_j\circ\gamma)_{j\in J}\in Q
\]
and we obtain a map $\phi\colon C^\ell(\Omega,F)\to Q$
which is continuous and linear as $\pi_i\circ\phi=C^\ell(\Omega,\lambda_i)$
is so for each $i\in I$.
Like the maps $\pr_i$ on~$P$, the $\lambda_i=\pr_i\circ\lambda$
separate points on~$F$.
This entails that the $C^\ell(\Omega,\lambda_i)=\pi_i\circ \phi$
separate points on $C^\ell(\Omega,F)$ and thus $\phi$ is injective.
If $(\gamma_j)_{j\in J}\in Q$,
then for all $x\in\Omega$ we have $\lambda_{i,j}(\gamma_j(x))=\gamma_i(x)$
for all $i\leq j$ in~$J$
whence $(\gamma_j(x))_{j\in J}\in P$
and
\[
\gamma(x):=\lambda^{-1}\big((\gamma_j(x))_{j\in J}\big)\in F.
\]
Consider $\eta\colon \Omega\to P$, $x\mto(\gamma_j(x))_{j\in J}$.
As the components of $\eta$ are $C^\ell$,
the map~$\eta$ is $C^\ell$ as a map to $\prod_{j\in J}F_j$
(see \cite[Lemma~1.3.3]{GaN}); as $\eta$ takes its values in
the closed vector subspace~$P$ of the direct product,
also $\eta$ is $C^\ell$ (see \cite[Lemma~1.3.19]{GaN}).
Hence also $\gamma=\lambda^{-1}\circ \eta$ is~$C^\ell$.
For all $j\in J$ and $x\in \Omega$,
\[
\phi(\gamma)(x)=\lambda_j(\phi(x))=\pr_j(\lambda(\phi(x)))=
\pr_j(\eta(x))=\gamma_j(x).
\]
Thus $\phi(\gamma)=(\gamma_j)_{j\in J}$,
whence~$\phi$ is surjective and an isomorphism
of vector spaces. For $i\in I$ and continuous seminorms
$q$ on~$F_i$, the seminorms $q_i\circ \lambda_i$ on~$F$
are continuous and define its locally convex topology.
As a consequence, the seminorms $\|\cdot\|_{C^k\!,K,q\circ\lambda_i}$
define the locally convex topology on $C^\ell(\Omega,F)$,
for $K\in\cK(\Omega)$ and $k\in\N_0$ with $k\leq \ell$.
Recall that $\phi^{-1}((\gamma_j)_{j\in J})(x)=\lambda^{-1}((\gamma_j(x))_{j\in J})$,
whence
$(\lambda_i\circ \phi^{-1}((\gamma_j)_{j\in J}))(x)=(\pr_i\circ\lambda\circ
\lambda^{-1})((\gamma_j(x))_{j\in J}=\gamma_i(x)$ and thus
\[
(\lambda_i\circ\phi^{-1})((\gamma_j)_{j\in J}=\gamma_i
=\pi_i((\gamma_j)_{j\in J}.
\]
As a consequence,
$\|\phi^{-1}((\gamma_j)_{j\in J})\|_{C^k\!,K,q\circ\lambda_i}
=\|\pi_i((\gamma_j)_{j\in J})\|_{C^k\!,K,q}$
and thus\linebreak
$\|\cdot\|_{C^k\!,K,q\circ\lambda_i}\circ\phi^{-1}\leq
\|\cdot\|_{C^k\!,K,q}\circ\pi_i$,
where the right-hand side is a continuous seminorm on~$Q$.
Hence $\phi^{-1}$ is continuous.\\[2.3mm]
(b) If $S\sub\prod_{q\in\Gamma}C^\ell_\cW(\Omega,\wt{F}_q)^\sbull$
is the standard projective limit,
with projections $\pi_q\colon S\to C^\ell_\cW(\Omega,\wt{F}_q)^\sbull$.
then $\psi(\gamma):=(\wt{\alpha}_q\circ\gamma)_{q\in\Gamma}\in S$
for each $\gamma\in C^\ell_\cW(\Omega,F)^\sbull$.
The map $\psi\colon C^\ell_\cW(\Omega,F)^\sbull\to S$
so obtained is linear, and continuous as $\|\cdot\|_{\|\cdot\|_q,f,k}\circ\pi_q
\circ\psi=\|\cdot\|_{q,f,k}\leq\|\cdot\|_{q,f,k}$.
It is injective as the $\wt{\alpha}_q$ separate points on~$F$.
Write $\psi^{-1}$ as a shorthand for
$(\psi|^{\im(\phi)})^{-1}$.
Since $\|\cdot\|_{q,f,k}\circ\psi^{-1}=\|\cdot\|_{\|\cdot\|_q,f,k}\circ\pi_q\leq
\|\cdot\|_{\|\cdot\|_q,f,k}\circ\pi_q$,
the map $\psi^{-1}$ is continuous.
If $(\gamma_q)_{q\in \Gamma}\in S$,
then $(\gamma_q)_{q\in \Gamma}$ is also
an element of the standard projective limit~$Q$
of $((C^\ell(\Omega,\wt{F}_q)_{q\in\Gamma},(C^\ell(\Omega,\wt{\alpha}_{p,q}))_{p\leq q})$,
as in~(a). By~(a), there exists $\gamma\in C^\ell(\Omega,F)$
such that $\wt{\alpha}_q\circ\gamma=\gamma_q$ for all $q\in\Gamma$.
Then $\alpha_q\circ\gamma\in C^\ell(\Omega,E_q)$
for all $q\in \Gamma$, by the Chain Rule,
entailing that $\gamma_q=\wt{\alpha}_q\circ\gamma\in C^\ell_\cW(\Omega,F_q)^\sbull$
actually. Thus $\gamma\in C^\ell_\cW(\Omega,F)^\sbull$
and $\psi(\gamma)=(\wt{\alpha}_q\circ\gamma)_{q\in\Gamma}=(\gamma_q)_{q\in\Gamma}$,
whence $\psi$ is surjective and hence an isomorphism of topological
vector spaces.\\[2.3mm]
Note that the weighted function spaces and their
topologies are unchanged if we add the functions
\[
|f_1|+\cdots+|f_n|
\]
to the set~$\cW$ for all $n\in\N$ and $f_1,\ldots, f_n\in \cW$.
By \cite[Lemma~3.4.9]{Wal},
$C^\ell_\cW(\Omega,\wt{E}_q)^\sbull$
is a closed vector subspace of a certain
locally convex space\linebreak
$C^\ell_\cW(\Omega,\wt{E}_q)$,
which is complete by \cite[Corollary~3.2.11]{Wal}.
Hence~$S$ is complete, being a closed vector subspace of a direct
product of complete locally convex spaces.
As a consequence, also $C^\ell_\cW(\Omega,F)^\sbull\cong S$ is complete. $\,\square$
\section{Details for Section~\ref{moreops}}\label{appC}
We prove Proposition~\ref{smoo-top} and
Proposition~\ref{onmfdsmoo}.\\[2.3mm]
{\bf Proof of Proposition~\ref{smoo-top}.}
Let $d$ be a metric on~$X$ defining its topology.
For $x\in X$ and $\ve>0$, write $B_\ve(x):=\{y\in X\colon d(x,y)<\ve\}$.
There are positive integers $m_1<m_2<\cdots$
such that
\[
B_{2/m_n}(K_n):=
\bigcup_{x\in K_n}B_{2/m_n}(x)\sub K_{n+1}^0\quad \mbox{for all $n\in\N$.}
\]
For each $n\in\N$, we let $(h_{n,x})_{x\in X}$
be a continuous partition of unity on~$X$
such that $\Supp(h_{n,x})\sub B_{1/m_n}(x)$
for all $x\in X$. Then
\[
\Phi_n:=\{x\in X\colon \Supp(h_{n,x})\cap K_n\not=\emptyset\}
\]
is a finite set. For $\gamma\in C(X,F)$,
we set
\[
S_n(\gamma):=\sum_{x\in \Phi_n}h_{n,x}\gamma(x)\in F\otimes C_c(X,\R).
\]
If $\Supp(\gamma)\cap K_n\not=\emptyset$,
then $\Supp(\gamma)\sub B_{2/m_n}(K_n)\sub K_{n+1}^0$,
by the triangle inequality. Hence
\[
\Supp(S_n(\gamma))\sub K_{n+1}^0.
\]
If $q$ is a continuous seminorm on~$F$, then
\[
q(S_n(\gamma)(y))\leq \sum_{x\in \Phi_n}h_{n,x}(y)q(\gamma(x))
\leq \|\gamma\|_{K_{n+1},q}\sum_{x\in\Phi_n}h_{n,x}(y)
\leq\|\gamma\|_{K_{n+1},q}
\]
for all $y\in X$, whence $\|S_n(\gamma)\|_{K_j,q}\leq \|\gamma\|_{K_{n+1},q}$
for all $j\in\N$. Hence $S_n\colon C(X,F)\to C(X,F)$
is continuous.
Let $B\sub C(X,F)$ be a compact set.
Given $n\in\N$, a continuous seminorm~$q$ on~$F$ and $\ve>0$,
we use that the map
\[
B\times K_{n+1}\to F,\quad (\gamma,y)\mto\gamma(y)
\]
is continuous and hence uniformly continuous,
by compactness of $B\times K_{n+1}$.
As a consequence, there exists $\delta>0$ such that
\[
q(\gamma(y)-\gamma(x))\leq \ve
\]
for all $\gamma\in B$ and $x,y\in K_{n+1}$ such that $d(x,y)\leq\delta$.
Let $j_0\geq n$ be an integer such that $1/{m_{j_0}}\leq \delta$.
Let $j\geq j_0$.
For each $\gamma\in B$, $y\in K_n$,
and $x\in \Phi_j$, if $h_{j,x}(y)\not=0$
then $y\in B_{1/m_j}(x)$ and thus $q(\gamma(x)-\gamma(y))\leq\ve$.
Now $\sum_{x\in \Phi_j}h_{j,x}(y)$ for all $y\in K_n$
(as this actually holds for all $y\in K_j$)
and hence
\begin{eqnarray*}
q(\gamma(y)-S_j(\gamma)(y)) &=&
q\left(
\sum_{x\in\Phi_j}h_{j,x}(y)\gamma(y)-\sum_{x\in\Phi_j}
h_{j,x}(y)\gamma(x)\right)\\
&\leq & \sum_{x\in\Phi_j}\underbrace{h_{j,x}(y)q(\gamma(y)-\gamma(x))}_{\leq
h_{j,x}(y)\ve}\leq\ve.
\end{eqnarray*}
Thus
$\|\gamma-S_j(\gamma)\|_{K_n,q}\leq\ve$ for all $\gamma\in B$ and $j\geq j_0$,
whence $S_j\to\id$ uniformly on compact sets.
Thus (a) and~(b) are established.\\[2.3mm]
To prove~(c), let
$\gamma\in C_{K_n}(X,F)$ and $j\geq n$.
Then $\gamma(x)\not=0$ for $x\in \Phi_j$
implies that $x\in K_n$,
whence $\Supp(h_{j,x})\sub B_{1/m_j}(K_n)\sub B_{2/m_n}(K_n)
\sub K_{n+1}^0$.
Thus $\Supp(S_j(\gamma))\sub K_{n+1}^0$
and thus $S_j(\gamma)\in C_{K_{n+1}}(X,F)\cap (F\otimes C_{K_{j+1}}(X,\R))
=F\otimes C_{K_{n+1}}(X,\R)$. $\,\square$
\begin{numba}
Recall that a topological sum (disjoint union)
$X=\coprod_{j\in J}X_j$ of topological spaces~$X_j$
is metrizable if every $X_j$ is so
(see \cite[Theorem 4.2.1]{Eng}).
\end{numba}
\begin{numba}\label{shrinkcover}
If $X$ is a paracompact topological space and $(U_j)_{j\in J}$
a locally finite open cover of~$X$, then there exists
a locally finite open cover $(V_j)_{j\in J}$ of~$X$ such that
$\wb{V_j}\sub U_j$ holds
for the closure in~$X$, for each $j\in J$.\\[2.3mm]
[Given $x\in X$, there exists $j(x)\in J$ such that $x\in U_{j(x)}$.
Since every paracompact space is normal and hence
regular, there exists a closed
$x$-neighbourhood $A_x$ such that $A_x\sub U_{j(x)}$.
By paracompactness, there exists a locally finite open cover
$(W_x)_{x\in X}$ of~$X$ such that $W_x\sub A_x^0$
for each $x\in X$. Thus
\[
\wb{W_x}\sub A_x\sub U_{j(x)}.
\]
For $j\in J$, let $\Phi_j:=\{x\in X\colon j(x)=j\}$.
Then $V_j:=\bigcup_{x\in\Phi_j}W_x\sub U_j$,
whence the open cover $(V_j)_{j\in J}$
is locally finite; moreover, $\wb{V_j}=\bigcup_{x\in\Phi_j}\wb{W_x}\sub U_j$.
\end{numba}
{\bf Proof of Proposition~\ref{are-metri}.}
Since $X$ is a topological sum of $\sigma$-compact open
subsets, we may assume that~$X$ is $\sigma$-compact.
Since~$X$ is a topological sum of the open subsets admitting
$\R^d$-charts for a fixed~$d$, we may assume that~$X$ is
a pure manifold modelled on~$\R^d$.\\[2.3mm]
We choose a locally finite open cover $(U_j)_{j\in J}$
of~$M$ subordinate to an open cover of relatively compact chart domains.
Thus, for each each $j\in J$ there exists
a $C^\ell$-diffeomorphism $\phi_j\colon U_j\to V_j$
onto a regular subset $V_j\sub\R^d$.
Using~\ref{shrinkcover} twice,
we find locally finite open covers $(Q_j)_{j\in J}$ and $(P_j)_{j\in J}$
of~$X$ such that $\wb{Q_j}\sub U_j$ and $\wb{P_j}\sub Q_j$.
Then $\wb{Q_j}$ and $\wb{P_j}$ are compact.
After replacing $J$ with the set of all $j\in J$ such that
$P_j\not=\emptyset$, we may assume that $P_j\not=\emptyset$
for all $j\in J$. We pick a $C^\ell$-map $g_j\in C^\ell(M,\R)$
such that $\Supp(g_j)\sub Q_j$, $g_j(M)\sub [0,1]$
and $g_j|_{\wb{P_j}}=1$ (cf.\ Lemma~\ref{partu-ex}).
Moreover, we choose a $C^\ell$-map $h_j\in C^\ell(M,\R)$
such that $\Supp(h_j)\sub U_j$, $h_j(M)\sub [0,1]$
and $h_j|_{\wb{Q_j}}=1$.
Then
\[
f_j\colon M\to\R^d\times \R,\quad
x\mto\left\{
\begin{array}{cl}
(h_j(x)\phi_j(x),g_j(x)) & \mbox{ if $x\in U_j$;}\\
(0,0) &\mbox{ if $x\in M\setminus\Supp(h_j)$}
\end{array}\right.
\]
is a $C^\ell$-function. Hence
\[
f:=(f_j)_{j\in J}\colon M\to (\R^d\times\R)^J,\quad
x\mto (f_j(x))_{j\in J}
\]
is a continuous function. Since $M$ is $\sigma$-compact,
the set $J$ is countable and thus $(\R^d\times\R)^J$
is metrizable in the product topology.
The proof will be complete if we can show that~$f$ is a topological
embedding. To see that~$f$ is injective, let $x,y\in M$
such that $f(x)=f(y)$.
There is $j\in J$ such that $x\in P_j$.
Then $f(x)=f(y)$ implies that $1=g_j(x)=g_j(y)$,
whence $y\in\Supp(g_j)\sub Q_j$ and thus $h_j(y)=1=h_j(x)$.
Now $\phi_j(x)=f_j(x)=f_j(y)=\phi_j(y)$ implies that $x=y$.
It remains to show that $f$ is open onto its image,
which will hold if we can show that $f(U)$ is open in $f(M)$
for each $j\in J$ and open subset $U\sub P_j$.
Let $\pr_j\colon (\R^d\times \R)^J\to\R^d\times\R$
be the projection onto the $j$th component.
Then
\[
f(U)=f(M)\cap\pr_j^{-1}(\phi_j(U)\times {\textstyle \,]0,\infty[})
\]
is indeed open in $f(U)$. $\,\square$\\[2.3mm]
We can say more for
manifolds modelled on locally convex sets.
Recall that a continuous map $X\to Y$ between
Hausdorff spaces is called \emph{proper}
if all preimages of compact subsets of~$Y$
are closed in~$X$. If~$Y$ is a $k$-space,
then every proper map is a closed mapping,
viz., each closed subset of~$X$ has a closed image
in~$Y$ (see \cite{Pal}).
\begin{prop}\label{W-embed}
Let $M$ be a $\sigma$--compact, locally
compact $C^\ell$-manifold
with rough boundary,
with $\ell\in\N\cup\{\infty\}$.
If~$M$ is modelled on~$\R^d$,
then there exists a $C^\ell$-map $M\to \R^{2d+2}$
which is a proper map and a topological embedding.
\end{prop}
The following lemma is useful for the proof
(cf.\ \cite[Lemma after Theorem 1.3.5]{Hir}
for manifolds without boundary). We let $\lambda_n$ be Lebesgue-Borel
measure~on~$\R^n$.
\begin{la}\label{sard}
Let $m\in\N$ and
$M$ be a $\sigma$-compact,
locally compact $C^1$-manifold with rough boundary,
such that $M$ is modelled on~$\R^m$.
\begin{itemize}
\item[\rm(a)]
If $n>m$ and $f\colon M\to \R^n$ is a $C^1$-map,
then $f(M)$ is a Borel set of measure
$\lambda_n(f(M))=0$.
\item[\rm(b)]
If $g\colon M\to N$ is a $C^1$-map to
a $C^1$-manifold~$N$ with rough boundary
modelled on~$\R^n$ with $n>m$,
then $N\setminus f(M)$ is dense in~$N$.
\end{itemize}
\end{la}
\begin{proof}
(a) For each $x\in M$, there exists a chart $\phi_x\colon U_x\to V_x$
such that $x\in U_x$ and $V_x$ is a locally convex subset of~$\R^m$
with dense interior. There exists a compact, convex
neighbourhood~$K_x$ of $\phi_x(x)$ in~$V_x$.
Then $f\circ \phi_x^{-1}|_{K_x}\colon K_x\to\R^n$
is a $C^1$-map on a compact, convex, regular subset
and thus Lipschitz continuous,
(see, e.g., \cite[Lemma~1.5.6]{GaN}). The image $f(\phi_x^{-1}(K_x))$
is compact and hence a Borel set.
Since $m<n$,
a standard covering argument with balls
shows that $\lambda_n(f(\phi_x^{-1}(K_x)))=0$.
As $M$ is $\sigma$-compact, there exists a countable subset
$\Phi\sub M$ such that $M=\bigcup_{x\in\Phi}\phi_x^{-1}(K_x)$.
Then $f(M)=\bigcup_{x\in \Phi}f(\phi_x^{-1}(K_x))$
is a Borel set and $\lambda_n(f(M))=0$.\\[2.3mm]
Let $\phi\colon U\to V\sub\R^n$ be a chart for~$N$
and $W:=g^{-1}(U)$. By~(a),
$\phi(g(W))$ is a Borel set of measure $\lambda_n(g(W))=0$,
whence $\R^n\setminus \phi(g(W))$ is dense
in~$\R^n$. As a consequence,
$V^0\setminus \phi(g(W))$ is dense
in $V^0$ and hence in~$V$,
whence $V\setminus \phi(g(W))$
is dense in~$V$. Thus $U\setminus g(W)$ is dense in~$U$.
The assertion follows.
\end{proof}
{\bf Proof of Proposition~\ref{W-embed}.}
Let $J$, $\phi_j\colon U_j\to V_j$, $Q_j$, $P_j$,
$f_j$, $g_j$, $h_j$, and $f$ be as in the proof of Proposition~\ref{are-metri}.
Let $K_1\sub K_2\sub\cdots$ be a compact exhaustion of~$M$
and set $K_{-1}:=K_0:=\emptyset$.\\[2.3mm]
(a) For all $n\in\N$,
choose a $C^\ell$-map $\theta_n\colon M\to\R$
such that $\theta_n(M)\sub [0,1]$, $\theta_n|_{K_{n-1}}=0$
and $\theta_n|_{M\setminus K_n^0}=1$
(cf.\ Lemma~\ref{partu-ex}).
Then $\theta(x):=\sum_{n=1}^\infty\theta_n(x)$
defines a non-negative $C^\ell$-map $\theta\colon M\to\R$
such that
\[
\theta(x)\in [n-1,n]\quad \mbox{for all $n\in\N$ and $x\in K_n\setminus K_{n-1}$.}
\]
Hence, if $\theta(x)\in [n-1,n[$, then $x\in K_n\setminus K_{n-2}$.
By the preceding, $\theta$ is proper.\\[2.3mm]
(b) \emph{It suffices to find an injective
$C^\ell$-map $g\colon M\to \R^{2d+1}$.}\\[2mm]
In fact, then $(g,\theta)\colon M\to\R^{2d+2}$
will be an injective $C^\ell$-map and proper,
hence a closed map and therefore a topological
embedding.\\[2.3mm]
(c) \emph{If $W\sub M$ is open
and $h\colon M\to\R^m$ a $C^\ell$-map
with $m\geq 2d+2$
such that $h|_W$ is injective,
then there exists a $C^\ell$-map $g\colon M\to\R^{m-1}$
with~$g|_W$~injective.}\\[2mm]
In fact, the map
\[
\kappa\colon (W\times W)\setminus \Delta_W\to \Sph_{m-1},\quad
(x,y)\mto \frac{h(x)-h(y)}{\|h(x)-h(y)\|_2}
\]
is $C^\ell$, where $\Delta_W:=\{(x,x)\colon x\in W\}$.
By Lemma~\ref{sard}, there exists $v\in\Sph_{m-1}$
which is not in the image of~$\kappa$,
nor in the image of $-\kappa$ (as this union of
images can be interpreted as the image of a $C^\ell$-map
on the disjoint union of two copies of $(W\times W)\setminus\Delta_W$).
Hence, if we write $\pi_v$ for the orthogonal projection
\[
\R^n=\R v \oplus (\R v)^\perp\to (\R v)^\perp,
\]
then $\pi_v(h(x)-h(y))\not=0$ for all $x\not=y$ in~$W$
and thus $\pi_v(h(x))\not=\pi_v(h(y))$.
We can therefore take $g:=\pi_v\circ h$,
identifying $(\R v)^\perp$ with $\R^{m-1}$.\\[2.3mm]
(d) For each $n\in\N$, the set
$J_n:=\{j\in J\colon \Supp(h_j)\cap K_{n+1}\not=\emptyset\}$
is finite. Since $f_j|_{K_{n+1}}=0$
for $j\in J\setminus J_n$, we see that
$(f_j)_{j\in J_n}\colon M\to(\R^d\times\R)^{J_n}$
is a $C^\ell$-map whose restriction to~$K_{n+1}$ is injective.
Thus, there exists a $C^\ell$-map $\psi_n\colon M\to \R^{k_n}$
for an integer $k_n\geq 2d+1$ such that $\psi_n|_{K_{n+1}^0}$
is injective. Iterating~(c), we may assume that $k_n=2d+1$.\\[2.3mm]
(e) We claim that there exists an injective $C^\ell$-map
$h\colon M\to \R^{6d+4}$.
If this is true, then~(c) yields
an injective $C^\ell$-map $g\colon M\to \R^{2d+1}$;
by~(b), this finishes the proof.
To prove the claim, let $L_{-3}:=L_{-2}:=L_{-1}:=L_0:=\emptyset$
and recursively find $\chi_n\in C^\ell_c(M,\R)$ for $n\in\N$
with $\chi_n|_{K_n\setminus K_{n-2}^0}=1$,
$\chi_n(M)\sub[0,1]$, and
\[
L_n:=\Supp(\chi_n)\sub K_{n+1}^0\setminus (L_{-2}\cup\cdots \cup L_{n-3})
\sub K_{n+1}^0\setminus K_{n-3}.
\]
Then $L_n\cap L_m\not=\emptyset$ implies $|n-m|\leq 2$
for all $n,m\in\N$.
We define $\alpha,\beta,\gamma\in C^\ell(M,\R^{2d+1})$ via
\[
\alpha(x):=\sum_{k=1}^\infty \chi_{3k-2}(x)\psi_{3k-2}(x),\;\;
\beta(x):=\sum_{k=1}^\infty \chi_{3k-1}(x)\psi_{3k-1}(x),
\]
and $\gamma(x):=\sum_{k=1}^\infty\chi_{3k}(x)\psi_{3k}(x)$
for~$x\in M$. Then
\[
h:=(\alpha,\beta,\gamma,\theta)\colon M\to \R^{2d+1}\times\R^{2d+1}
\times\R^{2d+1}\times\R
\]
is a $C^\ell$-map and injective. In fact, assume that $h(x)=h(y)$,
with $x,y\in M$.
Then $\theta(x)=\theta(y)\in [n-1,n[$ for some $n\in\N$,
whence $x,y\in K_n\setminus K_{n-2}$ (see~(a))
and thus $\chi_n(x)=1=\chi_n(y)$.
Assume that $n\equiv 0$ mod~$3$ (if $n\equiv 1$ or $n\equiv 2$ mod~$3$,
use $\alpha$ and $\beta$, respectively,
in place of $\gamma$).
Thus $n=3k$ for some $k\in\N$.
Then $\psi_n(x)=\psi_{3k}(x)=\gamma(x)=\gamma(y)=\psi_{3k}(y)=\psi_n(y)$
implies $x=y$. $\,\square$\\[2.3mm]
Of course, the preceding discussion varies
familiar arguments from the proof of the
Whitney Embedding Theorem (notably, cf.\ Theorems~3.4 and 3.5 in \cite{Hir}).\\[2.3mm]
{\bf Proof of Proposition~\ref{onmfdsmoo}.}
If $\ell=0$,
we can construct the operators $S_n$ as in the proof
of Proposition~\ref{smoo-top},
replacing all continuous partitions of unity occurring
in the proof with $C^r$-partitions of unity.
Now assume that $\ell\geq 1$.
Excluding a trivial case,
we may assume that $M\not=\emptyset$.
Let
$K_0:=\emptyset$.
Let $m_0:=0$.
Recursively, we find positive
integers $m_1\leq m_2\leq\cdots$
and full-dimensional compact submanifolds
$Q_j$ with corners of
$K_{n+1}^0\setminus K_{n-1}$
for $j\in \{m_{n-1}+1,\ldots, m_n\}$
such that $Q_j$ is diffeomorphic
to $[0,1]^{d_j}$ for some $d_j\in\N_0$,
and regular compact subsets $P_j\sub Q_j^0$
whose interiors $P_j^0$ cover
\[
K_n\setminus (P_1^0\cup\cdots\cup P_{m_{n-1}}^0)
\]
for $j\in\{m_{n-1}+1,\ldots, m_n\}$.\\[2.3mm]
[In fact, if $n\in \N$ and $m_0,\ldots,m_{n-1}$
and $P_j$ as well as $Q_j$ have been found for
$j\in \{1,\ldots, m_{n-1}\}$,
let $\kappa_{n,x}\colon B_{n,x}\to D_{n,x}$ be a chart for~$M$
around
$x\in K_n\setminus (P_1^0\cup\cdots\cup P_{m_{n-1}}^0)$
such that
$D_{n,x}$ is a relatively open subset of $[0,\infty[^{d(n,x)}$
for some $d(n,x)\in\N_0$.
For suitable $0\leq\alpha_i<\beta_i$
(depending on $n$ and $x$),
we have that
\[
C_{n,x}:=[\alpha_1,\beta_1]\times\cdots\times [\alpha_{d(n,x)},\beta_{d(n,x)}]\sub
D_{n,x}.
\]
and $\kappa_{n,x}(x)$ is in the interior $C_{n,x}^0$
of~$C_{n,x}$ relative~$D_{n,x}$. Then $Q_{n,x}:=\kappa_x^{-1}(C_x)$
is compact and
a full-dimensional submanifold with corners in~$M$
which is diffeomorphic to $[0,1]^d$,
and such that $x\in Q_{n,x}^0=\kappa_{n,x}^{-1}(C_{n,x}^0)$.
Let $P_{n,x}\sub Q_{n,x}^0$ be a compact $x$-neighbourhood.
By compactness,
there is a finite subset $\Psi_n$ of
$K_n\setminus (P_1^0\cup\cdots\cup P_{m_{n-1}}^0)$
with
\[
K_n\setminus (P_1^0\cup\cdots\cup P_{m_{n-1}}^0)
\,\sub\, \bigcup_{x\in\Psi_n}P_{n,x}^0.
\]
Write $x_{m_{n-1}+1},\ldots, x_{m_{n+1}}$
for the elements of $\Psi_n$ (without repetitions)
and abbreviate $Q_j:=Q_{n,x_j}$ and $P_j:=P_{n,x_j}$
for $j\in \{m_{n-1}+1,\ldots, m_n\}$.$\,$]\\[2.3mm]
Let $n_0:=\sup\{n\in\N\colon K_n\setminus K_{n-1}\not=\emptyset\}
\in\N\cup\{\infty\}$. Also, define
$m:=\sup\{m_n\colon n\in\N\}\in \N\cup\{\infty\}$,
$\Phi:=\{j\in\N\colon j\leq m\}$, and the set
$\Phi_n:=\{m_{n-1}+1,\ldots,m_n\}$.
Let $(h_j)_{j\in \Phi}$
be a $C^r$-partition of unity on~$M$ such that $L_j:=\Supp(h_j)\sub P_j^0$
for each $j\in \Phi$ (see Lemma~\ref{partu-ex}).
By Lemma~\ref{smoocube}, for each $j\in \Phi$
there exists a sequence $(S_{j,k})_{k\in\N}$
of continuous linear operators $S_{j,k}\colon
C^\ell(Q_j,F)\to C^r(Q_j,F)$
such that $S_{j,k}(\gamma)\in F\otimes C^r(Q_j,\R)$
for all $\gamma\in C^\ell(Q_j,F)$ and $S_{j,k}(\gamma)\to\gamma$
in $C^\ell(Q_j,F)$ as $k\to\infty$, uniformly for $\gamma$
in compact subsets of $C^\ell(Q_j,F)$.
For $n\in\N$ and $j\in\Phi$,
let $H_{n,j}\colon C^\ell(M,F)\to C^r_{L_j}(M,F)$
be the map given by
\[
H_{n,j}(\gamma)(x):=
\left\{
\begin{array}{cl}
0 & \mbox{ if $x\in M\setminus L_j$,}\\
h_j(x)S_{n,j}(\gamma|_{Q_j}) &\mbox{ if $x\in Q_j^0$}
\end{array}\right.
\]
for $\gamma\in C^\ell(M,F)$ and $x\in M$.
Then $H_{n,j}\in F\otimes C^r_{L_j}(M,\R)$
for all $\gamma\in C^\ell(M,F)$, entailing that
\[
S_n(\gamma):=\sum_{j=1}^{m_n}H_{n,j}(\gamma)
\in F\otimes C^r_{K_{n+1}}(M,\R)
\]
(as required in~(b)), using that $L_j\sub K_{n+1}$ for all $j\leq m_n$.
To see that each $H_{j,n}$
(and hence also $S_n\colon C^\ell(M,F)\to C^r_{K_{n+1}}(M,F)$)
is continuous, note that
the restriction operator
\[
\rho_j\colon C^r_{L_j}(M,F)\to C^r_{L_j}(Q_j^0,F)
\]
is an isomorphism of topological vector spaces
and the restriction maps $C^\ell(M,F)\to C^\ell(Q_j,F)$
and $C^r(Q_j,F)\to C^r(Q_j^0,F)$ are continuous and linear.
Moreover, the multiplication
operator
\[
C^r(Q_j^0,F)\to C^r_{L_j}(Q_j^0,F),\quad
\eta\mto h_j|_{Q_j^0}\cdot \eta
\]
is continuous and linear.
Hence
\[
H_{n,j}(\gamma)=\rho_j^{-1}(h_j|_{Q_j^0}\cdot S_{n,j}(\gamma)|_{Q_j^0})
\]
is continuous and linear in~$\gamma$.
If $\gamma\in C^\ell_{K_n}(M,F)$,
then $Q_j\cap K_n=\emptyset$ for all $j>m_n$,
whence
\[
S_k(\gamma)=\sum_{j=1}^{m_n}H_{k,j}(\gamma)
\in F\otimes C^r_{K_{n+1}}(M,\R)
\]
for all $k\geq n$ (establising (c)).
If $Y\sub C^\ell(M,F)$ is a compact subset,
then $Y_j:=\{\gamma|_{Q_j}\colon \gamma\in Y\}$
is a compact subset of $C^\ell(Q_j,F)$,
as the restriction operator $C^\ell(M,F)\to C^\ell(Q_j,F)$
is continuous.
Thus $S_{n,j}(\gamma|_{Q_j})\to \gamma|_{Q_j}$
in $C^\ell(Q_j,F)$ as $n\to\infty$,
uniformly in $\gamma\in Y$.
As the restriction map
$R_j\colon C^\ell(Q_j,F)\to C^\ell(Q_j^0,F)$,
the multiplication operator $\mu_j\colon C^\ell(Q_j^0,F)\to C^\ell_{L_j}(Q_j^0,F)$,
$\eta\mto h_j|_{Q_j^0}\cdot \eta$
and the inverse $r_j^{-1}$
of the restriction map $C^\ell_{L_j}(M,F)\to C^\ell_{L_j}(Q_j^0,F)$
are continuous linear maps and thus uniformly continuous,
we deduce that $H_{n,j}(\gamma)=(r_j^{-1}\circ \mu_j\circ R_j\circ S_{n,j})(\gamma)$
converges to
\[
(r_j^{-1}\circ \mu_j\circ R_j)(\gamma|_{Q_j})=h_j\cdot\gamma
\]
in $C^\ell(M,F)$ as $n\to\infty$, uniformly in $\gamma\in Y$.
Hence
\[
S_n(\gamma)=\sum_{j=1}^{m_n}H_{n,j}(\gamma)
\to \sum_{j=1}^{m_n} h_j\cdot \gamma
\]
in $C^\ell(M,F)$ as $n\to\infty$, uniformly in $\gamma\in Y$.
This establishes (a) and completes the proof. $\,\square$
\section{Differentiability properties of composition}\label{appcompo}
We now prove Proposition~\ref{strange-compo}
and deduce Proposition~\ref{eval-chain}
from it (in full generality).
The following lemma concerning G\^{a}teaux
differentials will be useful.
\begin{la}\label{via-gateaux}
Let $E$ and $F$ be locally convex spaces,
$R\sub E$ be a regular subset, $\ell\in\N_0\cup\{\infty\}$
and $f\colon R\to F$ be a continuous map.
If
\[
\delta^j_xf(y):=\frac{d^jt}{dt^j}\Big|_{t=0}f(x+ty)
\]
exists for all $j\in \N$ such that $j\leq \ell$
and $(x,y)\in R^0\times E$, and the mappings
$U\times E\to F$ so obtained admit continuous extensions
\[
\delta^j f\colon R\times E\to F,
\]
then $f$ is a $C^\ell$-map.
\end{la}
\begin{proof}
For finite-dimensional~$E$
and $R$ open, we can follow the detailed
outline
in \cite[Exercise~1.7.10]{GaN}
(which profited from the proof of the implication
(3)$\Rightarrow$(1) in Boman's Theorem,
\cite[p.\,26]{KaM}).\\[2.3mm]
In the general case, given $x\in R^0$, $j\in \N$ with
$j\leq\ell$ and $y_1,\ldots, y_j$
in a finite-dimensional vector subspace~$H$ of~$E$ with $x\in H$,
we deduce from the case of finite-dimensional domains that
\[
d^{\,(j)}(f|_{R^0})(x,y_1,\ldots,y_j)
=d^{\,(j)}(f|_{R^0\cap H})(x,y_1,\ldots,y_j)
\]
exists and is symmetric $j$-linear in $y_1,\ldots, y_j\in H$,
entailing that the map
$
d^{\,(j)}(f|_{R^0})(x,\cdot)\colon E^j\to F$
is symmetric $j$-linear and can be recovered
from $\delta^j_xf$ via the Polarization
Formula. However, applying the Polarization Formula
to $\delta^j_xf$ for arbitrary $x\in R$ we get a continuous map
\[
d^{\,(j)}f\colon R\times E^j\to F
\]
which, by the preceding, extends $d^{\,(j)}(f|_{R^0})$.
Hence $f$ is $C^\ell$.
\end{proof}
{\bf Proof of Proposition~\ref{strange-compo}.}
We may assume that $k<\infty$.
Since $C^\infty(R,F)=\pl\, C^\ell(R,F)$
for finite~$\ell$, we may also assume that $\ell\in\N_0$.
We prove by induction on $\ell\in\N_0$
that the assertion holds for all $k\in\N_0$,
and, moreover,
that there are non-negative integers $m_{k,i,j,a_1,\ldots,a_j}$
for $i\in\{0,1,\ldots,k-1\}$,
$j\in\{1,\ldots,k-1\}$ and $a=(a_1,\ldots, a_j)\in \N^j$
with $a_1+\cdots+a_j=k-i$
(independent of $E$, $F$, $X$, $Z$, $R$, $S$, and~$A$)
such that, for all $C^k$-maps
$\gamma$ and $\eta$ as in the proposition, the map
$\zeta\colon A\to C^\ell(R,F)$, $x\mto\gamma(x)\circ\gamma(x)$
has its $k$-th G\^{a}teaux differential $\delta^k_x\zeta(y)$
given by the following expression~($*$):
\[
\delta^k_x\gamma(y)\circ \eta(x)
+\sum_{i=0}^{k-1}\sum_{j=1}^{k-i}\sum_{a_1+\cdots+a_j=k-i}
\!\!\!\!\! m_{k,i,j,a}
d^{\,(j)}(\delta^i_x\gamma(y))\circ (\eta(x),
\delta_x^{a_1}\eta(y),\ldots,\delta_x^{a_j}\eta(y)).
\]
Let $\ell=0$ and $R$ be a topological space which is
a $k_\R$-space. If $k=0$ and $A$ is a $k_\R$-space,
the assertion holds by Proposition~\ref{compocseq}(d).
Now assume that $k>0$ and assume the assertion
holds for $k-1$ in place of~$k$.
As we assume that $R$ is a $k_\R$-space,
the map
\[
C(R,E)\to\prod_{K\in\cK(R)}C(R,E),\;\,
\zeta\mto (\zeta|_K)_{K\in\cK(R)}
\]
is linear and a topological embedding with closed
image. We therefore only need to show that the mappings
$A\to C(K,F)$, $x\mto\gamma(x)\circ\eta(x)|_K$
are $C^k$. We may therefore assume now that~$R$
is compact.
For $x\in A^0$, $y\in Z$, and $0\not=t\in\R$ with $x+ty\in A^0$, we have
\begin{eqnarray*}
\lefteqn{\frac{\gamma(x+ty)\circ\eta(x+ty)-\gamma(x)\circ\eta(x)}{t}}\qquad\qquad\\
&=&\frac{\gamma(x+ty)\circ \eta(x+ty)-\gamma(x)\circ\eta(x+ty)}{t}\\
& & +\frac{\gamma(x)\circ\eta(x+ty)-\gamma(x)\circ\eta(x)}{t}\\
&=& \gamma|_{A^0}^{[1]}(x,y,t)\circ \eta(x+ty)
+\frac{C(R,\gamma(x))(\eta(x+ty))-C(R,\gamma(x))(\eta(x))}{t}
\\
&\to & d\gamma(x,y)\circ\eta(x)+
d(\gamma(x))\circ (\eta(x),d\eta(x,y))
\end{eqnarray*}
in $C(R,F)$ as $t\to 0$,
using Proposition~\ref{compocseq}(d) (or~(c))
to get the convergence of the first summand,
and standard differentiability properties
of maps of the form $C(R,f)$ (as in \cite{GCX})
for the second.
Then $\zeta\colon A\to C(R,F)$, $x\mto \gamma(x)\circ\eta(x)$
is $C^1$ and
\begin{equation}\label{shapediff}
d\zeta(x,y)=d\gamma(x,y)\circ\eta(x)+
d(\gamma(x))\circ (\eta(x),d\eta(x,y)),
\end{equation}
as the right-hand side of this equation defines
a $C(R,F)$-valued continuous
function
of $(x,y)\in A\times Z$, by Proposition~\ref{compocseq}(d).
It remains to note that~(\ref{shapediff})
can be rewritten as
\[
\delta_x^1\zeta(y)=\delta^1_x\gamma(y)\circ\eta(x)+
d(\gamma(x))\circ (\eta(x),\delta^1_x\eta(y)).
\]
Now assume that $\gamma\colon A\to C^{k-1}(S,F)$ and
$\eta\colon A\to C(R,E)$ are $C^{k+1}$
and assume we already know that~$\zeta$ is a $C^k$
and that the formula~($*$) for $\delta^k_x\zeta(y)$ holds.
Then
\[
\delta^{k+1}_x\zeta(y)=\frac{d}{dt}\Big|_{t=0}\delta^k_{x+ty}\zeta(y)
\]
exists for all $x\in A^0$ and $y\in Z$, as
we can apply the case $k=1$ to each
summand in~($*$). Calculating derivatives
with the rule~(\ref{shapediff}),
we obtain the following derivatives
for the summands in~($*$):
The first summand has derivative
\begin{eqnarray*}
\lefteqn{\delta^{k+1}_x\gamma(y)\circ \eta(x)
+d(\delta^k_x\gamma(y))\circ (\eta(x),d\eta(x,y))}\\
&=&
\delta^{k+1}_x\gamma(y)\circ \eta(x)
+d(\delta^k_x\gamma(y))\circ (\eta(x),\delta^1_x\eta(y)).
\end{eqnarray*}
The summand corresponding to $(i,j,a_1,\ldots, a_j)$
has a derivative which is $m_{k,i,j,a}$ times the sum
\begin{eqnarray*}
\lefteqn{d^{\,(j)}(\delta^{i+1}_x\gamma(y))\circ (\eta(x),
\delta_x^{a_1}\eta(y),\ldots,\delta_x^{a_j}\eta(y))}\\
& & +
d^{\,(j+1)}(\delta^i_x\gamma(y))\circ (\eta(x),
\delta_x^{a_1}\eta(y),\ldots,\delta_x^{a_j}\eta(y),\delta^1_x\eta(y))\\[1mm]
& & +\sum_{\theta=1}^j
d^{\,(j)}(\delta^i_x\gamma(y))\circ (\eta(x),
\delta_x^{a_1}\eta(y),\ldots,\delta_x^{a_{\theta-1}}\eta(y),
\delta_x^{a_\theta+1}\eta(y),\delta_x^{a_{\theta+1}}\eta(y),\ldots,
\delta_x^{a_j}\eta(y)).
\end{eqnarray*}
Combining these summands, we get a formula
for $\delta^{k+1}_x\zeta(y)$ analogous to~($*$),
for $(x,y)\in A^0\times R$.
As the formula defines a continuous
function $\delta^{k+1}\zeta$
of $(x,y)\in R\times Z$,
Lemma~\ref{via-gateaux}
shows that $\zeta$ if $C^{k+1}$, with
G\^{a}teaux differential $\delta^{k+1}\zeta$ as just described.
This completes the induction on~$k$
for $\ell=0$.\\[2.3mm]
To pass to general $\ell\in\N_0$,
we use that the map
\[
C^\ell(R,F)\to\prod_{j\leq \ell}C(T^jR,T^jF),\;\,
\phi\mto T^j\phi
\]
is a linear topological embedding with closed
image. We therefore only need to show that the mappings
\[
T^j\circ \zeta\colon A\to C(T^jR,T^jF),\quad
x\mto T^j(\gamma(x)\circ \eta(x))=T^j(\gamma(x))\circ T^j(\eta(x))
\]
are $C^k$ for all $j\leq \ell$.
But this holds by the case $\ell=0$, as the mappings
$T^j\circ \gamma\colon A\to C^k(T^jS,T^jF)$
and $T^j\circ\eta\colon A\to C(T^jR,T^jE)$
are~$C^k$. $\,\square$\\[2.3mm]
{\bf Proof of Proposition~\ref{eval-chain}.}
Let $R:=X:=\{0\}$.
Then evaluation at~$0$,
\[
\ev_0\colon C(R,E)\to E,\quad\theta\mto\theta(0)
\]
is an isomorphism of topological vector spaces.
Since $g\colon U\to V\sub E$ is $C^k$, also the map
\[
\wt{g}:=(\ev_0)^{-1}\circ g\colon U\to C(R,E)
\]
is $C^k$.
Using Proposition~\ref{strange-compo},
we see that the map
\[
\zeta\colon U\to C(R,F),\;\, x\mto f(x)\circ\wt{g}(x)
\]
is $C^k$. Hence also $h=\ev_0\circ \, \zeta$ is~$C^k$. $\,\square$
Helge Gl\"{o}ckner, Universit\"{a}t Paderborn, Warburger Str.\ 100,
33098 Paderborn, Germany; glockner@math.uni-paderborn.de\vfill
\end{document}